\documentclass[a4paper, 11pt]{article}
\RequirePackage{amsthm,amsmath,amsfonts,amssymb}
\RequirePackage{natbib}
\usepackage{fullpage}
\usepackage{authblk}
\usepackage[utf8]{inputenc}
\usepackage{algorithm}
\usepackage{algpseudocode}
\usepackage{mathtools}
\usepackage{bbm, dsfont}
\usepackage[breaklinks,colorlinks,citecolor=blue,urlcolor=blue,linkcolor=blue,hypertexnames=false]{hyperref}
\usepackage[dvipsnames]{xcolor}

\usepackage{graphicx}
\usepackage{url}
\usepackage{tikz}
\usetikzlibrary{arrows}
\usepackage{paralist}
\usepackage{appendix}

\makeatletter
\let\old@float@makebox\float@makebox
\renewcommand{\float@makebox}[1]{
  \color@vbox\normalcolor
    \old@float@makebox{#1}
  \color@endbox}
\makeatother

\newcommand{\independent}{\mbox{${}\perp\mkern-11mu\perp{}$}}
\newcommand{\notindependent}{\mbox{${}\not\!\perp\mkern-11mu\perp{}$}}

\DeclareMathOperator*{\argmin}{argmin}

\newcommand{\mbb}{\boldsymbol}
\newcommand{\prob}{{\mathbb P}}
\newcommand{\pr}{{\mathbb P}}
\newcommand{\R}{{\mathbb R}}

\newcommand{\E}{{\mathbb E}}

\newcommand{\fhat}{\widehat{f}}
\newcommand{\ghat}{\widehat{g}}
\newcommand{\mhat}{\widehat{m}}
\newcommand{\hhat}{\widehat{h}}
\newcommand{\vhat}{\widehat{v}}

\newcommand{\sign}{\mathrm{sgn}}
 
\newcommand{\var}{{\mathrm{Var}}}
\newcommand{\Var}{{\mathrm{Var}}}
\newcommand{\cov}{{\mathrm{Cov}}}  
  
\newcommand{\corr}{{\mathrm{Corr}}}  
\newcommand{\tr}{{\mathrm{tr}}}

\newcommand{\given}{\,|\,}

\newcommand{\sgn}{\mathrm{sgn}}

\newcommand{\ceil}[1]{\left\lceil #1 \right\rceil}
\newcommand{\ind}{\mathbbm{1}}
\newcommand{\RN}[1]{
  \textup{\uppercase\expandafter{\romannumeral#1}}
}
\newcommand{\op}{\mathrm{op}}

\newcommand{\convD}{\overset{d}{\rightarrow}}
\newcommand{\target}{\tau}

\newtheorem{theorem}{Theorem}
\newtheorem{proposition}[theorem]{Proposition}
\newtheorem{assumption}{Assumption}
\newtheorem{lemma}[theorem]{Lemma}

\newtheorem{definition}[theorem]{Definition}
\newtheorem{corollary}[theorem]{Corollary}

\title{The Projected Covariance Measure for assumption-lean variable significance testing}
\date{\today}
\author{Anton Rask Lundborg$^{\dagger \diamond}$,
Ilmun Kim$^{\dagger \star}$,
Rajen D. Shah$^{\dagger}$ and
Richard J. Samworth$^{\dagger}$}
\affil{\normalsize{
	\textit{$^{\dagger}$Statistical Laboratory, University of Cambridge, United Kingdom} \\
	\textit{$^{\diamond}$Department of Mathematical Sciences, University of Copenhagen, Denmark} \\
	\textit{$^{\star}$Department of Statistics and Data Science, Department of Applied Statistics, Yonsei University, South Korea} 
}}

\begin{document}

\maketitle

\begin{abstract}
	Testing the significance of a variable or group of variables $X$ for predicting a response~$Y$, given additional covariates $Z$, is a ubiquitous task in statistics. A simple but common approach is to specify a linear model, and then test whether the regression coefficient for $X$ is non-zero. However, when the model is misspecified, the test may have poor power, for example when $X$~is involved in complex interactions, or lead to many false rejections.  In this work we study the problem of testing the model-free null of conditional mean independence, i.e.\ that the conditional mean of $Y$ given $X$ and $Z$ does not depend on $X$. We propose a simple and general framework that can leverage flexible nonparametric or machine learning methods, such as additive models or random forests, to yield both robust error control and high power. The procedure involves using these methods to perform regressions, first to estimate a form of projection of $Y$ on $X$ and $Z$ using one half of the data, and then to estimate the expected conditional covariance between this projection and $Y$ on the remaining half of the data. While the approach is general, we show that a version of our procedure using spline regression achieves what we show is the minimax optimal rate in this nonparametric testing problem. Numerical experiments demonstrate the effectiveness of our approach both in terms of maintaining Type I error control, and power, compared to several existing approaches.
\end{abstract}

\section{Introduction}
Understanding the relationship between a response and associated predictors is one of the most common problems faced by data analysts across many diverse areas of science and industry. Often a crucial step in this task is to determine which variables or groups of variables are important in this relationship. To fix ideas, consider data formed of independent copies of a triple $(X, Y, Z)$, where $Y \in \R$ is our response, and suppose we wish to assess the significance of a group of predictors $X \in \R^{d_X}$ after adjusting for confounding variables $Z \in \R^{d_Z}$; we will consider a more general setting later in this paper where $X$ and $Z$ can be potentially non-Euclidean. One simple but popular way of addressing this problem is to fit a linear regression model $Y = X^{\top}\beta + Z^{\top}\gamma + \varepsilon$, where we assume that the random error $\varepsilon$ satisfies $\E (\varepsilon \given X, Z)=0$, and perform an $F$-test for the significance of~$X$ (i.e.~test the null hypothesis that $\beta = 0$).  However, in the case that the linear model is not a sufficiently good approximation of the ground truth, this can result in wrongly declaring $X$ to be important or unimportant, and other significance tests based on parametric models suffer from similar issues. The fact that regressions based on parametric models are typically greatly outperformed by modern machine learning methods such as deep learning~\citep{goodfellow2016deep} and random forests~\citep{breiman2001random} in regression competitions such as those hosted by Kaggle~\citep{bojer2021kaggle}, suggests that such parametric models giving poor approximations to the truth is the norm rather than the exception, at least in contemporary datasets of interest.

In this work we consider the model-free null hypothesis of conditional mean independence, that is $\E(Y \given X, Z) = \E(Y \given Z)$; in words, $X$ does not feature in the regression function of $Y$ on $X$ and $Z$. It is interesting to compare this to the conditional independence null $Y \independent X \given Z$, which has attracted much attention in recent years. The latter asks not just for the regression function to be expressed as a function of $Z$ alone, but in fact for the entire conditional distribution of $Y$ given $(X, Z)$ to equal the conditional distribution of $Y$ given~$Z$. Any valid test of conditional mean independence may be used as a test for conditional independence as its size is no larger than its size over the larger null hypothesis of conditional mean independence. The two nulls in fact coincide when $Y$ is binary, but more generally there are important differences. One attractive property of the conditional mean independence null is that the alternative of conditional mean dependence may be characterised by the property that $X$ can improve the prediction of $Y$ in a mean-squared error sense, given knowledge of $Z$. For example, consider the setting where $X$ is a binary treatment variable,~$Z$ contains all pre-treatment confounders and $Y$ is the observed outcome. Under assumptions (including the absence of unmeasured confounders) that are standard in the causal inference literature \citep{Neyman23,Rubin74}, conditional mean dependence is equivalent to the existence of a subgroup average treatment effect, that is a (measurable) subset $\mathcal{A} \subseteq \R^{d_Z}$ where $\E\{\E(Y \given Z, X=1) \given Z \in \mathcal{A}\} > \E\{\E(Y \given Z, X=0) \given Z \in \mathcal{A}\}$.  On the other hand, rejection of the conditional independence null does not in general have an immediate interpretation in terms of its predictive implications.  

Despite the attractions of conditional mean independence, an important issue is that this property is not testable without further restrictions on the null hypothesis: if $(X, Y, Z)$ have a density that is absolutely continuous with respect to Lebesgue measure, then the power of any test at any alternative is at most its size. This comes as a direct consequence of the untestability of the smaller conditional independence null \citep{shah2020hardness}. The conclusion is that in order to test conditional mean independence, one must further constrain the null hypothesis in some way.

Given the success of machine learning methods in prediction problems, a natural and convenient way to specify these constraints is based on restricting the set of nulls to those where user-chosen regression methods can estimate certain conditional expectations sufficiently well.  One strategy, as adopted in the \emph{Generalised Covariance Measure} (GCM) of \citet{shah2020hardness}, involves, in the case where $X$ is univariate, regressing each of $X$ and $Y$ on $Z$, computing the covariance between the resulting residuals and estimating a normalised version of $\E \{\cov(X, Y \given Z)\}$, a quantity that is zero under conditional independence.  A drawback of this approach, however, is that it has no power against alternatives to conditional mean independence where $\E \{\cov(X, Y \given Z)\} =0$.

To gain greater power, \citet{shah2020hardness} suggest to apply the above with $X$ replaced by each component of $\bigl(\phi_1(X, Z), \ldots, \phi_m(X, Z)\bigr)$, where $\phi_1,\ldots, \phi_m : \R^{d_X \times d_Z} \to \R$ are a fixed user-chosen collection of transformations of the data. One may then base a final test on the maximum absolute value of the resulting test statistics. It is however not clear how one should choose these transformations, and if $m$ is large, or indeed $d_X$ is large and we use the strategy above but with the $\phi_j$ simply extracting the $j$th component of $X$, then performing all the regressions involved may be impractical.  A related approach to improve the power properties of the GCM is introduced by  \citet{scheidegger2021weighted}, who propose a carefully-weighted version of the GCM that, under conditions, can have power against alternatives where we do not have $\cov(X, Y \given Z) = 0$ almost surely; see also \citet{fernandez2022general}. Nevertheless, it is perfectly possible to have $\cov(X, Y \given Z) = 0$ under conditional mean dependence, and here even the weighted GCM would be powerless: for example, consider the simple setting where $(X, Z, \varepsilon) \sim N_3(0, \mbb{I})$ and $Y=X^2 + \varepsilon$.  In this case, $\cov(X, Y \given Z) = \cov(X, Y) = 0$ despite $X$ clearly being important for the prediction of $Y$.
It is therefore of great interest to develop methods for testing conditional mean independence whose validity, as in the case of the GCM and its weighted version, relies primarily on the predictive properties of user-chosen regression methods, but which have power against much wider classes of alternatives.

While there has a great deal of research effort on the problem of conditional independence testing in recent years (we review the contributions most relevant to our work here in Section~\ref{sec:review}), there has been comparatively little on testing conditional mean independence. One compelling approach is based on an equivalent way of stating the null hypothesis: defining 
\begin{align}
    \target &:= \E\bigl[ \bigl\{ \E(Y \given X,Z) - \E(Y \given Z) \bigr\}^2 \bigr], \label{eq:tau}
\end{align}
we have that $\target=0$ if and only if $Y$ is conditionally mean independent of $X$ given $Z$. This suggests a potential strategy for assessing conditional mean independence via the estimation of $\target$. Such an approach was adopted by \citet{williamson2019nonparametric}, who employed a plug-in estimator of $\tau$, and showed that, under conditions, it yields a semiparametric efficient estimator, provided that $\target > 0$.  However, as highlighted by \citet{williamson2019nonparametric}, under the null where $\target=0$, semiparametric approaches such as this face a fundamental difficulty as the influence function is identically zero, and as a consequence the test statistic has a degenerate distribution.

To circumvent this issue, \citet{williamson2020unified} and independently \citet{dai2021}, utilise an alternative representation of the target parameter $\target =\E[ \{Y - \E(Y \given Z)\}^2] - \E [\{Y - \E(Y \given X,Z)\}^2]$ and propose a testing procedure via sample splitting, where estimation of 
$\E[\{Y - \E(Y \given Z)\}^2]$ and $\E[\{Y - \E(Y \given X,Z)\}^2]$ is done on independent splits of the data. This restores the asymptotic normality of the test statistic under the null, but estimating these population quantities separately comes with a significant power loss.
Indeed, even in simple parametric settings, each of the population level quantities $\E[\{Y - \E(Y \given Z)\}^2]$ and $\E[\{Y - \E(Y \given X,Z)\}^2]$ can only be estimated at a $1/\sqrt{n}$ rate, and so the difference of the two estimates each coming from independent samples would also only converge to the true difference $\tau$ at a $1/\sqrt{n}$ rate. As a result, the test becomes asymptotically powerless if $\sqrt{n} \target \rightarrow 0$, even for a parametric linear model where the optimal testing rate is known to be of order $n^{-1}$. Moreover, the asymptotic normality fails when $Y$ is (close to) independent of $(X,Z)$, which raises concerns about uniform validity of the test. See Section~\ref{Section: a discussion of Williamson et al} of the supplementary material for details on these issues.

\subsection{Outline of our approach and contributions} \label{sec:outline}
In view of the considerations above, the goal of this paper is to propose a new framework for testing conditional mean independence that has the following properties:
	\begin{itemize}
		\item \textbf{Flexible Type I error control.} The user should be able to leverage flexible regression methods to ensure validity of the test uniformly over classes of distributions where these methods perform sufficiently well.
		\item \textbf{Rate-optimal power in diverse settings.} The test should have minimax rate-optimal power in both simple parametric models, as well as challenging nonparametric settings, when used with appropriate regression methods.
		\item \textbf{Computationally practical.} The test should involve performing only a small number of regressions.
	\end{itemize}
Our approach is based on the following alternative characterisation of conditional mean independence: $Y$ is conditionally mean independent of $X$ given $Z$ if and only if 
\begin{equation} \label{eq:ortho}
\E\bigl[\bigl\{ Y - \E(Y\given Z) \bigr\} f(X, Z)\bigr] = \E\bigl[\cov\bigl(Y, f(X, Z) \given Z\bigr)\bigr]=0
\end{equation}
for all functions $f$ such that $\E\bigl(f(X, Z)^2\bigr) < \infty$. In words, the residuals $Y - \E(Y\given Z)$ from regressing $Y$ on $Z$ alone are uncorrelated with any square-integrable function of $X$ and $Z$. On the other hand, under an alternative, these residuals should not be pure noise but contain some `signal' that can be exposed via an appropriate $f$ such that the left-hand side of \eqref{eq:ortho} is strictly positive.

To motivate our specific strategy, consider an oracular test statistic that uses knowledge of the conditional expectation $\E(Y \given Z)$: given
independent copies $(X_i, Y_i, Z_i)_{i=1}^n$ of $(X, Y, Z)$ and a function $f$, the random variables~$L_{i}^* := \{Y_i - \E(Y_i\given Z_i)\} f(X_i,Z_i)$ for $i=1,\ldots,n$ are independent and identically distributed, with zero mean under the null.
Writing $\widetilde{L}_{i}^* := \{Y_i - \E(Y_i \given X_i, Z_i)\} f(X_i,Z_i)$, we have that under regularity conditions, the studentised statistic
	\begin{align} \label{Eq: oracle statistic}
		T^* := \frac{\frac{1}{\sqrt{n}} \sum_{i=1}^{n} L_{i}^*}{\sqrt{\frac{1}{n} \sum_{i=1}^{n} \widetilde{L}_{i}^{* 2}}}
	\end{align}
converges to a standard normal distribution under the null, and may thus form the basis of a test. Note that since $\widetilde{L}_{i}^* = L_{i}^*$ under the null, we may alternatively studentise the test statistic using the empirical standard deviation of the $L_{i}^*$; however this version simplifies the derivation to follow.
Different choices of $f$ would lead to different power properties under an alternative. Ideally, we want to maximise the value of the test statistic under an alternative, so we would like $\E (L^*_i) / \sqrt{\Var(\widetilde{L}_i^{*})}$ to be as large as possible. It may be shown (see Proposition~\ref{Proposition: optimal f} in Section~\ref{Section: misc results} of the supplementary material) that this is uniquely maximised, up to an arbitrary positive scaling, by choosing $f(X, Z) = h(X, Z) / v(X, Z)$, where $h(X, Z) := \E(Y \given X, Z) - \E(Y \given Z)$ and $v(X, Z) := \Var(Y \given X, Z)$.  We therefore see that the optimal $f$ is a version of the projection $h$ of $Y$ onto the space of square-integrable functions of $(X, Z)$ that are orthogonal to functions of $Z$, inversely weighted by the conditional variance $v$.

The considerations above suggest the following approach: use one portion of the data to obtain an estimate $\fhat$ of the projection $f$, and then use the remaining data to evaluate a test statistic of the form \eqref{Eq: oracle statistic}, with the unknown conditional expectations there replaced with appropriate regression estimates.  This forms the basis of our proposed test statistic, which we call the \emph{Projected Covariance Measure} (PCM). In fact, it turns out to be advantageous to modify somewhat the basic blueprint described above, for instance by subtracting from $\fhat(X,Z)$ an estimate of its conditional expectation given $Z$, to reduce bias; a complete description of our methodology is given in Section~\ref{Section: Projected covariance measure}.

One important issue to be addressed is the fact that under the null, $h$ is the zero function, and as a consequence, both the numerator and denominator of $T^*$ are zero. This is not immediately problematic for the oracular statistic $T^*$, as one can always decide not to reject the null when the numerator is precisely $0$. However, it might appear to be potentially disastrous for an empirical version of $T^*$, where any bias terms in the numerator could be inflated by division with a denominator that is close to zero. One of our main contributions in this work is to show that by formulating our PCM test statistic appropriately, it has an asymptotic standard Gaussian limit in settings ranging from low- and high-dimensional linear models to fully nonparametric settings. Moreover, we demonstrate empirically that this limiting behaviour can be expected to hold more generally when machine learning methods
such as random forests \citep{breiman2001random} are used for the regressions involved.

The rest of the paper is organised as follows. After reviewing some related literature in Section~\ref{sec:review}, we present in Section~\ref{Section: Projected covariance measure} a full description of the PCM test in Algorithm~\ref{Algorithm: PCM}.  Since, as described above, this is a randomised procedure, we also introduce a derandomised variant in Algorithm~\ref{Algorithm: PCM multi} that we recommend for practical use.  In Section~\ref{Section: Linear models}, we examine the simplest instantiation of our general framework and study testing in the context of low-dimensional linear models. An important revelation of this analysis is that in contrast to the equally general testing frameworks of \citet{williamson2020unified} and \citet{dai2021}, our approach has power against local alternatives where $\tau$ is of order $n^{-1}$. We go on to show that, under conditions, the PCM maintains Type I error control in high-dimensional linear models, even when using an essentially arbitrary machine learning method to estimate the projection~$f$.  In Section~\ref{Section: A general theory}, we present a general theory of the PCM; this theory reveals that both Type I and II errors are controlled, as long as prediction errors of the user-chosen regression procedures are sufficiently small. 
We show in Section~\ref{Section: Series estimators} how our general conditions for Type I error control may be satisfied in a fully nonparametric regression setting when using series estimators for the relevant regressions. A modification of the PCM that incorporates an additional sample split for theoretical tractability enjoys what we show to be the minimax rate optimal separation rate, in terms of having power converging to 1 over classes of alternatives for which $\tau$ in~\eqref{eq:tau} satisfies a lower bound.

In Section~\ref{sec:numerical}, we conduct several simulation experiments that demonstrate the effectiveness of the PCM when used with generalised additive model-based regressions \citep{wood2017} and random forests, in terms of both Type I error control and power. We conclude with a discussion in Section~\ref{sec:discuss} outlining potential future research directions suggested by our work.

All of the sections and results in supplementary material \citep{lundborg2024projected} are preceded by an `S'.  In Sections~\ref{Section: Proofs} and \ref{Section: Auxiliary lemmas}, we include the proofs of all our main results and related auxiliary lemmas. In Section~\ref{Section: linear williamson theoretical comparison}, we provide a detailed discussion of the test proposed by \cite{williamson2020unified}, contrasting it with our method both theoretically and empirically. Section~\ref{Section: Splines} provides a self-contained description of spline regression and related results that we use for our analysis in Section~\ref{Section: Series estimators}.  In Section~\ref{Section: full linear analysis}, we give a more detailed analysis of our results for linear projections in Section~\ref{Section: Linear models}; in particular we derive an exact asymptotic power function of our test. 
Section~\ref{Section: additive models binary} contains the results from additional numerical experiments beyond those included in Section~\ref{sec:numerical}.

\subsection{Literature review} \label{sec:review}
There is a relatively small body of literature that is explicitly concerned with conditional mean independence. Early developments on this topic include the work of \citet{fan1996consistent}, \citet{lavergne2000nonparametric} and \citet{ait2001goodness} from the econometrics community. 
\citet{jin2018testing} propose an approach for testing conditional mean independence in cases where $\mathbb{E}(Y \given Z)$ is a linear function of $Z$, based on the martingale difference divergence proposed by \cite{shao2014martingale}.

Recent years have witnessed an increasing use of machine learning (ML) tools for statistical inference. For example, \cite{chernozhukov2018double} introduce an ML-driven approach for estimating causal parameters in the presence of complex nuisance parameters. \citet{shah2018goodness} and \citet{jankova2020goodness} propose methods for goodness-of-fit testing in high-dimensional (generalised) linear models that involve detecting remaining signal in residuals using ML methods.  More closely related to this work,  \citet{williamson2020unified}, and \citet{dai2021}, propose model-free methods for assessing conditional mean independence that can take advantage of existing ML algorithms.  \citet{williamson2020unified} derive a semiparametrically efficient estimator $\tau$, but recognise the difficulty of testing the null hypothesis that $\tau = 0$ caused by the fact that the efficient influence function is identically zero under the null. This means that their sample-splitting approach lacks validity when $(X,Y,Z)$ are independent, and moreover it turns out that the test may require larger values of $\tau$ than necessary in order to achieve power; see Section~\ref{Section: Linear models linear projection} for a more detailed discussion.  \citet{dai2021} alleviate the Type I error issue by adding noise to their test statistic, but this comes at a further price in terms of power, as pointed out by \citet{verdinelli2021decorrelated}.  
\cite{cai2022model} also propose model-free tests of conditional mean independence; one of their test statistics has a similar form to the one of \cite{williamson2019nonparametric} that compares the predictive performance of two regression models, and they use a permutation approach to calculate a $p$-value.  Another related work is that of \citet{zhang2020floodgate}, who provide a method for constructing confidence intervals for $\target$ in the case where the conditional distribution of $X$ given $Z$ is (almost) known; see also \citet{candes2016panning} and \citet{berrett2020conditional}, who employ similar assumptions in the context of testing conditional independence.

Many existing tests, including ours, determine their critical values based on asymptotic theory derived under the null.  However, most work (implicitly) targets pointwise Type I error control that holds only each fixed null.  This type of pointwise analysis leaves room for the existence of a sequence of null distributions for which the Type I error can be made arbitrarily large.
A classical example is the fact that the $t$-test that has pointwise asymptotic size $\alpha$ for the class of distributions with finite variance, has uniform asymptotic size 1 for the same class of distributions~\citep{romano2004non}.  While it is straightforward to introduce, for instance, moment conditions to restore uniform size control in that problem, we argue that the issue is even more pertinent in the context of testing conditional (mean) independence as there are no canonical choices of restrictions to the null that can yield this form of error control.  In this work, 
we therefore put great emphasis on uniform Type I error control over classes of distributions in order to present more practically-relevant error guarantees.  This uniform analysis is in line with recent work on conditional independence testing such as
\citet{shah2020hardness}, \citet{petersen2021quantile}, \citet{lundborg2021conditional}, \citet{scheidegger2021weighted} and \citet{neykov2021minimax}.

Our work builds on a classical technique, namely sample splitting, that involves partitioning the data into disjoint subsamples for different purposes: roughly speaking, a portion of the data is used for seeking a good direction that potentially contains a high signal and the other portion is used for conducting a test based on the data projected along the given direction. \cite{cox1975note} provides one of the earliest applications of  sample splitting to testing problems. Since then, many inference procedures have been developed by leveraging a similar technique to perform variable selection in high-dimensional models~\citep{wasserman2009high,meinshausen2009p,meinshausen2010stability,shah2013variable}, inference after model selection~\citep{rinaldo2019bootstrapping}, changepoint detection \citep{wang2018high} and inference based on maximum likelihood estimators~\citep{wasserman2020universal}, to name just a few.
In a similar vein, \citet{kim2020dimension} introduce splitting-based procedures that address an issue of degenerate $U$-statistics for high-dimensional inference.
While our main focus is on testing, sample splitting has also been considered for estimation problems, where it typically works as a device to reduce a bias and thus help to obtain a fast (often optimal) convergence rate~\citep{chernozhukov2018double,newey2018cross,wang2020debiased}. Some parts of our work are motivated by \citet{newey2018cross}, who propose cross-fit estimators of functionals involving conditional expectations.

\subsection{Preliminaries and notation}
\label{Section:Notation}
Throughout this paper, we adopt the convention that $0/0 := 0$ and let $\mathrm{sgn}(\cdot)$ denote the sign function on $\mathbb{R}$, with the convention that $\mathrm{sgn}(0) := 0$. We let $x \wedge y := \min(x, y)$ and for $n \in \mathbb{N}$, let $[n]:=\{1, \dots, n\}$. For two sequences $(a_n)$ and $(b_n)$, we write $a_n \asymp b_n$ and $a_n \lesssim b_n$ to mean that there exist $c,C > 0$ such that $0 < c \leq |a_n/b_n| \leq C < \infty$ for all $n$, and $a_n \leq Cb_n$ for all $n$ respectively. For a vector $x \in \mathbb{R}^n$ and $p \in [1,\infty]$, we denote its $\ell_p$ norm by $\|x\|_p$.
The operator norm of a matrix $\mbb{A} \in \mathbb{R}^{n \times m}$ is denoted by $\|\mbb{A}\|_{\mathrm{op}}$ and the maximum and minimum eigenvalues of a symmetric matrix $\mbb{B} \in \mathbb{R}^{n \times n}$ are denoted by $\lambda_{\max}(\mbb{B})$ and $\lambda_{\min}(\mbb{B})$, respectively. We use the notation $z_\alpha$ to denote the $\alpha$th quantile of the standard normal distribution, whose distribution function is denoted by $\Phi$.

In order to present our uniform results on testing, we require some conventions for probabilistic notation used in what follows. Let $(\Omega, \mathcal{F})$ be a measurable space equipped with a family of probability measures $(\mathbb{P}_P)_{P \in \mathcal{P}}$ where $\mathcal{P}$ is a collection of distributions on a Euclidean space. We will permit the family $\mathcal{P}$ to depend on $n$, to allow for settings where the number of parameters grows with $n$, but will typically suppress this in the notation.  Given a family of sequences of random variables $(X_{P,n})_{P \in \mathcal{P},n \in \mathbb{N}}$ on ($\Omega,\mathcal{F}$) whose distributions are determined by $P \in \mathcal{P}$, we write $X_{P,n} = o_\mathcal{P}(1)$ if $\sup_{P \in \mathcal{P}} \prob_P(|X_{P,n}| > \epsilon) \rightarrow 0$ for every $\epsilon >0$. Similarly, we write $X_{P,n} = O_\mathcal{P}(1)$ if, for any $\epsilon>0$, there exist $M_\epsilon,N_\epsilon>0$ such that $\sup_{n \geq N_\epsilon} \sup_{P \in \mathcal{P}} \prob_P(|X_{P,n}| > M_\epsilon) < \epsilon$. In addition, for another family of sequences of random variables $(Y_{P,n})_{P \in \mathcal{P},n \in \mathbb{N}}$, we write $X_{P,n} = o_\mathcal{P}(Y_{{P,n}})$ if there exists $R_{P,n}$ with $X_{P,n} = Y_{{P,n}} R_{P,n}$ and $R_{P,n} = o_{\mathcal{P}}(1)$; likewise, we write $X_{P,n} = O_{\mathcal{P}}(Y_{P,n})$ if $R_{P,n} = O_\mathcal{P}(1)$ in this representation. 
We say that $(X_{P, n})_{P \in \mathcal{P}, n \in \mathbb{N}}$ converges uniformly in distribution to random variable $X$ with distribution function $F$ if for all continuity points $x$ of $F$, we have
\[
\lim_{n \to \infty}   \sup_{P \in \mathcal{P}} \bigl| \prob_P(X_{P, n} \leq x) - F(x) \bigr| = 0 .
\]

We will denote different independent datasets by $\mathcal{D}_1, \mathcal{D}_2,\ldots$, each containing $n$ independent observations. We will frequently abuse notation and write conditional expectations conditioning on a random function, e.g. $\E_P\bigl(\fhat(X, Z) \given \fhat, Z\bigr)$ where $\fhat$ is a function produced by some regression estimator. By this we mean formally that we condition on the sample used to construct the regression estimator and any additional randomness involved in the computation of the regression function.  We let $(X, Y, Z)$ be random variables in $\mathcal{X} \times \mathbb{R} \times \mathcal{Z}$, where $\mathcal{X}$ and $\mathcal{Z}$ are measurable spaces, although we will at times think of $\mathcal{X}$ and $\mathcal{Z}$ being specific $d_X$- and $d_Z$-dimensional Euclidean spaces, respectively.

\section{Projected covariance measure} \label{Section: Projected covariance measure}
In this section, we outline our PCM methodology in detail.  We first provide further motivation in Section~\ref{sec:motivation}, before presenting our final algorithm in Section~\ref{sec:alg}.  Given that our approach involves sample splitting, it is convenient to assume here and also throughout Sections~\ref{Section: Linear models} and \ref{Section: A general theory} that we have $2n$ independent and identically distributed observations $(X_i, Y_i, Z_i)_{i=1}^{2n}$ rather than the conventional $n$ observations.

\subsection{Motivation} \label{sec:motivation}
Recall that the approach sketched in Section~\ref{sec:outline} involves first computing an estimate $\fhat$ of the weighted projection 
\[
f(X, Z) = \frac{h(X, Z)}{v(X, Z)} = \frac{\E(Y \given X, Z) - \E(Y \given Z)}{\Var(Y \given X, Z)}
\]
using one portion of the data, say $\mathcal{D}_2 := (X_i, Y_i, Z_i)_{i=n+1}^{2n}$ (we refer to the remaining data as $\mathcal{D}_1$). We discuss how to construct the estimate $\fhat$ in Section~\ref{sec:alg}. Next, given an estimate $\mhat(\cdot)$ of $m(\cdot) := \E(Y \given Z = \cdot)$, the oracular test statistic \eqref{Eq: oracle statistic} suggests a numerator of our test statistic of the form
\begin{equation} \label{eq:basic_test_stat}
\frac{1}{\sqrt{n}} \sum_{i=1}^n \{Y_i - \mhat(Z_i)\} \fhat(X_i, Z_i).
\end{equation}
We would like this to have mean close to zero under the null; however it is well-known \citep{chernozhukov2018double} that when using a nonparametric estimator $\mhat$, the quantity above may carry a substantial bias and we should instead consider an orthogonalised version of the form
\[
\frac{1}{\sqrt{n}} \sum_{i=1}^n L_i \qquad \text{with} \qquad L_i:=\{Y_i - \mhat(Z_i)\} \{\fhat(X_i, Z_i) - \mhat_{\fhat}(Z_i)\},
\]
where $\mhat_{\fhat}$ is an estimate of $m_{\fhat}(\cdot) := \E(\fhat(X, Z) \given Z = \cdot, \fhat)$. Importantly, the bias term can then be controlled by a product of the mean squared prediction error (MSPE) of $\mhat$,
\begin{equation} \label{eq:in-sample_err}
	\frac{1}{n} \sum_{i=1}^n \{\mhat(Z_i)  - m(Z_i)\}^2,
\end{equation}
and that of $\mhat_{\fhat}$, a quantity that may be substantially smaller than the MSPE of $\mhat$ alone (which would drive the bias in \eqref{eq:basic_test_stat}). 

Turning to the denominator of our test statistic, instead of studentising by a quantity requiring an estimate of $\E(Y \given X, Z)$ as suggested by \eqref{Eq: oracle statistic}, it is practically more convenient to normalise using the empirical standard deviation of $L_1,\ldots,L_n$.  Thus we propose to take as our test statistic
\begin{equation} \label{eq:test_stat}
	T := \frac{\frac{1}{\sqrt{n}} \sum_{i=1}^n L_i}{\sqrt{\frac{1}{n} \sum_{i=1}^n L_i^2 - \bigl(\frac{1}{n} \sum_{i=1}^n L_i \bigr)^2}}.
\end{equation}
For local alternatives, both versions are near-identical and so any differences in power properties should be very slight, as we have also observed empirically.

We choose in practice to train $\mhat$ and $\mhat_{\fhat}$ on $\mathcal{D}_1$ rather than $\mathcal{D}_2$. The errors such as \eqref{eq:in-sample_err} that are required to be controlled are then \emph{in-sample errors}, that is the regression methods are trained on the same data on which they are evaluated, so the regression methods need not extrapolate to unseen data points, for example. While from a theoretical perspective in-sample errors and out-of-sample errors are often thought of similarly, in finite samples, these can behave differently: for example Figure~\ref{fig:gam regression errors} demonstrates that when using additive models (computed using the \texttt{R} package \texttt{mgcv} \citep{wood2017})  to estimate $m$ in a setup considered in Section~\ref{Section: additive models}, out-of-sample errors can be appreciably larger with non-negligible probability.

\begin{figure}
	\centering
	\includegraphics[scale=0.44]{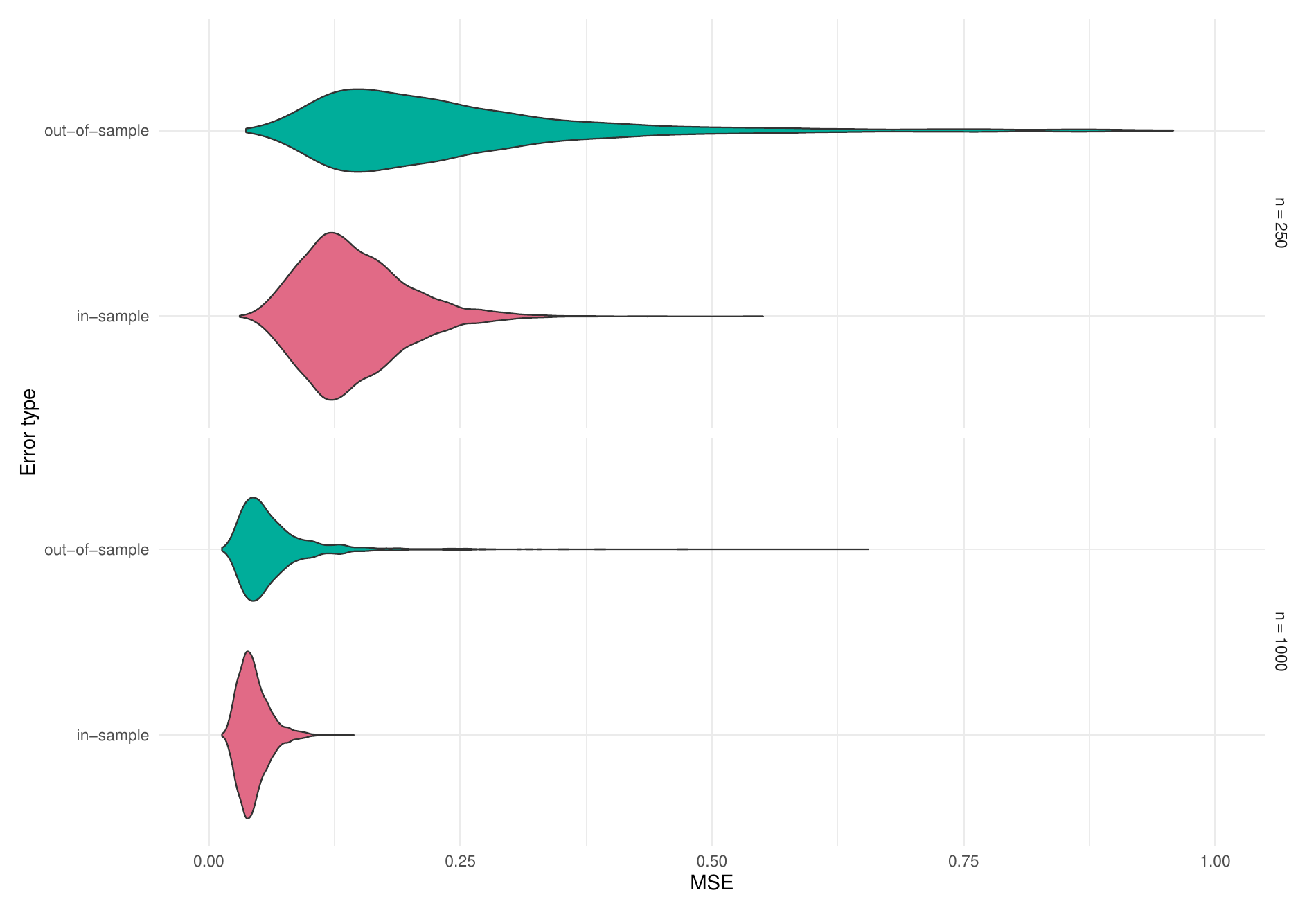}
	\caption{In-sample and out-of-sample errors for the model where $(Z_1, \dots, Z_7) \sim N_7(0, \mbb I)$, $Y=\sin(2\pi Z_1) + \varepsilon$ with $\varepsilon \sim N(0 ,1)$ independently of $Z_1, \dots, Z_7$, and regressions are performed using \texttt{mgcv}; see Section~\ref{Section: additive models} for more details on this setup.}
	\label{fig:gam regression errors}
\end{figure}

As the PCM may be thought of as the GCM applied to a transformed $X$, we would hope to obtain a standard Gaussian limit for $T$ as in the case of the regular GCM test statistic. Given that the transformation is designed to result in large values of $T$ under an alternative, we would perform a one-sided test by rejecting
when $T$ exceeds the appropriate normal quantile. Unfortunately however, the theory that guarantees asymptotic validity of
the GCM test statistic does not apply in our case:
it would require $\Var(\{Y - m(Z)\}\{\fhat(X, Z) - m_{\fhat}(Z)\} \given \fhat)$, i.e.\ the (square of the) target of the denominator, to be bounded away from zero under the null. But $f$ is identically $0$ under the null, so $\fhat$ and hence the above variance, and also both numerator and denominator of our test statistic, should all converge to $0$.

To see why we can expect
a standard Gaussian limit for our test statistic
despite this apparent degeneracy, consider a linear model setting where $(X, Y, Z) \in \R \times \R \times \R^d$ are related through
\begin{equation} \label{eq:lm_simple}
	Y = \beta X + Z^{\top} \mbb\gamma + \varepsilon \qquad \text{and} \qquad X = Z^{\top}\mbb\eta + \xi,
\end{equation}
with $\beta = 0$ and $\E (\varepsilon \given Z) = \E(\xi \given Z) = 0$. If we form estimates $\hhat$ and $\mhat$ using ordinary least squares, and for simplicity set $\vhat \equiv 1$ when forming $\fhat$, then $\fhat(x, z) = \hhat(x, z)$ takes the form $\widehat{\beta} x + z^{\top}\widetilde{\mbb\delta}$ for some $(\widehat{\beta}, \widetilde{\mbb\delta}) \in \R \times \R^d$, where both $\widehat{\beta}$ and $\|\widetilde{\mbb\delta}\|_2$ are
$O_P(n^{-1/2})$.

Let us write $\widehat{\mbb\gamma}$ and $\widehat{\mbb\eta}$ for the regression estimates of $\mbb\gamma$ and $\mbb\eta$ respectively. 
The next step of our procedure involves regressing each of $(Y_i)_{i=1}^n$ and $\bigl(\fhat(X_i, Z_i)\bigr)_{i=1}^n$ onto $(Z_i)_{i=1}^n$.  The residuals from the latter regression take the form $\widehat{\beta}\{Z_i^{\top}(\mbb\eta - \widehat{\mbb\eta}) + \xi_i\}$, so in our case
\[
L_i = \widehat{\beta} \{Z_i^{\top}(\mbb\gamma - \widehat{\mbb\gamma}) + \varepsilon_i\} \{Z_i^{\top}(\mbb\eta - \widehat{\mbb\eta}) + \xi_i\}.
\]
Thus, although $L_i$ and hence its standard deviation would be $O_P(n^{-1/2})$ due to the factor of $\widehat{\beta}$, writing $L_i' := L_i / |\widehat{\beta}|$, we see that our test statistic is of the form $\sgn(\widehat{\beta}) T'$, where $T'$ is a version of $T$ in~\eqref{eq:test_stat} with $L_i$ replaced by $L_i'$. But $L_i'$ is an order 1 quantity (in contrast of $L_i$), so under mild conditions $n^{-1/2}\sum_{i=1}^n L_i'$ will have a non-degenerate Gaussian limit, yielding a standard Gaussian limit for $T'$. As $\widehat{\beta}$ is independent of $T'$, having been constructed on $\mathcal{D}_2$, the final test statistic $T$ will also converge to a standard Gaussian.

While this argument provides a heuristic justification for the asymptotic validity of our proposed test under a simple linear model, there remain challenges in extending the basic intuition of this example to more general settings. In the above, it was possible to isolate the randomness from $\fhat$ simply via the sign of $\widehat{\beta}$, which helps bypass the $0/0$ issue. However, it is by no means straightforward to deal with the limits of the form $0/0$ in a nonparametric setting where $\fhat$ is entangled with other sources of randomness in a complicated way. Moreover, in nonparametric settings one needs to put more effort into ensuring that the convergence rates of $\fhat$, $\mhat$ and $\mhat_{\fhat}$ are fast enough that the bias term is asymptotically negligible. In this process, we are obliged to handle a nested regression problem that has rarely been touched in the literature, with a few exceptions~\citep[e.g.][]{kennedy2020optimal}.

\subsection{PCM algorithm} \label{sec:alg}
Our PCM approach developed in Section~\ref{sec:motivation} is set out in Algorithm~\ref{Algorithm: PCM}, with some recommendations for the constructions of $\hhat$ and $\vhat$ that we discuss in Sections~\ref{sec:hhat} and \ref{sec:vhat} below.
In Section~\ref{sec:multi}, we describe a derandomised variant of the PCM.

\begin{algorithm}\raggedright \caption{Projected Covariance Measure: single sample split} \label{Algorithm: PCM}
	\textbf{Input}: Data $(X_i,Y_i,Z_i)_{i=1}^{2n}$, significance level $\alpha \in (0,1)$,
	partition of $[2n] = \mathcal{I}_1 \cup \mathcal{I}_2$ into index sets $\mathcal{I}_1$ and $\mathcal{I}_2$, each of size $n$.\\
	\textbf{Options}: Regression methods for each of the regressions.\\
	\textbf{Define}: $\mathcal{D}_j = (X_i, Y_i, Z_i)_{i \in \mathcal{I}_j}$ for $j \in [2]$.
	\vskip .3em
	\begin{algorithmic}[1]
		\State \textit{Form $\hhat$.}
		\begin{enumerate}[(i)]
			\item Regress $Y$ onto $(X, Z)$ using $\mathcal{D}_2$ to give fitted regression function $\ghat$.
			\item If $\ghat$ can be modified so that all components involving only $Z$ are set to $0$, let $\widetilde{g}$ be this modified version of $\ghat$. Alternatively, set $\widetilde{g} := \ghat$.
			\item Regress $\widetilde{g}(X, Z)$ onto $Z$ using $\mathcal{D}_2$ to give fitted regression function $\widetilde{m}$, and then set $\widetilde{h}(x, z) := \widetilde{g}(x, z) - \widetilde{m}(z)$.
			\item Compute
						\[
							\widehat{\rho} := \frac{1}{n} \sum_{i \in \mathcal{I}_2} \bigl\{Y_i - \ghat(X_i, Z_i) + \widetilde{g}(X_i, Z_i) - \widetilde{m}(Z_i) \bigr\}\widetilde{h}(X_i, Z_i),
						\]
						and set $\hhat(x, z) := \sgn(\widehat{\rho}) \widetilde{h}(x, z)$,
		\end{enumerate}
		\State \textit{Form $\vhat$.}
		\begin{enumerate}[(i)]
			\item Regress $\{Y-\ghat(X,Z)\}^2$ onto $(X, Z)$ using $\mathcal{D}_2$ to give $\widetilde{v}$.
			\item Define $a: [0, \infty) \to [0, \infty]$ by 
			\[
			a(c) := \frac{1}{n} \sum_{i \in \mathcal{I}_2} \frac{\{Y_i-\ghat(X_i, Z_i)\}^2}{\max\{\widetilde{v}(X_i, Z_i), 0\} +c}.
			\]
			If $a(0) \leq 1$, set $\hat{c}:=0$; otherwise find $\hat{c}$ by solving $a(c) = 1$. Set $\vhat(x, z) := \max\{\widetilde{v}(x, z), 0\} + \hat{c}$.
		\end{enumerate}
		\State \textit{Compute test statistic.}
		\begin{enumerate}[(i)]
			\item Set $\fhat(x, z) := \hhat(x, z)/\vhat(x, z)$ and regress $\fhat(X, Z)$ onto $Z$ using $\mathcal{D}_1$, giving $\mhat_{\fhat}$.
			\item Regress $Y$ onto $Z$ using $\mathcal{D}_1$ to give $\mhat$.
			\item For $i \in \mathcal{I}_1$ set $L_i := \{Y_i - \mhat(Z_i)\}\{\fhat(X_i, Z_i) - \mhat_{\fhat}(Z_i) \} $ and let
			\[
			T := \frac{\frac{1}{\sqrt{n}} \sum_{i \in \mathcal{I}_1} L_i}{\sqrt{\frac{1}{n} \sum_{i \in \mathcal{I}_1} L_i^2 - \bigl(\frac{1}{n} \sum_{i \in \mathcal{I}_1} L_i \bigr)^2}}.
			\]
		\end{enumerate}
		\State \textit{Reject $H_0$ if $T > z_{1-\alpha}$.}
	\end{algorithmic}
\end{algorithm}

\subsubsection{Choice of \texorpdfstring{$\hhat$}{\hat{h}}} \label{sec:hhat}
	We would like $\hhat(X, Z)$ to be close to $h(X, Z) = \E(Y \given X, Z) - \E(Y \given Z)$ in order to maximise the power of the procedure. 	
	There are several ways of estimating $h$, with perhaps the most obvious being simply to take the difference of the estimated regression functions $\widehat{g}$ and $\check{m}$ from regressing $Y$ on each of $(X, Z)$ and $Z$. An alternative approach is based on observing that $h(X, Z) = g(X, Z) - \E\bigl(g(X, Z) \given Z\bigr)$ where $g(X, Z) := \E(Y \given X, Z)$. This suggests subtracting not $\check{m}$ but the output of regressing $\ghat(X, Z)$ onto $Z$. An advantage of this latter approach is that we are free to subtract any function $r$ of $Z$ from $\ghat(X, Z)$ prior to this second regression onto $Z$, as we also have $h(X, Z) = g(X, Z) - r(Z) - \E\bigl(g(X, Z) - r(Z) \given Z\bigr)$. Thus for example if $\ghat(x, z) = \ghat_x(x) + \ghat_z(z)$, then we may form an estimate of $h(X, Z)$ as the residuals from regressing $\ghat_x(X)$ onto $Z$. This second regression can then focus on removing any $Z$ signal in $ \ghat_x(X)$, rather than also having to cancel out $\ghat_z(Z)$.
	We do not make the claim that this always makes a large improvement on the first approach, and indeed for certain regression methods such as ordinary least squares (OLS), both approaches are identical and the `cancellation' is automatic. Nevertheless, we find the approach of Step 1 of Algorithm~\ref{Algorithm: PCM} to be a sensible default choice.
	
	In Step 1(iv) we make a final modification to the estimate thus constructed by potentially flipping its sign. The rationale for this is as follows: under an alternative, we have that $\E [\{Y - \E(Y \given Z)\} h(X, Z)] = \tau > 0$. As a basic check then, we can see if an empirical version of this inequality, with $\hhat$ taking place of $h$ and an estimate of $\E(Y \given Z)$ replacing the population quantity, holds; if not, we can at least flip the sign of $\hhat$. This does not require performing any further regressions to estimate $\E(Y \given Z)$: noting the identity $\E(Y \given Z) = \E\{ \E(Y \given X, Z)-r(Z) \given Z\} + r(Z)$, observe that $\widetilde{m}$ in Step~1(iii) is an estimate of the first of these quantities with $r(Z) = \widehat{g}(X, Z) - \widetilde{g}(X, Z)$, where $\widetilde{g}$ is defined in Step~1(ii). When using OLS for each of the regressions, $\widehat{\rho}$ is guaranteed to be non-negative, so no sign flip is performed.
	
	In high-dimensional settings, we would typically use a sparsity-inducing regression method such as the Lasso \citep{tibshirani1996regression}. Considering the simple case where $X$ is univariate, this can result in the coefficient for $X$ being set exactly to zero, and so the recommended construction of $\hhat$ given above would simply produce the zero function. While not a problem for Type I error control, as our convention (see Section~\ref{Section:Notation}) is not to reject the null when $L_i=0$ for all $i$, it is wasteful in terms of power and a better approach here would be to leave the coefficient for $X$ unpenalised. More generally for multivariate $X$, we can additionally regress on the first principal component of $X$ for example, and leave this unpenalised.

\subsubsection{Choice of \texorpdfstring{$\vhat$}{\hat{v}}} \label{sec:vhat}
A natural way to form $\vhat$ is to regress the square of the residuals from regressing $Y$ onto $(X, Z)$, and this is what we recommend in Step~2(i) of Algorithm~\ref{Algorithm: PCM} to produce $\widetilde{v}$. An issue is that while $v$ is clearly non-negative, and expected to be positive everywhere, $\widetilde{v}$ may in fact be negative.  Equally problematic is the possibility that $\widetilde{v}$ is very close to $0$ at some $(X_i, Z_i)$, as then taking $\vhat = \widetilde{v}$, we would have $\fhat(X_i, Z_i)$ very large and hence $\fhat(X_i, Z_i) - \mhat_{\fhat}(Z_i)$ and $L_i$ may be greatly inflated and dominate the test statistic. To mitigate these problems, we modify $\widetilde{v}$ by taking the positive part of our initial estimate, and then adding a non-negative constant $\hat{c}$. This constant is chosen such that $a(\hat{c})$ (see Step~2(ii) of Algorithm~\ref{Algorithm: PCM}) is at most $1$, the rationale coming from the population level identity $\E [\{Y - \E(Y \given X, Z)\}^2 / v(X, Z)] =1$. We also note that estimation of the conditional variance $v$ is not critical for good power properties. For example,
in Section~\ref{Section: Series estimators} we show that simply setting $\vhat \equiv 1$ delivers minimax rate optimal power in a fully nonparametric setting; 
however the power properties may improve empirically by a constant factor; see Section~\ref{Section: Linear models linear projection}.

\subsubsection{Derandomisation} \label{sec:multi}
The single sample split in Algorithm~\ref{Algorithm: PCM} crucially ensures independence between $\fhat$ and the remaining data $\mathcal{D}_1$, but has the consequence of introducing unwanted additional randomness into the test statistic. A drawback of randomised procedures such as this is that different random seeds may deliver different conclusions. To mitigate this issue, as is very common for methods involving sample splitting, it is recommended to derandomise Algorithm~\ref{Algorithm: PCM} by averaging the test statistics $T^{(1)},\ldots,T^{(B)}$ constructed from multiple random sample splits, as summarised in Algorithm~\ref{Algorithm: PCM multi}. Letting $\varphi$ denote the density function of $N(0,1)$, the choice $q=\alpha^{-1}\varphi(z_{1-\alpha})$ for the threshold for the averaged test statistic $\bar{T}$ is guaranteed to give an asymptotic Type I error of at most $\alpha$ (see Corollary~\ref{Corollary: average of Z values} in Section~\ref{Section: misc results}); when $\alpha = 0.05$, the corresponding threshold is $1.255 z_{0.95}$. However, our empirical experience is that the resulting test is typically very conservative as it is designed to maintain validity even in pathological cases.
\citet{guo2023rank} give a data-driven choice of $q$ that yields an asymptotic size that approaches a chosen level, along with theoretical guarantees; the procedure however is more computationally involved. Following \citet{wang2020debiased}, we therefore use the choice $q=z_{1-\alpha}$ in all of our simulations. This is still expected to be conservative
as by Jensen's inequality, $\bar{T}$ is less than or equal to $T$ in the convex ordering, so for example $\Var(\bar{T}) \leq \Var(T)$; however it gives a reasonable compromise between simplicity and statistical power.  We refer the reader to \citet{guo2023rank} and references therein for alternative derandomisation schemes that may be used here.
\setcounter{algorithm}{0}
\renewcommand{\thealgorithm}{\arabic{algorithm}$^{\mathrm{DR}}$}
\begin{algorithm} \raggedright  \caption{Projected Covariance Measure: \textbf{D}e\textbf{R}andomised} \label{Algorithm: PCM multi}
	\textbf{Input}: Data $(X_i,Y_i,Z_i)_{i=1}^{2n}$, significance level $\alpha \in (0,1)$, number of splits $B$. \\
	\textbf{Options}: Regression methods for each of the regressions; threshold $q$ (default choice $z_{1-\alpha}$).
	\vskip .3em
	\begin{algorithmic}[1]
		\State Form complementary pairs of index sets $\{(\mathcal{I}_1^{(b)}, \mathcal{I}_2^{(b)}): b \in [B]\}$ each of size $n$, where $\mathcal{I}_1^{(b)} \cup \mathcal{I}_2^{(b)} = [2n]$.
		\State For $b \in [B]$, apply Algorithm~\ref{Algorithm: PCM} with index sets $\mathcal{I}_1^{(b)}, \mathcal{I}_2^{(b)}$ to produce test statistic $T^{(b)}$.
		\State Define $\bar{T} := \sum_{b=1}^B T^{(b)} / B$ and reject $H_0$ if $\bar{T} > q$.
	\end{algorithmic}
\end{algorithm}
\renewcommand{\thealgorithm}{\arabic{algorithm}}

\section{Linear models} \label{Section: Linear models}
In this section we study our PCM methodology in the context of a linear model for $Y$ on $X$ and $Z$. We begin with the simplest version of this setup, where we assume that $g(x, z) := \E(Y \given X=x, Z=z)$ is a linear function that we estimate using ordinary least squares. This is not the sort of challenging setting where we would envision applying the PCM in practice, as clearly a $t$-test (modified to account for potential heteroscedasticity) would suffice to test for the significance of $X$. We nevertheless present it to show that in contrast to the general methodologies put forward by \citet{williamson2020unified} and \citet{dai2021}, here our method achieves the same (optimal) separation rate in terms of power as that of the $t$-test.  In Section~\ref{sec:gen_proj} below, we show that for both low- and high-dimensional $Z$, we retain Type I error control even under an arbitrary model for $X$ and when using an essentially arbitrary estimated projection $\fhat$.

\subsection{Linear projection function} \label{Section: Linear models linear projection}
We consider a family $\mathcal{P}$ of joint distributions $P$ of $(X,Y,Z) \in \mathbb{R} \times \mathbb{R} \times \mathbb{R}^d$ satisfying the linear model
\begin{align} \label{Eq: linear model with univariate X}
	Y = \beta_P X + \mbb{\gamma}_P^\top Z + \zeta_P,
\end{align}
where $\beta_P \in \mathbb{R}$ and $\mbb{\gamma}_P \in \mathbb{R}^{d}$ are regression coefficients and $\zeta_P$ is a random noise term with $\E_P(\zeta_P \given X,Z) = 0$.  Here, $Y$ and $X$ are conditionally independent given $Z$ if and only if $\beta_P = 0$.  We further impose the following moment conditions on $\mathcal{P}$.

\begin{assumption} \label{Assumption: OLS}  \leavevmode \normalfont
	\begin{enumerate}[(a)]
		\item There exist $C, \delta > 0$ such that 
		\[
		\sup_{P \in \mathcal{P}} \max\bigl\{ \E_P(\| Z \|_{\infty}^{4+\delta}), \E_P(|Y|^{4+\delta}), \E_P(|X|^{4+\delta}) \bigr\} \leq C.
		\]
		\item $\E_P(ZZ^\top) \in \mathbb{R}^{d \times d}$ is invertible, and writing $\mbb{\eta}_P := \E_P(ZZ^\top)^{-1}\E_P(XZ)$, $\xi_P := X - \mbb{\eta}_P^\top Z$, $\mbb{\Theta}_P := \E_P(ZZ^\top \xi_P^2)$, and $\mbb{\Sigma}_P^{XZ} := \E_P(W W^\top)$ where $W := (X,Z) \in \mathbb{R} \times \mathbb{R}^d$, there exists $c>0$ such that 
		\[
		\inf_{P \in \mathcal{P}} \min\{\var_P(\zeta_P \given X,Z),\var_P(\xi_P), \lambda_{\min}(\mbb{\Sigma}_P^{XZ}), \lambda_{\min}(\mbb{\Theta}_P) \} \geq c.
		\]
	\end{enumerate}
\end{assumption}

\begin{proposition} \label{Proposition: OLS power}
  Let $T$ be computed using Algorithm~\ref{Algorithm: PCM} but with $\vhat \equiv 1$ and where OLS is used for all remaining regressions. Let $\mathcal{P}$ denote a family of distributions satisfying Assumption~\ref{Assumption: OLS} and \eqref{Eq: linear model with univariate X} and define $\mathcal{P}_1(\kappa) := \{ P \in \mathcal{P} : |\beta_P| \geq \kappa / \sqrt{n}\}$.
	Given any $\alpha  \in  (0, 1)$, we have
	\[
	\lim_{\kappa \rightarrow \infty} \lim_{n \to \infty} \inf_{P \in \mathcal{P}_1(\kappa)} \pr_P(T >z_{1-\alpha}) = 1.
	\]
\end{proposition}
In this setting, $\tau_P = \beta_P^2\mathbb{E}_P\bigl(\mathrm{Var}_P(X \given Z)\bigr)$. 
Proposition~\ref{Proposition: OLS power} gives the reassuring conclusion that in the simplest of settings, our general PCM framework, when used with appropriately chosen regression methods, can match up to a constant the power properties of a $t$-test tailored to this setting.  In fact it turns out that the context is simple enough for us to derive an asymptotic power expression for our test.  We present such an analysis in Section~\ref{Section: full linear analysis} for a version of our test that uses $n_1$ and $n_2$ (with $n_1 + n_2 = 2n$) observations in $\mathcal{D}_1$ and $\mathcal{D}_2$ respectively, rather than the equal split that we consider here. This shows that the optimal splitting ratio depends on the unknown signal strength, and therefore supports a default choice of $n_1=n_2=n$ for simplicity. We also provide a simulation study in Section~\ref{Section: linear model sim} where we compare the local power properties of the PCM, the \citet{williamson2020unified} test and the $F$-test with a robust variance estimator.

\subsection{A general estimated projection} \label{sec:gen_proj}
We next consider a situation where the model is unspecified under the alternative, whereas $Y$ has a linear relationship with $Z$ under the null of conditional mean independence.  In this case, it is reasonable to employ a flexible regression method, such as neural networks or random forests, to estimate the projection $f$. 
Our goal here is to identify conditions on estimators, including $\fhat$, under which the proposed test controls the Type I error.  It turns out that, given a specified null model, the problem of testing whether $X$ is significant is closely connected to goodness-of-fit testing for the null model, and we are able to exploit this connection to study the asymptotic Type~I error of the proposed test. 

Consider first the case of low-dimensional $Z$. Let $\mathcal{P}_0$ denote a family of distributions of $(X, Y, Z)$ under the null where $Z \in \mathbb{R}^d$ has an arbitrary distribution and suppose that $Y = \mbb{\gamma}_P^\top Z + \varepsilon_P$, where $\E_P(\varepsilon_P \given X,Z) = 0$.  Then $m_P(z) = \mbb{\gamma}_P^\top Z$ and it is reasonable to use a linear regression model for $\mhat$.  
We will suppose that the regressions yielding $\mhat$ and $\mhat_{\fhat}$ in Algorithm~\ref{Algorithm: PCM} are performed using OLS, whereas we will leave the regression choices involved in the construction of $\fhat$ arbitrary. We make the following assumptions on $\mathcal{P}_0$ to ensure uniform asymptotic normality of the test statistic. 
\begin{assumption} \label{Assumption: a general projection}  \leavevmode \normalfont
	\begin{enumerate}[(a)]
		\item There exist $\delta \in (0,2]$, $c,C>0$ such that $\E_P(\varepsilon_P^2 \given X, Z) \geq c$ and ${\E_P(|\varepsilon_P|^{2+\delta} \given X,Z) \leq C}$ for all $P \in \mathcal{P}_{0}$.  
		\item For $i \in [n]$, let  $u_{n,i}:= \fhat(X_i,Z_i) - \mhat_{\fhat}(Z_i)$ and $\upsilon_{n,i} := u_{n,i}/\bigl(\sum_{i'=1}^{n}u_{n,i'}^2 \bigr)^{1/2}$. Assume that $\max_{i \in [n]} |\upsilon_{n,i}| = o_{\mathcal{P}_0}(1)$ and $\sum_{i=1}^{n} \upsilon_{n,i}^2= 1 + o_{\mathcal{P}_0}(1)$.
		\item Letting $\mbb{\widehat{\gamma}}$ denote the coefficient from the $\mhat$ regression, assume that $\max_{i \in [n]} \|Z_i\|_\infty \cdot \| \mbb{\widehat{\gamma}}  - \mbb{\gamma}_P\|_1 = o_{\mathcal{P}_0}(1)$.  
	\end{enumerate}	
\end{assumption}
Assumption~\ref{Assumption: a general projection}(a) concerns conditional moments of $\varepsilon$, and is used to establish the asymptotic normality of a suitably normalised version of the numerator of $T$. 
In contrast to prior work on goodness-of-fit testing, e.g.~\cite{jankova2020goodness},
we do not assume that the conditional variance of $\varepsilon$ is constant.  Assumption~\ref{Assumption: a general projection}(b) asks for no individual  $|\upsilon_{n,i}|$ to be significantly larger than the others, and, for large enough~$n$, that at least one of $\{u_{n,i}:i \in [n]\}$ is non-zero for all $P \in \mathcal{P}_0$, so $\fhat(X, Z)$ is not constant in $X$. The latter condition is important for establishing the asymptotic normality of our test statistic, but is not crucial for Type I error control.  Indeed, when $u_{n,i} = 0$ for all $i \in [n]$, the test statistic is zero, and we do not reject the null.  Finally, in settings where, for example each $Z_i$ has a uniformly bounded $(2+\delta)$th moment for some $\delta > 0$, we have $\max_{i \in [n]} \|Z_i\|_\infty \lesssim n^{1/(2+\delta)}$ and $\| \mbb{\widehat{\gamma}}  - \mbb{\gamma}_P\|_1 \lesssim n^{-1/2}$ with high probability, and in that case Assumption~\ref{Assumption: a general projection}(c) is satisfied.

\begin{proposition} \label{Proposition: Low-dimensional Z} Let $T$ be computed using Algorithm~\ref{Algorithm: PCM}, where OLS is used to construct both $\mhat$ and $\mhat_{\fhat}$.    Suppose in the above setting that Assumption~\ref{Assumption: a general projection} holds and that all OLS estimators exist almost surely. Then the test statistic  converges to $N(0,1)$ uniformly over $\mathcal{P}_{0}$; i.e.,
	\begin{align*}
		\sup_{P \in \mathcal{P}_{0}} \sup_{t \in \mathbb{R}} \big| \prob_{P}(T \leq t) - \Phi(t) \big| \rightarrow 0.
	\end{align*}  
\end{proposition}

Under the conditions of Proposition~\ref{Proposition: Low-dimensional Z}, the test that rejects the null when $T > z_{1-\alpha}$ is uniformly asymptotically of size $\alpha$.  We also note that the only requirement imposed on the projection~$\fhat$ is that it satisfies Assumption~\ref{Assumption: a general projection}(b).  \cite{jankova2020goodness} also consider this condition, providing supporting empirical evidence in general, and introducing a specific procedure to guarantee that the condition holds. 

We now extend these ideas and the setting described above Assumption~\ref{Assumption: a general projection} to the case where the dimension of $Z$ is potentially larger than the sample size. Here, the least squares estimator is not necessarily well-defined, so we construct $\mhat$ and $\mhat_{\fhat}$ using the Lasso or one of its variants. Letting $\mbb{\widehat{\gamma}}$ denote the coefficients from the $\mhat$ regression, the motivation for this comes from the decomposition
\begin{align}
\label{Eq:deltabias}
	\sum_{i=1}^{n} (Y_i - \mbb{\widehat{\gamma}}^\top Z_i) u_{n,i}  =  \sum_{i=1}^{n} \varepsilon_{P, i} u_{n,i}  - \delta_\mathrm{bias},
\end{align}
where $\delta_\mathrm{bias} := \sum_{i=1}^{n} (\mbb{\widehat{\gamma}} - \mbb{\gamma}_P)^\top Z_iu_{n,i}$.  This bias term is no longer exactly zero as for the least squares estimators considered in Proposition~\ref{Proposition: Low-dimensional Z}, but H\"{o}lder's inequality will nevertheless guarantee that it is sufficiently small for our purposes as long as
\begin{align} \label{Eq: condition on bias}
	\|\mbb{\widehat{\gamma}}  - \mbb{\gamma}_P\|_1 \cdot \bigg\| \sum_{i=1}^{n} Z_i \upsilon_{n,i}  \bigg\|_{\infty} = o_{\mathcal{P}_0}(1).
\end{align}
The next proposition is the analogue for high-dimensional $Z$ of Proposition~\ref{Proposition: Low-dimensional Z}.
\begin{proposition} \label{Proposition: High-dimensional Z}
  Let $T$ be computed using Algorithm~\ref{Algorithm: PCM}, where Lasso regressions are used to construct both $\mhat$ and $\mhat_{\fhat}$.	Suppose in the above setting that Assumption~\ref{Assumption: a general projection} and condition~(\ref{Eq: condition on bias}) hold. Then 
	\begin{align*}
		\sup_{P \in \mathcal{P}_{0}} \sup_{t \in \mathbb{R}} \big| \prob_{P}(T \leq t) - \Phi(t) \big| \rightarrow 0.
	\end{align*}  
\end{proposition}

In order to ensure that condition~(\ref{Eq: condition on bias}) holds, one can use the square-root Lasso~\citep{belloni2011square, sun2012scaled}, as suggested by \cite{jankova2020goodness}. In particular, for $\lambda_{\mathrm{sq}} >0$, we set $\mhat_{\fhat}(z) = \mbb{\widehat{\eta}}_{\mathrm{sq}}^\top z$ where
\begin{align*}
	\mbb{\widehat{\eta}}_{\mathrm{sq}} := \argmin_{\mbb{\eta} \in \mathbb{R}^d} \bigg\{ \biggl( \frac{1}{n} \sum_{i=1}^n \bigl( \fhat(X_i,Z_i) - \mbb{\eta}^\top Z_i \bigr)^2 \biggr)^{1/2} + \lambda_{\mathrm{sq}} \| \mbb{\eta}\|_1 \bigg\}.
\end{align*}
With this choice of $\mbb{\widehat{\eta}}_{\mathrm{sq}}$ and by letting $\lambda_{\mathrm{sq}} = c_{\mathrm{sq}}\sqrt{(\log d_Z)/n}$ for some constant $c_{\mathrm{sq}} > 0$, the Karush--Kuhn--Tucker conditions for the square-root Lasso guarantee that $\| \sum_{i=1}^{n} Z_i \upsilon_{n,i} \|_{\infty} \leq c_{\mathrm{sq}} \sqrt{\log d_Z}$. Furthermore, under appropriate conditions, the Lasso estimator~$\mbb{\widehat{\gamma}}$ has an error bound $\|\mbb{\widehat{\gamma}}  - \mbb{\gamma}_P\|_1 \lesssim s \sqrt{(\log d_Z)/n}$ with high probability, where $s$ denotes the number of non-zero coefficients of $\mbb{\gamma}_P$~\citep[e.g.~Corollary 6.2 of][]{buhlmann2011statistics}. Therefore, in this setting, condition~(\ref{Eq: condition on bias}) is satisfied provided that $s (\log d_Z) / \sqrt{n} \rightarrow 0$.

\section{General theory} \label{Section: A general theory}
In this section, we present general conditions ensuring uniform asymptotic validity and power of the test, primarily by imposing assumptions on the performance of the regressions involved. 
The results here hold for Algorithm~\ref{Algorithm: PCM} under various different possible additional assumptions, including that the estimator $\mhat$ satisfies a certain stability condition.  Alternatively, they also hold without this requirement for Algorithm~\ref{Algorithm: PCM theory}, where we employ additional sample splitting to ensure that $\mhat$ and $\mhat_{\fhat}$ are computed on independent data $\mathcal{D}_3$ and~$\mathcal{D}_4$.

\begin{algorithm}\raggedright \caption{Projected Covariance Measure: theoretical version} \label{Algorithm: PCM theory}
	\textbf{Input}: Data $(X_i,Y_i,Z_i)_{i=1}^{4n}$, significance level $\alpha \in (0,1)$,
	partition of $[4n] = \mathcal{I}_1 \cup \mathcal{I}_2 \cup \mathcal{I}_3 \cup \mathcal{I}_4$ into index sets $\mathcal{I}_1$, $\mathcal{I}_2$, $\mathcal{I}_3$ and $\mathcal{I}_4$, each of size $n$.\\
	\textbf{Options}: Regression methods for each of the regressions.\\
	\textbf{Define}: $\mathcal{D}_j = (X_i, Y_i, Z_i)_{i \in \mathcal{I}_j}$ for $j \in [4]$.
	\vskip .3em
	\begin{algorithmic}
 \State Perform Algorithm~\ref{Algorithm: PCM} but in steps 3 (i) and 3 (ii), perform the regressions on $\mathcal{D}_3$ and $\mathcal{D}_4$, respectively, instead of on $\mathcal{D}_1$.
	\end{algorithmic}
\end{algorithm}

Algorithm~\ref{Algorithm: PCM theory} could be made more data-efficient by exchanging the roles of $\mathcal{D}_1$, $\mathcal{D}_3$ and $\mathcal{D}_4$ and averaging the resulting test statistics, a process known as cross-fitting \citep{chernozhukov2018double}. However, for the reasons discussed in Section~\ref{sec:motivation} we do not recommend Algorithm~\ref{Algorithm: PCM theory} in practice and thus do not pursue this idea further. With the exception of Theorem~\ref{Theorem: Power of Spline} in Section~\ref{Section: Series estimators}, our results for Algorithm~\ref{Algorithm: PCM theory} remain valid if $\mathcal{D}_3 = \mathcal{D}_4$ but we have omitted this special case for brevity.

The following quantities relating to the performances of the regression methods $\mhat$ and $\mhat_{\fhat}$ will play a key role in our results. Let us introduce
\begin{equation}
	\label{eq: residuals}
	\varepsilon_{P,i}:=Y_i - m_P(Z_i), \qquad \xi_{P,i}:= \fhat(X_i,Z_i) - m_{P, \fhat}(Z_i),
\end{equation}
for $i \in [n]$, and an analogous version of \eqref{eq: residuals} without a subscript $i$.  Further, define
\begin{equation}
		\label{Eq: conditional variance}
	\sigma_P^2 := \var_{P}\bigl(\xi_P \given \fhat \bigr),
\end{equation}
as well as
\begin{equation}
	\mathcal{E}_{P, 1} := \frac{1}{n} \sum_{i=1}^{n} \{m_P(Z_i) - \mhat(Z_i)\}^2  \quad \text{and} \quad \mathcal{E}_{P, 2} := \frac{1}{n\sigma_P^2 } \sum_{i=1}^{n} \{m_{P, \fhat}(Z_i) - \mhat_{\fhat}(Z_i)\}^2.
\end{equation}
The second MSPE $\mathcal{E}_{P,2}$ in the display above is normalised by the variance of the errors $\xi_{P, i}$ featuring in the corresponding regression. Under the null, we expect this variance to be small as $\fhat$ is then estimating a zero function, and consequently $\mathcal{E}_{P, 2}$ may be inflated relative to the unnormalised version of this quantity. On the other hand, as $\fhat$ is small, we can expect that the unnormalised MSPE is particularly small. For example, writing $\mathcal{P}$ for the simple null linear model considered in \eqref{eq:lm_simple}, we would have
\[
\frac{1}{n}\sum_{i=1}^{n} \{m_{P, \fhat}(Z_i) - \mhat_{\fhat}(Z_i)\}^2 = O_{\mathcal{P}}(n^{-2}) \qquad \text{and} \qquad 1/\sigma_P^2= O_{\mathcal{P}}(n),
\]
giving $\mathcal{E}_{P, 2} = O_{\mathcal{P}}(n^{-1})$. 

Throughout this section we will impose the following assumption:
  \begin{assumption}
    \label{Assumption: general procedure}
  Assume that the class of distributions $\mathcal{P}$ of $(X, Y, Z)$ on $\mathcal{X} \times \mathbb{R} \times \mathcal{Z}$ satisfies:
  \begin{enumerate}[(a)]
    \item $\mathcal{E}_{P, 1} = o_{\mathcal{P}}(1)$, $\mathcal{E}_{P, 2} = o_{\mathcal{P}}(1)$ and their product satisfies $\mathcal{E}_{P, 1} \mathcal{E}_{P, 2} = o_{\mathcal{P}}(n^{-1})$.
    \item The weighted MSPE satisfies
    \[
      \frac{1}{n \sigma_P^2} \sum_{i=1}^{n}  \{m_P(Z_i) - \mhat(Z_i)\}^2 \xi_{P,i}^2 = o_{\mathcal{P}}(1).
    \]
    \item There exists $C > 0$ such that $\sup_{P \in \mathcal{P}} \Var_P(Y \given X, Z) \leq C$. 
  \end{enumerate}
  \end{assumption}
Assumption~\ref{Assumption: general procedure}(a) should be regarded as the primary restriction on $\mathcal{P}$, and along with Assumption~\ref{Assumption: general procedure}(b), relates directly to the performance of the user-chosen regression methods involved in the construction of the PCM. As alluded to above, in a simple linear model setting, even under the null we can expect $\mathcal{E}_{P, 1} \mathcal{E}_{P,2} = O_{\mathcal{P}}(n^{-2})$, which certainly satisfies the condition. The rate requirement on the product of MSPEs is however sufficiently slow to also accommodate nonparametric models.  It is well-known that an optimal convergence rate in terms of the squared prediction error under H\"{o}lder smoothness~$s$ (Definition~\ref{Definition: Holder class}) is $n^{-2s/(2s + d_Z)}$~\citep[e.g.,][]{nemirovski2000topics,gyorfi2002distribution}. If we assume that $m$ and $\sigma_P^{-1} m_{\fhat}$ have H\"{o}lder smoothness $s$, and that $\mhat_{\fhat}$ is scale equivariant as in Theorem~\ref{Theorem: general power result} below, then Assumption~\ref{Assumption: general procedure}(a) is satisfied provided that $s \ge d_Z/2$.  We note that the deterministic condition $\sup_{P \in \mathcal{P}} \{\E_P(\mathcal{E}_{P,1}) \E_P(\mathcal{E}_{P,2})\} = o(n^{-1})$ is sufficient to guarantee the product error condition in Assumption~\ref{Assumption: general procedure}(a), as can be verified via Markov's inequality and the Cauchy--Schwarz inequality. If in addition $\Var_P(\xi_{P} \given Z, \fhat) \leq C \sigma_P^2$, then Assumption~\ref{Assumption: general procedure}(b) is guaranteed to hold.  Assumption~\ref{Assumption: general procedure}(c) is standard in nonparametric regression (and is obviously satisfied when $Y$ is bounded, for instance).  

\subsection{Type I error control} \label{Section: Type I error for a general procedure}
The following result pertains to Type~I error control of the PCM test. One condition under which Algorithm~\ref{Algorithm: PCM} guarantees Type~I error control is where the estimator $\mhat$ does not change all that much on average over the training observations when individual data points are removed. To this end, we say that $\mhat$ is \emph{sufficiently stable} if conditions~(ii) and~(iii) of Proposition~\ref{Prop: stability} in Section~\ref{Section: Auxiliary lemmas} hold.  Similar stability conditions have recently been used by, e.g.~\citet{soloff2023bagging} and \citet{chen2022debiased}, and have been verified for classes of regularised empirical risk minimisation \citep{bousquet2002stability}, stochastic gradient descent \citep{hardt2016train} and bagged estimators \citep{chen2022debiased}, for example.
\begin{theorem}
\label{Theorem: General Type I error control}
Let $\mathcal{P}_0$ denote a class of null distributions, i.e.~that satisfy $\E_P(Y \given X, Z) = \E_P(Y \given Z)$ for $P \in \mathcal{P}_0$, and suppose that Assumption~\ref{Assumption: general procedure} holds for $\mathcal{P}_0$. Suppose in addition that
  \begin{enumerate}[(a)]
    \item $\sup_{P \in \mathcal{P}_{0}}\prob_{P}(\sigma_P^2 = 0) = o(1)$;
    \item there exists $\delta \in (0,2]$ such that $\E_P \bigl(| \varepsilon_{P} \xi_{P} |^{2+\delta} \given \fhat \bigr) / \sigma_P^{2+\delta} = o_{\mathcal{P}_0}(n^{\delta/2})$;
    \item there exists $c >0$ such that $\inf_{P \in \mathcal{P}_0}\E_P(\varepsilon_P^2 \given X, Z) \geq c$;
  \end{enumerate}
  and any of the following hold: 
  \begin{enumerate}[(i)]
      \item $T$ is computed using Algorithm~\ref{Algorithm: PCM theory};
      \item $T$ is computed using Algorithm~\ref{Algorithm: PCM} and $Y \independent X \given Z$ for all $P \in \mathcal{P}_0$;
      \item $T$ is computed using Algorithm~\ref{Algorithm: PCM} and $\mhat$ is a linear smoother;
      \item $T$ is computed using Algorithm~\ref{Algorithm: PCM} and $\mhat$ is a sufficiently stable estimator.
  \end{enumerate}
  Then 
  \begin{align*}
    \sup_{P \in \mathcal{P}_{0}} \sup_{t \in \mathbb{R}} \bigl| \prob_{P}(T \leq t) - \Phi(t) \bigr| \rightarrow 0.
   \end{align*}
\end{theorem}
The proof of Theorem~\ref{Theorem: General Type I error control} can be found in Section~\ref{Section: Proof of General Type I error control}, which formalises the brief explanation of asymptotic normality laid out in Section~\ref{sec:motivation}. 

Assumption~(a) of Theorem~\ref{Theorem: General Type I error control} asks that asymptotically $\fhat(X, Z)$ is not exactly constant in $X$ (even though we may expect it not to vary too much with $X$, for the reasons explained in the discussion in Section~\ref{sec:motivation}). 
Assumption~(b) is a conditional Lyapunov condition, and is used to apply the central limit theorem for triangular arrays.
A sufficient condition for this to hold when $\delta < 2$ is that $\sup_{P \in \mathcal{P}_0} \E_P (|\varepsilon_P|^{(8+4\delta)/(2-\delta)}) < \infty$ and $\E_P (\xi_P^{4} \given \widehat{f}) / \sigma_P^{4} = o_{\mathcal{P}_0}(n^{2\delta/(2+\delta)})$, which can be verified by H\"{o}lder's inequality with conjugate exponents $p=4/(2-\delta)$ and $q=4/(2+\delta)$. The latter condition holds under an $L_2$--$L_4$ norm equivalence \citep[e.g.][]{mendelson2020robust} for $\xi_P$, and is certainly satisfied if $\xi_P$ is log-concave conditional on $\widehat{f}$~\citep[Thm.~5.22]{lovasz2007geometry}.  Assumption~(c) is mild, and ensures in particular that $Y$ is not completely determined by $X$ and $Z$.

Theorem~\ref{Theorem: General Type I error control} indicates that the asymptotic normality of $T$ (hence the validity of the PCM test) is largely determined by the predictive performance of regression models used in the construction of the test statistic.  Moreover our numerical results in Sections~\ref{sec:numerical} and Section~\ref{Section: additive models binary} demonstrate that Type I error when applying Algorithm~\ref{Algorithm: PCM} continues to be well-controlled even in settings where $X \notindependent Y \given Z$, $\widehat{m}$ is not a linear smoother and where it is unclear whether $\widehat{m}$ is sufficiently stable.

\subsection{Power properties} \label{sec:power}

The following is our main result on the power of the PCM test.
\begin{theorem}
	\label{Theorem: general power result}
    Suppose that Assumption~\ref{Assumption: general procedure} holds over a class of alternative distributions $\mathcal{P}_1$. Let $(\epsilon_n)_{n \in \mathbb{N}}$ be a positive sequence and let
  \[
  \mathcal{P}_1(\epsilon_n) := \{P \in \mathcal{P}_1 : \tau_P \geq \epsilon_n\}.
  \]
   Assume that 
  \begin{enumerate}[(a)]
    \item $\mhat_{\fhat}$ is scale equivariant in the sense that $\mhat_{a \cdot \fhat}(Z) = a \cdot \mhat_{\fhat}(Z)$ for all $a > 0$;
    \item $\sup_{P \in \mathcal{P}_1} h_P(X, Z) \leq C$;
    \item $\epsilon_n \cdot n \to \infty$;
    \item There exists $\rho > 0$ such that
    \begin{equation*}
      \sup_{P \in \mathcal{P}_1(\epsilon_n)} \pr_P\Bigl( \corr_P\bigl(h_P(X, Z), \xi_P \given \fhat \bigr)  \leq \rho\Bigr)  = o(1).
    \end{equation*}
  \end{enumerate}
  Suppose further that either
  \begin{enumerate}[(i)]
    \item $T$ is computed using Algorithm~\ref{Algorithm: PCM theory};
    \item $T$ is computed using Algorithm~\ref{Algorithm: PCM} and $\mhat$ is sufficiently stable.
  \end{enumerate}
Then for any $\alpha \in (0, 1)$,
\[
	\inf_{P \in \mathcal{P}_1(\epsilon_n)}\prob_P (T > z_{1-\alpha}) \rightarrow 1.
\]
\end{theorem}
In addition to the primary restrictions on the MSPEs imposed by Assumption~\ref{Assumption: general procedure}, we now ask in part~(d) of Theorem~\ref{Theorem: general power result} for $\xi_P$, the population residual from regressing our estimated projection $\fhat$ onto $Z$, to be positively correlated with $h_P$ with high probability (there is no need for the correlation to approach $1$).  To interpret this, it is helpful to consider a stronger version of this condition with $\fhat(X, Z)$ replacing $\xi_P$; this results in a stronger condition because $\E_P\bigl(h_P(X, Z) \xi_P \given \fhat \bigr) = \E_P\bigl(h_P(X, Z) \fhat(X, Z) \given \fhat\bigr)$ and $\E (\xi_P^2 \given \fhat) = \E_P\bigl[\Var_P \bigl(\fhat(X, Z) \given Z, \fhat\bigr) \given \fhat\bigr] \leq \E\bigl( \fhat(X, Z)^2 \given \fhat\bigr)$.  This stronger assumption still permits $\fhat$ to be an inconsistent estimator of the true $f_P(X, Z) = h_P(X, Z) / \Var_P(Y \given X, Z)$ in that it only requires them to be positively correlated, with probability approaching one.  The flexibility afforded by this assumption relies on the regression method $\mhat_{\fhat}$ being scale equivariant in the sense of the assumption in part~(a) of Theorem~\ref{Theorem: general power result}; this is a mild condition, that is satisfied by many regression methods.  Nevertheless, in Section~\ref{Section: Series estimators} below, we will see that condition~(d) in Theorem~\ref{Theorem: general power result} provides the main constraint on the rate at which $\epsilon_n$ can converge to zero.  Finally, we observe by a union bound that the conclusion of Theorem~\ref{Theorem: general power result} also holds for the derandomised version of the PCM in Algorithm~\ref{Algorithm: PCM multi}.

\section{Series estimators} \label{Section: Series estimators}
Following the theory in the previous section for a general regression method, we now provide more concrete results for a specific class of linear smoothers, namely spline estimators.  In particular, our interest is to identify conditions under which our test is uniformly asymptotically valid and attains rate-optimal power in a nonparametric setting. Throughout this section, we assume that $(X, Z) \in [0,1]^{d_X} \times [0, 1]^{d_Z}$ and set $d:=d_X+d_Z$.
In Section~\ref{Section: Splines}, we give a self-contained description of spline spaces and their tensor product B-spline bases, containing all the results that we require for our analysis. Given a spline order $r \in \mathbb{N}$ (i.e.~degree $r-1$) and $N \in \mathbb{N}_0$ equi-spaced interior knots in each dimension, we denote by~$\mathcal{S}_{r,N}^{d_Z}$ the corresponding spline space on $[0,1]^{d_Z}$, and by $\mbb{\phi}^Z$ its $d_Z$-tensor B-spline basis, which consists of $K_Z := (N + r)^{d_Z}$ basis functions.  Writing $\mathcal{S}_{r,N}^{d_X}$ for the corresponding spline space on $[0,1]^{d_X}$ with $d_X$-tensor B-spline basis $\mbb{\phi}^X$, having $K_X := (N+r)^{d_X}$ basis functions, we can define the $d$-tensor product basis $\mbb{\phi}(x,z) := \mbb{\phi}^X(x) \otimes \mbb{\phi}^Z(z)$ for $\mathcal{S}_{r,N}^d$, where $\mbb{u} \otimes \mbb{v} := \mathrm{vec}(\mbb{u}\mbb{v}^\top)$, having $K_{XZ}:=K_XK_Z$ basis functions.  Further, we let~$\mbb{\psi}$ denote the tensor product B-spline basis for $\mathcal{S}_{2r-1,N}^{d_Z}$, and write $\widetilde{K}_Z:=(N+2r-1)^{d_Z}$; the higher order of the spline basis functions that make up $\mbb{\psi}$ affords better approximation properties that turn out to be useful for our theory.
In what follows, when we write that a variable is regressed on a spline basis, we mean that OLS is used for the regression.  In Theorems~\ref{Theorem: Asymptotic normality of Spline} and \ref{Theorem: Power of Spline} below, we use the terminology that $T$ is \emph{computed using spline regressions} if the regressions (following the notation of Algorithm~\ref{Algorithm: PCM}) are as follows:
\begin{itemize}
\item $\vhat \equiv 1$ for simplicity;
\item $\ghat$ is formed by regressing $Y$ onto $\mbb{\phi}(X, Z)$;
\item $\widetilde{m}$ is formed by regressing $\ghat(X, Z)$ onto $\mbb{\phi}^Z(Z)$, or equivalently by regressing $Y$ onto $\mbb{\phi}^Z(Z)$;
\item $\mhat_{\fhat}$ and $\mhat$ are formed by regressing $\fhat(X, Z)$ and $Y$, respectively, on $\mbb{\psi}(Z)$.
\end{itemize}
In addition, we also omit discussion of the sign correction step (Algorithm~\ref{Algorithm: PCM}, Step 1(iv)), since $\widehat{\rho}$ is always non-negative for the estimators given above.  

In Theorem~\ref{Theorem: Asymptotic normality of Spline} in Section~\ref{Section: Type I error control for spline regression} below, we demonstrate that the PCM computed using spline regressions (Algorithm~\ref{Algorithm: PCM} or Algorithm~\ref{Algorithm: PCM theory}) enjoys uniform asymptotic Type I error control under appropriate regularity conditions, while Theorem~\ref{Theorem: Power of Spline} and Proposition~\ref{Proposition: Lower bound} in Section~\ref{Section: power analysis using spline regression} reveal that the PCM (using Algorithm~\ref{Algorithm: PCM theory}) achieves the optimal testing rate for this problem.

\subsection{Type I error control} \label{Section: Type I error control for spline regression}
We start by stating our main distributional assumptions, which rely on the definitions of H\"older spaces $\mathcal{H}_s^d$ and H\"older norms $\|\cdot\|_{\mathcal{H}_s}$ given in Definition~\ref{Definition: Holder class}.
\begin{assumption} \label{Assumption: Nonparametric models} \normalfont 
	Let $\mathcal{P}$ be a class of distributions of $(X, Y, Z)$ on $[0, 1]^{d_X} \times \mathbb{R} \times [0, 1]^{d_Z}$, and for $P \in \mathcal{P}$, let
	$g_P(x,z) := \E_P(Y \given X=x,Z=z)$. Assume that there exist $C \geq 1$ and $c \in (0,1]$ with the following properties:
	\begin{enumerate}[(a)]
		\item For each $P \in \mathcal{P}$, we have $\E_P(\varepsilon_P^2 \given X, Z) \geq c$ and there exists $\delta \in (0,2]$ such that $\E_P(|\varepsilon_P|^{2+\delta} \given X, Z) \leq C$.
		\item For each $P \in \mathcal{P}$, the marginal distribution of $(X,Z)$ is absolutely continuous with respect to Lebesgue measure on $[0,1]^d$, with  density $p_P$ satisfying $\sup_{(x,z) \in [0,1]^d} p_P(x,z) \leq C$ and $\inf_{(x,z) \in [0, 1]^d} p_P(x,z) \geq c$.
		\item Let $s \in (0,r]$ and let $p_{X|Z, P}(\cdot \given z)$ denote the conditional density of $X$ given $Z=z$. Assume that for all $P \in \mathcal{P}$, we have $p_{X|Z, P}(x \given \cdot) \in \mathcal{H}^{d_Z}_s$ for all $x \in [0,1]^{d_X}$, and that $m_P \in \mathcal{H}^{d_Z}_s$ and $g_P \in \mathcal{H}^{d}_s$, with 
		\[
			\max \biggl\{ \sup_{x \in [0, 1]^{d_X}}\|p_{X|Z, P}(x, \cdot)\|_{\mathcal{H}_s}  , \|m_P\|_{\mathcal{H}_s}, \|g_P \|_{\mathcal{H}_s} \biggr\} \leq C.
		\]
	\end{enumerate}
\end{assumption}
Assumption~\ref{Assumption: Nonparametric models} is closely related to other assumptions commonly used in spline regression \citep[e.g.][]{belloni2015some,ichimura2015influence,newey2018cross}.  In order to state our Type I error control result for spline regressions, it will be convenient to define the projection $\mbb{\Pi}:\mathbb{R}^{K_{XZ}} \rightarrow \mathbb{R}^{K_{XZ}}$ by $\mbb{\Pi}(\mbb{x}) \equiv  \mbb{\Pi}(x_1,\ldots,x_{K_{XZ}}) := \mbb{x} - \mbb{1} \otimes \bar{\mbb{x}}$, with $\bar{\mbb{x}} = (\bar{x}_1,\ldots,\bar{x}_{K_Z})$ given by $\bar{x}_k := K_X^{-1}\sum_{\ell=1}^{K_X} x_{(k-1)K_X + \ell}$ for $k \in [K_Z]$.  Let $\mbb{\widehat{\beta}}_{XZ} \in \mathbb{R}^{K_{XZ}}$ denote the coefficients obtained when forming the $\ghat$ regression and similarly let $\mbb{\widehat{\beta}}_Z \in \mathbb{R}^{K_Z}$ denote the coefficients from the $\widetilde{m}$ regression. Using the fact that $\mbb{\phi}^X$ forms a partition of unity (Proposition~\ref{Proposition: b-spline properties}(a)), it follows that if we write $\widehat{\mbb{\beta}} := \mbb{\widehat{\beta}}_{XZ} -  \mbb{1} \otimes \mbb{\widehat{\beta}}_{Z}$, where $\mbb{1} \in \mathbb{R}^{K_X}$ denotes a vector of ones, then
$\fhat(x, z) = \widehat{\mbb{\beta}}^\top \mbb{\phi}(x, z)$. 

\begin{theorem} \label{Theorem: Asymptotic normality of Spline}
	Suppose that Assumption~\ref{Assumption: Nonparametric models} holds for a class of null distributions $\mathcal{P}_0$, i.e.~a class of distributions that also satisfies $\E_{P}(Y \given X,Z) = \E_{P}(Y\given Z)$ for every $P \in \mathcal{P}_0$. Assume that $\sup_{P \in \mathcal{P}_{0}} \prob_{P}(\|\mbb{\Pi}\mbb{\widehat{\beta}}\|_{\infty} = 0) = o(1)$ and that $\mbb{\Lambda}_P := \E_P\bigl\{\cov_P\bigl(\mbb{\phi}(X, Z) \given Z\bigr)\bigr\}$ satisfies 
	\begin{equation}
	\label{Eq:lambdamin}
	\tilde{\lambda}_{\min}(\mbb{\Lambda}_P) := \min_{\mbb{x}\in \mathbb{R}^{K_{XZ}}: \mbb{\Pi}\mbb{x} = \mbb{x}, \| \mbb{x} \|_2 = 1} \mbb{x}^\top \mbb{\Lambda}_P \mbb{x} \geq  \frac{c}{K_{XZ}},
	\end{equation}
	for each $P \in \mathcal{P}_0$, where $c \in (0,1]$ is taken from Assumption~\ref{Assumption: Nonparametric models}.  Finally, suppose that  
	\begin{equation} 
			n  K_{XZ} \biggl\{\widetilde{K}_Z^{-2 s/d_Z} + \frac{\widetilde{K}_Z}{n} \biggr\}^2 \rightarrow 0 \label{Eq: condition for the normality of spline} 
\end{equation}
and
\begin{equation}			\frac{K_{XZ}^{1+2/\delta}}{n} \rightarrow 0 \label{Eq: condition for the normality of spline 2}
	\end{equation}
where $\delta$ is taken from Assumption~\ref{Assumption: Nonparametric models}.  If $T$ is computed using spline regressions according to either Algorithm~\ref{Algorithm: PCM} or Algorithm~\ref{Algorithm: PCM theory}, then
	\begin{align*}
		\sup_{P \in \mathcal{P}_{0}} \sup_{t \in \mathbb{R}} \bigl| \prob_{P}(T \leq t) - \Phi(t) \bigr| \rightarrow 0.
	\end{align*}
\end{theorem}
The proof of Theorem~\ref{Theorem: Asymptotic normality of Spline} amounts to the verification of the conditions of Theorem~\ref{Theorem: General Type I error control}.  In addition to Assumption~\ref{Assumption: Nonparametric models}, Theorem~\ref{Theorem: Asymptotic normality of Spline} imposes several additional conditions. The assumption that $\sup_{P \in \mathcal{P}_{0}} \prob_{P}(\|\mbb{\Pi}\mbb{\widehat{\beta}}\|_{\infty} = 0) = o(1)$ simply avoids degeneracy of the test statistic and is used to show that assumption~(a) of Theorem~\ref{Theorem: General Type I error control} is satisfied.  If this condition is not satisfied, then since we defined $0/0 := 0$ in our test statistic, it can be shown that 
\[
\liminf_{n \rightarrow \infty} \inf_{P \in \mathcal{P}_{0}} \prob_{P}(|
T| \leq t) \geq \Phi(t) - \Phi(-t)
\]
for all $t \geq 0$ (i.e.~$|T|$ is asymptotically stochastically dominated by the absolute value of a standard Gaussian random variable), so the test retains uniform asymptotic Type I error control provided that $\alpha \leq 1/2$.

Condition~\eqref{Eq:lambdamin} can be regarded as a restricted minimum eigenvalue condition; for $\mbb{x} \in \mathbb{R}^{K_{XZ}}$ with $\mbb{\Pi}\mbb{x} = 0$, we have that $\mbb{x}^\top \mbb{\Lambda}_P \mbb{x} = 0$, but it turns out that we are able to restrict attention to the orthogonal complement of this subspace.  Motivation for the form of this condition is provided by the fact that, writing $\mbb{\Sigma}_P := \E_P\bigl( \mbb{\phi}(X,Z) \mbb{\phi}(X,Z)^\top \bigr) \in \mathbb{R}^{K_{XZ} \times K_{XZ}}$, we have
\[
\tilde{\lambda}_{\min}(\mbb{\Lambda}_P) \leq  \tilde{\lambda}_{\min}(\mbb{\Sigma}_P) \leq \lambda_{\max}(\mbb{\Sigma}_P) \leq C 2^d K_{XZ}^{-1}
\]
by Proposition~\ref{Proposition: b-spline properties}(d).  Moreover, Lemma~\ref{Lemma:lambdaminForIndependentXZ} in Section~\ref{Section: misc results} shows that the assumption holds when $X$ and $Z$ are independent.  

Condition~\eqref{Eq: condition for the normality of spline} is used to show that parts~(a) and~(b) of Assumption~\ref{Assumption: general procedure} are satisfied while Condition~\eqref{Eq: condition for the normality of spline 2} is employed to verify assumption~(b) of Theorem~\ref{Theorem: General Type I error control}.  These conditions control the interplay between the growth rate of the number of basis functions, the smoothness $s$ of the regression functions and conditional densities and $\delta$. When choosing the knot spacing to minimise the mean-squared error of the involved regressions, we would choose $\widetilde{K}_Z$ and $K_Z$ of order $n^{d_Z/(2s+d_Z)}$ and $K_X$ of order $n^{d_X/(2s+d_Z)}$. Thus for \eqref{Eq: condition for the normality of spline} to hold, we need $s > d_Z + d_X/2$ and for \eqref{Eq: condition for the normality of spline 2} to hold, we need $\delta > 2(d_X+d_Z)/(2s-d_X)$.  Both conditions could be weakened, at the expense of additional notational complexity, by choosing different knot spacings $N_X$ and $N_Z$ for the $d_X$- and $d_Z$-tensor B-spline bases $\mbb{\phi}^X$ and $\mbb{\phi}^Z$ for our spline spaces $\mathcal{S}_{r,N_X}^{d_X}$ and $\mathcal{S}_{r,N_Z}^{d_Z}$.  Indeed, by taking $N_X$, and hence $K_X$, to be of constant order, while retaining the original choices of $K_Z$ and $\widetilde{K}_Z$, we see that \eqref{Eq: condition for the normality of spline} holds when $s > d_Z$ and \eqref{Eq: condition for the normality of spline 2} holds when $\delta > d_Z/s$ (so it would suffice for Assumption~\ref{Assumption: Nonparametric models}(a) to hold with $\delta = 1$, provided again that $s > d_Z$).

\subsection{Power and minimax lower bound} \label{Section: power analysis using spline regression}

Up until this point, when analysing Algorithm~\ref{Algorithm: PCM theory} it was not crucial that $\mhat$ and $\mhat_{\fhat}$ were formed on separate auxiliary samples. However, this turns out to be helpful in demonstrating the optimality of our test.
To provide insight into the benefits of forming these estimators on separate samples, consider two generic spline estimators $\widehat{g}_1$ and $\widehat{g}_2$ of unknown functions $g_1$ and $g_2$, respectively. Suppose that we would like to choose $\widehat{g}_1$ and $\widehat{g}_2$ to minimise the empirical cross-product error on observations $Z_1,\ldots,Z_n$ that are independent of $\widehat{g}_1$ and $\widehat{g}_2$, given by
\begin{align*}
	\mathcal{E}_{\mathrm{cross}}:= \frac{1}{n} \sum_{i=1}^n \{\widehat{g}_1(Z_i) - g_1(Z_i)\} \{\widehat{g}_2(Z_i) - g_2(Z_i)\}. 
\end{align*}  
A naive way of approaching this problem would be to construct $\widehat{g}_1$ and $\widehat{g}_2$ on the same dataset and to choose the number of spline functions so as to minimise the mean-squared error of each of $\widehat{g}_1$ and $\widehat{g}_2$. The Cauchy--Schwarz inequality then guarantees that the cross-product error is small as long as the mean-squared errors are small. However, this indirect approach returns a potentially suboptimal rate of convergence due to its ``own observation'' bias, which arises from using the same datasets to form $\widehat{g}_1$ and $\widehat{g}_2$. When employing separate auxiliary samples to construct $\widehat{g}_1$ and $\widehat{g}_2$, we can eliminate this bias; thus a more refined analysis of terms like $\mathcal{E}_{\mathrm{cross}}$ that does not directly employ the Cauchy--Schwarz inequality can result in faster convergence rates; see for instance Proposition~\ref{Proposition: product error}.
Our main result in this section is as follows:

\begin{theorem} \label{Theorem: Power of Spline}
	Let $\mathcal{P}$ be a class of distributions satisfying Assumption~\ref{Assumption: Nonparametric models}, and let $\mathcal{P}_1(\epsilon_n) :=  \{P \in \mathcal{P}: \target_P \geq \epsilon_n \}$, where 
	\begin{align} \label{Eq: minimum separation rate in nonparametric model}
		\epsilon_n \cdot n^{\frac{4s}{4s + d}}   \rightarrow \infty.
	\end{align}
	Further, assume that the tuning parameters are chosen such that $K_X \asymp n^{\frac{2d_X}{4s + d}}$ and $K_Z \asymp \widetilde{K}_Z \asymp n^{\frac{2d_Z}{4s + d}}$ and that $r \geq s \geq 3d/4$.  If $T$ is computed using spline regressions according to Algorithm~\ref{Algorithm: PCM theory}, then
	\begin{align*}
		\inf_{P \in \mathcal{P}_1(\epsilon_n)}\prob_P (T > z_{1-\alpha}) \rightarrow 1.
	\end{align*}
\end{theorem}
Theorem~\ref{Theorem: Power of Spline} reveals that when using spline regressions and Algorithm~\ref{Algorithm: PCM theory}, the PCM has uniform asymptotic power 1 over a class of alternatives that are sufficiently separated from the null, as defined by $\mathcal{P}_1(\epsilon_n)$.

We remark that in Theorem~\ref{Theorem: Power of Spline}, we have operated in the context of a known smoothness parameter $s$ for theoretical purposes.  It is possible to construct more involved tests that adapt to unknown smoothness levels following the strategy of \citet{lepskii1991asymptotically} and \citet{ingster2000adaptive}, but we do not pursue this direction further.

The separation rate~(\ref{Eq: minimum separation rate in nonparametric model}) cannot be improved further from a minimax perspective, as illustrated by the following lower bound result.
\begin{proposition} \label{Proposition: Lower bound}
    	Consider a class of distributions, denoted by $\mathcal{P}$, that satisfy Assumption~\ref{Assumption: Nonparametric models}, let $\mathcal{P}_1(\epsilon_n) :=  \{P \in \mathcal{P}: \target_P \geq \epsilon_n \}$ and let $\alpha \in (0,1/2)$.  There exists $c > 0$ such that if $\limsup_{n \rightarrow \infty} \epsilon_n \cdot n^{\frac{4s}{4s + d}} < c$, then any test $\phi$ having uniform asymptotic size at most $\alpha$ satisfies
	\begin{align*}
		\limsup_{n \rightarrow \infty} \inf_{P \in \mathcal{P}_1(\epsilon_n)} \prob_P(\phi = 1) \leq \alpha + 1/2.
	\end{align*}
\end{proposition}
Proposition~\ref{Proposition: Lower bound} complements Theorem~\ref{Theorem: Power of Spline} by showing that when $\target_P$ is a small constant multiple of $n^{-\frac{4s}{4s+d}}$, no test can achieve uniform consistency under H\"{o}lder smoothness. 
The proof of Proposition~\ref{Proposition: Lower bound}, which can be found in Section~\ref{Section: Proof of Proposition: Lower bound}, follows a fairly standard argument \citep[e.g.][]{ingster1987minimax,arias2018remember} that bounds the $\chi^2$-divergence from a fixed null distribution to a mixture of distributions in the alternative class~$\mathcal{P}_1(\epsilon_n)$.

It can be shown that under Assumption~\ref{Assumption: Nonparametric models}, it is possible to satisfy the conditions of Assumption~\ref{Assumption: general procedure} using spline regressions with $K_X, K_Z$ and $\widetilde{K}_Z$, as well as $r$ and $s$, as in Theorem~\ref{Theorem: Power of Spline}. Therefore, Theorem~\ref{Theorem: Power of Spline} and Proposition~\ref{Proposition: Lower bound} reveal that assumption~(d) of Theorem~\ref{Theorem: general power result} can be the primary restriction on the separation rate $\epsilon_n$.

Despite the theoretical benefits described above, using Algorithm~\ref{Algorithm: PCM theory} instead of Algorithm~\ref{Algorithm: PCM} may degrade the practical performance of the algorithm, especially in small-sample scenarios. We therefore recommend using the PCM as in Algorithm~\ref{Algorithm: PCM} (in conjunction with Algorithm~\ref{Algorithm: PCM multi} for derandomisation), and its finite-sample performance is investigated in the next section.

\section{Numerical experiments} \label{sec:numerical}
In this section, we present the results of several simulation experiments that investigate the empirical performances of both the recommended derandomised version of the PCM (see Algorithm~\ref{Algorithm: PCM multi}) with $B=6$ splits, and the version in Algorithm~\ref{Algorithm: PCM}. We compare our tests to various conditional (mean) independence tests in the literature listed below.
\begin{itemize}
	\item[\texttt{gam}] The test based on the default $p$-value for a smooth when fitting a generalised additive model (GAM) using the \texttt{mgcv}-package in R \citep{wood2013p,wood2017}.
	\item[\texttt{wgsc}] The test resulting from applying the approach described in \citet[][Algorithm 3]{williamson2020unified} and employing sample splitting and cross-fitting as implemented in the \texttt{cv\_vim} function from the \texttt{vimp}-package in R (with $K=2$, resulting in $4$ folds).
	\item[\texttt{kci}] The \emph{kernel conditional independence test} \citep{zhang2012kernel} as implemented in the \texttt{KCI} function of the \texttt{CondIndTests} R package \citep{heinze2018}; we use the Bayesian hyperparameter tuning option for sample sizes of at most $500$ for computational stability reasons, and otherwise we use the default parameters.
	\item[\texttt{gcm}] The \emph{Generalised Covariance Measure} (GCM) as described in \citet{shah2020hardness}.
	\item[\texttt{wgcm.fix}]
	The `fixed weight function' variant of the \emph{Weighted Generalised Covariance Measure} (wGCM) \citep{scheidegger2021weighted} as implemented in the \texttt{wgcm.fix} function of the \texttt{weightedGCM} R package; we use $\texttt{weight.num}=7$ as in the simulations of the original paper.
	\item[\texttt{wgcm.est}] The `estimated weight function' variant of the wGCM as implemented in the \texttt{wgcm.est} function of the \texttt{weightedGCM} R-package.
\end{itemize}

In all of our numerical simulations, rejection rates were estimated based on 2500 repetitions. The code for all of our experiments (including those in the supplementary material) is available on GitHub: \url{https://github.com/ARLundborg/pcm_code/}.

\subsection{Additive models}
\label{Section: additive models}
We first investigate Type I error control in settings where both $\E(Y \given Z)$ and $\E(X \given Z)$ are additive functions, and $Z \sim N_7(0, \mbb I)$. For the methods, including the PCM, requiring choices of regression procedures, we use a generalised additive model fitted using \texttt{mgcv}. We employ default parameters for the generalised additive models (as given in the \texttt{smooth.terms} and \texttt{gam} functions in the \texttt{mgcv} package) except that we choose $\lfloor (N-1)/d \rfloor$ basis functions (where $N$ and $d$ are the number of observations and predictors on which the model is trained, respectively).  Since this is the largest number of basis functions per coordinate that can be taken without overparametrisation, this mitigates the risk of oversmoothing; as the fits are penalised, there is little risk of overfitting \citep{wood2017}.  For the $\widetilde{v}$ regression in the PCM, we apply a generalised additive model with logarithmic link function. We consider null settings consisting of $n \in \{250, 500, 1000\}$ independent and identically distributed copies of $(X, Y, Z)$ where  
\[
X = \sin(2\pi Z_1) + 0.1 \xi \quad \text{and} \quad Y = \sin(2\pi Z_1) + \varepsilon.
\]
and errors $\varepsilon$ and $\xi$ are independent $N(0, 1)$ random variables, independent of $Z$. Such a setup is challenging for Type I error control as $X$ and $Y$ are highly correlated yet are conditionally independent given $Z$.  Indeed we see from the left panel of Figure~\ref{fig:gam comparison} that several of the tests are anti-conservative, most notably \texttt{kci} and \texttt{gam}, which we omit from further comparisons as their power properties would be hard to interpret given the high rejection rates under the null. The \texttt{wgcm.est} test is also somewhat anti-conservative, but considerably less so.  In contrast, the derandomised PCM is conservative here. This is to be expected as the calibration following the derandomisation involved in its construction (Section~\ref{sec:multi}) is typically conservative; the single split version appears to have rejection rates close to the nominal 5\% mark for large $n$, as suggested by our theory. 

We investigate the power properties of the PCM in the following settings, where as before, $\varepsilon$ and $\xi$ are independent and independent of $Z$, and moreover $\varepsilon \sim N(0, 1)$.
\begin{enumerate}
	\item $\xi \sim N(0, 1)$, $X = \sin(2\pi Z_1) + \xi$ and $Y = \sin(2\pi Z_1)  + 0.2X^2 + \varepsilon$.
	\item $\xi+1 \sim \textrm{Exp}(1)$, $	X = \sin(2\pi Z_1) - \sin(2\pi Z_1) \xi$ and $Y = \sin(2\pi Z_1)  + 0.4X^2 + \varepsilon$.
	\item $\xi \sim N(0, 1)$, $X = \sin(2\pi Z_1) + \xi$ and $Y = \sin(2\pi Z_1) + 0.4 X^2 Z_2 + \varepsilon$.
\end{enumerate}
The settings are chosen such that in setting 1:~$\E\bigl(\cov(X,Y \given Z)\bigr) = 0$ but $\cov(X, Y \given Z) \neq 0$, in setting 2:~$\cov(X, Y \given Z) = 0$ but $\target \neq 0$ and in setting 3:~like setting 1, $\E\bigl(\cov(X,Y \given Z)\bigr) = 0$ and $\cov(X, Y \given Z) \neq 0$, but there is only an interaction effect. 

From the right-hand panels of Figure~\ref{fig:gam comparison}, we see that the PCM and \texttt{wgsc} exhibit good power in settings 1 and 2. 
\texttt{wgcm.est} also shows appreciable power in setting 1, though as expected has little power in setting 2 where $\cov(X, Y \given Z) = 0$.  It is initially somewhat surprising that the PCM has some power in setting~3 as the alternative is not an additive model, so the $\ghat$ regression is not able to correctly learn $g$. However, due to the estimation of $\vhat$, we do in fact see some power for sufficiently large $n$.

\begin{figure}
    \centering
    \includegraphics[scale=0.44]{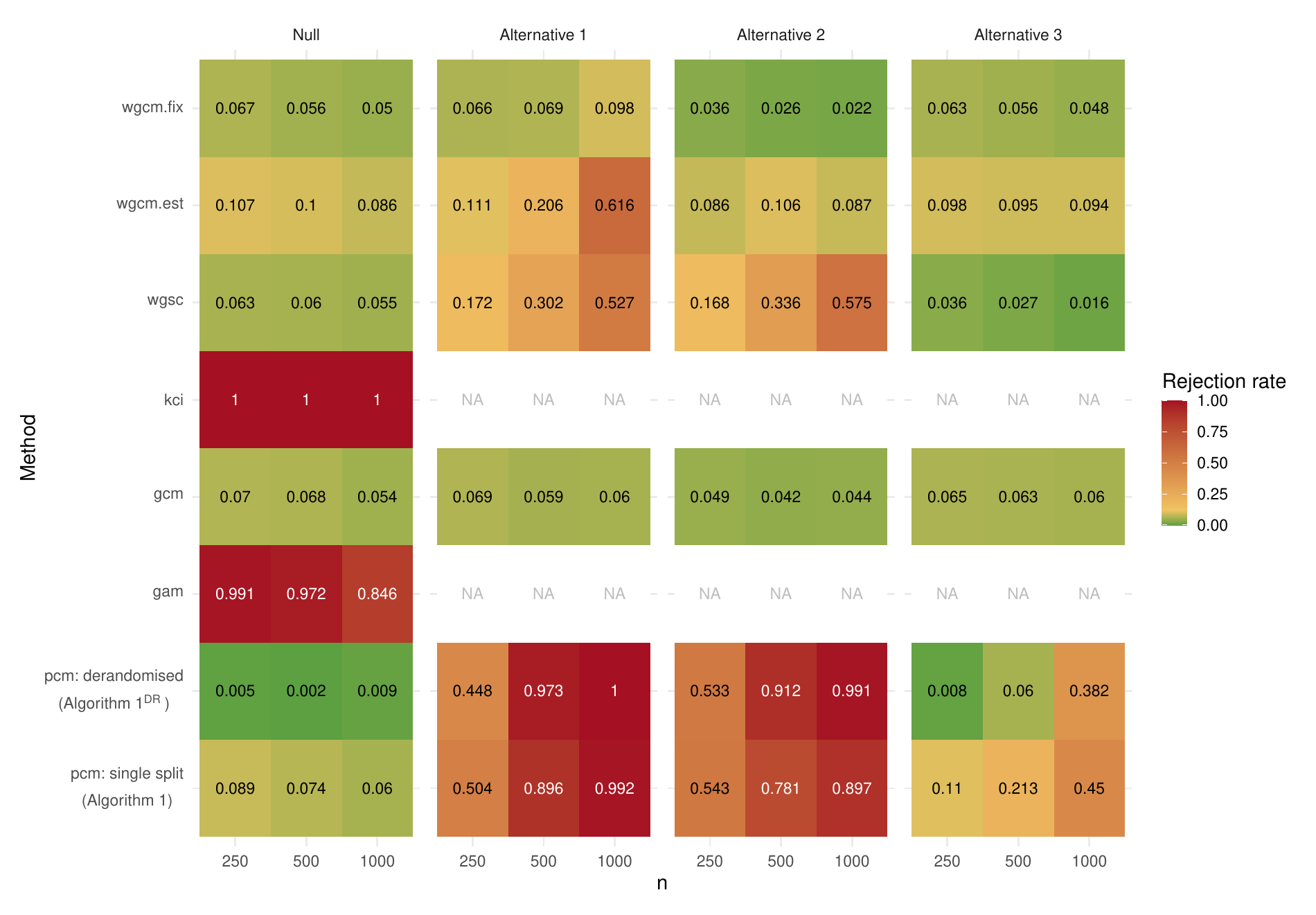}
    \caption{Rejection rates in the various settings considered in Section~\ref{Section: additive models} looking at additive models for nominal 5\%-level tests.  Note that Alternative~3 has only an interaction effect, so we cannot expect methods that fit additive models to have power.}
    \label{fig:gam comparison}
\end{figure}

\subsection{Non-additive models}
\label{Section: non-additive models}
In this section, we consider settings where the regression functions are non-additive and involve complex interactions.  We use random forests \citep{breiman2001random} implemented in the \texttt{ranger} R package \citep{wright2017ranger} as our regression procedure for the methods considered.  We use $500$ trees and set \texttt{mtry} equal to the number of predictors, with other tuning parameters set to be the defaults for the \texttt{ranger} function; the choice of \texttt{mtry} was made as this tended to give the smallest prediction errors in our preliminary experiments.

We consider null settings consisting of $n \in \{10^4, 2\cdot 10^4, 4 \cdot 10^4\}$ independent and identically distributed copies of $(X, Y, Z)$ where $Z \sim N_7(0, \mbb I)$ as before,
\[
X = \sin(\pi Z_1)(1+Z_3) + \xi \quad \text{and} \quad Y=\sin(\pi Z_1)(1+Z_3) + v(X)\varepsilon
\]
with $\varepsilon$ and $\xi$ independent $N(0, 1)$ random variables independent of $Z$, and  $v(X) := 0.5 + \ind_{\{X > 0\}}$ giving heteroscedastic errors for the $Y$ regression model.
The larger sample sizes considered here reflect the difficulty of estimating the more complicated regression functions in these examples.
Note that here we do not have $X \independent Y \given Z$, but the conditional mean independence $\E(Y \given X, Z) = \E(Y \given Z)$ does hold.  The results are presented in Figure~\ref{fig:ranger comparison}.  We see that the derandomised version of the PCM maintains Type I error control, and is in fact slightly conservative.  The \texttt{wgsc} test is highly conservative.  All other approaches considered appear to be anti-conservative to varying degrees, although this behaviour does appear to improve with increasing sample size.

Next we consider the following alternative settings, where as in Section~\ref{Section: additive models}, setting~1 has $\E\bigl(\cov(X, Y \given Z)\bigr)=0$, setting~2 has $\cov(X, Y \given Z) = 0$, and setting 3 has $\E\bigl(\cov(X, Y \given Z)\bigr)=0$ while also involving a pure interaction effect:
\begin{enumerate}
	\item $\xi \sim N(0, 1)$, $X = \sin(\pi Z_1)(1+Z_3) + \xi$ and $Y = \sin(\pi Z_1)(1+Z_3)  + 0.04X^2 + v(X)\varepsilon$;
	\item $\xi \sim \textrm{Exp}(1)$, $	X = \sin(\pi Z_1)(1+Z_3) - \sin(\pi Z_1) (\xi-1)$ and $Y = \sin(\pi Z_1)(1+Z_3)  + 0.04X^2 + v(X)\varepsilon$;
	\item $\xi \sim N(0, 1)$, $X = \sin(\pi Z_1)(1+Z_3) + \xi$ and $Y = \sin(\pi Z_1)(1+Z_3) + 0.04 X^2 Z_2 + v(X)\varepsilon$.
\end{enumerate}
Among the methods considered, here only the PCM appear to have good power across the settings considered.  Despite the fact that the derandomised PCM is conservative, it still obtains greater power in some settings than the slightly anti-conservative single-split version.  The \texttt{wgcm.est} has reasonable power in setting~1, though this should be interpreted with some care given that Type I error is not well controlled in the null settings. However in setting~2, \texttt{wgcm.est} is powerless as expected.

\begin{figure}
    \centering
    \includegraphics[scale=0.44]{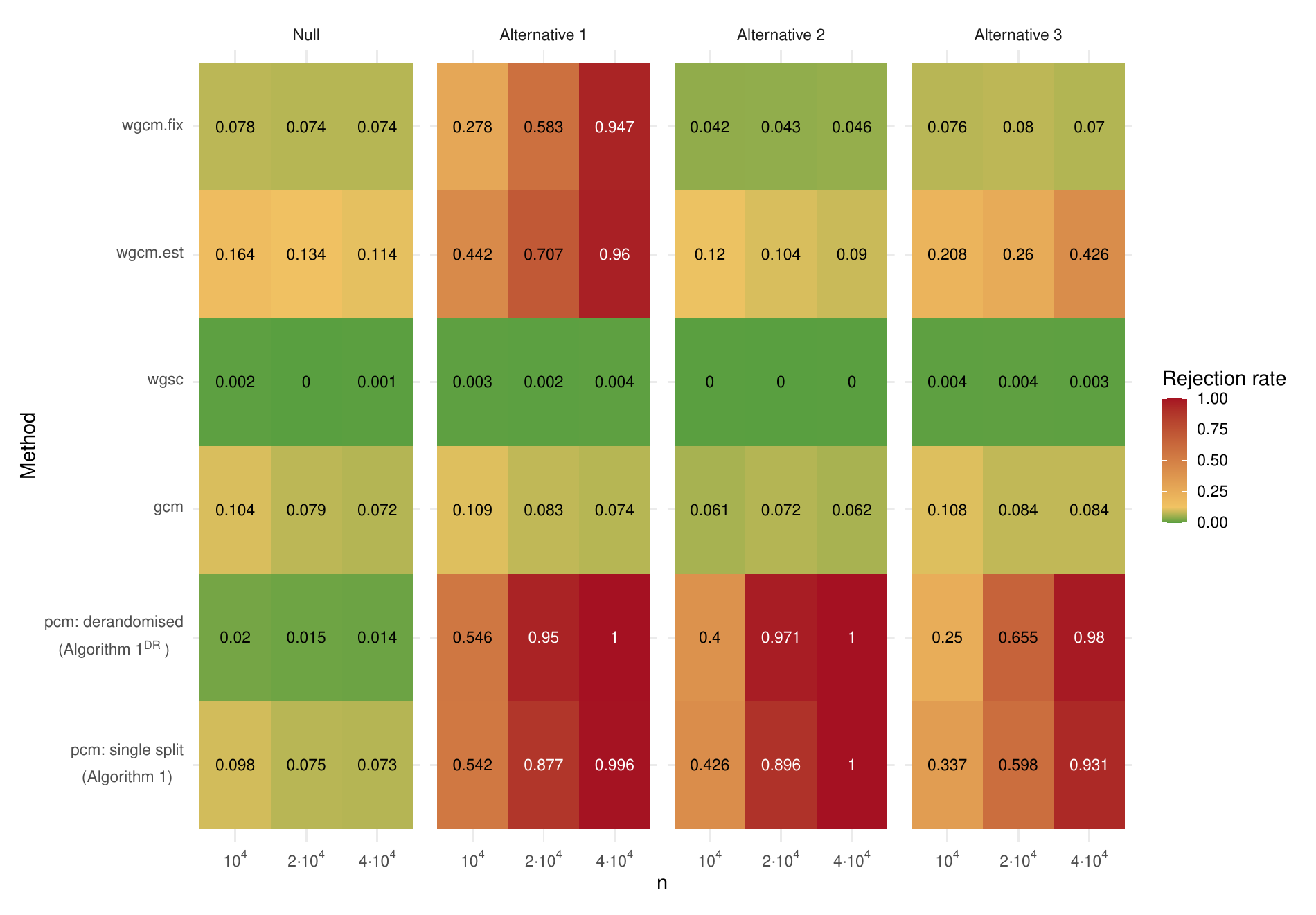}
    \caption{Rejection rates in the various settings considered in Section~\ref{Section: non-additive models} on non-additive models for nominal 5\%-level tests.}
    \label{fig:ranger comparison}
\end{figure}

\section{Conclusion} \label{sec:discuss}
In this work we have introduced a general test statistic called the PCM for testing conditional mean independence that: (a) can leverage machine learning methods to yield provable uniform Type I error control across a class of null distributions where these methods have sufficiently good predictive ability; and (b) when used in conjunction with appropriate regression methods attains rate-optimal power in both the parametric setting of the linear model and fully nonparametric settings. We believe the PCM fills an important gap in the data analyst's range of existing tools, which are unable simultaneously to achieve these desiderata. However, our work also offers several avenues for further work, some of which we mention below.

\paragraph*{Verifying the general assumptions for other regression methods} We have verified Assumption~\ref{Assumption: general procedure} for linear regression in linear model settings, and nonparametric series estimators in fully nonparametric settings.  Since we used the penalised regression splines of \texttt{mgcv} in several of our numerical experiments, it would be interesting to see for what classes $\mathcal{P}$ of distributions Assumption~\ref{Assumption: general procedure} is satisfied in that context. Similarly, it would be very interesting to ask the same question of random forests, which perform very well in our simulations; however this is likely to be challenging given the complex nature of the random forest procedure.

\paragraph*{Conditional independence testing}
Although the problem of testing conditional independence has been studied more intensively than that of testing conditional mean independence, there do not exist many practical conditional independence tests that achieve the two desiderata mentioned at the beginning of this section.  One starting point for constructing such a test may be the fact that the conditional independence null $Y \independent X \given Z$ may be viewed as the intersection of conditional mean independence nulls $\E\bigl(w(Y) \given X, Z\bigr) = \E\bigl(w(Y) \given Z\bigr)$ where function $w$ ranges over all monotone functions, for example. It might therefore be interesting to investigate procedures that seek two `projections': mappings $(X, Z) \to \fhat(X, Z)$ and also $Y \mapsto \widehat{w}(Y)$, after which one may apply the GCM. 

\paragraph*{Confidence intervals}
We have focused on the problem of testing conditional mean independence, but the problem of deriving confidence intervals for a parameter such as $\tau$ that is $0$ under our null is equally interesting. The pioneering work of \citet{williamson2019nonparametric} proposes an asymptotically optimal approach for this in the case where $\tau$ is bounded away from $0$. It would be interesting if the PCM could be used in conjunction with the proposal of \citet{williamson2019nonparametric} to extend the latter to yield confidence intervals with uniform coverage for all $\tau$.

\section*{Acknowledgements}
\sloppy ARL was supported by research grant 0069071 from the Novo Nordisk Fonden.  IK, RDS and RJS were supported by EPSRC programme grant EP/N031938/1; RJS was also supported by European Research Council Advanced grant 101019498. IK was partly supported by the National Research Foundation of Korea (2022R1A4A1033384), and the Korea government (MSIT) (RS-2023-00211073).  The authors are grateful to the anonymous reviewers for their constructive feedback, which helped to improve the paper.

\bibliographystyle{apalike}
\bibliography{reference}

\appendix
\renewcommand{\appendixpagename}{Supplementary material}

\setcounter{section}{0}
\setcounter{equation}{0}
\setcounter{theorem}{0}
\setcounter{assumption}{0}
\setcounter{figure}{0}
\def\theequation{S\arabic{equation}}
\def\thesection{S\arabic{section}}
\def\thetheorem{S\arabic{theorem}}
\def\thefigure{S\arabic{figure}}
\def\theassumption{S\arabic{assumption}}
\newpage 
\appendixpage

In Sections~\ref{Section: Proofs} and \ref{Section: Auxiliary lemmas} of the supplementary material, we include the proofs of all of our main results and related auxiliary lemmas. In Section~\ref{Section: a discussion of Williamson et al}, we present a detailed discussion of the test proposed by \cite{williamson2020unified}. In Section~\ref{Section: Splines}, we give a self-contained description of spline regression and related results that we use for our analysis in Section~\ref{Section: Series estimators}. Section~\ref{Section: full linear analysis} contains an analysis of the linear projections in Section~\ref{Section: Linear models} under more general assumptions. We also derive an exact asymptotic power function of our test in this setting.
Section~\ref{Section: additive models binary} contains the results from additional numerical experiments beyond those included in Section~\ref{sec:numerical}.

Throughout the supplementary material, for a positive semi-definite matrix $\mbb{A}$, we write $\mbb{A}^{-1}$ for its generalised inverse (i.e.~its Moore--Penrose pseudo-inverse).

\section{Proofs}
\label{Section: Proofs}
In our proofs we often suppress the dependence of quantities on $P$ for ease of notation.

\subsection{Proof of Proposition~\ref{Proposition: OLS power}}
There is no loss of generality in assuming that $\alpha \leq 1/2$, because for $\alpha > 1/2$, we have $\mathbb{P}(T > z_{1-\alpha}) \geq \mathbb{P}(T > z_{1/2}) = \mathbb{P}(T > 0)$.  We start by checking the assumptions of Lemma~\ref{Lemma: linear regression} for the regressions of $Y$ on $X$ and $Z$ (of which $\widehat{\beta}$ is one component), $Y$ on $Z$ (yielding $\mbb{\widehat{\theta}}$) and $X$ on $Z$ (yielding $\mbb{\widehat{\eta}}$). Let $\delta' := \delta/2$. 

Recalling that $W = (X, Z)$, we see that condition~(i) of Lemma~\ref{Lemma: linear regression} is satisfied for the $Y$ on $X$ and $Z$ regression by our assumption on $\mbb{\Sigma}^{XZ}$ and the fact that 
\[
	\lambda_{\min}\bigl(\E(WW^\top \zeta^2)\bigr) \geq c 	\lambda_{\min}(\mbb{\Sigma}^{XZ})
\]
by our assumption that $\Var(Y \given X, Z) \geq c$.  Condition~(ii) is satisfied with $\delta = \delta'$ by the Cauchy--Schwarz inequality and Jensen's inequality.

By Assumption~\ref{Assumption: OLS}, condition~(i) of Lemma~\ref{Lemma: linear regression} is satisfied for the $X$ on $Z$ regression. To see that condition~(ii) is satisfied with $\delta = \delta'$, we note that by the Cauchy--Schwarz inequality it suffices to check that $\E(|\xi |^{4+\delta})$ is bounded over $\mathcal{P}$. Letting $\mbb{\Sigma}:= \E(ZZ^\top)$, we have
\[
	\E(|\xi |^{4+\delta}) \leq 2^{3+\delta}\Bigl( \E(|X|^{4+\delta}) + \lambda_{\min}(\mbb{\Sigma})^{-(4+\delta)} \| \E(XZ) \|_2^{4+\delta} \E(\|Z\|_2^{4+\delta}) \Bigr)
\]
which is bounded under Assumption~\ref{Assumption: OLS}.

To see that condition~(i) of Lemma~\ref{Lemma: linear regression} is satisfied for the $Y$ on $Z$ regression, define $\mbb{\theta} := \E(ZZ^\top)^{-1}\E(YZ) \in \R^d$ and note that 
\begin{align*}
	\E\bigl(Z Z^\top (Y-\mbb{\theta}^\top Z)^2\bigr) &= \E\bigl(Z Z^\top (\zeta + \beta X- \beta \mbb{\eta}^\top Z)^2\bigr) \\
	&= \E(Z Z^\top \zeta^2) + \beta^2\E\bigl(Z Z^\top (X- \mbb{\eta}^\top Z)^2\bigr),
\end{align*}
so the minimum eigenvalue of $\E\bigl(Z Z^\top (Y-\mbb{\theta}^\top Z)^2\bigr)$ is bounded below by $c\lambda_{\min}(\mbb{\Sigma})$.  Condition~(ii) follows by similar arguments as those for the $X$ on $Z$ regression.  We therefore deduce from Lemma~\ref{Lemma: linear regression} that~\eqref{Eq: assumption 1 univariate-X linear model} holds with $\sigma_{\beta}^2$ given by the $(1,1)$th entry of the matrix $(\mbb{\Sigma}^{XZ})^{-1}\mathbb{E}(WW^\top \zeta^2)(\mbb{\Sigma}^{XZ})^{-1}$, and that
\begin{equation}
	\label{Eq: OLS error bounds}
	\sqrt{n} \| \mbb{\widehat{\eta}} - \mbb{\eta} \|_2 = O_{\mathcal{P}}(1) \quad \text{and} \quad \sqrt{n} \| \mbb{\widehat{\theta}} - \mbb{\theta} \|_2 = O_{\mathcal{P}}(1).
\end{equation}

We now verify that the remaining parts of Assumption~\ref{Assumption: univariate X} are satisfied.  First, 
\[
	\sup_{P \in \mathcal{P}} \E\biggl( \Bigl\| \frac{1}{n} \sum_{i=1}^n Z_i \xi_i \Bigr\|_2^2 \biggr) = \frac{1}{n} \sup_{P \in \mathcal{P}} \E\bigl( \|Z\xi\|_2^2 \bigr) \leq \frac{1}{n}\sup_{P \in \mathcal{P}} \E(\|Z\|_2^2) \cdot \sup_{P \in \mathcal{P}} \E(\xi^2) \to 0,
\]
so by Lemma~\ref{Lemma: convergence in probability from convergence of conditional expectations} and \eqref{Eq: OLS error bounds} we have that \eqref{Eq: assumption 2 univariate-X linear model} holds.  Similar arguments show that \eqref{Eq: assumption 3 univariate-X linear model} is satisfied.

Next,~\eqref{Eq: OLS error bounds} shows that 
\[
	\sqrt{n} \| \mbb{\widehat{\eta}} - \mbb{\eta} \|_2 \| \mbb{\widehat{\theta}} - \mbb{\theta} \|_2 = O_\mathcal{P}(n^{-1/2}) = o_{\mathcal{P}}(1).
\]
Moreover, by~\eqref{Eq: linear regression sigma convergence} in the proof of Lemma~\ref{Lemma: linear regression} and Assumption~\ref{Assumption: OLS}, we have 
\[
	\Bigl\| \frac{1}{n} \sum_{i=1}^n Z_i Z_i^\top \Bigr\|_{\mathrm{op}} \leq \Bigl\| \frac{1}{n} \sum_{i=1}^n Z_i Z_i^\top -  \mbb{\Sigma} \Bigr\|_{\mathrm{op}} + \| \mbb{\Sigma}\|_{\mathrm{op}} = O_{\mathcal{P}}(1),
\]
so~\eqref{Eq: assumption 4 univariate-X linear model} is satisfied. The remaining conditions hold by similar arguments using the moment bounds established earlier, \eqref{Eq: OLS error bounds} and Lemma~\ref{Lemma: conditional lln}.

To verify the remaining conditions of Proposition~\ref{Proposition: asymptotic power expression}, we note that
\[
\var(\varepsilon \xi) \geq \E\bigl\{\var(\varepsilon \xi \given X,Z)\bigr\} = \E\bigl\{\xi^2 \var(Y \given X,Z)\bigr\} \geq c^2.
\]
Finally, the moment bound condition in Proposition~\ref{Proposition: asymptotic power expression} follows by Cauchy--Schwarz and the arguments above.  The conclusion now follows from Proposition~\ref{Proposition: asymptotic power expression} together with the fact that $\psi_{\alpha,n}$ is an increasing function of $|\beta|$, so that 
\[
	\psi_{\alpha,n} \geq \Phi\biggl( \frac{\kappa}{\sigma_{\beta}} \biggr) \cdot \Phi\biggl( z_\alpha + \frac{\kappa \sigma_{\xi}^2}{\sigma_{\varepsilon\xi}} \biggr) + \Phi\biggl( - \frac{\kappa}{\sigma_{\beta}} \biggr) \cdot \Phi\biggl( z_\alpha - \frac{\kappa \sigma_{\xi}^2}{\sigma_{\varepsilon\xi}} \biggr) \rightarrow 1
\]
as $\kappa \rightarrow \infty$, as required.

\subsection{Proof of Proposition~\ref{Proposition: Low-dimensional Z}}
Throughout this proof we work on the event that at least one $u_{n,1},\ldots,u_{n,n}$ is non-zero which is a set of uniform asymptotic probability $1$ by Assumption~\ref{Assumption: a general projection}(b).  Let $\mbb{Z}:=(Z_1^\top,\ldots,Z_{n}^\top)^\top$, $\mbb{P}:= \mbb{Z} (\mbb{Z}^\top \mbb{Z})^{-1} \mbb{Z}^\top$, $\mbb{Y}:=(Y_1,\ldots,Y_{n})^\top$,  $\mbb{\varepsilon} := (\varepsilon_1,\ldots,\varepsilon_{n})^\top$, $\mbb{\fhat}:=\bigl(\fhat(X_1, Z_1), \dots, \fhat(X_n, Z_n) \bigr)^\top$ and $\mbb{I}$ denote the $d \times d$ identity matrix.  Since $\mbb{P}$ is a matrix representing an orthogonal projection such that $\mbb{Z}^{\top} (\mbb{I} - \mbb{P})$ is a zero vector, we have 
\begin{align*}
	\sum_{i=1}^{n} \{Y_i - \mbb{\widehat{\gamma}}^\top Z_i\}  \{\fhat(X_i,Z_i) - \mhat_{\fhat}(Z_i) \}&= \mbb{\fhat}^\top (\mbb{I}-\mbb{P}) \mbb{Y}  = \mbb{\fhat}^\top (\mbb{I}-\mbb{P}) \mbb{\varepsilon}\\
	&=  \sum_{i=1}^{n} \varepsilon_i \{ \fhat(X_i,Z_i) - \mhat_{\fhat}(Z_i) \}.
\end{align*}
Based on the above identity, we have that
\begin{align*}
	T = \frac{\frac{1}{\sqrt{n}\nu}\sum_{i=1}^{n} \varepsilon_i u_{n,i}}{\sqrt{\frac{1}{n \nu^2}\sum_{i=1}^{n} (Y_i - \mbb{\widehat{\gamma}}^\top Z_i)^2u_{n,i}^2  - \bigl( \frac{1}{n \nu} \sum_{i=1}^{n} \varepsilon_i u_{n,i} \bigr)^2}},
\end{align*}
where 
\[
\nu := \sqrt{\frac{1}{n} \sum_{i=1}^{n} \E(\varepsilon_i^2 \given X_i, Z_i) u_{n,i}^2} > 0.
\]
Let $\mathcal{F}_{n}$ denote the $\sigma$-algebra generated by $\fhat$ and $\bigl((X_i, Z_i)\bigr)_{i=1}^n$. Form the triangular array 
\[
W_{n, i} := \frac{\varepsilon_i u_{n,i}}{\nu}
\]
for $n \in \mathbb{N}$ and $i \in [n]$, and note that this  satisfies the first three assumptions of Lemma~\ref{Lemma: conditional clt}, since $u_{n,i}$ is measurable with respect to $\mathcal{F}_n$.  Finally, the fourth assumption of this lemma is also satisfied, because 
\begin{align*}
\frac{1}{n^{1+\delta/2}} &\sum_{i=1}^n \E(|W_{n, i}|^{2+\delta} \given \mathcal{F}_n) = \frac{1}{n^{1+\delta/2} \nu^{2+\delta}} \sum_{i=1}^n \E(|\varepsilon_i|^{2+\delta} \given X_i, Z_i) |u_{n,i}|^{2+\delta}\\
 &\leq \frac{C}{c^{1+\delta/2}} \sum_{i=1}^n  |\upsilon_{n,i}|^{2+\delta} \leq \frac{C}{c^{1+\delta/2}} \biggl( \sum_{i=1}^n  |\upsilon_{n,i}|^2 \biggr) \max_{i \in [n]} |\upsilon_{n,i}|^\delta = o_{\mathcal{P}_0}(1)
\end{align*}
by Assumptions~\ref{Assumption: a general projection}(a) and~(b) and Lemmas~\ref{Lemma: product of o and O} and~\ref{Lemma: uniform convergence in probability under continuous transformation}. Lemma~\ref{Lemma: conditional clt} thus yields that the numerator of $T$ is uniformly asymptotically standard Gaussian.

For the denominator of $T$, the uniform version of Slutsky's theorem \citep[][Theorem 6.3]{bengs2019uniform} yields that $\frac{1}{n \nu} \sum_{i=1}^{n} \varepsilon_i u_{n,i} = o_{\mathcal{P}_0}(1)$. 
For the first term in the denominator of $T$, we consider the decomposition
\begin{align*}
	\frac{1}{n \nu^2} \sum_{i=1}^{n} (Y_i - \mbb{\widehat{\gamma}}^\top Z_i)^2u_{n,i}^2  = \underbrace{\frac{1}{n \nu^2}\sum_{i=1}^{n} \varepsilon_i^2u_{n,i}^2}_{\RN{1}_n} &+ \underbrace{\frac{1}{n \nu^2}\sum_{i=1}^{n} \{(\mbb{\widehat{\gamma}}  - \mbb{\gamma})^\top Z_i\}^2u_{n,i}^2}_{\RN{2}_n}\\
	&- \underbrace{ \frac{2}{n \nu^2}\sum_{i=1}^{n} (\mbb{\widehat{\gamma}}  - \mbb{\gamma})^\top Z_i \varepsilon_i u_{n,i}^2}_{\RN{3}_n}.
\end{align*}
Similarly to our previous argument, define the triangular array $\tilde{W}_{n, i} := W_{n,i}^2$ for $n \in \mathbb{N}$ and $i \in [n]$, and note that  $(\tilde{W}_{n, i})_{n \in \mathbb{N},i \in [n]}$ satisfies the conditions of Lemma~\ref{Lemma: conditional lln} with $\mu_n=1$ in that result, so $\RN{1}_n = 1 + o_{\mathcal{P}_0}(1)$.  Now, by H\"{o}lder's inequality,
\begin{align*}
	|\RN{2}_n| \leq \frac{1}{n \nu^2} \|\mbb{\widehat{\gamma}} -  \mbb{\gamma} \|_1^2 \sum_{i=1}^{n} \|Z_i\|_\infty^2 u_{n,i}^2 \leq \frac{1}{c}  \max_{i \in [n]} \|Z_i\|_\infty^2  \|\mbb{\widehat{\gamma}} -  \mbb{\gamma} \|_1^2 \sum_{i=1}^n \upsilon_{n,i}^2 = o_\mathcal{P}(1),
\end{align*}
by Assumption~\ref{Assumption: a general projection} and Lemma~\ref{Lemma: product of o and O}. Finally, the Cauchy--Schwarz inequality yields that 
\begin{align*}
	|\RN{3}_n| &\leq 2 \sqrt{\frac{1}{n\nu^2}\sum_{i=1}^{n} \varepsilon_i^2 u_{n,i}^2} \cdot \sqrt{\frac{1}{n \nu^2}\sum_{i=1}^{n} \{(\mbb{\widehat{\gamma}}  - \mbb{\gamma})^\top Z_i\}^2u_{n,i}^2} = 2 \sqrt{\RN{1}_n} \cdot \sqrt{\RN{2}_n} = o_\mathcal{P}(1)
\end{align*}
by Lemmas~\ref{Lemma: product of o and O} and \ref{Lemma: uniform convergence in probability under continuous transformation}.  The result follows by the uniform version of Slutsky's theorem.

\subsection{Proof of Proposition~\ref{Proposition: High-dimensional Z}}
As in the proof of Proposition~\ref{Proposition: Low-dimensional Z}, we work on the event that at least one $u_{n,1},\ldots,u_{n,n}$ is non-zero, which is a set of uniform asymptotic probability $1$ by Assumption~\ref{Assumption: a general projection}(b).  Recall the definitions of $\nu$ from the proof of Proposition~\ref{Proposition: Low-dimensional Z}, and  $\delta_{\mathrm{bias}}$ from just after~\eqref{Eq:deltabias}.  Our test statistic can be written as
\begin{align*}
	T = \frac{\frac{1}{\sqrt{n}\nu}\sum_{i=1}^{n} \varepsilon_i u_{n,i} - \frac{1}{\sqrt{n}\nu}\delta_\mathrm{bias}}{\sqrt{ \frac{1}{n s_{n}^2}\sum_{i=1}^{n} (Y_i - \mbb{\widehat{\gamma}}^\top Z_i)^2u_{n,i}^2  - \bigl(\frac{1}{n \nu}\sum_{i=1}^{n} \varepsilon_i u_{n,i} + \frac{1}{n \nu}\delta_\mathrm{bias} \bigr)^2}}.
\end{align*}
Following the proof of Proposition~\ref{Proposition: Low-dimensional Z}, we know that $\frac{1}{\sqrt{n}\nu}\sum_{i=1}^{n} \varepsilon_i u_{n,i}$ converges uniformly to $N(0,1)$. Further, by Assumption~\ref{Assumption: a general projection}(a) and Hölder's inequality, we have
\[
\biggl| \frac{1}{\sqrt{n}\nu}\delta_\mathrm{bias} \biggr| \leq \frac{1}{c^{1/2}} \biggl| \sum_{i=1}^{n} (\mbb{\widehat{\gamma}} - \mbb{\gamma})^\top Z_i \upsilon_{n,i} \biggr| \leq \frac{1}{c^{1/2}} \| \mbb{\widehat{\gamma}} - \mbb{\gamma} \|_1 \biggl\| \sum_{i=1}^{n} Z_i \upsilon_{n,i} \biggr\|_\infty = o_{\mathcal{P}_0}(1)
\]
under condition~(\ref{Eq: condition on bias}). A uniform version of Slutsky's theorem \citep[][Theorem 6.3]{bengs2019uniform} now yields that the numerator of $T$ is uniformly asymptotically standard Gaussian. We can repeat the arguments of Proposition~\ref{Proposition: Low-dimensional Z} to show that the denominator is $1+o_{\mathcal{P}_0}(1)$ under Assumption~\ref{Assumption: a general projection}, so the uniform version of Slutsky's theorem yields the desired result.

\subsection{Proof of Theorem~\ref{Theorem: General Type I error control}} \label{Section: Proof of General Type I error control}
We first prove the result under (i), (ii) or (iii).  For distinct $i,j \in [n]$, let 
$R_{ij} := \E\bigl(M_{i}M_{j} \given (X_{i'}, Z_{i'})_{i'=1}^n\bigr) - \E\bigl(M_{i}M_{j} \given (Z_{i'})_{i'=1}^n \bigr)$, where $M_i := m(Z_i)-\mhat(Z_i)$, and assume that
\begin{equation}
	\label{Eq: type I error abstract condition}
\frac{1}{n \sigma_n^2} \sum_{i \neq j} \bigl| \E\bigl( R_{ij} \xi_{i} \xi_{j} \given (Z_{i'})_{i'=1}^n, \fhat \bigr) \bigr| = o_{\mathcal{P}_0}(1).
\end{equation}
In Proposition~\ref{Proposition: general type 1 error d sufficient}, we show that this condition is satisfied if either (i), (ii) or (iii) hold.  Define $\nu^2 := \var(\varepsilon \xi \given \widehat{f})$ and note that $\nu^2 \geq c \sigma^2$ by assumption~(c). Throughout this proof we work on the event $\Omega_0 := \{\sigma > 0\}$, which satisfies $\mathbb{P}(\Omega_0^c) = o_{\mathcal{P}_0}(1)$ by assumption~(a).
Define
\[
T^{(\mathrm{N})} :=  \frac{\frac{1}{\sqrt{n}} \sum_{i=1}^{n} L_{i}}{\nu} \quad \text{and} \quad T^{(\mathrm{D})} :=   \frac{\sqrt{\frac{1}{n} \sum_{i=1}^{n} L_{i}^2 - \bigl(\frac{1}{n} \sum_{i=1}^{n} L_{i} \bigr)^2 }}{\nu},
\]
so that 
$T = T^{(\mathrm{N})} / T^{(\mathrm{D})}$. We will show that $T^{(\mathrm{N})}$ converges uniformly in distribution to $N(0, 1)$ and $|(T^{(\mathrm{D})})^2 - 1| = o_{\mathcal{P}_0}(1)$, which yields the desired result by combining Lemma~\ref{Lemma: uniform convergence in probability under continuous transformation} and the uniform version of Slutsky's lemma \citep[][Theorem 6.3]{bengs2019uniform}. 

Define $\widetilde{M}_i := m_{\fhat}(Z_i)-\mhat_{\fhat}(Z_i)$ for $i \in [n]$ and note that
\begin{equation}
	\label{Eq: type 1 numerator decomposition}
T^{(\mathrm{N})} = \underbrace{\frac{1}{\sqrt{n} \nu} \sum_{i=1}^{n}  M_i \widetilde{M}_{i}}_{a_n} + \underbrace{\frac{1}{\sqrt{n}\nu} \sum_{i=1}^{n} \widetilde{M}_{i} \varepsilon_i }_{b_n} + \underbrace{\frac{1}{\sqrt{n}\nu} \sum_{i=1}^{n} M_i \xi_{i}}_{c_n} + \underbrace{\frac{1}{\sqrt{n}\nu} \sum_{i=1}^n \varepsilon_i \xi_{i} }_{U_n}.
\end{equation}
By the Cauchy--Schwarz inequality,
\[
a_n \leq \sqrt{\frac{1}{c n} \biggl( \sum_{i=1}^{n} M_i^2 \biggr) \biggl( \frac{1}{\sigma^2} \sum_{i=1}^{n} \widetilde{M}_{i}^2 \biggr) } = o_{\mathcal{P}_0}(1),
\]
by Assumption~\ref{Assumption: general procedure}(b).

To see that $b_n = o_{\mathcal{P}_0}(1)$, we note that 
\begin{equation}
	\label{eq: b_n^2}
	b_n^2 = \frac{1}{n \nu^2} \sum_{i=1}^{n}  \widetilde{M}_{i}^2 \varepsilon_i^2  + \frac{1}{n \nu^2} \sum_{i \neq j} \widetilde{M}_{i} \widetilde{M}_{j} \varepsilon_i \varepsilon_j. 
\end{equation}
By Assumption~\ref{Assumption: general procedure}(a) and (c) and assumption (c) of the theorem, the first term in \eqref{eq: b_n^2} satisfies
\[
  \E\biggl( \frac{1}{n \nu^2} \sum_{i=1}^{n}  \widetilde{M}_{i}^2 \varepsilon_i^2 \biggm| (X_{i'},Z_{i'})_{i'=1}^n, \fhat , \mhat_{\fhat} \biggr) \leq \frac{C}{n c \sigma^2} \sum_{i=1}^{n}  \widetilde{M}_{i}^2 = o_{\mathcal{P}_0}(1)
\]
so the same is true unconditionally by Lemma~\ref{Lemma: convergence in probability from convergence of conditional expectations}.  Now, for $i\neq j$,
\begin{align*}
  \E(Y_i Y_j \given X_i, X_j, Z_i, Z_j) = \E(Y_i \given X_i,Z_i)\E(Y_j \given X_j, Z_j) = m(Z_i)m(Z_j),
\end{align*}
using the fact that $m(Z) = \E(Y \given Z) = \E(Y \given X, Z)$ under $\mathcal{P}_0$.  Hence,
\begin{align*}
\E(\varepsilon_i \varepsilon_j \given X_i, X_j, Z_i, Z_j) &= \E\bigl\{\bigl(Y_i - m(Z_i)\bigr)\bigl(Y_j-m(Z_j)\bigr) \given X_i, X_j, Z_i, Z_j\bigr\}\\
 &= \E(Y_i Y_j \given X_i, X_j, Z_i, Z_j) - m(Z_i) m(Z_j) = 0.
\end{align*}
It follows that the second term in \eqref{eq: b_n^2} satisfies
\begin{align*}
\E\biggl( \frac{1}{n \nu^2} \sum_{i \neq j} \widetilde{M}_{i} \widetilde{M}_{j} \varepsilon_i \varepsilon_j &\biggm| (X_{i'},Z_{i'})_{i'=1}^n, \fhat , \mhat_{\fhat} \biggr) \\
&= \frac{1}{n \nu^2} \sum_{i\neq j} \widetilde{M}_{i} \widetilde{M}_{j} \E(\varepsilon_i \varepsilon_j \given X_i, X_j, Z_i, Z_j) = 0,
\end{align*}
and we deduce by Lemmas~\ref{Lemma: convergence in probability from convergence of conditional expectations} and~\ref{Lemma: uniform convergence in probability under continuous transformation} that $b_n = o_{\mathcal{P}_0}(1)$.

To see that $c_n = o_{\mathcal{P}_0}(1)$, we proceed as above and write
\begin{align*}
	c_n^2  =  \frac{1}{n \nu^2} \sum_{i=1}^{n}  M_i^2 \xi_{i}^2 +  \frac{1}{n \nu^2} \sum_{i \neq j} M_iM_j \xi_{i} \xi_{j},
\end{align*}
where we note that the first term is $o_{\mathcal{P}_0}(1)$ by Assumption~\ref{Assumption: general procedure}(b).  Moreover,  
\begin{align*}
	\E\Bigl( \frac{1}{n \nu^2} \sum_{i \neq j} M_iM_j \xi_{i} \xi_{j} &\biggm| (Z_{i'})_{i'=1}^n, \fhat \Bigr) \\
	& =\frac{1}{n \nu^2} \sum_{i \neq j} \E\bigl\{\E \bigl(M_iM_j \given (X_{i'
	}, Z_{i'})_{i'=1}^n\bigr) \xi_{i} \xi_{j} \given (Z_{i'})_{i'=1}^n, \fhat \bigr\}  \\
	&=  \frac{1}{n \nu^2} \sum_{i \neq j}\E(R_{ij} \xi_{i} \xi_{j} \given (Z_{i'})_{i=1}^n, \fhat),
\end{align*}
where the last equality holds since 
\begin{align*}
\E\bigl\{ \E\bigl(M_iM_j \given (Z_{i'})_{i'=1}^n\bigr)   \xi_{i} \xi_{j} \given (Z_{i'})_{i'=1}^n, \fhat \bigr\}  &=  \E\bigl(M_iM_j \given (Z_{i'})_{i'=1}^n\bigr) \E( \xi_{i} \given Z_i, \fhat) \E( \xi_{j} \given Z_j, \fhat) \\
&= 0.
\end{align*}
Continuing, we have by \eqref{Eq: type I error abstract condition} that
\begin{align*}
	\frac{1}{n \nu^2} \sum_{i \neq j} \E\bigl( R_{ij} \xi_{i} \xi_{j} \given (Z_{i'})_{i'=1}^n, \fhat\bigr) \leq &   \frac{1}{c n \sigma^2} \sum_{i \neq j}  \big|\E( R_{ij} \xi_{n,i} \xi_{j} \given (Z_{i'})_{i'=1}^n, \fhat)\big| = o_{\mathcal{P}_0}(1).
\end{align*}
Therefore, by Lemmas~\ref{Lemma: convergence in probability from convergence of conditional expectations} and~\ref{Lemma: uniform convergence in probability under continuous transformation} we conclude that $c_n = o_{\mathcal{P}_0}(1)$ as desired.

To deal with the final term, we define the triangular array
$W_{n, i} := \nu^{-1} \varepsilon_i \xi_{i} $ for $n \in \mathbb{N}$ and $i \in [n]$,  and note that $W_{n, i}$ satisfies all the conditions of Lemma~\ref{Lemma: conditional clt} by assumptions~(b) and~(c) (here we condition on $\fhat$ in applying this result). Hence, $U_n = n^{-1/2}\sum_{i=1} W_{n,i}$, and therefore $T^{(\mathrm{N})}$, converges uniformly in distribution to $N(0,1)$.

We now show that $|(T^{(\mathrm{D})})^2-1| = o_{\mathcal{P}_0}(1)$, from which the desired result follows from Lemma~\ref{Lemma: uniform convergence in probability under continuous transformation}. Note that 
\begin{equation}
	\label{Eq: type 1 denominator decomposition 1}
(T^{(\mathrm{D})})^2 = \underbrace{\frac{1}{n\nu^2} \sum_{i=1}^{n} L_{i}^2}_{p_n} - \biggl(\underbrace{ \frac{1}{\sqrt{n} \nu} \sum_{i=1}^{n} L_{i}}_{q_n} \biggr)^2
\end{equation}
and that $q_n = \frac{1}{\sqrt{n}} T^{(\mathrm{N})}$. We have just shown that $T^{(\mathrm{N})} = O_{\mathcal{P}_0}(1)$, so $q_n = o_{\mathcal{P}_0}(1)$ and we are therefore done if we can show that $|p_n - 1| = o_{\mathcal{P}_0}(1)$. Now
\begin{equation}
	\label{Eq: type 1 denominator decomposition 2}
\begin{aligned}
p_n &= \underbrace{\frac{1}{n \nu^2} \sum_{i=1}^{n} \varepsilon_i^2\xi_{i}^2}_{\RN{1}_n}  + \underbrace{\frac{1}{n \nu^2} \sum_{i=1}^{n}  M_i^2 \widetilde{M}_{i}^2 }_{\RN{2}_n} + \underbrace{\frac{4}{n \nu^2} \sum_{i=1}^{n}   M_i \widetilde{M}_{i} \varepsilon_i \xi_{i}}_{\RN{3}_n} + \underbrace{\frac{1}{n \nu^2} \sum_{i=1}^{n}  \widetilde{M}_{i}^2 \varepsilon_i^2}_{\RN{4}_n^\varepsilon}\\ &\hspace{2cm}+ \underbrace{\frac{1}{n \nu^2} \sum_{i=1}^{n}  M_i^2 \xi_{i}^2}_{\RN{4}_n^\xi} 
+ \underbrace{\frac{2}{n \nu^2} \sum_{i=1}^{n} \widetilde{M}_{i}^2 M_i \varepsilon_i}_{\RN{5}_n^\varepsilon} +  \underbrace{\frac{2}{n \nu^2} \sum_{i=1}^{n}  M_i^2 \widetilde{M}_{ i}\xi_{i}}_{\RN{5}_n^\xi}\\
&\hspace{2cm}+ \underbrace{\frac{2}{n \nu^2} \sum_{i=1}^{n}  \widetilde{M}_{i}\xi_{i}\varepsilon_i^2}_{\RN{6}_n^\varepsilon} + \underbrace{\frac{2}{n \nu^2} \sum_{i=1}^{n}  M_i \varepsilon_i\xi_{i}^2}_{\RN{6}_n^\xi}.
\end{aligned}
\end{equation}
Consider the triangular array
$\tilde{W}_{n, i}:= W_{n, i}^2$ for $n \in \mathbb{N}$ and $i \in [n]$, and note that it satisfies all the conditions of Lemma~\ref{Lemma: conditional lln} by assumptions~(b) and~(c) with $\mu_n = 1$ (again conditioning on $\fhat$ in that result), so $|\RN{1}_n-1| = o_{\mathcal{P}_0}(1)$. It remains to show that the remaining terms are $o_{\mathcal{P}_0}(1)$. Now
\[
0 \leq \RN{2}_n \leq \frac{1}{c n} \biggl( \sum_{i=1}^{n} M_{i}^2 \biggr) \biggl( \frac{1}{\sigma^2} \sum_{i=1}^{n} \widetilde{M}_i^2 \biggr) = o_{\mathcal{P}_0}(1)
\]
by Assumption~\ref{Assumption: general procedure}(a). By the Cauchy--Schwarz inequality,
\[
|\RN{3}_n| \leq 4  \biggl( \frac{1}{n \nu^2} \sum_{i=1}^{n} \varepsilon_i^2\xi_i^2  \biggr)^{1/2}  \biggl( \frac{1}{n \nu^2} \sum_{i=1}^{n} M_i^2 \widetilde{M}_i^2 \biggr)^{1/2} = 4 \RN{1}_n^{1/2} \RN{2}_n^{1/2} = o_{\mathcal{P}_0}(1)
\]
by the above and Lemma~\ref{Lemma: uniform convergence in probability under continuous transformation}. Now  
\[
|\RN{4}_n^\varepsilon| \leq   \frac{1}{n \nu^2} \sum_{i=1}^{n} \widetilde{M}_{i}^2 \varepsilon_i^2 = o_{\mathcal{P}_0}(1)
\]
by our argument for the first term in \eqref{eq: b_n^2}.  A similar argument shows that $\RN{4}_n^\xi = o_{\mathcal{P}_0}(1)$ by Assumption~\ref{Assumption: general procedure}(b).  By the triangle inequality and the Cauchy--Schwarz inequality, 
\[
|\RN{5}_n^\varepsilon| \leq 2  \biggl(\frac{1}{n \nu^2} \sum_{i=1}^{n}M_i^2  \widetilde{M}_{i}^2  \biggr)^{1/2} \biggl(\frac{1}{n \nu^2} \sum_{i=1}^{n} \widetilde{M}_{i}^2 \varepsilon_i^2\biggr)^{1/2} = 2 \RN{2}_n^{1/2} (\RN{4}_n^\varepsilon)^{1/2} = o_{\mathcal{P}_0}(1)
\]
by the above and Lemma~\ref{Lemma: uniform convergence in probability under continuous transformation}. A similar argument can be used for $\RN{5}_n^\xi$. Finally, again by the triangle inequality and the Cauchy--Schwarz inequality,
\[
|\RN{6}_n^\varepsilon| \leq 2 \biggl(\frac{1}{n \nu^2} \sum_{i=1}^{n} \varepsilon_i^2\xi_{i}^2 \biggr)^{1/2} \biggl(\frac{1}{n \nu^2} \sum_{i=1}^{n} \widetilde{M}_{ i}^2 \varepsilon_i^2\biggr)^{1/2} = 2 \RN{1}_n^{1/2} (\RN{4}_n^\varepsilon)^{1/2} = o_{\mathcal{P}_0}(1)
\]
by the above and Lemma~\ref{Lemma: uniform convergence in probability under continuous transformation}; $\RN{6}_n^\xi$ is handled similarly. This proves the desired result under (i), (ii) or (iii).

To prove the result under (iv), we note that only the argument for the $c_n$ term needs modification. That $c_n = o_{\mathcal{P}_0}(1)$ under the assumption that $\mhat$ is sufficiently stable follows immediately from Proposition~\ref{Prop: stability} and Assumption~\ref{Assumption: general procedure}(b), with $\xi_i$ in that result taken to be $\xi_i /\sigma$ here.

\subsection{Proof of Theorem~\ref{Theorem: general power result}}
\label{Section: Proof of power result}

We first prove the result under assumption~(i).  Without loss of generality, we may assume that $\alpha \in (0,1/2)$, so that $z_{1-\alpha} > 0$.  Let $s$ denote the denominator in the definition of $T$. Suppose there exists $c > 0$ such that
\begin{align}
	\label{Eq: power L claim}
	\sup_{P \in \mathcal{P}_1(\epsilon_n)} \pr \Bigl( \frac{1}{n} \sum_{i=1}^n L_{ i} \leq c \tau \Bigr) \to 0,\\
	\label{Eq: power s claim}
	\frac{s}{\sqrt{n}} = o_{\mathcal{P}_1(\epsilon_n)}(\tau).
\end{align}
Note that, since $0/0 := 0$ and $\tau > 0$, we have that
\begin{align*}
	\pr(T \leq z_{1-\alpha}) &= \pr\biggl( \frac{1}{n} \sum_{i=1}^n L_{ i} \leq z_{1-\alpha} \frac{s}{\sqrt{n}} \biggr)\\  &\leq \pr\biggl( \frac{1}{n} \sum_{i=1}^n L_{i} \leq c \tau \biggr) + \pr\biggl( z_{1-\alpha} \frac{s}{\sqrt{n} \tau} \geq c \biggr).
\end{align*}
Thus, from \eqref{Eq: power L claim} and \eqref{Eq: power s claim}, 
\[
\sup_{P \in \mathcal{P}_1(\epsilon_n)} \pr(T \leq z_{1-\alpha}) \leq \sup_{P \in \mathcal{P}_1(\epsilon_n)} \Biggl\{ \pr\biggl( \frac{1}{n} \sum_{i=1}^n L_{i} \leq c \tau \biggr) + \pr\biggl( z_{1-\alpha} \frac{s_{n}}{\sqrt{n} \tau} \geq c \biggr) \Biggr\} \to 0
\]
and hence the result will follow if we can prove \eqref{Eq: power L claim} and \eqref{Eq: power s claim}.

Observe that if we define 
\begin{equation}
	\label{Eq: definition of check f}
	\check{f}(X, Z) := \frac{\tau^{1/2}}{\sigma} \fhat(X, Z)	
\end{equation}
and let $\check{T}$ denote the test using $\check{f}$ in place of $\fhat$, then $T = \check{T}$, since we have assumed in~(a) that $m_{\fhat}$ is scale equivariant. It follows that we may put $\check{f}$ in place of $\fhat$ and assume without loss of generality that 
\[
	\sigma^2 = \E\bigl(\xi^2 \given \fhat \bigr)  = \tau.	
\]
For both claims \eqref{Eq: power L claim} and \eqref{Eq: power s claim}, we therefore work with $\check{f}$ instead of $\fhat$, $\check{\xi} := \check{f}(X, Z) - \E(\check{f}(X, Z) \given \check{f}, Z)$ instead of $\xi$ and similarly $\check{\xi}_i$ instead of $\xi_i$ for $i \in [n]$.

To prove \eqref{Eq: power L claim}, we write $Y_i = m(Z_i) + h(X_i,Z_i) + \zeta_i$, and have
\begin{align*}
\frac{1}{n} \sum_{i=1}^n L_{i} &= \underbrace{\frac{1}{n} \sum_{i=1}^n h(X_i, Z_i) \check{\xi}_i}_{\RN{1}_n} + \underbrace{\frac{1}{n} \sum_{i=1}^n \zeta_i \check{\xi}_i}_{\RN{2}_n}\\
&+ \underbrace{\frac{1}{n} \sum_{i=1}^n \{m(Z_i) - \mhat(Z_i) \}\check{\xi}_i}_{\RN{3}_n} + \underbrace{\frac{1}{n} \sum_{i=1}^n h(X_i, Z_i) \{m_{\check{f}}(Z_i) - \mhat_{\check{f}}(Z_i)\}}_{\RN{4}_n}\\
&+ \underbrace{\frac{1}{n} \sum_{i=1}^n \zeta_i \{m_{\check{f}}(Z_i) - \mhat_{\check{f}}(Z_i)\}}_{\RN{5}_n} + \underbrace{\frac{1}{n} \sum_{i=1}^n  \{m(Z_i) - \mhat(Z_i)\}\{m_{\check{f}}(Z_i) - \mhat_{\check{f}}(Z_i) \}}_{\RN{6}_n}.
\end{align*}
Defining the triangular array $W_{n, i} := h(X_i, Z_i) \check{\xi}_i/\tau$ for $i \in [n]$, note by assumption~(b) that
\[
    \sum_{i=1}^n \E\bigl(|W_{n, i}|^2 \given \check{f} \bigr) = \frac{    n \E\bigl(h(X, Z)^2 \check{\xi}^2 \given \check{f} \bigr) }{\tau^2} \leq \frac{ C^2   n }{\tau}.
\]
Therefore, defining $\mu_n := \E(\RN{1}_n \given \check{f})$ (the numerator of $\mathrm{Corr}(h(X, Z), \check{\xi} \given \check{f})$), assumption~(ii) of Lemma~\ref{Lemma: conditional lln} is satisfied with $\delta = 1$ on $\mathcal{P}_1(\epsilon_n)$ by assumption (c). We deduce that
\[
\sup_{P \in \mathcal{P}_1(\epsilon_n)} \pr(|\RN{1}_n - \mu_n| \geq \eta \tau ) = o(1)
\]
for any $\eta > 0$. 

To deal with the $\RN{2}_n$ term, we first note that for $i \neq j$, 
\[
	\E(\zeta_i \zeta_j \given X_i, Z_i, X_j, Z_j) = \E(\zeta_i \given X_i, Z_i) \E(\zeta_j \given X_j, Z_j) = 0.
\]
Hence, using the fact that $\E(\zeta_i^2 \given X_i, Z_i) = \Var(Y_i \given X_i, Z_i) \leq C$ for all $i \in [n]$ by Assumption~\ref{Assumption: general procedure}(c), we have 
\begin{align*}
\E\bigl(|\RN{2}_n| \given \check{f}, (X_i, Z_i)_{i=1}^n \bigr) &\leq \frac{1}{n} \biggl(\sum_{i=1}^n \E(\zeta_i^2 \given X_i, Z_i ) \check{\xi}_i^2 \biggr)^{1/2}\\
&\leq \frac{C^{1/2}}{n^{1/2}} \biggl(\frac{1}{n} \sum_{i=1}^n  \check{\xi}_i^2 \biggr)^{1/2},
\end{align*}
and therefore 
\[
	\E\bigl(|\RN{2}_n| \given \check{f}\bigr) \leq \frac{C^{1/2}}{n^{1/2}} \tau^{1/2}.
\]
We conclude by Lemma~\ref{Lemma: unconditionalisation via Markov} that $\RN{2}_n = O_{\mathcal{P}_1(\epsilon_n)}(n^{-1/2}\tau^{1/2})$. To deal with $\RN{3}_n$, we first write
\[
  \RN{3}_n^2 = \frac{1}{n^2} \sum_{i=1}^n \check{\xi}_i^2 \{m(Z_i) - \mhat(Z_i) \}^2 + \frac{1}{n^2} \sum_{i \neq j} \check{\xi}_i \check{\xi}_j \{m(Z_i) - \mhat(Z_i) \}\{m(Z_j) - \mhat(Z_j) \}
\]
and note that the first term is $o_{\mathcal{P}_1(\epsilon_n)}(\tau n^{-1})$ by Assumption~\ref{Assumption: general procedure}(b). For the second term, note similarly to $\RN{2}_n$ that 
\begin{equation}
	\label{Eq:tilde f crossterm}
	\E(\check{\xi}_i \check{\xi}_j \given \check{f}, Z_i, Z_j) = \E(\check{\xi}_i \given \check{f},Z_i) \E(\check{\xi}_j \given \check{f},Z_j) = 0.
\end{equation}
Thus, since $\mhat$ is formed on auxiliary data in Algorithm~\ref{Algorithm: PCM theory},
\begin{align*}
	\E\biggl( \frac{1}{n \nu^2} \sum_{i \neq j} \check{\xi}_i \check{\xi}_j \given \check{f}, Z_i, Z_j &\biggm| \check{f}, (Z_{i'})_{i'=1}^n, \mhat \biggr) = 0.
\end{align*}
We deduce by Lemma~\ref{Lemma: convergence in probability from convergence of conditional expectations} that $\RN{3}_n = o_{\mathcal{P}_1(\epsilon_n)}(\tau^{1/2} n^{-1/2})$.  Next, we note that 
\begin{equation}
	\label{Eq: m_check bound}
	\frac{1}{n} \sum_{i=1}^n  \{m_{\check{f}}(Z_i) - \mhat_{\check{f}}(Z_i) \}^2 =  \frac{\tau}{n \sigma^2} \sum_{i=1}^n  \{m_{\fhat}(Z_i) - \mhat_{\fhat}(Z_i) \}^2 = o_{\mathcal{P}_1(\epsilon_n)}(\tau),
\end{equation}
by Assumption~\ref{Assumption: general procedure}(a).  By the Cauchy--Schwarz inequality, Markov's inequality for the first factor and \eqref{Eq: m_check bound} for the second factor, we have 
\[
  |\RN{4}_n| \leq  \biggl(\frac{1}{n} \sum_{i=1}^n h(X_i, Z_i)^2  \biggr)^{1/2} \biggl( \frac{1}{n}\sum_{i=1}^n \{m_{\check{f}}(Z_i) - \mhat_{\check{f}}(Z_i) \}^2 \biggr)^{1/2} = o_{\mathcal{P}_1(\epsilon_n)}(\tau).
\]
For the $\RN{5}_n$ term, we note that by similar arguments as above,
\begin{align} 
\E\bigl(|\RN{5}_n| \given \check{f}, (X_i, Z_i)_{i=1}^n \bigr) 
\leq \frac{C^{1/2}}{n^{1/2}} \biggl(\frac{1}{n} \sum_{i=1}^n  \{m_{\check{f}}(Z_i) - \mhat_{\check{f}}(Z_i) \}^2 \biggr)^{1/2} = o_{\mathcal{P}_1(\epsilon_n)}(\tau^{1/2}n^{-1/2}).
\end{align}
Hence by Lemma~\ref{Lemma: unconditionalisation via Markov}, we deduce that $\RN{5}_n =  o_{\mathcal{P}_1(\epsilon_n)}(\tau^{1/2}n^{-1/2})$. 
For the final term, by the Cauchy--Schwarz inequality,
\begin{align*}
	|\RN{6}_n| &\leq \biggl(\frac{1}{n} \sum_{i=1}^n  \{m(Z_i)- \mhat(Z_i)\}^2 \biggr)^{1/2}	 \biggl(\frac{1}{n} \sum_{i=1}^n  \{m_{\check{f}}(Z_i) - \mhat_{\check{f}}(Z_i) \}^2 \biggr)^{1/2}\\
	&= o_{\mathcal{P}_1(\epsilon_n)}(\tau^{1/2} n^{-1/2})
\end{align*}
using Assumptions~\ref{Assumption: general procedure}(a) and \eqref{Eq: m_check bound}. Letting $R_n := \RN{2}_n + \RN{3}_n + \RN{4}_n + \RN{5}_n + \RN{6}_n$, we have therefore shown that $R_n = o_{\mathcal{P}_1(\epsilon_n)}(\tau)$ by assumption (c). We conclude that
\begin{align*}
	\sup_{P \in \mathcal{P}_1(\epsilon_n)} \pr \Bigl( \frac{1}{n} \sum_{i=1}^n L_{n, i} \leq \rho \tau / 3 \Bigr) \leq 	\sup_{P \in \mathcal{P}_1(\epsilon_n)}\pr ( \mu_n \leq \rho \tau )  &+  \sup_{P \in \mathcal{P}_1(\epsilon_n)}\pr ( |\RN{1}_n - \mu_n| \geq \rho \tau / 3 ) \\
	&+ \sup_{P \in \mathcal{P}_1(\epsilon_n)}\pr (|R_n| \geq \rho\tau /3 ),
\end{align*}
so \eqref{Eq: power L claim} is satisfied with $c:= \rho/3$ by assumption (d).

To see that \eqref{Eq: power s claim} holds, note that 
\begin{align*}
	\frac{s_n}{n^{1/2}} &\leq \biggl( \frac{1}{n^2} \sum_{i=1}^n L_{i}^2 \biggr)^{1/2} \leq 5^{1/2} \biggl[ \underbrace{\biggl( \frac{1}{n^2} \sum_{i=1}^n h(X_i, Z_i)^2 \check{\xi}_i^2 \biggr)^{1/2}}_{\widetilde{\RN{1}}_n}\\
	&+ \underbrace{\biggl(\frac{1}{n^2} \sum_{i=1}^n \zeta_i^2 \check{\xi}_i^2\biggr)^{1/2}}_{\widetilde{\RN{2}}_n}
	+ \underbrace{\biggl(\frac{1}{n^2} \sum_{i=1}^n \{m(Z_i) - \mhat(Z_i) \}^2 \check{\xi}_i^2 \biggr)^{1/2}}_{\widetilde{\RN{3}}_n}\\
	&+ \underbrace{\biggl(\frac{1}{n^2} \sum_{i=1}^n h(X_i, Z_i)^2\{m_{\check{f}}(Z_i) - \mhat_{\check{f}}(Z_i) \}^2\biggr)^{1/2}}_{\widetilde{\RN{4}}_n}\\
	&+ \underbrace{\biggl(\frac{1}{n^2} \sum_{i=1}^n \zeta_i^2\{m_{\check{f}}(Z_i) - \mhat_{\check{f}}(Z_i) \}^2\biggr)^{1/2}}_{\widetilde{\RN{5}}_n}\\
	&+ \underbrace{\biggl(\frac{1}{n^2} \sum_{i=1}^n   \{m(Z_i) - \mhat(Z_i)\}^2\{m_{\check{f}}(Z_i) - \mhat_{\check{f}}(Z_i) \}^2\biggr)^{1/2}}_{\widetilde{\RN{6}}_n}\biggr].
\end{align*}
Now 
\[
	\E\bigl(\widetilde{\RN{1}}_n \given \check{f} \bigr) \leq \frac{C}{n^{1/2}} \biggl( \frac{1}{n} \sum_{i=1}^n \E\bigl(\check{\xi}_i^2 \given \check{f} \bigr) \biggr)^{1/2} = \frac{C}{n^{1/2}} \tau^{1/2},
\]
so by Lemma~\ref{Lemma: unconditionalisation via Markov} we see that $\widetilde{\RN{1}}_n = O_{\mathcal{P}_1(\epsilon_n)}(n^{-1/2} \tau^{1/2})$. The remaining terms are of the same uniform stochastic order as the  corresponding terms without tildes using the bounds above. Thus,  \eqref{Eq: power s claim} is satisfied by assumption (c), and the result follows.

If we assume as in (ii) that the test statistic is computed via Algorithm~\ref{Algorithm: PCM} and that $\mhat$ is sufficiently stable, then the only changes to the above proof are when dealing with the $\RN{3}_n$ term. 
 But this remains $o_{\mathcal{P}_1(\epsilon_n)}(\tau^{1/2}n^{-1/2})$ by the stability assumptions on $\mhat$ and Proposition~\ref{Prop: stability}.

\subsection{Proof of Theorem~\ref{Theorem: Asymptotic normality of Spline}} 
It suffices to verify that Assumption~\ref{Assumption: general procedure} holds and to check assumptions (a)-(c) of Theorem~\ref{Theorem: General Type I error control} as the conclusion then follows from Theorem~\ref{Theorem: General Type I error control}. 

To see that assumption (a) of Theorem~\ref{Theorem: General Type I error control} holds, note that by Proposition~\ref{Proposition: tensor product projection} with $\fhat$ in place of $f$ in that result,
\begin{align}
	\label{Eq: eps pi beta}
	\xi &:= \fhat(X,Z) - \E\bigl\{\fhat(X,Z) \given Z,\fhat\bigr\} = (\mbb{\Pi}\mbb{\widehat{\beta}})^\top \bigl\{\mbb{\phi}(X, Z)-\E\bigl(\mbb{\phi}(X, Z) \given Z\bigr) \bigr\}.
\end{align}
Thus, from the definition in~\eqref{Eq: conditional variance},
\begin{align}
	\label{Eq: lower bound for the variance}
	\sigma^2 = \E(\xi^2 \given \fhat)
	= (\mbb{\Pi}\mbb{\widehat{\beta}})^\top \mbb{\Lambda} (\mbb{\Pi}\mbb{\widehat{\beta}}) \geq  \tilde{\lambda}_{\min}(\mbb{\Lambda}) \| \mbb{\Pi}\mbb{\widehat{\beta}}\|_2^2 \geq c K^{-1}_{XZ} \| \mbb{\Pi}\mbb{\widehat{\beta}}\|_2^2,
\end{align}
where the last inequality holds by our assumption \eqref{Eq:lambdamin}.  Hence,
\[
\sup_{P \in \mathcal{P}_0} \mathbb{P}(\sigma^2 = 0) = \sup_{P \in \mathcal{P}_{0}} \prob_{P}(\|\mbb{\Pi}\mbb{\widehat{\beta}}\|_{\infty} = 0) = o(1),
\]
as desired.

Next, for any $\eta \geq 1$, we have by~\eqref{Eq: eps pi beta} that
\begin{equation}
	\label{Eq: cond exp eps^r}
\E\bigl(|\xi|^\eta \given Z, \fhat\bigr) \leq 2^\eta \E\bigl( \bigl| (\mbb{\Pi} \mbb{\widehat{\beta}})^\top \mbb{\phi}(X, Z) \bigr|^\eta \given Z, \fhat \bigr) \leq 2^\eta \| \mbb{\Pi} \mbb{\widehat{\beta}} \|_\infty^\eta
\end{equation}
by Hölder's inequality and Proposition~\ref{Proposition: b-spline properties}(a).  
By~\eqref{Eq: lower bound for the variance}, Assumption~\ref{Assumption: Nonparametric models}(a),~\eqref{Eq: cond exp eps^r} with $\eta=2+\delta$ and~\eqref{Eq: condition for the normality of spline 2}, we have
\[
\frac{\E(|\xi \varepsilon|^{2+\delta} \given \fhat)}{\sigma^{2+\delta}} \leq \frac{2^{2+\delta}C}{c^{2+\delta}} \frac{\| \mbb{\Pi} \mbb{\widehat{\beta}} \|_\infty^{2+\delta}}{\| \mbb{\Pi} \mbb{\widehat{\beta}} \|_2^{2+\delta}} K_{XZ}^{1+\delta/2} \leq \frac{2^{2+\delta}C}{c^{2+\delta}} K_{XZ}^{1+\delta/2} = o(n^{\delta/2}),
\]
so assumption (b) of Theorem~\ref{Theorem: General Type I error control} is satisfied.  Assumption~(c) of Theorem~\ref{Theorem: General Type I error control} is also a hypothesis of the result we seek to prove (see Assumption~\ref{Assumption: Nonparametric models}), so there is nothing to check.

To see that Assumption~\ref{Assumption: general procedure}(a) holds, we first define $M_i := m(Z_i)-\mhat(Z_i)$ and $\widetilde{M}_i := m_{\fhat}(Z_i)-\mhat_{\fhat}(Z_i)$ for $i \in \mathcal{D}_1$. By Corollary~\ref{Corollary: spline regression} and \eqref{Eq: condition for the normality of spline},  
\begin{equation}
\label{Eq: Ns}
\mathcal{E}_{P, 1} = \frac{1}{n} \sum_{i=1}^n M_i^2 = O_{\mathcal{P}_0}\bigl(\widetilde{K}_Z^{-2s/d_Z} + \widetilde{K}_Z/n\bigr) = o_{\mathcal{P}_0}(1).
\end{equation}
Now, suppose that $g^\dagger = \mbb{\beta}_{XZ}^\top \mbb{\phi}$ is the $L_2(P)$-best approximant of $g$ over $\mathcal{S}_{r,L}^d$.  Then, by Propositions~\ref{Proposition: tensor product projection} and \ref{Proposition: b-spline properties}(b),
\begin{equation}
\label{Eq:PiBeta0}
\begin{aligned}
\| \mbb{\Pi} \mbb{\widehat{\beta}} \|_{\infty} &=\| \mbb{\Pi} \mbb{\widehat{\beta}}_{XZ} \|_{\infty} \leq 2 \| \mbb{\widehat{\beta}}_{XZ} \|_{\infty} \\
&\leq   2 \| \mbb{\widehat{\beta}}_{XZ} - \mbb{\beta}_{XZ} \|_{\infty} + 2c_s(r)^{-d} \|  g - g^\dagger\|_{\infty} + 2c_s(r)^{-d} \| g\|_{\infty}.
\end{aligned}
\end{equation}
Hence, by Corollary~\ref{Corollary: spline regression},  Propositions~\ref{Proposition: spline approximation} and \ref{Proposition: spline least squares approximation} and Assumption~\ref{Assumption: Nonparametric models},  
\begin{equation}
	\label{Eq: pi beta bound}
\| \mbb{\Pi} \mbb{\widehat{\beta}} \|_{\infty} = O_{\mathcal{P}_0}(K_{XZ} n^{-1/2} + 1) = O_{\mathcal{P}_0}(1),
\end{equation}
where the last equality uses~\eqref{Eq: condition for the normality of spline 2}.  Now $\widetilde{m}(Z)$ is in the span of $\mbb{\psi}(Z)$, so the residuals $m_{\fhat} - \mhat_{\fhat}$ are identical to those resulting from a $\ghat(X, Z)$ on $\mbb{\psi}(Z)$ regression.  Thus, by Proposition~\ref{Proposition: spline on spline regression}, \eqref{Eq: lower bound for the variance} and \eqref{Eq: condition for the normality of spline},  
\begin{equation}
\label{Eq: Ms}
\mathcal{E}_{P, 2} = \frac{1}{n \sigma^2} \sum_{i=1}^n \widetilde{M}_{i}^2 =  O_{\mathcal{P}_0}\biggl(\frac{\| \mbb{\Pi} \mbb{\widehat{\beta}} \|_\infty^2}{\| \mbb{\Pi} \mbb{\widehat{\beta}} \|_2^2} K_{XZ} \{\widetilde{K}_Z^{-2s/d_Z} + \widetilde{K}_Z/n\} \biggr) = o_{\mathcal{P}_0}(1). 
\end{equation}
Combining~\eqref{Eq: Ns} and~\eqref{Eq: Ms}, we have 
\[
  \mathcal{E}_{P, 1}  \mathcal{E}_{P, 2} = O_{\mathcal{P}_0}\biggl(\frac{\| \mbb{\Pi} \mbb{\widehat{\beta}} \|_\infty^2}{\| \mbb{\Pi} \mbb{\widehat{\beta}} \|_2^2} K_{XZ} \{\widetilde{K}_Z^{-2s/d_Z} + \widetilde{K}_Z/n\}^2 \biggr) = o_{\mathcal{P}_0}(n^{-1}),
\]
by~\eqref{Eq: condition for the normality of spline}, so Assumption~\ref{Assumption: general procedure}(a) holds.

To verify Assumption~\ref{Assumption: general procedure}(b), note by Lemma~\ref{Lemma: unconditionalisation via Markov} that, since $\E(\varepsilon^2 \given X, Z) \leq \E(\varepsilon^{2+\delta} \given X, Z)^{2/(2+\delta)} \leq C^{2/(2+\delta)}$, we have
\begin{align*}
\E \Biggl( \frac{1}{\sigma^2}  &\frac{1}{n} \sum_{i=1}^n \widetilde{M}_{i}^2 \varepsilon_{i}^2 \biggm| \fhat, (X_{i'}, Z_{i'})_{i'=1}^n, \mhat_{\fhat} \Biggr) = \frac{1}{\sigma^2}  \biggl\{ \frac{1}{n} \sum_{i=1}^n \widetilde{M}_{i}^2 \E(\varepsilon_i^2 \given X_i, Z_i) \biggr\}\\
&= O_{\mathcal{P}_0}\biggl(\frac{\| \mbb{\Pi} \mbb{\widehat{\beta}} \|_\infty^2}{\| \mbb{\Pi} \mbb{\widehat{\beta}} \|_2^2} K_{XZ} \{ \widetilde{K}_Z^{-2s/d_Z} +  \widetilde{K}_Z/n\} \biggr) = o_{\mathcal{P}_0}(K_{XZ}^{1/2}n^{-1/2}) = o_{\mathcal{P}_0}(1)
\end{align*}
by~\eqref{Eq: condition for the normality of spline} and~\eqref{Eq: condition for the normality of spline 2}.  Assumption~\ref{Assumption: general procedure}(c) is also the hypothesis of Assumption~\ref{Assumption: Nonparametric models}(a), and this completes the proof.

\subsection{Proof of Theorem~\ref{Theorem: Power of Spline}}

Without loss of generality, we may assume that $\alpha \in (0,1/2)$, so that $z_{1-\alpha} > 0$.  Let $s_n$ denote the denominator in the definition of $T$.  Suppose we can show that
\begin{align}
	\label{Eq: spline power L claim}
	\frac{1}{n} \sum_{i=1}^n L_{i} &= \tau(1+R_n), \quad \text{where $R_n=o_{\mathcal{P}_1(\epsilon_n)}(1)$}\\
	\label{Eq: spline power s claim}
	\frac{s}{\sqrt{n}} & = \tau U_{n}, \quad \text{where $U_n=o_{\mathcal{P}_1(\epsilon_n)}(1)$}.
\end{align}
Note that, since $0/0 := 0$ and $\tau > 0$, we have from \eqref{Eq: spline power L claim} and \eqref{Eq: spline power s claim} that
\begin{align*}
	\pr(T \leq z_{1-\alpha}) &= \pr\biggl( \frac{1}{n} \sum_{i=1}^n L_{ i} \leq z_{1-\alpha} \frac{s}{\sqrt{n}} \biggr)  = \pr(z_{1-\alpha}U_n-R_n \geq 1) \\
	&\leq \pr\biggl(|U_n| \geq \frac{1}{2 z_{1-\alpha}}\biggr) + \pr\biggl(|R_n| \geq \frac{1}{2}\biggr).
\end{align*}
Thus, from \eqref{Eq: spline power L claim} and \eqref{Eq: spline power s claim}, 
\[
\sup_{P \in \mathcal{P}_1(\epsilon_n)} \pr(T \leq z_{1-\alpha}) \leq \sup_{P \in \mathcal{P}_1(\epsilon_n)} \pr\biggl(|U_n| \geq \frac{1}{2 z_{1-\alpha}}\biggr) + \sup_{P \in \mathcal{P}_1(\epsilon_n)}\pr\biggl(|R_n| \geq \frac{1}{2}\biggr) \to 0
\]
and hence the result will follow if we can prove \eqref{Eq: spline power L claim} and \eqref{Eq: spline power s claim}.

To see that \eqref{Eq: spline power L claim} holds, we write
\begin{align*}
	&\frac{1}{n} \sum_{i=1}^n L_{i} = \underbrace{\frac{1}{n} \sum_{i=1}^n \varepsilon_i f(X_i, Z_i)}_{\RN{1}_n}  + \underbrace{\frac{1}{n} \sum_{i=1}^n \varepsilon_i \{m(Z_i) - \widetilde{m}(Z_i)\}}_{\RN{2}_n}\\
	&+  \underbrace{\frac{1}{n} \sum_{i=1}^n f(X_i, Z_i) \{\ghat(X_i, Z_i) - g(X_i, Z_i)\}}_{\RN{3}_n}\\
	&+ \underbrace{\frac{1}{n} \sum_{i=1}^n \{Y_i - g(X_i, Z_i)\} \{\ghat(X_i, Z_i) - g(X_i, Z_i)\}}_{\RN{4}_n} - \underbrace{\frac{1}{n} \sum_{i=1}^n \varepsilon_i m_{\fhat}(Z_i)}_{\RN{5}_n}\\
	&+ \underbrace{\frac{1}{n} \sum_{i=1}^n \varepsilon_i \{m_{\fhat}(Z_i) - \mhat_{\fhat}(Z_i)\}}_{\RN{6}_n} + \underbrace{\frac{1}{n} \sum_{i=1}^n \{m(Z_i)-\mhat(Z_i)\}\{m_{\fhat}(Z_i) - \mhat_{\fhat}(Z_i)\}}_{\RN{7}_n}\\
	&+ \underbrace{\frac{1}{n} \sum_{i=1}^n \{m(Z_i)-\mhat(Z_i)\}\{\fhat(X_i, Z_i) - m_{\fhat}(Z_i)\}}_{\RN{8}_n}.
\end{align*}
Using the fact that $\bigl(\varepsilon_i f(X_i, Z_i)\bigr)_{i=1}^n$ are independent and identically distributed with mean $\tau$, we have that
\begin{equation}
	\label{Eq: spline power rn1 bound}
	\begin{aligned}
&\E\biggl(\biggl|\frac{1}{n} \sum_{i=1}^n \varepsilon_i f(X_i, Z_i)-\tau\biggr|\biggr) \leq \E \biggl\{ \biggl( \frac{1}{n} \sum_{i=1}^n \varepsilon_i f(X_i, Z_i)-\tau\biggr)^2 \biggr\}^{1/2}\\
&= \frac{1}{n^{1/2}}
 \bigl\{\Var\bigl(\varepsilon f(X, Z)\bigr)\bigr\}^{1/2} \leq \frac{1}{n^{1/2}}\bigl\{\E\bigl(\varepsilon^2 f(X, Z)^2\bigr)\bigr\}^{1/2} \leq \Bigl(\frac{C^{2/(2+\delta)} \tau}{n}\Bigr)^{1/2},
	\end{aligned}
\end{equation}
so $\RN{1}_n - \tau = O_{\mathcal{P}}\bigl(\tau^{1/2}/n^{1/2}\bigr)$, by Lemma~\ref{Lemma: unconditionalisation via Markov}.
Now note that for $i \neq j$,
\begin{equation}
	\label{Eq: cond exp xi given z spline proof}
	\E(\varepsilon_i \varepsilon_j \given Z_i, Z_j) = \E(\varepsilon_i \given Z_i)\E(\varepsilon_j \given Z_j) = 0.
\end{equation}
Hence, by Assumption~\ref{Assumption: Nonparametric models}(a), we have that 
\begin{align*}
	\E\bigl(|\RN{2}_n| \given \widetilde{m}, (Z_i)_{i=1}^n\bigr) &\leq \frac{1}{n}  \biggl( \sum_{i=1}^n \E(\varepsilon_i^2 \given Z_i) \{m(Z_i) - \widetilde{m}(Z_i)\}^2 \biggr)^{1/2}\\
	&\leq \frac{C^{1/(2+\delta)}}{n^{1/2}} \biggl( \frac{1}{n} \sum_{i=1}^n \{m(Z_i) - \widetilde{m}(Z_i)\}^2 \biggr)^{1/2}.
\end{align*}
Thus, by Corollary~\ref{Corollary: spline regression} and Lemma~\ref{Lemma: unconditionalisation via Markov},
\[
\RN{2}_n = O_{\mathcal{P}}\bigl(K_Z^{-s/d_Z} n^{-1/2} + K_Z^{1/2} n^{-1}\bigr) = O_{\mathcal{P}}\bigl(n^{-\frac{(4s+d/2)}{4s+d}} + n^{-\frac{(4s + d_X)}{4s+d}}\bigr) = O_{\mathcal{P}}\bigl(n^{-\frac{4s}{4s+d}}\bigr).
\]
Since $\tau = \E\bigl(f(X, Z)^2\bigr)$, we have by Proposition~\ref{proposition: error times smooth function} that
\begin{align*}
	\RN{3}_n &= O_{\mathcal{P}}\bigl(K_{XZ}^{-2s/d} + K_{XZ}^{-(s/d-1/2)} n^{-1} + \tau^{1/2} n^{-1/2} \{1+ K_{XZ}^{-(s/d-1/2)}\}\bigr)\\
	&= O_{\mathcal{P}}\bigl(n^{-\frac{4s}{4s+d}} + \tau^{1/2}n^{-1/2}\bigr).
\end{align*}
Next, by Assumption~\ref{Assumption: Nonparametric models}(a) and (c), 
\begin{align*}
\E\bigl(\{Y - g(X, Z)\}^2 \given X, Z\bigr) &= \E\bigl(\{Y - m(Z)\}^2 \given X, Z\bigr) - 2m(Z)^2 + f(X, Z)^2\\
&\leq C^{2/(2+\delta)} + 4C^2
\end{align*}
and for $i \neq j$,
\begin{align*}
&\E\bigl(\{Y_i - g(X_i, Z_i)\}\{Y_j - g(X_j, Z_j)\}\given X_i, Z_i,X_j,Z_j\bigr) \\
&= \E\bigl(\{Y_i - g(X_i, Z_i)\}\given X_i,Z_i\bigr)\E\bigl(\{Y_j - g(X_j, Z_j)\}\given X_j,Z_j\bigr) = 0. 
\end{align*}
Therefore, 
\begin{align*}
\E\bigl(&|\RN{4}_n| \given \ghat, (X_i, Z_i)_{i=1}^n\bigr)\\
&\leq \frac{1}{n} \biggl( \sum_{i=1}^n \E\bigl(\{Y_i - g(X_i, Z_i)\}^2 \given X_i, Z_i\bigr) \{\ghat(X_i, Z_i) - g(X_i, Z_i)\}^2 \biggr)^{1/2}\\
&\leq \frac{(C^{2/(2+\delta)} + 4C^2)^{1/2}}{n^{1/2}} \biggl(\frac{1}{n} \sum_{i=1}^n \{\ghat(X_i, Z_i) - g(X_i, Z_i)\}^2 \biggr)^{1/2}.
\end{align*}
By Corollary~\ref{Corollary: spline regression} and Lemma~\ref{Lemma: unconditionalisation via Markov} we thus have
\[
\RN{4}_n = O_{\mathcal{P}}\bigl(K_{XZ}^{-s/d} n^{-1/2} + K_{XZ}^{1/2} n^{-1} \bigr) = O_{\mathcal{P}}\bigl(n^{-\frac{(4s+d/2)}{4s+d}} + n^{-\frac{4s}{4s+d}}\bigr) = O_{\mathcal{P}}\bigl(n^{-\frac{4s}{4s+d}}\bigr).
\]
Now, using \eqref{Eq: cond exp xi given z spline proof} and the fact that $m_{\fhat}(Z) = \E\bigl(\fhat(X, Z) - f(X, Z) \given Z, \fhat\bigr)$, we have 
\begin{align*}
&\E\bigl(|\RN{5}_n| \given \fhat, (Z_i)_{i=1}^n\bigr) \leq \frac{C^{1/(2+\delta)}}{n^{1/2}} \biggl( \frac{1}{n} \sum_{i=1}^n \E\bigl(\{\fhat(X_i, Z_i)- f(X_i, Z_i)\}^2 \given Z_i, \fhat \bigr) \biggr)^{1/2}\\
&\leq \frac{2^{1/2}C^{1/(2+\delta)}}{n^{1/2}} \biggl( \frac{1}{n} \sum_{i=1}^n \E\bigl(\{\ghat(X_i, Z_i)- g(X_i, Z_i)\}^2 \given Z_i, \fhat \bigr)\\
&\hspace{5cm}+ \frac{1}{n} \sum_{i=1}^n \E\bigl(\{\widetilde{m}(Z_i) - m(Z_i)\}^2 \given Z_i, \fhat \bigr) \biggr)^{1/2}.
\end{align*}
By Corollary~\ref{Corollary: spline regression} and Lemma~\ref{Lemma: unconditionalisation via Markov}, we deduce that
\[
\RN{5}_n =O_{\mathcal{P}}\bigl(K_{XZ}^{-s/d} n^{-1/2} + K_{XZ}^{1/2} n^{-1} + K_Z^{-s/d_Z} n^{-1/2} + K_Z^{1/2} n^{-1} \bigr) = O_{\mathcal{P}}\bigl(n^{-\frac{4s}{4s+d}}\bigr).
\]
As in the proof of Theorem~\ref{Theorem: Asymptotic normality of Spline}, $\widetilde{m}(Z)$ is in the span of $\mbb{\psi}(Z)$, the residuals $m_{\fhat} - \mhat_{\fhat}$ are identical to those resulting from a $\ghat(X, Z)$ on $\mbb{\psi}(Z)$ regression.  Moreover, by~\eqref{Eq: pi beta bound},
\begin{equation}
\label{Eq:PiBeta2}
\| \mbb{\Pi} \mbb{\widehat{\beta}} \|_{\infty} = O_{\mathcal{P}}(K_{XZ} n^{-1/2} + 1) = O_{\mathcal{P}}(1),
\end{equation}
where the final equality uses the fact that $s \geq 3d/4$.  We deduce by Proposition~\ref{Proposition: spline on spline regression} that
\begin{equation}
    \label{Eq: spline power fhat on Z regression}
    \frac{1}{n} \sum_{i=1}^n \{m_{\fhat}(Z_i) - \mhat_{\fhat}(Z_i)\}^2 = O_{\mathcal{P}}(\widetilde{K}_Z^{-2s/d_Z} + \widetilde{K}_Z n^{-1}).
\end{equation}
By a similar argument as for the $\RN{2}_n$ term, but conditioning on $\fhat$ and $\mhat_{\fhat}$ instead of $\widetilde{m}$ and applying \eqref{Eq: spline power fhat on Z regression}, we conclude that 
\[
\RN{6}_n =  O_{\mathcal{P}}\bigl(n^{-\frac{4s}{4s+d}}\bigr).
\]

We now intend to apply Proposition~\ref{Proposition: product error} to the $\RN{7}_n$ term with $(X, Y, Z) = \bigl(\fhat(X, Z), Y, Z\bigr)$. By~\eqref{Eq:PiBeta2}, we can choose $\sigma_n^2 = \max(\| \mbb{\Pi} \mbb{\widehat{\beta}} \|_{\infty}^2, C^2)$ to fulfill Assumption~(iii) of that result, and Assumption~(ii) is satisfied by Assumption~\ref{Assumption: Nonparametric models}(b). Assumption~(i) is satisfied with $\zeta_f = \zeta_g = s/d_Z$ by Propositions~\ref{Proposition: spline approximation} and \ref{Proposition: spline least squares approximation}, Lemma~\ref{Lemma: spline on spline approximation}, \eqref{Eq:PiBeta2} and Assumption~\ref{Assumption: Nonparametric models}(c). We therefore have by Proposition~\ref{Proposition: product error} that
\begin{align*}
\RN{7}_n &= O_{\mathcal{P}}\bigl(\widetilde{K}_Z^{-2s/d_Z} + \widetilde{K}_Z^{1/2} n^{-1} + \widetilde{K}_Z^{2-s/d_Z} \log(\widetilde{K}_Z)n^{-2}\bigr)\\
 &= O_{\mathcal{P}}(n^{-\frac{4s}{4s+d}} + n^{-\frac{4s+d_X}{4s+d}} + \log(n) n^{-\frac{(10s+2d-4d_Z)}{4s+d}}) = O_{\mathcal{P}}\bigl(n^{-\frac{4s}{4s+d}}\bigr)
\end{align*}
using that $s \geq 3d/4$.  For the final error term, similar to previous terms, for $i \neq j$,
\[
\E\bigl(\{\fhat(X_i, Z_i) - m_{\fhat}(Z_i)\}\{\fhat(X_j, Z_j) - m_{\fhat}(Z_j)\} \given \fhat, Z_i,Z_j\bigr) = 0, 
\]
so, by Hölder's inequality and the triangle inequality, 
\begin{align*}
\E\bigl(&|\RN{8}_n| \given \fhat, \mhat, (Z_i)_{i=1}^n\bigr)\\
&\leq \frac{1}{n} \biggl( \sum_{i=1}^n \E\bigl(\{\fhat(X_i, Z_i) - m_{\fhat}(Z_i)\}^2 \given \fhat, Z_i\bigr) \{\mhat(Z_i)-m(Z_i)\}^2 \biggr)^{1/2}\\
& \leq \frac{2\|\ghat\|_\infty}{n^{1/2}} \biggl( \frac{1}{n} \sum_{i=1}^n \{\mhat(Z_i)- m(Z_i)\}^2 \biggr)^{1/2}.
\end{align*} 
Combining Proposition~\ref{Proposition: b-spline properties}(b), the argument leading to~\eqref{Eq:PiBeta0}, and \eqref{Eq:PiBeta2} yields that $\|\ghat\|_\infty \leq \|\mbb{\widehat{\beta}}_{XZ}\|_\infty = O_{\mathcal{P}}(1)$.  We can therefore apply Corollary~\ref{Corollary: spline regression} and Lemma~\ref{Lemma: unconditionalisation via Markov} as for $\RN{2}_n$ to conclude that
\[
\RN{8}_n = O_{\mathcal{P}}(n^{-\frac{4s}{4s+d}}).
\]
Combining these bounds, we have 
\[
\frac{1}{n} \sum_{i=1}^n L_{i} = \tau(1+R_n),
\]
where
\[
R_n = O_{\mathcal{P}}\bigl(\tau^{-1} n^{-\frac{4s}{4s+d}} + \tau^{-1/2} n^{-1/2}\bigr).
\]
It follows that
\[
R_n = O_{\mathcal{P}_1(\epsilon_n)}\bigl(\epsilon_n^{-1} n^{-\frac{4s}{4s+d}} + \epsilon_n^{-1/2}n^{-1/2}\bigr),
\]
so by Lemma~\ref{Lemma: relationship between big op and small op} and \eqref{Eq: minimum separation rate in nonparametric model}, \eqref{Eq: spline power L claim} holds.

To see that \eqref{Eq: spline power s claim} holds, note that 
\begin{align*}
	\frac{s_n}{n^{1/2}} &\leq \biggl(\frac{1}{n^2} \sum_{i=1}^n L_{i}^2 \biggr)^{1/2} \leq 8^{1/2} \biggl[\underbrace{\biggl(\frac{1}{n^2} \sum_{i=1}^n \varepsilon_i^2 f(X_i, Z_i)^2 \biggr)^{1/2}}_{\widetilde{\RN{1}}_n}\\
	&\hspace{1cm}+ \underbrace{\biggl( \frac{1}{n^2} \sum_{i=1}^n \varepsilon_i^2 \{m(Z_i) - \widetilde{m}(Z_i)\}^2 \biggr)^{1/2}}_{\widetilde{\RN{2}}_n} \\
	&\hspace{1cm}+  \underbrace{\biggl(\frac{1}{n^2} \sum_{i=1}^n f(X_i, Z_i)^2 \{\ghat(X_i, Z_i) - g(X_i, Z_i)\}^2 \biggr)^{1/2}}_{\widetilde{\RN{3}}_n}\\
	&\hspace{1cm}+ \underbrace{\biggl(\frac{1}{n^2} \sum_{i=1}^n \{Y_i - g(X_i, Z_i)\}^2 \{\ghat(X_i, Z_i) - g(X_i, Z_i)\}^2 \biggr)^{1/2}}_{\widetilde{\RN{4}}_n}\\
	&\hspace{1cm}+ \underbrace{\biggl(\frac{1}{n^2} \sum_{i=1}^n \varepsilon_i^2 m_{\fhat}(Z_i)^2 \biggr)^{1/2}}_{\widetilde{\RN{5}}_n} + \underbrace{\biggl(\frac{1}{n^2} \sum_{i=1}^n \varepsilon_i^2 \{m_{\fhat}(Z_i) - \mhat_{\fhat}(Z_i)\}^2 \biggr)^{1/2}}_{\widetilde{\RN{6}}_n} \\
	&\hspace{1cm}+ \underbrace{\biggl(\frac{1}{n^2} \sum_{i=1}^n \{m(Z_i)-\mhat(Z_i)\}^2\{m_{\fhat}(Z_i) - \mhat_{\fhat}(Z_i)\}^2\biggr)^{1/2}}_{\widetilde{\RN{7}}_n}\\
	&\hspace{1cm}+ \underbrace{\biggl(\frac{1}{n^2} \sum_{i=1}^n \{m(Z_i)-\mhat(Z_i)\}^2\{\fhat(X_i, Z_i) - m_{\fhat}(Z_i)\}^2\biggr)^{1/2}}_{\widetilde{\RN{8}}_n}\biggr].
\end{align*}
Combining the bound in \eqref{Eq: spline power rn1 bound} with Lemma~\ref{Lemma: unconditionalisation via Markov} yields that $\widetilde{\RN{1}}_n = O_{\mathcal{P}}(\tau^{1/2}/n^{1/2})$.  All other terms except $\widetilde{\RN{7}}_n$ are of the same uniform stochastic order as the same expressions for the corresponding terms without tildes.  For the final term, then, 
\[
\widetilde{\RN{7}}_n \leq n^{-1/2} \| \mhat - m \|_\infty \biggl( \frac{1}{n} \sum_{i=1}^n \{m_{\fhat}(Z_i) - \mhat_{\fhat}(Z_i)\}^2\biggr)^{1/2}.
\]
Now $\mhat = \mbb{\widehat{\gamma}}^\top \mbb{\psi}$, and we can let $m^\dagger = \mbb{\gamma}^\top \mbb{\psi}$ denote the $L_2(P)$-best approximant of $m$ over $\mathcal{S}_{2r-1,L}^d$.  Then by Proposition~\ref{Proposition: b-spline properties}(b), Corollary~\ref{Corollary: spline regression}, the fact that $s \geq 3d/4$, Propositions~\ref{Proposition: spline approximation} and \ref{Proposition: spline least squares approximation}, we have
\begin{align*}
\| \mhat - m \|_\infty \leq \| \mbb{\widehat{\gamma}} - \mbb{\gamma} \|_{\infty} + \| m^\dagger - m\|_{\infty} &= O_{\mathcal{P}}(\widetilde{K}_Zn^{-1/2} + \widetilde{K}_Z^{-s/d_Z})\\
&= O_{\mathcal{P}}\bigl(n^{-\frac{2d_X}{4s+d}} + n^{-\frac{2s}{4s+d}}\bigr) = o_{\mathcal{P}}(1),
\end{align*}
so 
\[
\widetilde{\RN{7}}_n = o_{\mathcal{P}}\bigl(n^{-\frac{4s}{4s+d}}\bigr)
\]
by~\eqref{Eq: spline power fhat on Z regression}.  We conclude that $s_n/n^{1/2} = \tau U_n$, where
\[
U_n = O_{\mathcal{P}}(\tau^{-1/2} n^{-1/2} + \tau^{-1} n^{-\frac{4s}{4s+d}}) = o_{\mathcal{P}_1(\epsilon_n)}(1),
\]
and hence \eqref{Eq: spline power s claim} is satisfied.  This completes the proof.

\subsection{Proof of Proposition~\ref{Proposition: Lower bound}} \label{Section: Proof of Proposition: Lower bound}

Recall that given two probability measures $\mu$ and $\nu$ on a measurable space $(E,\mathcal{E})$ such that $\mu$ has density $p$ with respect to $\nu$, we define the $\chi^2$-divergence from $\nu$ to $\mu$ by
\[
\chi^2(\mu,\nu) := \int_E p^2 \, d\nu - 1.
\]
Let $\mathcal{A} := [0,1]^{d_X} \times \{-1, 1\} \times [0, 1]^{d_Z}$, and let $P_0 \in \mathcal{P}$ denote a fixed null distribution supported on $\mathcal{A}$.  Further, for each $n \in \mathbb{N}$, let $\mathcal{Q}_n \subseteq \mathcal{P}_1(\epsilon_n)$ denote a finite family of alternative distributions supported on $\mathcal{A}$.  Suppose that $Q \in \mathcal{Q}_n$ has density $q_Q:\mathcal{A} \to [0, \infty)$ with respect to $P_0$ and define $P_0^n := P_0^{\otimes n}$ and 
\[
P_1^n := \frac{1}{|\mathcal{Q}_n|} \sum_{Q \in \mathcal{Q}_n} Q^{\otimes n},
\]
where $\otimes n$ denotes the $n$-fold product of a measure with itself. Suppose that
\begin{equation}
	\label{Eq: lower bound requirement}
\limsup_{n \to \infty} \chi^2(P_1^n, P_0^n) \leq 1.
\end{equation}
Now, for all $n \in \mathbb{N}$ and tests $\phi$,
\begin{align*}
\inf_{P \in \mathcal{P}_1(\epsilon_n)} \prob_{P}(\phi = 1) &\leq \min_{Q \in \mathcal{Q}_n} \prob_{Q}(\phi = 1) \leq \frac{1}{|\mathcal{Q}_n|} \sum_{Q \in \mathcal{Q}_n} \prob_{Q}(\phi = 1)\\
&= \int_{\mathcal{A}^n} \phi \, \mathrm{d}P_1^n \leq   \prob_{P_0}(\phi = 1) + d_{\mathrm{TV}}(P_0^n, P_1^n).
\end{align*}
Defining $q_Q^{\otimes n}(x_1,y_1,z_1,\ldots,x_n,y_n,z_n) := \prod_{i=1}^n q_Q(x_i,y_i,z_i)$, we have by Jensen's inequality that 
\begin{align*}
d_{\mathrm{TV}}(P_0^n, P_1^n)^2 &= \frac{1}{4} \biggl( \int_{\mathcal{A}^n} \biggl| \frac{1}{|\mathcal{Q}_n|} \sum_{Q \in \mathcal{Q}_n} q_Q^{\otimes n} - 1 \biggr| \, \mathrm{d}P_0^n \biggr)^2\\
&\leq \frac{1}{4} \biggl\{ \int_{\mathcal{A}^n} \biggl( \frac{1}{|\mathcal{Q}_n|} \sum_{Q \in \mathcal{Q}_n} q_Q^{\otimes n} - 1 \biggr)^2 \, \mathrm{d}P_0^n \biggr\} = \frac{1}{4}\chi^2(P_1^n,P_0^n).
\end{align*}
Thus
\[
\limsup_{n \to \infty} \inf_{P \in \mathcal{P}_1(\epsilon_n)} \prob_{P}(\phi = 1) \leq \alpha + \frac{1}{2}
\]
for all asymptotically valid tests $\phi$, by \eqref{Eq: lower bound requirement}. 

We now construct $P_0$ and $\mathcal{Q}_n$ such that \eqref{Eq: lower bound requirement} holds. We let $P_0$ denote the uniform distribution on $\mathcal{A}$.  Then
\begin{align*}
	g_{P_0}(x,z) = \E_{P_0}(Y \given X=x,Z=z) = \E_{P_0}(Y) = 0 = \E_{P_0}(Y \given Z=z) = m_{P_0}(z),
\end{align*}
so $\target_{P_0} = 0$.  We also note that $g$ and $m$ are constant functions and so $g, m \in \mathcal{H}_s$ with $\|m\|_{\mathcal{H}_s} = \|g\|_{\mathcal{H}_s} = 0$. It is immediate from similar arguments that the remaining conditions of Assumption~\ref{Assumption: Nonparametric models} are satisfied for $P_0$, so $P_0 \in \mathcal{P}$.

We now aim to construct $\mathcal{Q}_n$.   To this end, define the bump function $K:[0,1/2] \rightarrow [0,\infty)$ by $K(x) := e^{-\frac{1}{x \cdot (1/2-x)^2}}$, let $I_0 := \bigl(\int_0^{1/2} K(u)^2 \, \mathrm{d}u\bigr)^{1/2} \in (0,\infty)$ and define $v: \mathbb{R} \to \mathbb{R}$ by $v(x) := \frac{1}{\sqrt{2}I_0} \cdot K(x)\mathbbm{1}_{\{x \in [0,1/2]\}} - \frac{1}{\sqrt{2}I_0} \cdot K(x-1/2)\mathbbm{1}_{\{x \in [1/2,1]\}}$ and $v(x) := 0$ for $x \in \mathbb{R} \setminus [0,1]$, so that $v$ is infinitely differentiable with $v(0) = v(1) = 0$, $\int_0^1 v(x) \, \mathrm{d}x = 0$ and $\int_0^1 v(x)^2 \, \mathrm{d}x = 1$.  Now define $h: \mathbb{R}^d \to \mathbb{R}$ by $h(x_1, \dots, x_d) := \prod_{j=1}^d v(x_j)$ and note that $h$ is $0$ outside $[0, 1]^d$, $h$ is infinitely differentiable, $\int_{\mathbb{R}^d} h^2(x_1,\ldots,x_d) \, \mathrm{d}x_1\ldots \mathrm{d}x_d = 1$ and $\int_0^1 h(x_1, \dots, x_j, \dots, x_d) \, \mathrm{d}x_j = 0$ for $j \in [d]$.

Define $\rho_n := \big\lfloor{n^{\frac{2}{4s+d}}}\big\rfloor$ and, for $j \in [\rho_n]^d$, define $h_{n,j}:\mathbb{R}^{d_X+d_Z} \rightarrow \mathbb{R}$ by $h_{n,j}(x, z) := \rho_n^{d/2} h\bigl(\rho_n \cdot (x, z) - j + 1\bigr)$, so that $(h_{n,j})_{j \in [\rho_n]^d}$ have disjoint support, $\| h_{n, j} \|_2 = 1$ and $\|h_{n,j}\|_\infty = \rho_n^{d/2}\|h\|_\infty$.  Let $\gamma_n := c^{1/2} n^{-\frac{2s+d}{4s+d}}$, where $c \in \bigl(0,\rho_n^{-d}\|h\|_\infty^{-2}\bigr)$ will be specified later.  For $\mbb{\eta} := (\eta_j)_{j \in [\rho_n]^d} \in \{-1,1\}^{\rho_n^{d}}$, define $g_{n, \mbb{\eta}}:\mathbb{R}^{d_X+d_Z} \rightarrow (-1,1)$ by
\begin{align*}
	g_{n, \mbb{\eta}}(x, z) :=\gamma_n \sum_{j \in [\rho_n]^{d}} \eta_j h_{n,j}(x, z).
\end{align*}   
To see that $g_{n, \mbb{\eta}} \in \mathcal{H}_s^d$, we first note that for any multi-index $\mbb{\alpha} \in \mathbb{N}^d_0$ with $|\mbb{\alpha}| \leq s$, we have
\begin{equation}
\label{Eq: minimax alternative holder 1}
\|D^{\mbb{\alpha}} g_{n, \mbb{\eta}} \|_\infty =  \gamma_n \rho_n^{d/2 + |\mbb{\alpha}|} \| D^{\mbb{\alpha}} h \|_\infty \leq c^{1/2} \max_{\mbb{\widetilde{\alpha}} \in \mathbb{N}_0^d: |\mbb{\widetilde{\alpha}}| \leq s}\| D^{\mbb{\widetilde{\alpha}}} h \|_\infty.
\end{equation}
Now fix $(x, z), (x', z') \in \mathbb{R}^{d_X + d_Z}$; let $j \in [\rho_n]^d$ denote the unique index such that $h_{n,j}(x, z) \neq 0$ if it exists, and otherwise arbitrarily set $j=(1,\ldots,1)^d$.  Similarly, let $j' \in [\rho_n]^d$ denote the unique index such that $h_{n,j'}(x', z') \neq 0$ if it exists, and otherwise set $j'=(1,\ldots,1)^d$.  Then for any $\mbb{\alpha} \in \mathbb{N}_0^d$ with $|\mbb{\alpha}| = \ceil{s} - 1 =: s_0$, we have 
\begin{align}
\label{Eq: minimax alternative holder 2}
&|D^{\mbb{\alpha}} g_{n, \mbb{\eta}}(x, z) - D^{\mbb{\alpha}} g_{n, \mbb{\eta}}(x', z')| \nonumber\\ 
&\leq \gamma_n \rho_n^{d/2 +s_0} \bigl\{ \bigl| D^{\mbb{\alpha}} h\bigl(\rho_n \cdot (x, z) - j + 1 \bigr) - D^{\mbb{\alpha}} h\bigl(\rho_n \cdot (x', z') - j + 1\bigr) \bigr| \nonumber \\
&\hspace{2cm} + \bigl| D^{\mbb{\alpha}} h\bigl(\rho_n \cdot (x, z) - j' + 1\bigr) - D^{\mbb{\alpha}} h\bigl(\rho_n \cdot (x', z') - j' + 1\bigr) \bigr| \bigr\} \nonumber\\
&\leq \gamma_n \rho_n^{d/2 +s_0} \min\biggl(4 \| D^{\mbb{\alpha}}h \|_\infty, 2 \max_{\mbb{\widetilde{\alpha}} \in \mathbb{N}_0^d: |\mbb{\widetilde{\alpha}}| = s_0+1} \| D^{\mbb{\widetilde{\alpha}}}h \|_\infty \rho_n  \|(x, z) - (x', z') \|_1 \biggr) \nonumber\\
&\leq  \gamma_n \rho_n^{d/2 +s_0} \max\biggl(4 \| D^{\mbb{\alpha}}h \|_\infty, 2 \max_{\mbb{\widetilde{\alpha}} \in \mathbb{N}_0^d: |\mbb{\widetilde{\alpha}}| \leq s_0+1} \| D^{\mbb{\widetilde{\alpha}}}h \|_\infty \biggr) \min\bigl(1,  \rho_n  \|(x, z) - (x', z') \|_1 \bigr) \nonumber\\
&\leq c^{1/2} d^{1/2} \max\biggl(4 \| D^{\mbb{\alpha}}h \|_\infty, 2 \max_{\mbb{\widetilde{\alpha}} \in \mathbb{N}_0^d: |\mbb{\widetilde{\alpha}}| \leq s_0+1} \| D^{\mbb{\widetilde{\alpha}}}h \|_\infty \biggr)   \|(x, z) - (x', z') \|^{s-s_0},
\end{align}
where the final inequality uses the fact that $\min(1, t)^y \leq t^y$ for any $t  > 0$ and $y \in (0, 1)$.
Using \eqref{Eq: minimax alternative holder 1} and \eqref{Eq: minimax alternative holder 2} and reducing $c > 0$ such that 
\[
c^{1/2} < \max\biggl(4 \| D^{\mbb{\alpha}} \|_\infty, 2 \max_{\mbb{\alpha} \in \mathbb{N}_0^d: |\mbb{\alpha}| \leq s_0+1} \| D^{\mbb{\alpha}} \|_\infty, \max_{\mbb{\alpha} \in \mathbb{N}_0^d: |\mbb{\alpha}| \leq s}\| D^{\mbb{\alpha}} h \|_\infty \biggr)^{-1} \frac{C}{d^{1/2}},
\]
if necessary, we ensure that $g_{n, \mbb{\eta}} \in \mathcal{H}_s^{d_X+d_Z}$ with $\|g_{n, \mbb{\eta}} \|_{\mathcal{H}_s} \leq C$.

Define now $Q_{n, \mbb{\eta}}$ such that $(X, Z)$ is uniform on $[0, 1]^{d_X} \times [0, 1]^{d_Z}$ and $Y$ is Rademacher with
\[
	\E_{Q_{n, \mbb{\eta}}}(Y \given X=x, Z=z) = g_{n, \mbb{\eta}}(x, z).
\]
Note that by construction
\begin{align*}
	m_{n,\mbb{\eta}}(z) := \int_{[0,1]^{d_X}} g_{n, \mbb{\eta}}(x,z) \, \mathrm{d}x = 0 \quad \text{for any $\mbb{\eta} \in \{-1,1\}^{\rho_n^d}$},
\end{align*}
so $m_{n,\mbb{\eta}} \in \mathcal{H}_s^{d_Z}$ with $\|m_{n,\mbb{\eta}} \|_{\mathcal{H}_s} = 0$. Further,
\begin{align*}
	\target_{Q_{n, \mbb{\eta}}}= \E_{Q_{n, \mbb{\eta}}}\bigl[\bigl\{g_{n, \mbb{\eta}}(X,Z) - m(Z)\bigr\}^2\bigr] = \gamma_n^2 \rho_n^{d} \leq c^2 n^{-\frac{4s}{4s + d}},
\end{align*}
and we deduce from the definition of $\epsilon_n$ that $Q_{n, \mbb{\eta}} \in \mathcal{P}_1(\epsilon_n)$ for sufficiently large $n$.  We let $\mathcal{Q}_n := \bigl\{Q_{n, \mbb{\eta}}: \mbb{\eta} \in \{-1, 1\}^{\rho_n^d}\bigr\}$.

To see that \eqref{Eq: lower bound requirement} is satisfied, we note that
\begin{align*}
\chi^2(P_1^n, P_0^n) &=  - 1 + \frac{1}{|\mathcal{Q}_n|^2} \sum_{Q, Q' \in \mathcal{Q}_n} \int_{\mathcal{A}^n} q_Q^{\otimes n} q_{Q'}^{\otimes n} \, \mathrm{d}P_0^n\\
&= - 1 + \frac{1}{|\mathcal{Q}_n|^2} \sum_{\mbb{\eta}, \mbb{\eta}' \in \{-1,1\}^{\rho_n^d}} \biggl( \int_{\mathcal{A}} q_{Q_{n, \mbb{\eta}}} q_{Q_{n, \mbb{\eta}'}} \, \mathrm{d}P_0 \biggr)^n.
\end{align*}
so it suffices to show that $\limsup$ of the second term is at most $2$ as $n \rightarrow \infty$.  But
\begin{align*}
q_{Q_{n,\mbb{\eta}}}(x, y, z)
&= \{1 + g_{n, \mbb{\eta}}(x, z)\}^{(1+y)/2} \{1 - g_{n, \mbb{\eta}}(x, z)\}^{(1-y)/2},
\end{align*}
so, for $Q_{n,\mbb{\eta}}, Q_{n,\mbb{\eta}'} \in \mathcal{Q}_n$, we have 
\begin{align*}
\int_{\mathcal{A}} q_{Q_{n,\mbb{\eta}}} q_{Q_{n,\mbb{\eta}'}} \, \mathrm{d}P_0 &= \frac{1}{2}\E_{P_0}\bigl(q_{Q_{n, \mbb{\eta}}}(X, 1, Z) q_{Q_{n, \mbb{\eta}'}}(X, 1, Z) \given Y=1\bigr)\\
&\hspace{1cm}+ \frac{1}{2}\E_{P_0}\bigl(q_{Q_{n, \mbb{\eta}}}(X, -1, Z) q_{Q_{n, \mbb{\eta}'}}(X, -1, Z) \given Y=-1\bigr)\\
&= \frac{1}{2} \E_{P_0}\bigl(  \bigl\{1 + g_{n, \mbb{\eta}}(X, Z)\bigr\} \bigl\{1 + g_{n, \mbb{\eta}'}(X, Z)\bigr\}\bigr)\\
&\hspace{1cm}+ \frac{1}{2}\E_{P_0}\bigl(\bigl\{1 - g_{n, \mbb{\eta}}(X, Z)\bigr\} \bigl\{1 - g_{n, \mbb{\eta}'}(X, Z)\bigr\}\bigr)\\
&= 1 + \E_{P_0}\bigl(g_{n, \mbb{\eta}}(X, Z)g_{n, \mbb{\eta}'}(X, Z)\bigr) \\
&= 1 + \gamma_n^2 \sum_{j, j' \in [\rho_n]^{d}} \eta_j \eta_{j'}' \int_{[0,1]^{d_X+d_Z}} h_{n,j}(x, z)h_{n,j'}(x, z) \, \mathrm{d}x \, \mathrm{d}z\\
&= 1 + \gamma_n^2 \mbb{\eta}^\top \mbb{\eta}'.
\end{align*}
Let $\mbb{U} = (U_1,\ldots,U_{\rho_n^d})$ and $\mbb{W} = (W_1,\ldots,W_{\rho_n^d})$ be independent random vectors, each with independent Rademacher components.  Then
\begin{align*}
\frac{1}{|\mathcal{Q}_n|^2} &\sum_{Q, Q' \in \mathcal{Q}_n} \biggl( \int_{\mathcal{A}} q_Q q_{Q'} \, \mathrm{d}P_0 \biggr)^n =  \frac{1}{2^{2 \rho_n^d}} \sum_{\mbb{\eta}, \mbb{\eta}' \in \{-1, 1\}^{\rho_n^d}} (1 + \gamma_n^2\mbb{\eta}^\top \mbb{\eta}')^n\\
&\leq \frac{1}{2^{2 \rho_n^d}} \sum_{\mbb{\eta}, \mbb{\eta}' \in \{-1, 1\}^{\rho_n^d}} e^{n\gamma_n^2\mbb{\eta}^\top \mbb{\eta}'} = \mathbb{E}(e^{n\gamma_n^2 \mbb{U}^\top \mbb{W}}) = \prod_{j=1}^{\rho_n^d} \mathbb{E}(e^{n\gamma_n^2U_jW_j})\\
&= \cosh(n\gamma_n^2)^{\rho_n^d} \leq e^{n^2\gamma_n^4 \rho_n^d/2}.
\end{align*}
Thus, 
\[
	\limsup_{n \to \infty} \frac{1}{|\mathcal{Q}_n|^2} \sum_{Q, Q' \in \mathcal{Q}_n} \biggl( \int_{\mathcal{A}} q_Q q_{Q'} \, \mathrm{d}P_0 \biggr)^n \leq \limsup_{n \to \infty} e^{n^2\gamma_n^4 \rho_n^d/2} \leq \exp(c^2/2).
\]
Taking $c \leq \sqrt{2\log 2}$, we have proved~\eqref{Eq: lower bound requirement} for $P_0$ and $\mathcal{Q}_{n}$, and the result follows.

\section{Auxiliary lemmas}
\label{Section: Auxiliary lemmas}
\subsection{Uniform convergence results}
Recall the `uniformly small in probability' notation $o_\mathcal{P}(1)$ defined in Section~\ref{Section:Notation}.  As above, we sometimes omit the subscript $P$ from quantities depending on $P$ to simplify the presentation. In what follows we collect several technical lemmas that are used in the proofs in Section~\ref{Section: Proofs} of the supplementary material.

\begin{lemma}  \label{Lemma: bounded convergence theorem}
	Let $(X_n)_{n \in \mathbb{N}}$ be a sequence of real-valued random variables. Let $C > 0$ and suppose that $|X_{n}| \leq C$ for all $n \in \mathbb{N}$ and $X_n = o_\mathcal{P}(1)$. Then $\sup_{P \in \mathcal{P}}\E_P(|X_n|) = o(1)$. 
	\begin{proof}
		For any given $\epsilon >0$, 
		\begin{align*}
			|X_n| = |X_n| \ind_{\{|X_n| > \epsilon\}} +  |X_n| \ind_{\{|X_n| \leq \epsilon\}}	\leq  C \ind_{\{|X_n| > \epsilon\}} +  \epsilon.
		\end{align*}
	By the assumption that $X_n = o_\mathcal{P}(1)$, we can choose $N \in \mathbb{N}$ such that $\sup_{P \in \mathcal{P}} \mathbb{P}_P(|X_n| > \epsilon) < \epsilon/C$ for $n \geq N$. It follows that for $n \geq N$,
	\[
	\sup_{P \in \mathcal{P}}\E_P(|X_n|) \leq C \sup_{P \in \mathcal{P}} \prob_P(|X_n| > \epsilon) + \epsilon < 2 \epsilon.
	\]
	Since $\epsilon > 0$ was arbitrary, the result follows.
	\end{proof}
\end{lemma}

The following lemma derives uniform stochastic boundedness of a sequence $(X_n)$ based on a conditional moment condition.
\begin{lemma} \label{Lemma: unconditionalisation via Markov}
	Let $(X_n)_{n \in \mathbb{N}}$ be a sequence of real-valued random variables on $(\Omega, \mathcal{F})$ and let $(\mathcal{F}_n)_{n \in \mathbb{N}}$ be a sequence of sub-$\sigma$-algebras of $\mathcal{F}$. For a positive sequence $(a_n)_{n \in \mathbb{N}}$, possibly depending on $P$, suppose that $\E_P(|X_n| \given \mathcal{F}_n) = O_{\mathcal{P}}(a_n)$. Then $X_n = O_{\mathcal{P}}(a_n)$.
	\begin{proof}
		By hypothesis, given $\epsilon > 0$, there exist $M_{\epsilon} > 0$, $N_\epsilon \in \mathbb{N}$, both depending only on $\epsilon$, such that 
		\begin{align} \label{Eq: argument based on Op}
			\sup_{n \geq N_\epsilon} \sup_{P \in \mathcal{P}}  \prob_{P}(\mathcal{A}_{n,P}) \leq \frac{\epsilon}{2},
		\end{align}  
		where $\mathcal{A}_{n,P} := \bigl\{ \E_P(|X_n| \given  \mathcal{F}_n) \geq M_\epsilon a_n \bigr\}$. Then, by Markov's inequality, for any $K_\epsilon >0$,
		\begin{align*}
			\sup_{n \geq N_\epsilon} \sup_{P \in \mathcal{P}} \prob_{P}( |X_n| &\geq K_\epsilon a_n) =  \sup_{n \geq N_\epsilon} \sup_{P \in \mathcal{P}} \prob_{P}\biggl( \frac{|X_n|}{a_n} \wedge K_\epsilon \geq K_\epsilon\biggr) \\
			&\leq  \frac{1}{K_\epsilon} \sup_{n \geq N_\epsilon} \sup_{P \in \mathcal{P}} \E_P\biggl(\frac{|X_n|}{a_n} \wedge K_\epsilon\biggr) \\
			&\overset{\mathrm{(i)}}{\leq} \frac{1}{K_\epsilon} \sup_{n \geq N_\epsilon} \sup_{P \in \mathcal{P}} \E_P\biggl(\frac{\E_P(|X_n|\given \mathcal{F}_n)}{a_n} \wedge K_\epsilon\biggr) \\
			&\leq \frac{1}{K_\epsilon} \sup_{n \geq N_\epsilon} \sup_{P \in \mathcal{P}} \E_P\biggl\{\biggl(\frac{\E_P(|X_n|\given \mathcal{F}_n)}{a_n} \wedge K_\epsilon\biggr)\ind_{\mathcal{A}_{n,P}}\biggr\}  \\
			& \hspace{1cm}+  \frac{1}{K_\epsilon} \sup_{n \geq N_\epsilon} \sup_{P \in \mathcal{P}} \E_P\biggl\{\biggl(\frac{\E_P(|X_n|\given \mathcal{F}_n)}{a_n} \wedge K_\epsilon\biggr)\ind_{\mathcal{A}_{n,P}^c}\biggr\} \\
			&\leq \sup_{n \geq N_\epsilon} \sup_{P \in \mathcal{P}} \prob_P(\mathcal{A}_{n,P})  + \frac{M_\epsilon}{K_\epsilon} ~\overset{\mathrm{(ii)}}{\leq}~ \frac{\epsilon}{2} + \frac{M_\epsilon}{K_\epsilon},  
		\end{align*}
		where step~(i) uses conditional Jensen's inequality and step~(ii) uses the inequality~(\ref{Eq: argument based on Op}). Then the desired result follows by taking $K_\epsilon \geq 2M_\epsilon/\epsilon$.  
	\end{proof}
\end{lemma}

\begin{lemma} \label{Lemma: product of o and O}
	Let $(X_n)_{n \in \mathbb{N}}$ and $(Y_n)_{n \in \mathbb{N}}$ be sequences of real-valued random variables.  If $X_n = o_\mathcal{P}(1)$ and $Y_n = O_\mathcal{P}(1)$ then $X_nY_n = o_\mathcal{P}(1)$. 
	\begin{proof}
		Let $\epsilon > 0$ be given. Then for any $M > 0$
		\[
			\sup_{P \in \mathcal{P}} \prob_P(|X_n Y_n| > \epsilon) \leq \sup_{P \in \mathcal{P}} \prob_P(|X_n| > \epsilon/M) + \sup_{P \in \mathcal{P}} \prob_P(|Y_n| > M).
		\]
		Choose $M > 0$ and $n_0 \in \mathbb{N}$ large enough that $\sup_{P \in \mathcal{P}} \prob_P(|Y_n| > M)  \leq \epsilon/2$ for all $n \geq n_0$.  By increasing $n_0$ if necessary, we can ensure that  $\sup_{P \in \mathcal{P}} \prob_P(|X_n| > \epsilon/M) \leq \epsilon/2$ for all $n \geq n_0$.  The result follows.
	\end{proof}
\end{lemma}

\begin{lemma} \label{Lemma: relationship between big op and small op}
	Let $(X_n)_{n \in \mathbb{N}}$ and $(R_n)_{n \in \mathbb{N}}$ be sequences of real-valued random variables. Suppose that $R_n > 0$ for all $n \in \mathbb{N}$, $X_n = O_\mathcal{P}(R_n)$ and $R_n = o_\mathcal{P}(1)$. Then $X_n = o_\mathcal{P}(1)$.
	\begin{proof}
		By hypothesis, for any $\epsilon>0$, there exist constants $M_\epsilon > 0$ and $N_\epsilon \in \mathbb{N}$, both depending only on $\epsilon$, such that
		\begin{align*}
			\sup_{n \geq N_\epsilon} \sup_{P \in \mathcal{P}} \pr_P(R_n > \epsilon/M_\epsilon) \leq \epsilon/2 \quad \text{and} \quad \sup_{n \geq N_\epsilon}\sup_{P \in \mathcal{P}} \prob_{P} (|X_n| > R_n M_\epsilon) \leq \epsilon/2.
		\end{align*}
		Therefore,
		\begin{align*}
			 \sup_{n \geq N_\epsilon} \sup_{P \in \mathcal{P}} \prob_P(|X_n| > \epsilon) &\leq \sup_{n \geq N_\epsilon} \sup_{P \in \mathcal{P}} \prob_P\bigl(|X_n| > \epsilon, R_n > \epsilon /M_\epsilon \bigr) \\
			 &\hspace{2cm}+ \sup_{n \geq N_\epsilon}\sup_{P \in \mathcal{P}} \prob_P \bigl(|X_n| > \epsilon, R_n \leq  \epsilon /M_\epsilon \bigr) \\
			 &\leq \sup_{n \geq N_\epsilon}\sup_{P \in \mathcal{P}} \prob_P(R_n > \epsilon/M_\epsilon) +  \sup_{n \geq N_\epsilon} \sup_{P \in \mathcal{P}} \prob_{P} (|X_n| > R_n M_\epsilon) \leq \epsilon,
		\end{align*}
		as required. 
	\end{proof}
\end{lemma}

\begin{lemma} \label{Lemma: product rule}
	Let $(X_n)_{n \in \mathbb{N}}$ and $(Y_n)_{n \in \mathbb{N}}$ be sequences of real-valued random variables. For positive sequences $(a_n)_{n \in \mathbb{N}}$, $(b_n)_{n \in \mathbb{N}}$, suppose that $X_n = O_\mathcal{P}(a_n)$ and $Y_n = O_\mathcal{P}(b_n)$. Then $X_n Y_n = O_\mathcal{P}(a_nb_n)$. 
	\begin{proof}
		For any $\epsilon >0$, there exist $N_\epsilon \in \mathbb{N}$ and $M_\epsilon, K_\epsilon >0$, all depending only on $\epsilon$, such that
		\begin{align*}
			\sup_{n \geq N_\epsilon} \sup_{P \in \mathcal{P}} \prob_{P} (|X_n| > a_n M_\epsilon) \leq \epsilon/2 \quad \text{and} \quad \sup_{n \geq N_\epsilon} \sup_{P \in \mathcal{P}} \prob_{P} (|Y_n| > b_n K_\epsilon) \leq \epsilon/2.
		\end{align*}
		Notice that if $|X_nY_n| > a_nb_n M_\epsilon K_\epsilon$, then either $|X_n| > a_n M_\epsilon$ or $|Y_n| > b_n K_\epsilon$. Therefore, by a union bound,
		\begin{align*}
			\sup_{n \geq N_\epsilon}\sup_{P \in \mathcal{P}} \prob_{P} (|X_nY_n| > a_nb_n M_\epsilon K_\epsilon) &\leq \sup_{n \geq N_\epsilon} \sup_{P \in \mathcal{P}} \prob_{P} (|X_n| > a_n M_\epsilon)\\
			&\hspace{1cm}+ \sup_{n \geq N_\epsilon}\sup_{P \in \mathcal{P}} \prob_{P} (|Y_n| > b_n K_\epsilon) \leq \epsilon,
		\end{align*}
		as desired. 
	\end{proof}
\end{lemma}

\begin{lemma} \label{Lemma: convergence in probability from convergence of conditional expectations}
Let $(X_n)_{n \in \mathbb{N}}$ be a sequence of real-valued random variables on $(\Omega,\mathcal{F})$, and let $(\mathcal{F}_n)_{n \in \mathbb{N}}$ be a sequence of sub-$\sigma$-algebras of $\mathcal{F}$. Suppose that $|X_n| = Y_n + Z_n$. If $Y_n = o_{\mathcal{P}}(1)$ and $\E_P(Z_n \given \mathcal{F}_n) = o_{\mathcal{P}}(1)$, then $X_n = o_\mathcal{P}(1)$.
\end{lemma}
\begin{proof}
	Let $\epsilon \in (0,1/2]$ be given. By Markov's inequality,
	\begin{align*}
		\sup_{P \in \mathcal{P}} \mathbb{P}_P(|X_n| \geq \epsilon) &= \sup_{P \in \mathcal{P}} \mathbb{P}_P(|X_n| \wedge \epsilon \geq \epsilon) \leq \frac{1}{\epsilon} \sup_{P \in \mathcal{P}} \E_P(|X_n| \wedge \epsilon ) \\
		&\leq  \frac{1}{\epsilon} \sup_{P \in \mathcal{P}} \E_P\bigl((\epsilon^2 + Z_n) \wedge \epsilon\bigr) +  \sup_{P \in \mathcal{P}} \prob_P(|Y_n| > \epsilon^2).
	\end{align*}
	The second term converges to $0$ by assumption. For the first term, by Jensen's inequality,
	\begin{align*}
		\frac{1}{\epsilon} \sup_{P \in \mathcal{P}} \E_P\bigl((\epsilon^2 + Z_n) \wedge \epsilon\bigr) &= \frac{1}{\epsilon} \sup_{P \in \mathcal{P}} \E_P\Bigl[\E_P\bigl((\epsilon^2 + Z_n) \wedge \epsilon \given \mathcal{F}_n \bigr)\Bigr]\\
		&\leq \frac{1}{\epsilon} \sup_{P \in \mathcal{P}} \E_P\Bigl[ \bigl\{\epsilon^2 + \E_P( Z_n \given \mathcal{F}_n)\bigr\} \wedge \epsilon \Bigr]\\
		&\leq  2\epsilon +  \sup_{P \in \mathcal{P}} \prob_P\bigl(|\E_P(Z_n \given \mathcal{F}_n)| > \epsilon^2\bigr).
	\end{align*}
The result therefore follows by our hypothesis on $\E_P(Z_n \given \mathcal{F}_n)$.
\end{proof}

\begin{lemma} \label{Lemma: uniform convergence in probability under continuous transformation}
Let $(X_n)_{n \in \mathbb{N}}$ be a sequence of real-valued random variables and let $X$ be another such variable. Assume that $|X_n - X| = o_{\mathcal{P}}(1)$ and let $h:\mathbb{R} \rightarrow \mathbb{R}$ be a continuous function. Suppose that at least one of the following conditions hold:
\begin{enumerate}[(i)]
	\item $h$ is uniformly continuous,
	\item $X$ is uniformly tight, that is,
	\[
		\lim_{M \to \infty}\sup_{P \in \mathcal{P}} \mathbb{P}_P(|X| > M) = 0.
	\]
\end{enumerate}
Then $|h(X_n) - h(X)| = o_\mathcal{P}(1)$.
\end{lemma}
\begin{proof}
	Let $\epsilon > 0$ be given. We need to show that 
	\[
		\lim_{n \to \infty} \sup_{P \in \mathcal{P}} \mathbb{P}_P(|h(X_n)-h(X)| > \epsilon) = 0.
	\]
	If $h$ is uniformly continuous, then we can find $\delta >0$ such that $|h(x) - h(y)| \leq \epsilon$ whenever $|x-y| \leq \delta$.  Thus,
\[
\sup_{P \in \mathcal{P}} \mathbb{P}_P(|h(X_n)-h(X)| > \epsilon) \leq \sup_{P \in \mathcal{P}} \mathbb{P}_P(|X_n-X| > \delta) \rightarrow 0
\]
as $n \to \infty$.   On the other hand, suppose now that $X$ is uniformly tight and let $M > 0$. Since $h$ is continuous, it is uniformly continuous on $[-M, M]$, so we can choose $\delta > 0$ such that  $|h(x) - h(y)| \leq \epsilon$ whenever $x,y \in [-M,M]$ satisfy $|x-y| \leq \delta$.  Hence, for $M > \delta$,
	\begin{align*}
	\sup_{P \in \mathcal{P}} \mathbb{P}_P(|h(X_n)-h(X)| > \epsilon) &\leq \sup_{P \in \mathcal{P}} \mathbb{P}_P(|X_n-X| > \delta)\\
	&\hspace{1cm}+ \sup_{P \in \mathcal{P}} \mathbb{P}_P(|X_n| \vee |X| > M,|X_n-X| \leq \delta) \\
	&\leq \sup_{P \in \mathcal{P}} \mathbb{P}_P(|X_n-X| > \delta) + \sup_{P \in \mathcal{P}} \mathbb{P}_P(|X| > M-\delta) \rightarrow 0
	\end{align*}
as $n, M \to \infty$.
\end{proof}

\begin{lemma} \label{Lemma: conditional clt}
	Let $(X_{n, i})_{n \in \mathbb{N}, i \in [n]}$ be a triangular array of real-valued random variables and let $(\mathcal{F}_n)_{n \in \mathbb{N}}$ be a filtration on $\mathcal{F}$. Assume that
	\begin{enumerate}[(i)]
		\item $X_{n, 1}, \dots, X_{n, n}$ are conditionally independent given $\mathcal{F}_n$, for each $n \in \mathbb{N}$;
		\item $\mathbb{E}_P(X_{n, i} \given \mathcal{F}_n) = 0$ for all $n \in \mathbb{N}, i \in [n]$;
		\item $\bigl| n^{-1} \sum_{i=1}^n \mathbb{E}_P(X_{n, i}^2 \given \mathcal{F}_n) - 1\bigr| = o_\mathcal{P}(1)$;
		\item there exists $\delta > 0$ such that
		\[
			\frac{1}{n}\sum_{i=1}^n \E_{P}\bigl(|X_{n, i}|^{2+\delta} \given \mathcal{F}_n\bigr) = o_{\mathcal{P}}(n^{\delta/2}).
		\]
	\end{enumerate}
	Then $S_n := n^{-1/2} \sum_{m=1}^n X_{n,m}$ converges uniformly in distribution to $N(0, 1)$, i.e. 
	\[
		\lim_{n \to \infty} \sup_{P \in \mathcal{P}} \sup_{x \in \mathbb{R}} |\mathbb{P}_P(S_n \leq x) - \Phi(x)| = 0.
	\]
	\begin{proof}
We will make the dependence of $X_{n, m}$ and $\mathcal{F}_n$ on $P$ clear by instead writing $X_{P, n, i}$ and $\mathcal{F}_{P, n}$ throughout.  By \citet[][Lemma 1]{kasy2019uniformity} it suffices to show that  
		\[
			\frac{1}{\sqrt{n}} \sum_{i=1}^n X_{P_n, n, i} \overset{d}{\to} N(0, 1)
		\]
		for any sequence $(P_n)_{n \in \mathbb{N}}$ in $\mathcal{P}$. Define the triangular array $W_{n, i} := n^{-1/2} X_{P_n, n, i}$ for $n \in \mathbb{N}$ and $i \in [n]$, and let $\tilde{\mathcal{F}}_{n, i}$ be the smallest $\sigma$-algebra containing $\mathcal{F}_{P_n, n}$ that makes $X_{P_n, n, 1}, \dots, X_{P_n, n, i}$ measurable (and $\tilde{\mathcal{F}}_{n,0} := \mathcal{F}_{P_n,n}$). We claim that $(W_{n, i}, \tilde{\mathcal{F}}_{n, i})$ form a martingale difference array. To see this, observe that $W_{n, i}$ is $\tilde{\mathcal{F}}_{n, i}$-measurable and 
		\begin{align*}
			\E_{P_n}(W_{n, i} \given \tilde{\mathcal{F}}_{n, i-1} ) &= \frac{1}{n^{1/2}}\E_{P_n}(X_{P_n, n, i} \given \mathcal{F}_{P_n, n}, X_{P_n, n, 1}, \dots, X_{P_n, n, i-1} )\\
			&= \frac{1}{n^{1/2}}\E_{P_n}(X_{P_n, n, i} \given \mathcal{F}_{P_n, n}) = 0,
		\end{align*}
		where we have used assumptions (i) and (ii) in the penultimate and final equalities, respectively, and this establishes our claim.  Now
		\[
			\sum_{i=1}^n \E_{P_n}(W_{n, i}^2 \given \tilde{\mathcal{F}}_{n, i-1}) = \frac{1}{n} \sum_{i=1}^n \E_{P_n}(X_{P_n, n, i}^2 \given \mathcal{F}_{P_n, n}) \overset{P}{\to} 1, 
		\]
		by assumptions (i) and (iii), and  
		\[
			\sum_{i=1}^n \E_{P_n}\bigl(|W_{n, i}|^{2+\delta} \given \tilde{\mathcal{F}}_{n, i-1}\bigr) = \frac{1}{n^{1+\delta/2}} \sum_{i=1}^n \E_{P_n}\bigl(|X_{P_n, n, i}|^{2+\delta} \given \mathcal{F}_{P_n, n}\bigr) \overset{P}{\to} 0,
		\]
		by assumptions (i) and (iv). It follows that for any $c > 0$,
		\begin{align*}
		    \sum_{i=1}^n \E_{P_n}\bigl(|W_{n, i}|^2 \ind_{\{|W_{n, i}| > c\}} \given \tilde{\mathcal{F}}_{n, i-1}\bigr) 		    &< 	\frac{1}{c^\delta}	    \sum_{i=1}^n \E_{P_n}\bigl(|W_{n, i}|^{2+\delta} \given \tilde{\mathcal{F}}_{n, i-1}\bigr) \overset{P}{\to} 0,
		\end{align*}
		so the conditional Lindeberg condition is satisfied.  The result therefore follows by the Lindeberg--Feller central limit theorem for martingales \citep[e.g.][Theorem 8.2.4]{durrett2019probability}.
	\end{proof}
\end{lemma}

\begin{lemma}
	\label{Lemma: conditional lln}
	Let $(X_{n, i})_{n \in \mathbb{N}, i \in [n]}$ be a triangular array of real-valued random variables and let $(\mathcal{F}_n)_{n \in \mathbb{N}}$ be a filtration on $\mathcal{F}$. Assume that
	\begin{enumerate}[(i)]
		\item $X_{n, 1}, \dots, X_{n, n}$ are conditionally independent given $\mathcal{F}_n$ for all $n \in \mathbb{N}$;
		\item there exists $ \delta \in (0,1]$ such that
		\[
			\sum_{i=1}^n \E_{P}\bigl(|X_{n, i}|^{1+\delta} \given \mathcal{F}_n\bigr) = o_{\mathcal{P}}(n^{1+\delta}).
		\]
	\end{enumerate}
	Then $S_n := n^{-1} \sum_{i=1}^n X_{n,i}$ and $\mu_{P, n} := n^{-1} \sum_{i=1}^n \E_P(X_{n,i} \given \mathcal{F}_n)$ satisfy $|S_n - \mu_{P, n}| =  o_\mathcal{P}(1)$; i.e., for any $\epsilon > 0$
	\[
		\lim_{n \to \infty} \sup_{P \in \mathcal{P}} \mathbb{P}_P(|S_n-\mu_{P, n}| > \epsilon) = 0.
	\]
	\begin{proof}
		For $n \in \mathbb{N}$, $i \in [n]$, define $W_{n, i} := X_{n, i} - \mu_{P, n}$. Note that 
		\begin{equation}
			\label{eq: lln decay}
			\begin{aligned}
				\sup_{P \in \mathcal{P}} \sum_{i=1}^n \E_P\bigl(|W_{n, i}|^{1+\delta} \given \mathcal{F}_n\bigr)
				&\leq 2^{\delta} \biggl(\sup_{P \in \mathcal{P}}  \sum_{i=1}^n  \E_P\bigl(|X_{n, i}|^{1+\delta} \given \mathcal{F}_n\bigr) + n |\mu_{P, n}|^{1+\delta} \biggr)\\
				&\leq 2^{\delta+1} \biggl(\sup_{P \in \mathcal{P}}  \sum_{i=1}^n  \E_P\bigl(|X_{n, i}|^{1+\delta} \given \mathcal{F}_n\bigr)\biggr) = o_{\mathcal{P}}(n^{1+\delta}),
			\end{aligned}
		\end{equation}
		by assumption (ii).  We need to show that for any $\epsilon > 0$,
		\[
		\lim_{n \to \infty} \sup_{P \in \mathcal{P}} \mathbb{P}_P \biggl( \biggl| \frac{1}{n} \sum_{i=1}^n  W_{n, i} \biggr| \geq \epsilon \biggr) = 0.
		\]
		Define $W_{n, i}^< := W_{n, i} \ind_{\{ |W_{n, i}| \leq n \}}$ and $W_{n, i}^> := W_{n, i} \ind_{\{ |W_{n, i}| > n \}}$. By the triangle inequality we can write
		\begin{align*}
			\sup_{P \in \mathcal{P}} \mathbb{P}_P \biggl(& \biggl| \frac{1}{n} \sum_{i=1}^n  W_{n, i} \biggr| \geq \epsilon \biggr) \leq \underbrace{\sup_{P \in \mathcal{P}} \mathbb{P}_P \biggl( \biggl| \frac{1}{n} \sum_{i=1}^n  [W_{n, i}^< - \E_P(W_{n, i}^< \given \mathcal{F}_n)] \biggr| \geq \frac{\epsilon}{3} \biggr)}_{\RN{1}_n}\\
			&+ \underbrace{\sup_{P \in \mathcal{P}}\mathbb{P}_P \biggl( \biggl| \frac{1}{n} \sum_{i=1}^n  W_{n, i}^> \biggr| \geq \frac{\epsilon}{3} \biggr)}_{\RN{2}_n} + \underbrace{\sup_{P \in \mathcal{P}}\mathbb{P}_P \biggl( \biggl| \frac{1}{n} \sum_{i=1}^n  \E_P(W_{n, i}^< \given \mathcal{F}_n) \biggr| \geq \frac{\epsilon}{3} \biggr) }_{\RN{3}_n},
		\end{align*}
		and we will treat each term separately. Considering first $\RN{1}_n$, we note that
		\begin{align*}
			\RN{1}_n &= \sup_{P \in \mathcal{P}} \mathbb{P}_P \biggl( \biggl| \frac{1}{n} \sum_{i=1}^n  [W_{n, i}^< - \E_P(W_{n, i}^< \given \mathcal{F}_n)] \biggr| \wedge \frac{\epsilon}{3} \geq \frac{\epsilon}{3} \biggr)\\
			&\leq \frac{9}{\epsilon^2} \sup_{P \in \mathcal{P}} \E_P \biggl( \biggl\{ \frac{1}{n} \sum_{i=1}^n  \bigl[W_{n, i}^< - \E_P(W_{n, i}^< \given \mathcal{F}_n)\bigr] \biggr\}^2 \wedge \frac{\epsilon^2}{9}  \biggr) \\
			&\leq \frac{9}{\epsilon^2}\sup_{P \in \mathcal{P}} \E_P \biggl( \E_P\biggl[ \biggl\{ \frac{1}{n} \sum_{i=1}^n  [W_{n, i}^< - \E_P(W_{n, i}^< \given \mathcal{F}_n)] \biggr\}^2 \biggm| \mathcal{F}_n \biggr] \wedge \frac{\epsilon^2}{9}\biggr),
		\end{align*}
		where we have applied Markov's inequality and the tower property combined with the monotonicity of conditional expectations to move the minimum inside the conditional expectation. By assumption (i), the terms in the sum of squares are conditionally independent, so the cross terms vanish, and we find
		\begin{align*}
			\RN{1}_n &\leq \frac{9}{\epsilon^2} \sup_{P \in \mathcal{P}} \E_P \biggl( \biggl\{\frac{1}{n^2} \sum_{i=1}^n  \Var_P( W_{n, i}^<  \given \mathcal{F}_n )\biggr\} \wedge \frac{\epsilon^2}{9}  \biggr)\\
			&\leq \frac{9}{\epsilon^2}\sup_{P \in \mathcal{P}} \E_P \biggl( \frac{1}{n^2} \sum_{i=1}^n  \E_P\bigl\{ (W_{n, i}^<)^2  \given \mathcal{F}_n \bigr\} \wedge \frac{\epsilon^2}{9}   \biggr).
		\end{align*}
		Now, for $\delta \in (0,1)$, 
		\begin{align*}
			\E_P\bigl\{ &(W_{n, i}^<)^2  \given \mathcal{F}_n  \bigr\} = \int_0^\infty \mathbb{P}_P\bigl((W_{n, i}^<)^2 > t \given \mathcal{F}_n \bigr) \, \mathrm{d}t =  \int_0^\infty 2 y \, \mathbb{P}_P(|W_{n, i}^<| > y \given \mathcal{F}_n) \, \mathrm{d}y\\
			&\leq \int_0^n 2 y \, \mathbb{P}_P(|W_{n, i}| > y \given \mathcal{F}_n) \, \mathrm{d}y = n^2 \int_0^1 2 u \, \mathbb{P}_P(|W_{n, i}| > nu \given \mathcal{F}_n) \, \mathrm{d}u\\
			&\leq n^{1-\delta} \biggl(\int_0^1 2 u^{-\delta} \, \mathrm{d}u \biggr) \E_P(|W_{n, i}|^{1+\delta} \given \mathcal{F}_n) = \frac{2}{1-\delta}n^{1-\delta}\E_P(|W_{n, i}|^{1+\delta} \given \mathcal{F}_n),
		\end{align*}
		where we have used the substitutions $y = \sqrt{t}$ and $u=(1/n)y$, as well as the conditional version of Markov's inequality.  We deduce that for any $\delta \in (0,1]$, 
		\[
		\RN{1}_n \leq \frac{9}{\epsilon^2} \biggl(\frac{2\mathbbm{1}_{\{\delta \in (0,1)\}}}{1-\delta} + \mathbbm{1}_{\{\delta = 1\}}\biggr) \sup_{P \in \mathcal{P}} \E_P \biggl( \frac{1}{n^{1+\delta}} \sum_{i=1}^n  \E_P\bigl\{( |W_{n, i}|^{1+\delta}  \given \mathcal{F}_n  \bigr\} \wedge \frac{\epsilon^2}{9}  \biggr) \rightarrow 0, 
		\]
		by \eqref{eq: lln decay} and Lemma~\ref{Lemma: bounded convergence theorem}.  
		
		We now deal with $\RN{2}_n$ and $\RN{3}_n$ by first noting that, using similar $\epsilon/3$-thresholding as above, we obtain
		\begin{align*}
		\RN{3}_n &= \sup_{P \in \mathcal{P}}\mathbb{P}_P \biggl( \biggl| \frac{1}{n} \sum_{i=1}^n  \E_P(W_{n, i}^< \given \mathcal{F}_n) \wedge \frac{\epsilon}{3} \biggr| \geq \frac{\epsilon}{3} \biggr)\\
		&\leq \frac{3}{\epsilon} \sup_{P \in \mathcal{P}} \mathbb{E}_P \biggl(  \biggl| \frac{1}{n} \sum_{i=1}^n   \E_P(W_{n, i}^< \given \mathcal{F}_n) \biggr|  \wedge \frac{\epsilon}{3}   \biggr) 
		\end{align*}
		by Markov's inequality. Now, by construction, we can write 
		\[
			\frac{1}{n} \sum_{i=1}^n   \E_P(W_{n, i}^< \given \mathcal{F}_n) = -\frac{1}{n} \sum_{i=1}^n   \E_P(W_{n, i}^> \given \mathcal{F}_n),
		\]
		and thus by the triangle inequality,
		\begin{align*}
			\RN{3}_n \leq \frac{3}{\epsilon} \sup_{P \in \mathcal{P}} \mathbb{E}_P \biggl( \biggl[ \frac{1}{n} \sum_{i=1}^n \bigl|  \E_P(W_{n, i}^> \given \mathcal{F}_n) \bigr|  \biggr] \wedge \frac{\epsilon}{3}  \biggr) +  \frac{3}{\epsilon} \sup_{P \in \mathcal{P}} \mathbb{E}_P \Bigl(  |R_n|  \wedge \frac{\epsilon}{3}   \Bigr).
		\end{align*}
		The second term converges to $0$ by Lemma~\ref{Lemma: bounded convergence theorem}, so it remains to show that the first term converges to $0$. Now~$\RN{2}_n$ can be seen to also be upper bounded by the first term by a similar argument to the one given above, so we are done if we can show that the first term converges to $0$. 
		Applying conditional Hölder's inequality \citep[][Theorem 10.1.6]{gut2013probability} followed by conditional Markov's inequality yields
		\begin{align*}
			\frac{3}{\epsilon} \sup_{P \in \mathcal{P}} &\mathbb{E}_P \biggl( \biggl[ \frac{1}{n} \sum_{i=1}^n \E_P\bigl(| W_{n, i}^>| \given \mathcal{F}_n\bigr)  \biggr] \wedge \frac{\epsilon}{3} \biggr)\\
			&\leq \frac{3}{\epsilon} \sup_{P \in \mathcal{P}} \mathbb{E}_P \biggl( \biggl[ \frac{1}{n} \sum_{i=1}^n \E_P\bigl(| W_{n, i} |^{1+\delta} \given \mathcal{F}_n\bigr)^{\frac{1}{1+\delta}} \mathbb{P}_P\bigl(| W_{n, i} | > n \given \mathcal{F}_n\bigr)^{\frac{\delta}{1+\delta}}   \biggr] \wedge \frac{\epsilon}{3} \biggr)\\
			&\leq \frac{3}{\epsilon} \sup_{P \in \mathcal{P}} \mathbb{E}_P \biggl( \biggl[ \frac{1}{n^{1+\delta}} \sum_{i=1}^n \E_P\bigl(| W_{n, i} |^{1+\delta} \given \mathcal{F}_n\bigr)   \biggr] \wedge \frac{\epsilon}{3} \biggr).
		\end{align*}
		Finally, combining the above with \eqref{eq: lln decay} and Lemma~\ref{Lemma: bounded convergence theorem} yields the desired result.
	\end{proof}
\end{lemma}

\begin{lemma} \label{Lemma: linear regression}
Let $\mathcal{P}$ denote a family of distributions of $(Y, Z)$ taking values in $\mathbb{R} \times \mathbb{R}^d$. Define $\mbb{\Sigma} \equiv \mbb{\Sigma}_P:=\E_P(Z Z^{\top}) \in \mathbb{R}^{d \times d}$ and suppose that this is invertible for all $P \in \mathcal{P}$. Let $\mbb{\beta} \equiv \mbb\beta_P := \mbb{\Sigma}_P^{-1} \E_P(Z Y)$, $\varepsilon \equiv \varepsilon_P := Y - \mbb\beta_P^\top Z$ and $\mbb{\Theta} \equiv \mbb{\Theta}_P := \E(Z Z^{\top} \varepsilon^2)$. Suppose there exist $C, c, \delta > 0$ such that the following hold:
\begin{enumerate}[(i)]
	\item $\inf_{P \in \mathcal{P}} \min\{\lambda_{\min}(\mbb{\Theta}), \lambda_{\min}(\mbb{\Sigma}) \} \geq c$
	\item $\sup_{P \in \mathcal{P}} \max\bigl\{\E_P(\|Z \varepsilon \|_2^{2+\delta}), \E_P(\|Z \|_\infty^{2+\delta})\bigr\} \leq C$.
\end{enumerate}
Given independent copies  $(Y_1, Z_1), \dots, (Y_n, Z_n)$ of $(Y, Z)$, let $\mbb{\widehat{\beta}}$ denote the ordinary least squares estimator from regressing $Y$ on $Z$. Then,
\[
	\sqrt{n} \mbb{\Theta}^{-1/2} \mbb{\Sigma} (\mbb{\widehat{\beta}} - \mbb{\beta})
\]
converges uniformly to a standard $d$-variate Gaussian distribution.

\begin{proof}
Let $\mbb{\widehat{\Sigma}} \equiv (\mbb{\widehat{\Sigma}}_{jk})_{j,k=1}^d := n^{-1} \sum_{i=1}^n Z_i Z_i^\top$ and write $\mbb{\Sigma}_{jk}$ for the $(j,k)$th entry of $\mbb{\Sigma}$.  We first argue that 
\begin{equation}
	\label{Eq: linear regression sigma convergence}
	\| \mbb{\widehat{\Sigma}} - \mbb{\Sigma}_P \|_{\op} = o_{\mathcal{P}}(1).
\end{equation}
By the equivalence of finite-dimensional norms, it suffices to show that $\max_{j,k \in [d]} |\mbb{\widehat{\Sigma}}_{jk} - \mbb{\Sigma}_{jk} | = o_{\mathcal{P}}(1)$, which is equivalent to $\mbb{\widehat{\Sigma}}_{jk} - \mbb{\Sigma}_{jk} =  o_{\mathcal{P}}(1)$ for all $j,k \in [d]$. To show this final claim, let $Z_{ij}$ denote the $j$th component of $Z_i$.  Then 
\[
\mbb{\widehat{\Sigma}}_{jk} - \mbb{\Sigma}_{jk}  = \frac{1}{n} \sum_{i=1}^n \bigl\{Z_{ij} Z_{ik} - \E_P\bigl(Z_{ij}Z_{ik} \bigr)\bigr\}.
\]
Fix $j,k \in [d]$ and define the triangular array $X_{n, i} := Z_{ij}Z_{ik} - \E_P(Z_{ij}Z_{ik})$ for $i \in [n]$ and $n \in \mathbb{N}$.  We aim to apply Lemma~\ref{Lemma: conditional lln} to $(X_{n,i})_{n \in \mathbb{N},i \in [n]}$, where we condition on the trivial $\sigma$-algebra and have $\mu_{P,n} = 0$.  Now, condition~(i) of Lemma~\ref{Lemma: conditional lln} is satisfied by hypothesis, and for condition~(ii) we have by Jensen's inequality and the Cauchy--Schwarz inequality that 
\[
\frac{1}{n} \sum_{i=1}^n \E_P(|X_{n, i}|^{1+\delta/2}) \leq 2^{1+\delta/2} \E_P(|Z_{1j} Z_{1k}|^{1+\delta/2}) \leq 2^{1+\delta/2} \E_P(\|Z\|_\infty^{2+\delta}) \leq 2^{1+\delta/2} C.
\]
Thus \eqref{Eq: linear regression sigma convergence} follows.

We now argue that 
\begin{equation}
	\label{Eq: linear regression sigma_inv convergence}
	\| \mbb{\widehat{\Sigma}}^{-1} - \mbb{\Sigma}^{-1} \|_{\op} = o_{\mathcal{P}}(1).
\end{equation}
It follows immediately our assumption on $\lambda_{\min}(\mbb{\Sigma})$ by Lemma~\ref{Lemma: locally lipschitz matrix inversion} that for any $\epsilon > 0$, we have
\[
\pr_P(\| \mbb{\widehat{\Sigma}}^{-1} - \mbb{\Sigma}^{-1} \|_{\op} \geq \epsilon) \leq \pr_P(\| \mbb{\widehat{\Sigma}} - \mbb{\Sigma} \|_{\op} > c/2) + \pr_P(\| \mbb{\widehat{\Sigma}} - \mbb{\Sigma} \|_{\op} \geq c^2 \epsilon/2).
\]
Thus taking suprema over $\mathcal{P}$ and applying \eqref{Eq: linear regression sigma convergence}, we have shown \eqref{Eq: linear regression sigma_inv convergence}.

We now turn to proving the stated result. Defining $U_n := n^{-1/2}\sum_{i=1}^n  \mbb{\Theta}^{-1/2} Z_i \varepsilon_i$, we have 
\[
	\sqrt{n} \mbb{\Theta}^{-1/2} \mbb{\Sigma} (\mbb{\widehat{\beta}} - \mbb{\beta}) = \mbb{\Theta}^{-1/2} \mbb{\Sigma} \mbb{\widehat{\Sigma}}^{-1} \mbb{\Theta}^{1/2}  U_n.
\]
The summands in the definition of $U_n$ are mean zero with identity covariance matrix and they satisfy Lyapunov's condition since
\[
	\E \bigl( \| \mbb{\Theta}^{-1/2} Z \varepsilon\|_2^{2+\delta} \bigr) \leq \lambda_{\min}(\mbb{\Theta})^{-(1+\delta/2)} \E \bigl( \|  Z \varepsilon\|_2^{2+\delta}\bigr) \leq c^{-(1+\delta/2)} C.
\]
The Lindeberg--Feller central limit theorem \citep[][Proposition 2.27]{vaart1998} therefore yields that $U_n$ converges uniformly to a $d$-variate standard Gaussian. Combining this with a uniform version of Slutsky's theorem \citep[][Theorem 6.3]{bengs2019uniform} and \eqref{Eq: linear regression sigma_inv convergence} yields the desired result.
\end{proof}
\end{lemma}

\subsection{Miscellaneous results} \label{Section: misc results}

\begin{proposition}
\label{Proposition: optimal f}
	Let $X, Y, Z$ be random variables with $Y \in \R$, $\E(Y^4) < \infty$ and $\Var(Y \given X, Z) >0$ almost surely. Then
	\begin{equation} \label{eq:oracle_min}
	\frac{\left(\E[\{Y - \E(Y \given Z)\} f(X, Z)]\right)^2}{\E[\{Y - \E(Y \given X, Z)\}^2 f(X, Z)^2] }
	\end{equation}
	is maximised over $f$ with  $0<\E\{f(X, Z)^4\} <\infty$ by
	\[
	f(X, Z) \propto \frac{\E(Y \given X, Z) - \E(Y \given Z)}{\Var(Y \given X, Z)},
	\]
	and up to positive scaling this is almost surely the unique maximiser.
\end{proposition}
\begin{proof}
The denominator of \eqref{eq:oracle_min} may be written as $\E \{\Var(Y \given X, Z) f(X, Z)^2\}$. Turning to the numerator, we have by the Cauchy--Schwarz inequality that
	\begin{align}
	\label{Eq:NumeratorBound}
		\bigl(\E[\{Y - \E(Y \given Z)\} & f(X, Z)]\bigr)^2 = \Bigl(\E[\{\E(Y \given X, Z) - \E(Y \given Z)\} f(X, Z)]\Bigr)^2 \nonumber \\
		&\leq \E \{\Var(Y \given X, Z) f(X, Z)^2\}  \,\, \E \biggl[\frac{\{\E(Y \given X, Z) - \E(Y \given Z)\}^2}{\Var(Y \given X, Z)}\biggr].
	\end{align}
Since the first factor in this final expression cancels with the denominator of~\eqref{eq:oracle_min}, and since we have equality in~\eqref{Eq:NumeratorBound} if and only if $f(X, Z) \propto \{\E(Y \given X, Z) - \E(Y \given Z)\} /\Var(Y \given X, Z)$ almost surely, the result follows.
\end{proof}

\begin{proposition}
    \label{Proposition: average of Z values}
    Let $Z_1,\ldots,Z_B$ be arbitrarily dependent standard normal random variables. For any $\alpha \in (0,1)$ and $B \in \mathbb{N}$, 
    \begin{equation*}
        \pr\biggl( \frac{1}{B} \sum_{b=1}^B Z_b \geq \frac{\varphi(z_{1-\alpha})}{\alpha}  \biggr) \leq \alpha,
    \end{equation*}
    where $\varphi$ denotes the density function of $N(0,1)$. Moreover, 
    \begin{equation*}
        \lim_{\alpha \rightarrow 0} \frac{\alpha^{-1}\varphi(z_{1-\alpha})}{z_{1-\alpha}} = 1.
    \end{equation*}
\begin{proof}
    Denoting $\bar{Z} := B^{-1}\sum_{b=1}^B Z_b$, observe that, for any $t,a \in \mathbb{R}$ with $t > a$,
    \begin{align*}
        \pr(\bar{Z} \geq t)  & =  \pr\bigl(\max\{\bar{Z} - a, 0\} \geq t - a\bigr) \\
        & \leq  \frac{\E\bigl(\max\{\bar{Z} - a, 0\} \bigr)}{t - a} \leq \frac{\E\bigl(\max\{Z_1 - a, 0\} \bigr)}{t - a},
    \end{align*}
    where the first inequality holds by Markov's inequality and the second by Jensen's inequality. Using the identity $\max\{x,0\} = (x+|x|) / 2$ for $x \in \mathbb{R}$, as well as the mean formula for a folded normal random variable, we compute the expectation above as
    \begin{equation*}
        \E\bigl(\max\{Z_1 - a, 0\} \bigr) = \varphi(a) + a\Phi(a) - a.
    \end{equation*}
    Therefore, taking $t := \alpha^{-1}\{\varphi(a) + a\Phi(a) - a\} + a > a$, we have
    \begin{equation*}
        \pr\biggl(\bar{Z} \geq \frac{\varphi(a) + a\Phi(a) - (1-\alpha)a}{\alpha} \biggr) \leq \alpha.
    \end{equation*}
    Now the first claim follows by observing that 
    \begin{equation*}
        \min_{a \in \mathbb{R}} \bigl\{ \varphi(a) + a\Phi(a) - (1-\alpha)a \bigr\} = \varphi(z_{1-\alpha}).
    \end{equation*}
For the second claim, we invoke the Mills ratio bounds
\[
\frac{z}{z^2+1} < \frac{1-\Phi(z)}{\varphi(z)} < \frac{1}{z}
\]
for $z > 0$.  Applying this bound at $z = z_{1 - \alpha}$ yields 
\[
1 < \frac{\alpha^{-1}\varphi(z_{1-\alpha})}{z_{1-\alpha}} < 1 + \frac{1}{z_{1-\alpha}^2},
\]
so the result follows.
\end{proof}
\end{proposition}
\begin{corollary} \label{Corollary: average of Z values}
    Let $(X_{P,n}^b)_{n \in \mathbb{N},b \in [B]}$ be sequences of real-valued random variables whose distributions are determined by $P \in \mathcal{P}$. Suppose that for each $b \in [B]$, 
    \begin{align} \label{Eq: asymptotic normality}
        \lim_{n \rightarrow \infty} \sup_{P \in \mathcal{P}}\sup_{x \in \mathbb{R}}|\pr_P(X_{P,n}^{b} \leq x) - \Phi(x)| =0.
    \end{align}
    Then for any fixed value of $\alpha \in (0,1)$,
   \begin{align*}
        \limsup_{n \rightarrow \infty} \sup_{P \in \mathcal{P}} \pr_P\biggl( \frac{1}{B} \sum_{b=1}^B X_{P,n}^{b} \geq \frac{\varphi(z_{1-\alpha})}{\alpha}  \biggr) \leq \alpha.
    \end{align*}
\begin{proof}
We aim to prove that for any $\epsilon >0$, there exists $N_\epsilon >0$ such that
\begin{align}
\label{Eq:Neps}
    \sup_{n \geq N_\epsilon}\sup_{P \in \mathcal{P}} \pr_P\bigl(\bar{X}_{P,n}  \geq \alpha^{-1} \varphi(z_{1-\alpha}) \bigr) \leq \alpha + \epsilon,
\end{align}
with $\bar{X}_{P,n}:=B^{-1} \sum_{b=1}^B X_{P,n}^b$. Given $\epsilon >0$, the asymptotic normality condition~\eqref{Eq: asymptotic normality} ensures the existence of $K_\epsilon > 0$ and $N_{\epsilon,1} \in \mathbb{N}$ such that $\mathcal{A}_{P,n,\epsilon}:=\{X_{P,n}^{1} - z_{1-\alpha} \leq K_\epsilon,\ldots, X_{P,n}^{B} - z_{1-\alpha} \leq K_\epsilon\}$ satisfies
    \begin{align*}
        \inf_{n \geq N_{\epsilon,1}}\inf_{P \in \mathcal{P}} \pr_P\bigl(\mathcal{A}_{P,n,\epsilon}\bigr) \geq 1 - \epsilon/2.
    \end{align*}
Then following the proof of Proposition~\ref{Proposition: average of Z values} with $t = \alpha^{-1}\varphi(z_{1-\alpha}) + z_{1-\alpha}$ and $a =z_{1-\alpha}$, we have
\begin{align*}
    & \sup_{n \geq N_{\epsilon,1}}\sup_{P \in \mathcal{P}} \pr_P\bigl(\bar{X}_{n,P}  \geq \alpha^{-1} \varphi(z_{1-\alpha}) \bigr) \\
    & \leq \sup_{n \geq N_{\epsilon,1}}\sup_{P \in \mathcal{P}}\frac{\E_P\bigl(\max\{X_{P,n}^1 - z_{1-\alpha}, 0\}\mathbbm{1}_{\{X_{P,n}^1 - z_{1 - \alpha} \leq K_{\epsilon}\}}\bigr)}{\alpha^{-1}\varphi(z_{1-\alpha})} + \sup_{n \geq N_{\epsilon,1}}\sup_{P \in \mathcal{P}} \pr_P\bigl(\mathcal{A}_{P,n,\epsilon}^c \bigr) \\
    & \leq \sup_{n \geq N_{\epsilon,1}}\sup_{P \in \mathcal{P}}\frac{\E_P\bigl(\max\{X_{P,n}^1 - z_{1-\alpha}, 0\}\ind_{\{X_{P,n}^1 - z_{1 - \alpha} \leq K_{\epsilon}\}}\bigr)}{\alpha^{-1}\varphi(z_{1-\alpha})} + \frac{\epsilon}{2}.
\end{align*}
Observe that $f: x \mapsto \max\{x-z_{1-\alpha},0\} \ind_{\{x - z_{1-\alpha} \leq K_\epsilon\}}$ is a bounded, upper semi-continuous function. Therefore, \citet[][Theorem 2.1 (5)]{bengs2019uniform} along with the asymptotic normality condition~\eqref{Eq: asymptotic normality} guarantees the existence of $N_{\epsilon,2}$ such that 
\begin{align*}
    \sup_{n \geq N_{\epsilon,2}}\sup_{P \in \mathcal{P}} & \frac{\E_P\bigl(\max\{X_{P,n}^1 - z_{1-\alpha}, 0\}\mathbbm{1}_{\{X_{P,n}^1 - z_{1-\alpha} \leq K_{\epsilon}\}}\bigr)}{\alpha^{-1}\varphi(z_{1-\alpha})}  \\
    & \leq \frac{\E\bigl(\max\{Z - z_{1-\alpha}, 0\}\ind_{\{Z - z_{1-\alpha}\leq K_{\epsilon}\}}\bigr)}{\alpha^{-1}\varphi(z_{1-\alpha})} + \frac{\epsilon}{2} \leq \alpha + \frac{\epsilon}{2},
\end{align*}
where $Z \sim N(0,1)$.  The claim~\eqref{Eq:Neps} therefore follows with $N_\epsilon := \max\{N_{\epsilon,1},N_{\epsilon,2}\}$.
\end{proof}
\end{corollary}
We can further develop a more refined threshold than the one obtained in Proposition~\ref{Proposition: average of Z values}. For $Z \sim N(0,1)$, $\alpha \in (0,1/2)$ and $B \geq 2$, define
\begin{align*}
    & H_{\alpha}(x) := (B-1)\Phi^{-1}\bigl(1 - \alpha + (B-1)x\bigr) + \Phi^{-1}(1-x) \quad \text{for $x \in [0,\alpha/B]$,}
\end{align*}
and 
\begin{align*}
    c_B(\alpha) := \min\biggl\{ c \in \bigl[0,\alpha/B\bigr] : \int_c^{\frac{\alpha}{B}} H_{\alpha}(t) \mathrm{d}t \geq \biggl( \frac{\alpha}{B} - c\biggr) H_{\alpha}(c) \biggr\}.
\end{align*}
Write
\begin{align*}
    q_{1-\alpha,B} := \begin{cases}
        B^{-1} H_{\alpha}\bigl(c_B(\alpha)\bigr) \quad & \text{if $c_B(\alpha) > 0$,} \\[.5em]
        \alpha^{-1}\varphi(z_{1-\alpha}) \quad & \text{if $c_B(\alpha) = 0$.}
    \end{cases} 
\end{align*}
Having this notation in place, \citet[][Corollary 3.7]{wang2013bounds} yields 
\begin{align*}
    \pr\biggl( \frac{1}{B} \sum_{b=1}^B Z_b > q_{1-\alpha,B} \biggr) \leq \alpha.
\end{align*}
Notably, according to \citet[][Corollary 3.7]{wang2013bounds}, this threshold is sharp in the sense that
\begin{align*}
    q_{1-\alpha,B} = \sup_{Z_b \sim N(0,1), b \in [B]} \inf \biggl\{s \in \mathbb{R} : \pr\biggl( \frac{1}{B} \sum_{b=1}^B Z_b \leq s\biggr) \geq 1-\alpha\biggr\}.
\end{align*}
It is also worth noting that we have $q_{1-\alpha,2} = z_{1-\alpha/2}$. 
Moreover, $\lim_{B \rightarrow \infty} q_{1-\alpha,B} = \alpha^{-1} \varphi(z_{1-\alpha})$, which illustrates that Proposition~\ref{Proposition: average of Z values} is not improvable for large $B$.

\begin{proposition} \label{Proposition: general type 1 error d sufficient}
Consider the setting of Theorem~\ref{Theorem: General Type I error control}, and suppose that Assumption~\ref{Assumption: general procedure} holds.
\begin{enumerate}[(i)]
	\item If $Y \independent X \given Z$, then \eqref{Eq: type I error abstract condition} is satisfied.
	\item If $\mhat$ is formed using a sample independent from $\mathcal{D}_1$, then \eqref{Eq: type I error abstract condition} is satisfied.
	\item If $\mhat$ is a linear smoother, then \eqref{Eq: type I error abstract condition} is satisfied.
\end{enumerate}
\begin{proof}
As in the main proofs in Section~\ref{Section: Proofs} of the supplementary material, we suppress dependence on $P$ in what follows.

(i) Under conditional independence, $R_{ij} = 0$ for $i \neq j$, so~\eqref{Eq: type I error abstract condition} is satisfied.

(ii) Define $w(z_1, z_2) := \E\bigl(\{\mhat(z_1)-m(z_1)\}\{\mhat(z_2)-m(z_2)\}\bigr)$ and $M_i := m(Z_i)-\mhat(Z_i)$. Note that since $(X_i, Z_i)_{i=1}^n \independent \mhat$, we have 
\[
    \E\bigl(M_iM_j \given (X_{i'}, Z_{i'})_{i'=1}^n \bigr) = w(Z_i, Z_j) = \E\bigl(M_iM_j \given (Z_{i'})_{i'=1}^n\bigr) 
\]
by, e.g.,~\citet[][Example~4.1.7]{durrett2019probability}. Thus $R_{ij}=0$ for $i \neq j$ and \eqref{Eq: type I error abstract condition} is satisfied.

(iii) It suffices to show that $\E\bigl( R_{ij} \xi_{i} \xi_{j} \given (Z_{i'})_{i'=1}^n, \fhat\bigr) = 0$ for $i \neq j$.  Note that when $\mhat(\cdot) = \sum_{k=1}^n \omega(Z_k, \cdot) Y_k$ is a linear smoother,  
\begin{align*}
	\E \bigl( \mhat(Z_i) \given (X_{i'}, Z_{i'})_{i'=1}^n \bigr) &=  \sum_{k=1}^n \omega(Z_k, Z_i) \E\bigl(Y_k \given (X_{i'}, Z_{i'})_{i'=1}^n \bigr) \\
	&= \sum_{k=1}^n \omega(Z_k, Z_i) \E\bigl(Y_k \given (Z_{i'})_{i'=1}^n \bigr) \\
	&= \E\bigl( \mhat(Z_i) \given  (Z_{i'})_{i'=1}^n \bigr).
\end{align*}
Based on this identity, it can be seen that 
\begin{align*}
	R_{ij} &= \E\bigl\{\mhat(Z_i)\mhat(Z_j) \given (X_{i'},Z_{i'})_{i'=1}^n\bigr\} - \E\bigl\{\mhat(Z_i)\mhat(Z_j) \given (Z_{i'})_{i'=1}^n\bigr\} \\
	&= \sum_{k=1}^n  \omega(Z_k,Z_i) \omega(Z_k,Z_j) \bigl\{ \E\bigl(Y_k^2 \given (X_{i'},Z_{i'})_{i'=1}^n\bigr) - \E\bigl(Y_k^2 \given (Z_{i'})_{i'=1}^n\bigr) \bigr\}  \\
	&\hspace{0.5cm}+  \sum_{1 \leq k \neq k' \leq n} \omega(Z_k,Z_i) \omega(Z_{k'},Z_j)  \bigl\{ \E\bigl(Y_kY_{k'} \given (X_{i'},Z_{i'})_{i'=1}^n\bigr) - \E\bigl(Y_kY_{k'} \given (Z_{i'})_{i'=1}^n\bigr) \bigr\} \\
	&= \sum_{k=1}^n  \omega(Z_k,Z_i) \omega(Z_k,Z_j) \bigl\{ \E(Y_k^2 \given X_k, Z_k) - \E(Y_k^2 \given Z_k) \bigr\}.
\end{align*}
Thus
\begin{align*}
	\E\bigl( R_{ij} \xi_{i} \xi_{j} &\given (Z_{i'})_{i'=1}^n, \fhat\bigr)\\
	&=  \sum_{k=1}^n  \omega(Z_k,Z_i) \omega(Z_k,Z_j) \E\Bigl[ \bigl\{ \E(Y_k^2 \given X_k, Z_k) - \E(Y_k^2 \given Z_k) \bigr\} \xi_{i} \xi_{j} \given (Z_{i'})_{i'=1}^n, \fhat \Bigr].
\end{align*}
When $i \neq j$ at least one of $i$ and $j$ differs from $k \in [n]$. Without loss of generality, assume $k \neq i$, so that 
\begin{align*}
\E\Bigl[ \bigl\{ &\E(Y_k^2 \given X_k, Z_k) - \E(Y_k^2 \given Z_k) \bigr\} \xi_{i} \xi_{j} \given (Z_{i'})_{i'=1}^n, \fhat \Bigr] \\[.5em]
	=~ & \E\Bigl[ \bigl\{ \E(Y_k^2 \given X_k, Z_k) - \E(Y_k^2 \given Z_k) \bigr\} \underbrace{\E(\xi_{i} \given Z_i, \fhat)}_{=0} \xi_{j} \given (Z_{i'})_{i'=1}^n, \fhat \Bigr] = 0.
\end{align*}
Therefore, we conclude that $\E\bigl( R_{ij} \xi_{i} \xi_{j} \given (Z_{i'})_{i'=1}^n, \fhat\bigr) = 0$ when $\mhat$ is a linear smoother and thus \eqref{Eq: type I error abstract condition} is satisfied.
\end{proof}
\end{proposition}
Proposition~\ref{Prop: stability} below seeks to control an average product error between population residuals $\xi_i$ and regression errors $\mhat(Z_i) - m_P(Z_i)$.  To obtain the desired $o_{\mathcal{P}}(n^{-1/2})$ conclusion, a naive application of the Cauchy--Schwarz inequality would require stronger assumptions on the performance of our regression estimate than would be reasonable.  Sample splitting provides one way of restoring independence and allows the conclusion to hold, but the purpose of this result is to show that sample splitting is not needed provided that our regression estimate is \emph{sufficiently stable} in the sense that conditions~(ii) and (iii) below hold (in addition to other standard conditions that would be required for the conclusion under sample splitting).
\begin{proposition}
\label{Prop: stability}
Let $(\mathcal{F}_n)_{n \in \mathbb{N}}$ be a sequence of $\sigma$-algebras and let $(Y_i, Z_i, \xi_i)_{i \in [n]}$ be a sequence of random variables taking values in $\mathbb{R} \times \mathcal{Z} \times \mathbb{R}$ that are independent and identically distributed conditional on $\mathcal{F}_n$. Let $\mhat$ denote an estimate of the regression function $m_P: \mathcal{Z} \to \mathbb{R}$ given by $m_P(z) := \E_P(Y \given Z=z)$ formed by regressing $(Y_i)_{i \in [n]}$ on $(Z_i)_{i \in [n]}$. For indices $i_1, \dots, i_k \in [n]$, we define $\mhat^{-i_1, \dots, i_k}$ to be the estimate formed by regressing $Y$ on~$Z$ with the same algorithm as for $\mhat$ except we do not use the observations with indices $i_1, \dots, i_k$. Suppose that $\mathcal{P}$ is a class of distributions for which 
\begin{enumerate}[(i)]
  \item \label{Prop: stability, ass 1} $
    \frac{1}{n} \sum\limits_{i=1}^n \xi_i^2 \bigl(m_P(Z_i) - \mhat(Z_i)\bigr)^2 = o_{\mathcal{P}}(1).
  $
  \item \label{Prop: stability, ass 3} $
    \max_{j \in [n]} \frac{1}{n} \sum\limits_{\substack{i=1\\i \neq j}}^n \bigl(\mhat^{-j}(Z_i) - \mhat(Z_i)\bigr)^2 = o_{\mathcal{P}}(n^{-1}).
    $
  \item \label{Prop: stability, ass 4}$
    \max_{j \in [n]} \frac{1}{n} \sum\limits_{\substack{i=1\\i \neq j}}^n \bigl(\mhat^{-i, j}(Z_i) - \mhat^{-j}(Z_i)\bigr)^2 = o_{\mathcal{P}}(n^{-1}).
  $
  \item\label{Prop: stability, ass 5} For every $n \in \mathbb{N}$ and $i \in [n]$ we have $\E_P(\xi_i \given Z_i, \mathcal{F}_n) = 0$ for all $P \in \mathcal{P}$ and there exists a constant $C > 0$ such that $\sup_{P \in \mathcal{P}} \E_P(\xi_i^2) \leq C$. 
\end{enumerate}
Then 
\[
  \frac{1}{n} \sum_{i=1}^n \xi_i \bigl(m_P(Z_i) - \mhat(Z_i)\bigr) = o_{\mathcal{P}}(n^{-1/2}).
\]
\begin{proof}
As in the main proofs in Section~\ref{Section: Proofs}, we suppress dependence on $P$ in what follows. Assumption \eqref{Prop: stability, ass 4} implies that 
\[
\frac{1}{n} \sum\limits_{i=1}^{n-1} \bigl(\mhat^{-i, n}(Z_i) - \mhat^{-n}(Z_i)\bigr)^2 \leq \max_{j \in [n]} \frac{1}{n} \sum\limits_{\substack{i=1\\i \neq j}}^n \bigl(\mhat^{-i, j}(Z_i) - \mhat^{-j}(Z_i)\bigr)^2 = o_{\mathcal{P}}(n^{-1}).
\]
which, by re-indexing implies, 
\begin{equation}
\label{Prop: stability, ass 2} 
\frac{1}{n} \sum\limits_{i=1}^n \bigl(\mhat^{-i}(Z_i) - \mhat(Z_i)\bigr)^2 = o_{\mathcal{P}}(n^{-1}).
\end{equation}
By Markov's inequality and~\eqref{Prop: stability, ass 5} we have 
\begin{equation}
  \label{Eq: stability_xi^2}
  \frac{1}{n} \sum_{i=1}^n \xi_i^2 = O_{\mathcal{P}}(1).
\end{equation}
To prove the desired result, we write
\[
  \frac{1}{n} \sum_{i=1}^n \xi_i \bigl(m(Z_i) - \mhat(Z_i)\bigr) = \frac{1}{n} \sum_{i=1}^n \xi_i \bigl(m(Z_i) - \mhat^{-i}(Z_i)\bigr) + \frac{1}{n} \sum_{i=1}^n \xi_i \bigl(\mhat^{-i}(Z_i) - \mhat(Z_i)\bigr).
\]
By the Cauchy--Schwarz inequality, \eqref{Eq: stability_xi^2} and \eqref{Prop: stability, ass 2}, we have 
\begin{align*}
  \biggl|\frac{1}{n} \sum_{i=1}^n \xi_i \bigl(\mhat^{-i}(Z_i) - \mhat(Z_i)\bigr) \biggr| &\leq \biggl(\frac{1}{n} \sum_{i=1}^n \xi_i^2\biggr)^{1/2} \biggl(\frac{1}{n} \sum_{i=1}^n \bigl(\mhat^{-i}(Z_i) - \mhat(Z_i)\bigr)^2  \biggr)^{1/2}\\
  &= o_{\mathcal{P}}(n^{-1/2}).
\end{align*}
To deal with the first term in the initial decomposition, we write
\begin{align*}
  \biggl( \frac{1}{n} \sum_{i=1}^n \xi_i \bigl(m(Z_i) - \mhat^{-i}(Z_i)\bigr) \biggr)^2 &= \frac{1}{n^2} \sum_{i=1}^n \xi_i^2 \bigl(m(Z_i) - \mhat^{-i}(Z_i)\bigr)^2\\
  &\hspace{0.25cm}+ \frac{1}{n^2} \sum_{i \neq j}^n \xi_i \xi_j \bigl(m(Z_i) - \mhat^{-i}(Z_i)\bigr) \bigl(m(Z_j) - \mhat^{-j}(Z_j)\bigr).
\end{align*}
For the first term, we note that 
\begin{align*}
  &\frac{1}{n^2} \sum_{i=1}^n \xi_i^2 \bigl(m(Z_i) - \mhat^{-i}(Z_i)\bigr)^2 \leq \frac{2}{n^2} \sum_{i=1}^n \xi_i^2 \bigl\{\bigl(m(Z_i) - \mhat(Z_i)\bigr)^2 + \bigl(\mhat(Z_i) - \mhat^{-i}(Z_i)\bigr)^2 \bigr\}\\
  &\leq \frac{2}{n^2} \sum_{i=1}^n \xi_i^2 \bigl(m(Z_i) - \mhat(Z_i)\bigr)^2 +  2 \biggl(\frac{1}{n} \sum_{i=1}^n \xi_i^2 \biggr) \biggl(\frac{1}{n} \sum_{i=1}^n \bigl(\mhat(Z_i) - \mhat^{-i}(Z_i)\bigr)^2 \biggr) = o_{\mathcal{P}}(n^{-1})
\end{align*}
by \eqref{Eq: stability_xi^2}, \eqref{Prop: stability, ass 1} and \eqref{Prop: stability, ass 2}. For the second term, we write
\begin{align*}
\frac{1}{n^2} \sum_{i \neq j}^n \xi_i \xi_j &\bigl(m(Z_i) - \mhat^{-i}(Z_i)\bigr) \bigl(m(Z_j) - \mhat^{-j}(Z_j)\bigr)\\
&= \underbrace{\frac{1}{n^2} \sum_{i \neq j}^n \xi_i \xi_j \bigl(m(Z_i) - \mhat^{-i, j}(Z_i)\bigr) \bigl(m(Z_j) - \mhat^{-i, j}(Z_j)\bigr)}_{\RN{1}}\\
&+\underbrace{\frac{2}{n^2} \sum_{i \neq j}^n \xi_i \xi_j \bigl(\mhat^{-i, j}(Z_i) - \mhat^{-i}(Z_i)\bigr) \bigl(m(Z_j) - \mhat^{-i, j}(Z_j)\bigr)}_{\RN{2}}\\
&+\underbrace{\frac{1}{n^2} \sum_{i \neq j}^n \xi_i \xi_j \bigl(\mhat^{-i, j}(Z_i) - \mhat^{-j}(Z_i)\bigr) \bigl(\mhat^{-i, j}(Z_j) - \mhat^{-i}(Z_j)\bigr)}_{\RN{3}}\\
&+\underbrace{\frac{1}{n^2} \sum_{i \neq j}^n \xi_i \xi_j \bigl(\mhat^{-j}(Z_i) - \mhat(Z_i)\bigr) (\mhat^{-i}(Z_j) - \mhat(Z_j))}_{\RN{4}}\\
&+\underbrace{\frac{1}{n^2} \sum_{i \neq j}^n \xi_i \xi_j \bigl(\mhat(Z_i) - \mhat^{-i}(Z_i)\bigr) \bigl(\mhat(Z_j) - \mhat^{-j}(Z_j)\bigr)}_{\RN{5}}\\
&+\underbrace{\frac{2}{n^2} \sum_{i \neq j}^n \xi_i \xi_j \bigl(\mhat^{-i, j}(Z_i) - \mhat^{-j}(Z_i)\bigr) \bigl(\mhat^{-i}(Z_j) - \mhat(Z_j)\bigr)}_{\RN{6}}\\
&+\underbrace{\frac{2}{n^2} \sum_{i \neq j}^n \xi_i \xi_j \bigl(\mhat^{-i, j}(Z_i) - \mhat^{-j}(Z_i)\bigr) \bigl(\mhat(Z_j) - \mhat^{-j}(Z_j)\bigr)}_{\RN{7}}\\
&+\underbrace{\frac{2}{n^2} \sum_{i \neq j}^n \xi_i \xi_j \bigl(\mhat^{-j}(Z_i) - \mhat(Z_i)\bigr) \bigl(\mhat(Z_j) - \mhat^{-j}(Z_j)\bigr)}_{\RN{8}}.
\end{align*}
Using the fact that the triples $(Y_i, Z_i, \xi_i)_{i \in [n]}$ are conditionally independent given $\mathcal{F}_n$, together with~\eqref{Prop: stability, ass 5}, for each of the summands in $\RN{1}$ we have
\begin{align*}
&\E\biggl(\xi_i \xi_j \bigl(m(Z_i) - \mhat^{-i, j}(Z_i)\bigr) \bigl(m(Z_j) - \mhat^{-i, j}(Z_j)\bigr) \biggm| (Y_{i'})_{\substack{i'=1\\i' \neq i}}^n, (Z_{i'})_{i'=1}^n, \xi_j , \mathcal{F}_n \biggr)\\
&= \E\biggl(\xi_i \biggm| (Y_{i'})_{\substack{i'=1\\i' \neq i}}^n, (Z_{i'})_{i'=1}^n, \xi_j , \mathcal{F}_n\biggr) \xi_j \bigl(m(Z_i) - \mhat^{-i, j}(Z_i)\bigr) \bigl(m(Z_j) - \mhat^{-i, j}(Z_j)\bigr)\\
&= \E(\xi_i \given Z_i, \mathcal{F}_n) \xi_j \bigl(m(Z_i) - \mhat^{-i, j}(Z_i)\bigr) \bigl(m(Z_j) - \mhat^{-i, j}(Z_j)\bigr) = 0.
\end{align*}
Thus, 
\[
  \E\bigl(\RN{1} \given (Z_i)_{i=1}^n, \mathcal{F}_n\bigr) = 0.
\]
We can apply a similar argument to the summands of $\RN{2}$ and thus also conclude that
\[
  \E\bigl(\RN{2} \given (Z_i)_{i=1}^n, \mathcal{F}_n\bigr) = 0.
\]
By the Cauchy--Schwarz inequality, \eqref{Eq: stability_xi^2} and \eqref{Prop: stability, ass 4}, we have
\begin{align*}
  |\RN{3}| &\leq \frac{1}{n^2} \sum_{i \neq j}^n \xi_j^2 \bigl(\mhat^{-i, j}(Z_i) - \mhat^{-j}(Z_i)\bigr)^2\\
  & \leq \biggl(\frac{1}{n} \sum_{j=1}^n \xi_j^2 \biggr) \biggl(\max_{j \in [n]} \frac{1}{n} \sum_{\substack{i=1\\i \neq j}}^n \bigl(\mhat^{-i, j}(Z_i) - \mhat^{-j}(Z_i)\bigr)^2 \biggr)= o_\mathcal{P}(n^{-1}).
\end{align*}
By Cauchy--Schwarz, \eqref{Eq: stability_xi^2} and \eqref{Prop: stability, ass 3}, we have
\begin{align*}
  |\RN{4}| &\leq \frac{1}{n^2} \sum_{i \neq j}^n \xi_j^2 \bigl(\mhat^{-j}(Z_i) - \mhat(Z_i)\bigr)^2\\
  & \leq \biggl(\frac{1}{n} \sum_{j=1}^n \xi_j^2 \biggr) \biggl(\max_{j \in [n]} \frac{1}{n} \sum_{\substack{i=1\\i \neq j}}^n \bigl(\mhat^{-j}(Z_i) - \mhat(Z_i)\bigr)^2 \biggr)= o_\mathcal{P}(n^{-1}).
\end{align*}
By Cauchy--Schwarz, \eqref{Eq: stability_xi^2} and \eqref{Prop: stability, ass 2}, we have
\begin{align*}
  |\RN{5}| &\leq \frac{1}{n^2} \sum_{i \neq j}^n \xi_j^2 \bigl(\mhat(Z_i) - \mhat^{-i}(Z_i)\bigr)^2\\
  &\leq \biggl(\frac{1}{n} \sum_{j=1}^n \xi_j^2 \biggr) \biggl(\frac{1}{n}\sum_{i=1}^n \bigl(\mhat^{-i}(Z_i) - \mhat(Z_i)\bigr)^2 \biggr)= o_\mathcal{P}(n^{-1}).
\end{align*}
By Cauchy--Schwarz and the bounds above we have
\begin{align*}
  |\RN{6}| &\leq 2 \biggl(\frac{1}{n^2} \sum_{i \neq j}^n \xi_i^2 \bigl(\mhat^{-i}(Z_j) - \mhat(Z_j)\bigr)^2 \biggr)^{1/2} \biggl(\frac{1}{n^2} \sum_{i \neq j}^n \xi_j^2 (\mhat^{-i, j}(Z_i) - \mhat^{-j}(Z_i))^2 \biggr)^{1/2} \\
  &= o_{\mathcal{P}}(n^{-1}).
\end{align*}
Similarly,
\begin{align*}
  |\RN{7}| &\leq 2 \biggl(\frac{1}{n^2} \sum_{i \neq j}^n \xi_i^2 \bigl(\mhat(Z_j) - \mhat^{-j}(Z_j)\bigr)^2 \biggr)^{1/2} \biggl(\frac{1}{n^2} \sum_{i \neq j}^n \xi_j^2 \bigl(\mhat^{-i, j}(Z_i) - \mhat^{-j}(Z_i)\bigr)^2 \biggr)^{1/2} \\
  &= o_{\mathcal{P}}(n^{-1})
\end{align*}
and
\begin{align*}
  |\RN{8}| &\leq 2 \biggl( \frac{1}{n^2} \sum_{i \neq j}^n \xi_i^2 \bigl(\mhat(Z_j) - \mhat^{-j}(Z_j)\bigr)^2 \biggr)^{1/2} \biggl( \frac{1}{n^2} \sum_{i \neq j}^n \xi_j^2 \bigl(\mhat^{-j}(Z_i) - \mhat(Z_i)\bigr)^2 \biggr)^{1/2} \\
  &= o_{\mathcal{P}}(n^{-1}).
\end{align*}
Combining the above with Lemma~\ref{Lemma: convergence in probability from convergence of conditional expectations} yields the desired result.
\end{proof}
\end{proposition}

\begin{lemma}
	\label{Lemma: locally lipschitz matrix inversion}
	Let $\mbb{A}, \mbb{B} \in \mathbb{R}^{k \times k}$ be symmetric matrices and suppose that $\lambda_{\min}(\mbb{A}) \geq c$ and $\| \mbb{A}- \mbb{B} \|_{\op} \leq c/2$ for some $c > 0$. Then $\mbb{B}$ is invertible and
	\[
		\| \mbb{A}^{-1}- \mbb{B}^{-1} \|_{\op} \leq 2 c^{-2} \| \mbb{A}- \mbb{B} \|_{\op}.
	\]
	\begin{proof}
	By Weyl's inequality and our assumptions, we have 
	\[
	\lambda_{\min}(\mbb{B}) \geq \lambda_{\min}(\mbb{A}) - \| \mbb{A} - \mbb{B}\|_{\op} \geq c/2,
	\]
	so $\mbb{B}$ is invertible. We can therefore write
	\begin{align*}
	\| \mbb{A}^{-1} - \mbb{B}^{-1} \|_{\op} &=  \| \mbb{A}^{-1}(\mbb{B} - \mbb{A}) \mbb{B}^{-1} \|_{\op}\\
	&\leq \bigl(\| \mbb{A}^{-1} - \mbb{B}^{-1} \|_{\op} +  \|\mbb{A}^{-1} \|_{\op} \bigr) \| \mbb{A} - \mbb{B}\|_{\op} \| \mbb{A}^{-1} \|_{\op}.
	\end{align*}
	Moreover since $ \| \mbb{A} - \mbb{B}\|_{\op} \| \mbb{A}^{-1} \|_{\op} \leq 1/2$, we deduce that 
	\begin{align*}
		\| \mbb{A}^{-1} - \mbb{B}^{-1} \|_{\op} &\leq \frac{ \| \mbb{A} - \mbb{B}\|_{\op} \| \mbb{A}^{-1} \|_{\op}^2}{1 - \| \mbb{A} - \mbb{B}\|_{\op} \| \mbb{A}^{-1} \|_{\op}}\\
	&\leq 2 \| \mbb{A} - \mbb{B}\|_{\op} \| \mbb{A}^{-1} \|_{\op}^2 \leq 2 c^{-2}\| \mbb{A} - \mbb{B}\|_{\op},
	\end{align*}
	as required.
	\end{proof}
\end{lemma}

\begin{lemma}
\label{Lemma:lambdaminForIndependentXZ}
Consider the setting of Theorem~\ref{Theorem: Asymptotic normality of Spline} and assume that $X$ and $Z$ are independent for all $P \in \mathcal{P}_0$. Then \eqref{Eq:lambdamin} is satisfied for sufficiently small $c > 0$.
\begin{proof}
    Recalling the definitions of $V$ and $\mbb{1} \in \mathbb{R}^{K_X}$ from Proposition~\ref{Proposition: tensor product projection}, we note that for every $\mbb{v} \in \mathbb{R}^{K_Z}$, we can define $\mbb{w} := \mbb{1} \otimes \mbb{v} = (\mbb{I}_{K_Z} \otimes_{\mathrm{Kron}} \mbb{1}) \mbb{v} \in V$ so that
    \[
        \mbb{w}^\top \mbb{\phi}(X, Z) = \mbb{v}^\top \mbb{\phi}^Z(Z),
    \]
    by Proposition~\ref{Proposition: tensor product projection}. Therefore
    \[
        \mbb{w}^\top \mbb{\Lambda}_P \mbb{w} = \E_P\bigl(\Var_P(\mbb{v}^\top \mbb{\phi}^Z(Z) \given Z)\bigr) = 0.
    \]
    Since $\{x \in \mathbb{R}^{K_{XZ}}:\mbb{\Pi}x=x\} = V^\perp$ is $(K_{XZ} - K_Z)$-dimensional and $\mbb{\Lambda}_P$ is non-negative definite,  we conclude that 
    \[
        \widetilde{\lambda}_{\min}(\mbb{\Lambda}_P) = \lambda_{K_{XZ}-K_Z-1}(\mbb{\Lambda}_P),
    \]
    where $\lambda_k(\mbb{\Lambda}_P)$ denotes the $k$th largest eigenvalue of $\mbb{\Lambda}_P$.
    We can write
    \[
        \mbb{\Lambda}_P = \underbrace{\E_P\bigl(\mbb{\phi}(X, Z) \mbb{\phi}(X, Z)^\top\bigr)}_{\mbb{\Sigma}_P} - \underbrace{\E_P\bigl(\E_P(\mbb{\phi}(X, Z) \given Z)\E(\mbb{\phi}(X, Z) \given Z)^\top \bigr)}_{\mbb{\Gamma}_P}.
    \]
    Denote the Kronecker product by $\otimes_{\mathrm{Kron}}$ and note that for $\mbb{x} \in \mathbb{R}^{K_X}$ and $\mbb{z} \in \mathbb{R}^{K_Z}$, we have
    \[
       \mbb{x} \otimes \mbb{z} = (\mbb{I}_{K_Z} \otimes_{\mathrm{Kron}} \mbb{x}) \mbb{z}.
    \]
    Write $\mbb{A}_P := \mbb{I}_{K_Z} \otimes_{\mathrm{Kron}} \E_P\bigl(\mbb{\phi}^X(X)\bigr) \in \mathbb{R}^{K_XK_Z \times K_Z}$.  Then, since $X$ and $Z$ are independent, 
    \[
        \E_P\bigl(\mbb{\phi}(X, Z) \given Z\bigr) =  \E_P\bigl(\mbb{\phi}^X(X)\bigr) \otimes \mbb{\phi}^Z(Z)= \mbb{A}_P\mbb{\phi}^Z(Z).
    \]
    Defining $\mbb{\Sigma}_{Z, P}:= \E_P\bigl(\mbb{\phi}^Z(Z) \mbb{\phi}^Z(Z)^\top\bigr) \in \mathbb{R}^{K_Z \times K_Z}$, it follows that $\mbb{\Gamma}_P = \mbb{A}_P \mbb{\Sigma}_{Z, P} \mbb{A}_P^\top$, so we deduce that $\mathrm{rank}(\mbb{\Gamma}_P) \leq \mathrm{rank}(\mbb{\Sigma}_{Z, P}) \leq K_Z$.  Hence, by Weyl's inequality, 
    \begin{align*}
        &\lambda_{K_{XZ}-K_Z-1}(\mbb{\Lambda}_P) \geq \lambda_{K_{XZ}}(\mbb{\Sigma}_P) +  \lambda_{K_{XZ}-K_Z-1}(-\mbb{\Gamma}_P) \\
        &\hspace{1cm}= \lambda_{K_{XZ}}(\mbb{\Sigma}_P) - \lambda_{K_Z+1}(\mbb{\Gamma}_P) = \lambda_{K_{XZ}}(\mbb{\Sigma}_P)
         \geq  c_s(r)^d K_{XZ}^{-1}\inf_{(x,z) \in [0,1]^d} p_P(x,z),
    \end{align*}
    by Proposition~\ref{Proposition: b-spline properties}(d). This proves the desired claim.
    \end{proof}
\end{lemma}

\begin{corollary}
	\label{Cor:PainfulAnalysis}
	Consider the setting of Proposition~\ref{Proposition: asymptotic power expression}. Assume that $\beta = \lambda n^{-1/2}$ and denote $r := n_2/n_1$. Then, given any $\delta > 0$, we can choose $\lambda_0 \equiv \lambda_0(\alpha,\delta,\sigma_{\xi},\sigma_{\varepsilon\xi}) > 0$ and $r_0 \equiv r_0(\lambda_0,\delta,\sigma_\beta) > 0$ such that 
	\[
		\psi < \frac{1}{2} + \delta,
	\]
	for all $\lambda \geq \lambda_0$ and $r \in (0,r_0]$.  Further, given $\delta > 0$, we can choose $\lambda_1 \equiv \lambda_1(\alpha,\delta,\sigma_{\xi},\sigma_{\varepsilon\xi}) > 0$ and $r_1 \equiv r_1(\lambda_1,\delta,\sigma_\beta) > 0$ such that 
	\[
		\psi < \alpha + \delta
	\]
	for all $\lambda \in (0, \lambda_1]$ and $r \geq r_1$.
	\begin{proof}
	To prove the first claim, note that, for $r < 1/2$, 
	\[
		\psi \leq \Phi\biggl( \frac{\lambda r^{1/2}}{\sigma_{\beta}} \biggr) + \Phi\biggl(z_\alpha- \frac{\lambda \sigma_{\xi}^2}{\sqrt{2}\sigma_{\varepsilon\xi}} \biggr).
	\]
	We can now choose $\lambda_0 \equiv \lambda_0(\alpha,\delta,\sigma_{\xi},\sigma_{\varepsilon\xi}) > 0$ large enough that the second term is at most $\delta/2$ for $\lambda \geq \lambda_0$ and then choose $r_0 \equiv r_0(\lambda_0,\delta,\sigma_\beta) > 0$ small enough that the first term is less than $1/2 + \delta/2$ for $r \in (0,r_0]$. 

	To prove the second claim, note that 
	\[
		\psi \leq \Phi\biggl(z_\alpha+ \frac{\lambda \sigma_{\xi}^2}{\sigma_{\varepsilon\xi}} \biggr) + \Phi\biggl( -\frac{\lambda r^{1/2}}{\sigma_{\beta}} \biggr).
	\]
	We can now choose $\lambda_1 \equiv \lambda_1(\alpha,\delta,\sigma_{\xi},\sigma_{\varepsilon\xi}) > 0$ small enough that the first term is at most $\delta/2 + \alpha$ for $\lambda \in (0, \lambda_1]$ and then choose $r_1 \equiv r_1(\lambda_1,\delta,\sigma_\beta) > 0$ large enough that the second term is less than $\delta/2$ for $r \geq r_1$.
	\end{proof}

\end{corollary}

\section{A discussion of the test of \texorpdfstring{\cite{williamson2020unified}}{Williamson et al}} \label{Section: a discussion of Williamson et al}
\label{Section: linear williamson theoretical comparison}
Like our proposal, the test proposed by \cite{williamson2020unified} relies on sample splitting.  However, their test suffers from two issues as we describe below. To this end, we start by formalising their testing procedure. First split the data $\mathcal{D} = \{(X_i,Y_i,Z_i)\}_{i=1}^{2n}$ randomly into $\mathcal{D}_1$ and $\mathcal{D}_2$ both of size $n$ and let $\mathcal{I}_1$ and $\mathcal{I}_2$ denote the corresponding data indices. We write the sample mean of $Y$ for each split as $\overline{Y}_1 := n^{-1} \sum_{i\in \mathcal{I}_1} Y_i$ and $\overline{Y}_2 := n^{-1} \sum_{i\in \mathcal{I}_2} Y_i$ respectively. Recall the definitions of $g$ and $m$ from Section~\ref{sec:motivation} and let $\ghat$ and $\mhat$ denote generic estimators of these, where $\ghat$ is constructed on $\mathcal{D}_1$ and $\mhat$ is constructed on $\mathcal{D}_2$. For notational convenience, we define $\mu_0 := \E_P(Y)$, $\sigma_Y^2 := \Var_P(Y)$, $\tau_{xz,0} := \E_P\bigl\{\bigl(g(X, Z) - \mu_0\bigr)^2\bigr\}$ and $\tau_{z,0} := \E_P\bigl\{\bigl(m(Z) - \mu_0\bigr)^2\bigr\}$. Let us further define
\begin{align*}
	\widehat v_1 &:= \frac{ \frac{1}{n} \sum_{i \in \mathcal{I}_1} (Y_i - \overline{Y}_1)^2  -  \frac{1}{n} \sum_{i \in \mathcal{I}_1} \{Y_i - \ghat(X_i,Z_i) \}^2 }{\frac{1}{n} \sum_{i \in \mathcal{I}_1} (Y_i - \overline{Y}_1)^2}, \\
	\widehat v_2 &:= \frac{ \frac{1}{n} \sum_{i \in \mathcal{I}_2} (Y_i - \overline{Y}_2)^2  -  \frac{1}{n} \sum_{i \in \mathcal{I}_2} (Y_i - \mhat(Z_i) )^2 }{\frac{1}{n} \sum_{i \in \mathcal{I}_2} (Y_i - \overline{Y}_2)^2},
\end{align*}
and denote their population counterparts by 
\begin{align*}
	v_1 := \frac{\tau_{xz,0}}{\sigma_Y^2} \quad \text{and} \quad v_2 :=  \frac{\tau_{z,0}}{\sigma_Y^2}.
\end{align*}
As shown in  \citet[][Lemma~1]{williamson2019nonparametric}, the influence functions of $v_1$ and $v_2$ are given by
\[
	\varphi_{1}(x,y,z) := \frac{2\{y - g(x,z) \} \{ g(x,z)  - \mu_0 \} + \{ g(x,z) - \mu_0\}^2}{\sigma_Y^2}  - \tau_{xz,0} \biggl\{ \frac{y - \E_P(Y)}{\sigma_Y^2} \biggr\}^2 
	\]
	and
	\[
	\varphi_{2}(y,z) := \frac{2\{y - m(z) \} \{ m(z)  - \mu_0 \} + \{ m(z) - \mu_0\}^2}{\sigma_Y^2}  - \tau_{z,0} \biggl\{ \frac{y - \E_P(Y)}{\sigma_Y^2} \biggr\}^2		
\]
respectively.  Finally, by letting $\widehat{\eta}_1$ and $\widehat{\eta}_2$ be consistent estimators of $\eta_1 := \E_P (\varphi_1(X,Y,Z)^2)$ and $\eta_2 := \E_P(\varphi_2(Y,Z)^2)$, the test statistic proposed by \citet{williamson2020unified} is given as
\begin{align*}
	T_{\mathrm{W}} := \frac{\widehat{v}_1 - \widehat{v}_2}{\sqrt{n^{-1} (\widehat{\eta}_1 + \widehat{\eta}_2)}}.
\end{align*} 
The test statistic $T_{\mathrm{W}}$ is calibrated based on a normal approximation and the null of $\tau_P = 0$ is rejected if $T_{\mathrm{W}}  > z_{1-\alpha}$. Having specified the test function, we are now ready to describe the issues mentioned above.  

\vskip 1em

\subsection{Lack of power} \label{Section: Lack of power}
We shall see that the \citet{williamson2020unified} test has the asymptotic power equal to its size whenever $\sqrt{n} \tau_P \rightarrow 0$, under some regularity conditions. This property is true even for the simple linear model where the optimal detection boundary is known to be $\tau_P \asymp n^{-1}$. To see this, suppose that the assumptions of \citet[][Theorem~1]{williamson2019nonparametric} are satisfied for $\widehat{v}_1$ and $\widehat{v}_2$. That is, $\widehat{v}_1$ and $\widehat{v}_2$ are asymptotically linear with influence functions $\varphi_{1}$ and $\varphi_{2}$, respectively, so that
\[
	 \widehat v_1 - v_1= \frac{1}{n} \sum_{i \in \mathcal{I}_1} \varphi_1(X_i,Y_i,Z_i) + o_P\bigl(n^{-1/2}\bigr) 
	\]
and
\[
\widehat v_2 - v_2 = \frac{1}{n} \sum_{i \in \mathcal{I}_2} \varphi_2(Y_i,Z_i) + o_P\bigl(n^{-1/2}\bigr). 
\]
The asymptotic validity of the approach of \cite{williamson2020unified} comes from the fact that the individual influence functions $\varphi_{1}$ and $\varphi_{2}$ are not necessarily degenerate under the null. In particular, when $\eta_1$ and~$\eta_2$ are non-zero, the central limit theorem guarantees that $T_{\mathrm{W}}$ converges in distribution to $N(0,1)$ under the null (where $v_1 = v_2$). Similarly, we can also establish the asymptotic normality of $T_{\mathrm{W}}$ under the alternative in the case where $\eta_1$ and $\eta_2$ are non-zero. This asymptotic normality allows us to describe the asymptotic power expression of the given test. More formally, the central limit theorem yields
\begin{align*}
	\sqrt{n}(\widehat v_1 - v_1) \convD N(0,  \eta_1) \quad \text{and} \quad \sqrt{n}(\widehat v_2 - v_2) \convD N(0, \eta_2).
\end{align*}
Hence, by Slutsky's theorem and the independence of $\widehat{v}_1$ and $\widehat{v}_2$, we conclude that
\begin{align*}
	\frac{(\widehat v_1 - \widehat v_2) - (v_1 - v_2)}{\sqrt{n^{-1}( \widehat{\eta}_1 + \widehat{\eta}_2)}} \convD N(0,1).
\end{align*}
This shows that 
\begin{align*}
	\mathbb{P}_P (T_{\mathrm{W}} > z_{1-\alpha}) \rightarrow \Phi \biggl( z_\alpha + \frac{\sqrt{n} \tau_P}{\sigma_Y^2\sqrt{\eta_1 + \eta_2}} \biggr),
\end{align*}
where we have used the fact that $v_1 - v_2 = \tau_P / \sigma_Y^2$. Therefore, when $\sigma_Y^2, \eta_1$ and $\eta_2$ are strictly bounded below by some positive constant, the power converges to the nominal level $\alpha$ whenever $\sqrt{n} \tau_P \rightarrow 0$.

\vskip 1em

\subsection{Asymptotic validity} \label{Section: Asymptotic validity}
The previous argument hinges on the condition that $\eta_1$ and $\eta_2$ are non-zero. As acknowledged by \cite{williamson2020unified}, the asymptotic validity of their test is no longer guaranteed when $\eta_1$ and $\eta_2$ are zero. We illustrate this by considering a specific example. 

Consider a simple linear model where $Y = \beta_0 + \beta_1 X + \beta_2 Z + \varepsilon$ and $(X,Z,\varepsilon)^\top$ follows a multivariate normal distribution with zero mean and identity covariance matrix. Assume that $\beta_1 = \beta_2 = 0$. In this scenario, $\varphi_1$ and $\varphi_2$ are the zero functions on their respective domains, so $\eta_1 = \eta_2 = 0$.  Letting $(\widehat \beta_0, \widehat \beta_1, \widehat \beta_1)$ denote the least squares estimator of $(\beta_0,\beta_1, \beta_2)$ based on $\mathcal{D}_1$, and letting $F_{k_1,k_2}$ denote the $F$-distribution with $(k_1,k_2)$ degrees of freedom, we have
\begin{align*}
	\widehat v_1  =  \frac{ \frac{1}{n} \sum_{i \in \mathcal{I}_1} (Y_i - \overline{Y}_1)^2  -  \frac{1}{n} \sum_{i \in \mathcal{I}_1} (Y_i - \widehat \beta_0 - \widehat \beta_1 X_i - \widehat \beta_2 Z_i )^2 }{\frac{1}{n} \sum_{i \in \mathcal{I}_1} (Y_i - \overline{Y}_1)^2} \sim \frac{2}{n-1} F_{2,n-1}.   
\end{align*} 
Similarly, writing $(\widetilde{\beta}_0, \widetilde{\beta}_2)$ for the least squares estimator of $(\beta_0,\beta_2)$ based on $\mathcal{D}_2$, we have
\begin{align*}
	\widehat v_2 = \frac{ \frac{1}{n} \sum_{i \in \mathcal{I}_2} (Y_i - \overline{Y}_2)^2  -  \frac{1}{n} \sum_{i \in \mathcal{I}_2} (Y_i - \widetilde{\beta}_0 -\widetilde{\beta}_2 Z_i )^2 }{\frac{1}{n} \sum_{i \in \mathcal{I}_2} (Y_i - \overline{Y}_2)^2}  \sim \frac{1}{n-1} F_{1,n-1}.
\end{align*}
Since $F_{2,n-1} \convD \chi_2^2/2$ and $F_{1,n-1} \convD \chi_1^2$ where $\chi_k^2$ denotes the chi-square distribution with $k$ degrees of freedom, we observe that
\begin{align*}
	n  (\widehat v_1 - \widehat v_2) \convD 2 U  - 2 V,
\end{align*}
where $U$ and $V$ are independent (due to the sample splitting) with $U \sim \chi_2^2$ and $V \sim \chi_1^2$.  It turns out that the denominator of $T_{\mathrm{W}}$ can also affect the limiting behaviour of $T_{\mathrm{W}}$ non-trivially in this degenerate situation, and the exact form of the limiting distribution relies on the choice of $\widehat{\eta}_1$ and $\widehat{\eta}_2$. Nonetheless, we can still argue that the limiting distribution is not Gaussian. To see this, note that  
\begin{align*}
	\mathbb{P}_P(T_{\mathrm{W}} < 0) \rightarrow \mathbb{P}_P( U  - V < 0) 
	=  \mathbb{P}_P (F_{2,1} < 1/2) \neq \Phi(0).
\end{align*}
Therefore, when $\eta_1 = \eta_2 = 0$, the test based on $T_{\mathrm{W}}$ can be either conservative or anti-conservative depending on the choice of $\alpha$. 

We remark that our result in Section~\ref{Section: Lack of power}
 did not assume linear models, and it also applies to the practical approach suggested by \cite{williamson2020unified} using cross-fitting.  Here, on the other hand, we consider a linear model because it allows us to show explicitly that the asymptotic distribution of $T_W$ is non-Gaussian, so the test is not asymptotically valid.
As acknowledged in \cite{williamson2020unified}, we believe that the same validity issue arises when nonparametric procedures are considered, potentially using cross-fitting. However, the argument becomes unnecessarily complicated so we do not pursue this direction here.

\subsection{Linear model comparison}
\label{Section: linear model sim}
To provide an empirical comparison of the local power properties of the PCM with the approach considered in \citet{williamson2019nonparametric, williamson2020unified} and the more conventional $F$-test with robust standard error \citep{white1980} (as implemented in the R package \texttt{lmtest} \citep{lmtest}), we consider the following setup where $Z$ and $\xi$ are independent $N_5(0, \mbb I)$ random vectors, $\varepsilon \sim N(0, 1)$  independently of $Z$ and $\xi$, $\mbb{\beta} = (1, 2, 3, 4, 5)/\sqrt{n}$ and
\begin{align*}
X &= Z  + \xi,\\
Y &= \mbb{\beta}^\top X  + 2\Bigl(\bigl(1+e^{-3X_1} \bigr)^{-1} + \bigl(1+e^{-3Z_1} \bigr)^{-1} \Bigr) \varepsilon.
\end{align*}
We simulate $n \in \{100, 400, 1600, 6400\}$ observations from the above model. All regressions are performed using OLS, except $\vhat$ which uses a random forest and we only apply Algorithm~\ref{Algorithm: PCM} for simplicity (rather than the derandomised version). We apply the \texttt{wgsc} in two different ways. The test labelled \texttt{wgsc} follows \citet[][Algorithm~3]{williamson2020unified} as in the rest of the paper. The test labelled \texttt{wgsc\_no\_x} is still employs sample-splitting to compute the test but no cross-fitting is done to compute the regression estimators. The resulting test remains valid as the Donsker conditions are satisfied.

The results can be seen in Figure~\ref{fig:linear rates}.
They confirm our theoretical observations in Section~\ref{Section: linear williamson theoretical comparison} on the power properties of the PCM and \texttt{wgsc} tests in this linear model setting.  
\begin{figure}
    \centering
    \includegraphics[scale=0.44]{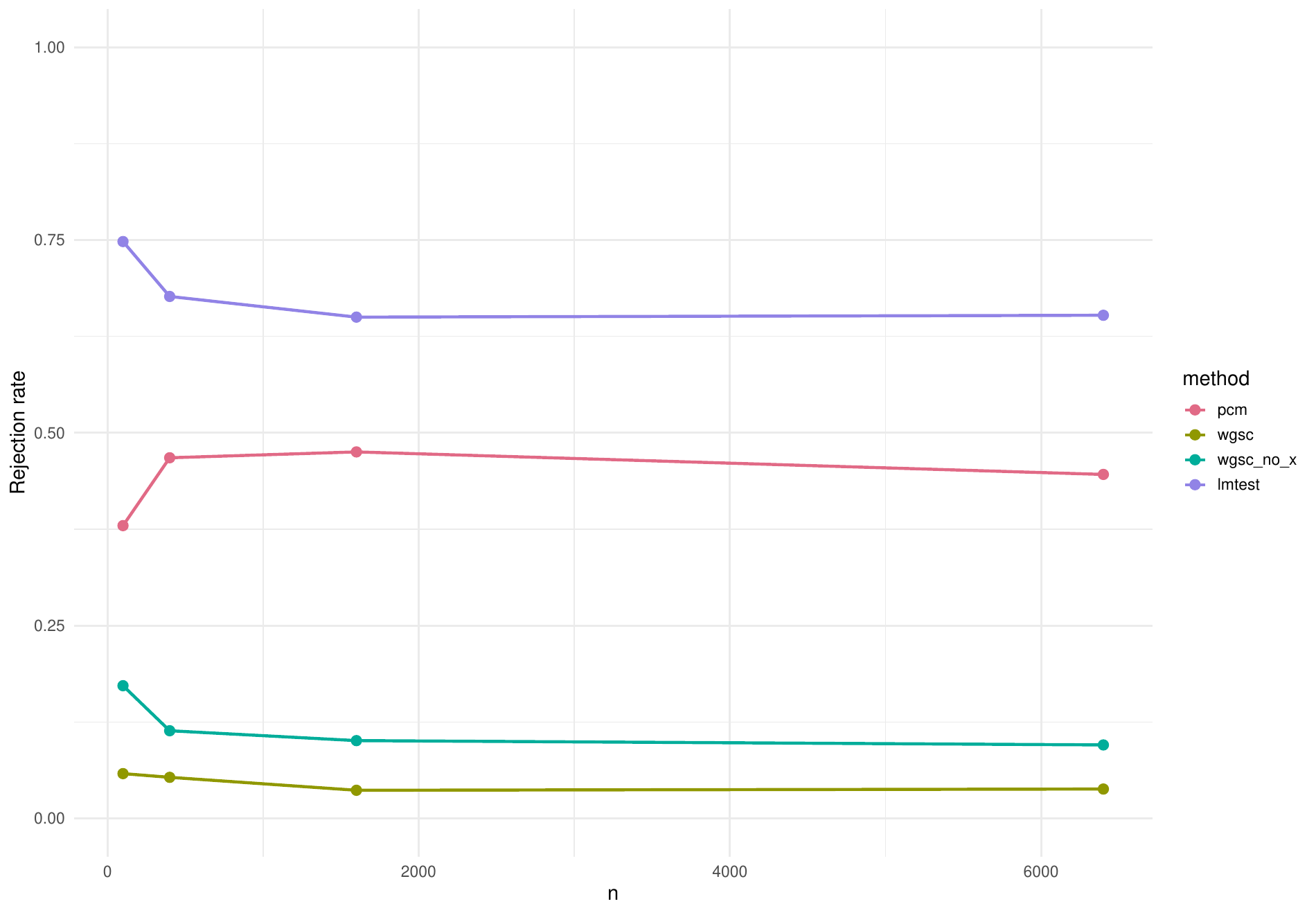}
    \caption{Power in the alternative settings considered in Section~\ref{Section: linear model sim} for nominal 5\%-level tests.}
    \label{fig:linear rates}
\end{figure}

\section{Splines} \label{Section: Splines}
This section is a self-contained description of spline spaces and some of their properties relevant for the spline regressions in Section~\ref{Section: Series estimators}. The definitions given here are not standard in the spline literature, in the sense that they are less general than the usual definitions, but they suffice for the purposes of regression with splines. One particular simplification that we will adhere to throughout is to restrict attention to splines with equi-spaced knots that are defined on the unit hypercube.

We start by considering function spaces of piecewise polynomials with adjustable degrees of smoothness and give a definition of uniform B-splines.

\begin{definition}  \leavevmode \normalfont 
	\label{definition: univariate splines}
Let $N \in \mathbb{N}_0$, and let $\Delta = (\Delta_\ell)_{\ell=0}^{N+1}$ be the knots of an equi-spaced partition of $[0, 1]$, with $\Delta_0:=0$ and $\Delta_{N+1} := 1$.  For $r \in \mathbb{N}$, define the \emph{spline space} $\mathcal{S}_{r, N}$ to be the set of functions $f:[0,1] \to \mathbb{R}$, where the restriction of $f$ to $[\Delta_{\ell-1}, \Delta_{\ell}]$ is a polynomial of degree at most $r-1$ for $\ell \in [N+1]$ and where $f$ is $(r-2)$-times continuously differentiable when $r \geq 2$ (we interpret this as meaning `continuous' when $r = 2$). We say that $r$ is the \emph{order} of $\mathcal{S}_{r, N}$.  Define the vector $t = (t_1,\ldots,t_{N+2r}) \in [0, 1]^{N+2r}$ by 
\[
t := (\underbrace{\Delta_0, \dots, \Delta_0}_{r}, \Delta_1, \dots, \Delta_{N} \underbrace{\Delta_{N+1}, \dots, \Delta_{N+1}}_{r}).
\] 
For $s \in [r]$ and $k \in [N+2r-s]$, define the functions $B_{k, s, N}$ recursively for $x \in [0,1)$ by
\[
B_{k, 1, N}(x) := \mathbbm{1}_{[t_k,t_{k+1})}(x), 
\]
and, for $s \in \{2,\ldots,r\}$,
\begin{align*}
B_{k, s, N}(x) &:= \frac{x-t_k}{t_{k+s-1}-t_k}B_{k, s-1, N}(x) + \frac{t_{k+s}-x}{t_{k+s}- t_{k+1}}B_{k+1, s-1,N}(x) 
\end{align*}
(with the convention that $0/0:=0$).  We also define $B_{k,s,N}(1) := \lim_{x \nearrow 1} B_{k,s,N}(x)$.  The $K := N+r$ functions $B_{1,r,N},\ldots,B_{K,r,N}$ are called \emph{B-splines}.
\end{definition}
It is standard in the spline literature to parametrise spline spaces in terms of the order $r$ of the polynomials rather than the degree ($r-1$). The Curry--Schoenberg theorem gives a relationship between the two definitions above. 
\begin{proposition}[Curry--Schoenberg]
The set of B-splines $\{B_{k, r, N}\}_{k=1}^K$ is a basis for $\mathcal{S}_{r,N}$.
\begin{proof}
	See \citet[][Theorem 4.13]{schumaker2007}.
\end{proof}
\end{proposition}
It is worth noting some properties of B-splines and the spline space $\mathcal{S}_{r,N}$.
\begin{proposition}
\label{Proposition: univariate spline properties}
\begin{enumerate}[(a)]
	\item The B-splines $\{B_{k, r, N}\}_{k=1}^K$are non-negative and form a partition of unity; i.e., 
	\[
		\sum_{k=1}^K B_{k, r, N}(x) = 1.
	\]
	for all $x \in [0, 1]$.
	\item For any $f\in \mathcal{S}_{r, N}$ of the form $f(x)=\sum_{k=1}^K \beta_k B_{k, r, N}(x)$ with $\mbb{\beta}=(\beta_1, \dots, \beta_K) \in \mathbb{R}^K$, there exists $c_s(r) > 0$, depending only on $r$, such that 
	\begin{equation}
	    \label{Eq:deBoor}
		c_s(r) K^{-1/p} \| \mbb{\beta} \|_p  \leq \| f \|_p \leq 2^{1/p} K^{-1/p} \| \mbb{\beta} \|_p .
	\end{equation}
	for all $p \in [1, \infty]$.  In particular, 
	\[
	c_s(r) K^{-1/p}  \leq \| B_{k, r, N} \|_p \leq 2^{1/p} K^{-1/p}
	\]
	for all $k \in [K]$.
\end{enumerate}
\begin{proof}
\emph{(a)} This follows from Equations (4.5) and (4.10) of \citet{boor76}. 

\medskip

\emph{(b)} The conclusion of \citet[][Theorem 5.2]{boor76} yields the existence of $c_s(r) > 0$ such that 
\begin{equation}
	\label{eq:de boor spline bound}
c_s(r) (N+1)^{-1/p} \| \mbb{\beta} \|_p \leq \| f \|_p \leq r^{-1/p} (N+1)^{-1/p} \| \mbb{\beta} \|_p.
\end{equation}
But $(N+1)^{-1/p} \geq K^{-1/p}$ since $K=N+r$, yielding the lower bound in~\eqref{Eq:deBoor}. For the upper bound in~\eqref{Eq:deBoor}, we note that 
\begin{equation}
\label{Eq:UnivariateKBound}
K \leq 2 \max(N, r) \leq 2 (N+1) r
\end{equation}
and rearranging yields the desired result.
\end{proof}
\end{proposition}

We will require splines on $[0, 1]^d$ instead of just $[0,1]$, and to that end we tensorise our earlier spline constructions.
\begin{definition}  \leavevmode \normalfont 
	\label{definition: tensor splines}
Recall Definition~\ref{definition: univariate splines}. Let $d \in \mathbb{N}$ and define the \emph{$d$-tensor spline space}
\[
\mathcal{S}_{r, N}^d := \biggl\{ f: [0, 1]^d \to \mathbb{R}, f(x_1, \dots, x_d)=\prod_{j=1}^d f_j(x_j) : f_j \in \mathcal{S}_{r, N} \ \forall j \in [d] \biggr\}.
\]
Let $\otimes$ denote the vectorised outer product operator, so that $x \otimes y := \mathrm{vec}(xy^\top)$ for Euclidean vectors $x$ and $y$, where $\mathrm{vec}$ denotes the vectorisation operator.  Let $\mbb{B}_{r, N}: [0,1] \rightarrow \mathbb{R}^{N+r}$ have $k$th component function $B_{k,r,N}$, so that $\mbb{B}_{r, N}(x) = \bigl(B_{1,r,N}(x),\ldots,B_{N+r,r,N}(x)\bigr)^\top$.  Now redefine $K := (N+r)^d$; since $\otimes$ is associative, we may define tensor-basis functions $\mbb{\phi} \equiv \mbb{\phi}_{r,N}:[0,1]^d \rightarrow \mathbb{R}^K$ by
\[
\mbb{\phi}(x_1, \dots, x_d) \equiv \bigl(\phi_1(x_1,\ldots,x_d),\ldots,\phi_K(x_1,\ldots,x_d)\bigr)^\top := \mbb{B}_{r, N}(x_1) \otimes \cdots \otimes \mbb{B}_{r, N}(x_d).
\]
\end{definition}
By properties of the tensor product and the Curry--Schoenberg theorem, the collection $\{\phi_k\}_{k=1}^K$ forms a basis for the $d$-tensor spline space $\mathcal{S}_{r,N}^d$ under the usual pointwise addition and scalar multiplication operations, and we refer to it as the \emph{$d$-tensor B-spline basis of $\mathcal{S}_{r,N}^d$}.  In our subsequent asymptotic results, the first of which is Lemma~\ref{lemma: rudelson lln}, we will think of $d$ and $r$ as fixed, but allow $N$ (and consequently $K$) to depend on $n$.  Proposition~\ref{Proposition: b-spline properties} below shows that the properties of univariate B-splines given in Proposition~\ref{Proposition: univariate spline properties} carry over to the $d$-tensor splines.
\begin{proposition}
	\label{Proposition: b-spline properties}
\begin{enumerate}[(a)]
	\item The basis functions $\{\phi_k\}_{k=1}^K$ are non-negative and form a  partition of unity.
	\item For any $f \in \mathcal{S}^d_{r, N}$ of the form $f(x)=\sum_{k=1}^K \beta_k \phi_k(x)$ where $\mbb{\beta}=(\beta_1, \dots, \beta_K) \in \mathbb{R}^K$, and for any $p \in [1, \infty]$,
	\[
		c_s(r)^d K^{-1/p} \| \mbb{\beta} \|_p \leq \| f \|_p \leq 2^{d/p} K^{-1/p} \| \mbb{\beta} \|_p,
	\]
	where $c_s(r) > 0$ is taken from Proposition~\ref{Proposition: univariate spline properties}(b). 
	In particular, for all $k \in [K]$,
	\[
		c_s(r)^d K^{-1/p} \leq \| \phi_k \|_p \leq 2^{d/p}  K^{-1/p}.
	\]
	\item For any $Z$ with distribution $P$ on $[0,1]^d$, the matrix $\mbb{\Sigma}_P := \E_P\bigl( \mbb{\phi}(Z) \mbb{\phi}(Z)^\top \bigr) \in \mathbb{R}^{K \times K}$ satisfies 
	\[
		\lambda_{\min}(\mbb{\Sigma}_P) \leq K^{-1}
	\]
	and
	\[
		\lambda_{\max}(\mbb{\Sigma}_P) \geq K^{-2}.
	\]
	\item Now suppose that $P$ is absolutely continuous with respect to Lebesgue measure on $[0,1]^d$ with density $p$.  If $C := \sup_{z \in [0,1]^d} p(z) < \infty$, then
	\[
		\lambda_{\max}\bigl(\mbb{\Sigma}_P \bigr) \leq C 2^d K^{-1}.
	\]
	If instead $c :=\inf_{z \in [0,1]^d} p(z) > 0$, then
	\[
		\lambda_{\min}\bigl(\mbb{\Sigma}_P \bigr) \geq c c_s(r)^d K^{-1},
	\]
	where $c_s(r) > 0$ is taken from (b).
\end{enumerate}
\begin{proof}
\emph{(a)} This follows from Proposition~\ref{Proposition: univariate spline properties}\emph{(a)} and the definition of the $d$-tensor B-spline basis. 

\medskip

\emph{(b)} 
We will only prove the case $d=2$, since the full result will then follow by induction on $d$.  For the lower bound, we can write 
\[
f(x_1, x_2) = \sum_{k=1}^{\sqrt{K}} \sum_{\ell=1}^{\sqrt{K}} \beta_{k\ell} B_{k, r, N}(x_1) B_{\ell, r, N}(x_2) =:
\sum_{k=1}^{\sqrt{K}} \gamma_k(x_2)B_{k,r,N}(x_1).
\]
For $p \in [1, \infty)$, we have by using \eqref{eq:de boor spline bound} twice that
\begin{align*}
\|f\|_p^p &= \int_0^1 \int_0^1 \biggl|\sum_{k=1}^{\sqrt{K}} \gamma_k(x_2) B_{k, r, N}(x_1) \biggr|^p \, \mathrm{d}x_1 \, \mathrm{d}x_2 \geq \frac{c_s(r)^p}{N+1} \int_0^1 \sum_{k=1}^{\sqrt{K}} |\gamma_k(x_2)|^p  \, \mathrm{d}x_2\\
 &\geq \frac{c_s(r)^{2p}}{(N+1)^2} \sum_{k=1}^{\sqrt{K}} \sum_{\ell=1}^{\sqrt{K}} |\beta_{k\ell}|^p = \frac{c_s(r)^{2p}}{(N+1)^2} \|\mbb{\beta}\|^p_p \geq \frac{c_s(r)^{2p}}{K^2} \|\mbb{\beta}\|^p_p,
\end{align*}
as desired. For $p=\infty$, we have by a similar argument that
\begin{align*}
\|f\|_\infty &= \sup_{x_1, x_2 \in [0, 1]} \biggl| \sum_{k=1}^{\sqrt{K}} \gamma_k(x_2) B_{k, r, N}(x_1) \biggr| \geq c_s(r) \sup_{x_2 \in [0,1]} \max_{k \in [\sqrt{K}]} | \gamma_k(x_2)| \\
&\geq c_s(r)^2 \max_{k, \ell \in [\sqrt{K}]} |\beta_{k\ell}| = c_s(r)^2 \| \mbb{\beta} \|_\infty,
\end{align*}
again as desired. 
For the upper bound, we argue very similarly, and use the fact that, with $K$ redefined as $(N+r)^d$, we have $K \leq 2^d(N+1)^dr^d$ by~\eqref{Eq:UnivariateKBound}.

\medskip

\emph{(c)} Note that 
\[
K \lambda_{\min}(\mbb{\Sigma}_P) \leq \tr(\mbb{\Sigma}_P) = \E\biggl( \sum_{k=1}^K \phi_k^2(Z) \biggr) \leq \E\biggl( \biggl\{\sum_{k=1}^K \phi_k(Z) \biggr\}^2 \biggr) = 1,
\]
and by Cauchy--Schwarz,
\[
K \lambda_{\max}(\mbb{\Sigma}_P) \geq \tr(\mbb{\Sigma}_P) = \E\biggl( \sum_{k=1}^K \phi_k^2(Z) \biggr) \geq \frac{1}{K} \E\biggl( \biggl\{ \sum_{k=1}^K \phi_k(Z) \biggr\}^2 \biggr) = \frac{1}{K}.
\]

\medskip

\emph{(d)} We have
\begin{align*}
\lambda_{\max}(\mbb{\Sigma}_P) &= \sup_{\mbb{\beta} \in \mathbb{R}^K: \| \mbb{\beta} \|_2 = 1} \mbb{\beta}^\top \mbb{\Sigma}_P \mbb{\beta} = \sup_{\mbb{\beta} \in \mathbb{R}^K: \| \mbb{\beta} \|_2 = 1} \E\biggl(  \biggl\{ \sum_{k=1}^K \beta_k \phi_k(Z) \biggr\}^2 \biggr)\\
&\leq C \sup_{\mbb{\beta} \in \mathbb{R}^K: \| \mbb{\beta} \|_2 = 1} \biggl\| \sum_{k=1}^K \beta_k \phi_k \biggr\|_2^2 \leq \frac{2^dC}{K},
\end{align*}
where the final inequality uses~\emph{(b)}.  By a similar argument,
\begin{align*}
	\lambda_{\min}(\mbb{\Sigma}_P) &= \inf_{\mbb{\beta} \in \mathbb{R}^K: \| \mbb{\beta} \|_2 = 1} \mbb{\beta}^\top \mbb{\Sigma}_P \mbb{\beta} = \inf_{\mbb{\beta} \in \mathbb{R}^K: \| \mbb{\beta} \|_2 = 1} \E\biggl(  \biggl\{ \sum_{k=1}^K \beta_k \phi_k(Z) \biggr\}^2 \biggr) \\
	&\geq c \inf_{\mbb{\beta} \in \mathbb{R}^K: \| \mbb{\beta} \|_2 = 1} \biggl\| \sum_{k=1}^K \beta_k \phi_k \biggr\|_2^2 \geq \frac{c}{c_s(r)^d K},
\end{align*}
as required.
\end{proof}
\end{proposition}

We will now argue that splines are strong approximators over classes of sufficiently smooth functions, as defined by H\"older smoothness:
\begin{definition} \leavevmode \normalfont 
	\label{Definition: Holder class}
	Given a multi-index $\mbb{\alpha} =(\alpha_1, \dots, \alpha_d) \in \mathbb{N}_0^d$ with $|\mbb{\alpha}| := \sum_{j=1}^d \alpha_j$ and an $|\mbb{\alpha}|$-times differentiable function $f:[0,1]^d \rightarrow \mathbb{R}$, we define
	\[
	D^{\mbb{\alpha}} f:= \frac{\partial^{\alpha_1}}{\partial x_1^{\alpha_1}} \dots \frac{\partial^{\alpha_d}}{\partial x_d^{\alpha_d}} f.
	\]
	For $s > 0$, write $s_0 := \ceil{s} - 1$ and define $\mathcal{H}_s \equiv \mathcal{H}_s^d$ to be the set $f: [0,1]^d \to \mathbb{R}$ that are $s_0$-times differentiable and that satisfy
	\begin{equation}
	\label{Eq:Holder1}
	\max_{\mbb{\alpha} \in \mathbb{N}_0^d: |\mbb{\alpha}| = s_0} |D^{\mbb{\alpha}} f(x) - D^{\mbb{\alpha}}f(\tilde{x})| \leq C \| x - \tilde{x} \|_2^{s-s_0} \quad \forall x, \tilde{x} \in [0,1]^d
	\end{equation}
	and
	\begin{equation}
	\label{Eq:Holder2}
	\max_{\mbb{\alpha} \in \mathbb{N}_0^d: |\mbb{\alpha}| = s_0} \| D^{\mbb{\alpha}} f \|_\infty \leq C
	\end{equation}
	for some $C > 0$.  If $f \in \mathcal{H}_s$, then the infimum of the set of $C > 0$ for which both~\eqref{Eq:Holder1} and~\eqref{Eq:Holder2} hold is called the $s$-Hölder norm, and is denoted by $\|f\|_{\mathcal{H}_s}$.
\end{definition}
The following basic result shows that given normed space of real-valued functions on $[0,1]^d$ containing $\mathcal{S}_{r,N}^d$, we can find a best approximant within $\mathcal{S}_{r,N}^d$.
\begin{lemma}
\label{Lemma:BestApproximant}
Let $(\mathcal{V},\|\cdot\|)$ denote a normed space of real-valued functions on $[0,1]^d$ that contains $\mathcal{S}_{r,N}^d$ as a subspace.  Then given any $f \in \mathcal{V}$, there exists $f^* \in \mathcal{S}_{r,N}^d$ such that $\|f - f^*\| = \inf_{g \in \mathcal{S}_{r,N}^d} \|f - g\|$.  If $\|\cdot\|$ is strictly convex, then this best approximant is unique.
\end{lemma}
\begin{proof}
Since $\mathcal{S}_{r,N}^d$ is a finite-dimensional subspace of $\mathcal{V}$, \citet[][Theorem 1.2]{powell1981} guarantees the existence of the best approximant $f^*$.  The uniqueness property follows from \citet[][Theorem 2.4]{powell1981} since $\mathcal{S}_{r,N}^d$ is convex.  
\end{proof}

	The approximation properties of splines over Hölder smoothness classes are characterised below.
\begin{proposition}
	\label{Proposition: spline approximation}
	Suppose $f \in \mathcal{H}_s^d$. Then there exists $C(d, r) > 0$ and $f^* \in \mathcal{S}_{r,N}^d$ such that
	\begin{equation}
	\label{Eq:InfinityNormApproximantBound}
	\| f - f^*\|_\infty \leq \frac{C(d, r)}{(N+1)^{\min(s, r)}} \|f\|_{\mathcal{H}_s} \leq \frac{C(d, r)}{(2rK)^{\min(s,r)/d}} \|f\|_{\mathcal{H}_s}.
	\end{equation}
	\begin{proof}
	Given $g:[0,1]^d \rightarrow \mathbb{R}$, $j \in [d]$, $h > 0$ and $k \in \mathbb{N}$, we define the \emph{$k$th forward difference of $g$ in coordinate $j$ with spacing $h$ at $x$} by
	\[
	\Delta_{j,h}^k g(x) := \sum_{\ell=0}^k (-1)^{k-\ell} \binom{k}{\ell} g(x + \ell he_j),
	\]
	where $e_j \in \mathbb{R}^d$ denotes the $j$th standard basis vector.  The \emph{$k$th modulus of smoothness of $g$ in coordinate $j$ of radius $t \in (0,1/k]$} is then defined as
	\[
	\omega_j^k(g;t) := \sup_{h \in [0,t]} \sup_{x \in [0,1-kh]} |\Delta_{j,h}^k g(x)|.
	\]
	By Lemma~\ref{Lemma:BestApproximant} and \citet[][Theorem 12.8]{schumaker2007}, there exists $f^* \in \mathcal{S}_{r,N}^d$ such that 
	\[
	\| f - f^* \|_\infty \leq C'(d, r) \sum_{j=1}^d \omega_j^r\bigl(f; 1/(N+1)\bigr),
	\]
	for some $C'(d,r) > 0$ depending only on $d$ and $r$.
	First consider the case $r \geq s$, and recall the notation $s_0 := \lceil s \rceil - 1$.   By \citet[][(2.119) and (2.117) in Theorem~2.59]{schumaker2007},
	\begin{align*}
	\omega_j^r\bigl(f; 1/(N+1)\bigr) &\leq \frac{1}{(N+1)^{s_0}} \omega_j^{r-s_0}\bigl(D^{s_0 e_j} f; 1/(N+1)\bigr)\\
	 &\leq \frac{2^{r-s_0-1}}{(N+1)^{s_0}} \omega_j^1\bigl(D^{s_0 e_j} f; 1/(N+1)\bigr) \leq \frac{2^{r-1}}{(N+1)^s}\| f \|_{\mathcal{H}_s}.
	\end{align*}
	On the other hand, if $r < s$, then by \citet[][(2.120) in Theorem~2.59]{schumaker2007},
	\[
	\omega_j^r\bigl(f; 1/(N+1)\bigr) \leq \frac{1}{(N+1)^r} \| D^{r e_j} f \|_\infty \leq  \frac{1}{(N+1)^r} \| f \|_{\mathcal{H}_s} .
	\]
	Combining these bounds yields the first inequality in~\eqref{Eq:InfinityNormApproximantBound} with $C(d,r) := 2^{r-1}dC'(d,r)$.  The final bound again follows from the fact that $K \leq 2^d(N+1)^dr^d$. 
	\end{proof}
\end{proposition}
Our next result provides a way of translating properties between population least squares approximants and supremum norm approximants.  
\begin{proposition}
	\label{Proposition: spline least squares approximation}
Let $Z$ be a random vector taking values in $[0,1]^d$, and let $\mathcal{F}$ be a class of functions $f: [0, 1]^d \to \mathbb{R}$ with $\mathbb{E}\bigl(f(Z)^2\bigr) < \infty$ for all $f \in \mathcal{F}$.  Then for each $f \in \mathcal{F}$, there exists a unique $f^\dagger \in \mathcal{S}_{r, N}^d$ such that
\[
\E\bigl(\{f(Z) - f^\dagger(Z)\}^2\bigr) = \inf_{g \in \mathcal{S}^d_{r, N}} \E\bigl(\{f(Z) - g(Z)\}^2\bigr).
\]
Now fix $f \in \mathcal{F}$ and $f^* \in \mathcal{S}_{r,N}^d$.  Further, suppose that $Z$ has a density $p$ with respect to Lebesgue measure on $[0,1]^d$ satisfying $c := \inf_{z \in [0,1]^d} p(z) > 0$ and $C := \sup_{z \in [0,1]^d} p(z) < \infty$.  Then there exists $M(c, C, d, r) > 0$ such that
\[
\| f - f^\dagger \|_\infty \leq M(c, C, d, r)\|f-f^*\|_\infty.
\]
\begin{proof}
Let $P$ denote the distribution of $Z$, and let $L_2(P)$ denote the normed space of equivalence classes of measurable functions $g:[0,1]^d \rightarrow \mathbb{R}$ satisfying
\[
	\|g\|_{2,P} := \bigl\{\mathbb{E}\bigl(g(Z)^2\bigr)\bigr\}^{1/2} < \infty
\] 
under the binary relation where $g \sim g^\circ$ if $g(Z)=g^\circ(Z)$ almost surely\footnote{We do not distinguish between a function with finite $\|\cdot\|_{2,P}$ norm and its equivalence class in what follows.}.  The existence of the unique $f^\dagger \in \mathcal{S}_{r,N}^d$ follows from  Lemma~\ref{Lemma:BestApproximant} since the $L_2(P)$ norm is strictly convex.

Now define $\tilde{g} := f - f^*$, so the unique $L_2(P)$-best approximant $\tilde{g}^\dagger$ to $\tilde{g}$ in $\mathcal{S}_{r,N}^d$ is given by $\tilde{g}^\dagger = f^\dagger - f^*$.
We now verify that Conditions~A.1,~A.2 and~A.3 of \citet[][Theorem A.1]{huang2003} hold.  Condition A.1 is satisfied by our hypotheses on $c, C$; Condition A.2 holds since the knots of the splines in $\mathcal{S}^d_{r, N}$ are equi-spaced; and Condition A.3 is satisfied by Proposition~\ref{Proposition: b-spline properties}(b), where we again use the bounds $(N+1)^d \leq K \leq 2^d(N+1)^dr^d$.  Thus, \citet[][Theorem A.1]{huang2003} yields the existence of $M'(c, C, d, r) > 0$ such that 
\[
\|\tilde{g}^\dagger \|_\infty \leq M'(c, C, d, r) \|\tilde{g}\|_\infty.
\]
We conclude that
\[
\| f - f^\dagger \|_\infty \leq \| \tilde{g} \|_\infty  + \| \tilde{g}^\dagger \|_\infty \leq \bigl(1+M'(c, C, d, r)\bigr) \| \tilde{g} \|_\infty,
\]
so the desired result holds with $M(c, C, d, r) := 1 + M'(c, C, d, r)$.
\end{proof}
\end{proposition}
Lemma~\ref{lemma: rudelson lln} below will ensure that, provided $K$ increases slightly slower than $n$ (so that $K \log(K)/n \rightarrow 0$), performing ordinary least squares with the $d$-tensor B-spline basis will yield consistent estimators. 
\begin{lemma}
	\label{lemma: rudelson lln}
Let $\mathcal{P}$ denote a family of distributions for a random vector $Z$ taking values in $[0,1]^d$, and let $(Z_n)_{n \in \mathbb{N}}$ be a sequence of independent and identically distributed copies of $X$.  Recall the notation $\mbb{\phi} = (\phi_1,\ldots,\phi_K)^\top$, where $\{\phi_k\}_{k=1}^K$ denotes the $d$-tensor B-spline basis of $\mathcal{S}_{r,N}^d$.  For $P \in \mathcal{P}$, define $\mbb{\Sigma}_P := \E_P\bigl(\mbb{\phi}(Z) \mbb{\phi}(Z)^\top\bigr) \in \mathbb{R}^{K \times K}$ and $\mbb{\widehat{\Sigma}} := n^{-1} \sum_{i=1}^n \mbb{\phi}(Z_i) \mbb{\phi}(Z_i)^\top$.  Suppose that each $P \in \mathcal{P}$ is absolutely continuous with respect to Lebesgue measure on $[0,1]^d$, with corresponding density $p_P$ satisfying $C := \sup_{P \in \mathcal{P}} \sup_{z \in [0,1]^d} p_P(z) < \infty$.  Then
\[
K \bigl\| \mbb{\widehat{\Sigma}} - \mbb{\Sigma}_P \bigr\|_{\op} = O_{\mathcal{P}}\biggl(\frac{K \log (eK)}{n} + \sqrt{\frac{K \log (eK)}{n}}\biggr).
\]
If, in addition, $c := \inf_{P \in \mathcal{P}} \inf_{z \in [0, 1]^d} p_P(z) > 0$ and $K \log (K) /n \to 0$, then
\[
K^{-1} \bigl\| \mbb{\widehat{\Sigma}}^{-1} - \mbb{\Sigma}_P^{-1} \bigr\|_{\op} = O_{\mathcal{P}}\biggl(\frac{K \log (eK)}{n} + \sqrt{\frac{K \log (eK)}{n}}\biggr),
\]
and
\begin{equation}
	\label{Eq: sigma operator norm}
\|\mbb{\widehat{\Sigma}} \|_{\op} = O_{\mathcal{P}}(K^{-1}), \quad  \|\mbb{\widehat{\Sigma}}^{-1} \|_{\op} = O_{\mathcal{P}}(K).
\end{equation}
\begin{proof}
For the first claim, by Markov's inequality, it suffices to show that 
\[
\sup_{P \in \mathcal{P}} K \E_P\bigl(\bigl\| \mbb{\widehat{\Sigma}}  - \mbb{\Sigma}_P \bigr\|_{\op} \bigr) = O\biggl(\frac{K \log (eK)}{n} + \sqrt{\frac{K \log (eK)}{n}}\biggr)
\]
as $n \to \infty$. By the Rudelson law of large numbers for matrices \citep[][Lemma 6.2]{belloni2015some} (and Chebyshev's inequality when $K=1$), there exists a universal constant $C_* > 0$ such that 
\begin{align*}
\sup_{P \in \mathcal{P}} K \E_P\bigl( \bigl\| \mbb{\widehat{\Sigma}}  - \mbb{\Sigma}_P \bigr\|_{\op} \bigr) &\leq \frac{C_* K \log(eK)}{n} + C_* \sqrt{\frac{K^2 \log (eK)}{n}} \sup_{P \in \mathcal{P}} \sqrt{\| \mbb{\Sigma}_P \|_{\op}} \\
&\leq \frac{C_* K \log(eK)}{n} + C_* \sqrt{C 2^d} \sqrt{\frac{K \log (eK)}{n}},
\end{align*}
since $\|\mbb{\phi}(Z) \|_2 \leq \|\mbb{\phi}(Z) \|_1 = 1$ and $\| \mbb{\Sigma}_P \|_{\op} \leq C 2^d K^{-1}$, by Proposition~\ref{Proposition: b-spline properties}(a) and~(d) respectively.

For the second claim, note first that $K \lambda_{\min}(\mbb{\Sigma}_P) \geq c c_s(r) =: b > 0$ by Proposition~\ref{Proposition: b-spline properties}(d). 
From the first claim of the lemma and the hypothesis that $K \log(K)/n \rightarrow 0$, given $\epsilon > 0$, we can choose $n_0 \in \mathbb{N}$ large enough that 
\[
\sup_{P \in \mathcal{P}} \pr_P\biggl(K \| \mbb{\widehat{\Sigma}} - \mbb{\Sigma}_P \|_{\op} \geq \frac{b}{2} \biggr) \leq \frac{\epsilon}{2}
\]
for $n \geq n_0$.  Then, by another application of the first claim of the lemma, by increasing $n_0$ if necessary, we can find $M_0 > 0$ such that 
\[
\sup_{P \in \mathcal{P}} \pr_P\biggl(K \| \mbb{\widehat{\Sigma}} - \mbb{\Sigma}_P \|_{\op} \geq \frac{b^2M}{2}\biggl\{\frac{K \log (eK)}{n} + \sqrt{\frac{K \log (eK)}{n}}\biggr\}\biggr) \leq \frac{\epsilon}{2}
\]
for all $n \geq n_0$ and $M \geq M_0$.  It follows by Lemma~\ref{Lemma: locally lipschitz matrix inversion} that for $n \geq n_0$ and $M \geq M_0$, we have 
\begin{align*}
\sup_{P \in \mathcal{P}} \pr_P\biggl(&\frac{1}{K} \| \mbb{\widehat{\Sigma}}^{-1} - \mbb{\Sigma}_P^{-1} \|_{\op} \geq M\biggl\{\frac{K \log (eK)}{n} + \sqrt{\frac{K \log (eK)}{n}}\biggr\}\biggr) \\
&\leq \sup_{P \in \mathcal{P}} \pr_P\biggl(K \| \mbb{\widehat{\Sigma}} - \mbb{\Sigma}_P \|_{\op} \geq \frac{b^2M}{2}\biggl\{\frac{K \log (eK)}{n} + \sqrt{\frac{K \log (eK)}{n}}\biggr\}\biggr)\\
&\hspace{1cm}+ \sup_{P \in \mathcal{P}} \pr_P\biggl(K \| \mbb{\widehat{\Sigma}} - \mbb{\Sigma}_P \|_{\op} > \frac{b}{2} \biggr) \leq \epsilon,
\end{align*}
which establishes the second claim.

Finally, by the first part of Proposition~\ref{Proposition: b-spline properties}(d),
\[
K\|\mbb{\widehat{\Sigma}} \|_{\op} \leq K\| \mbb{\widehat{\Sigma}} - \mbb{\Sigma}_P \|_{\op} + K\|\mbb{\Sigma}_P\|_{\op} = O_{\mathcal{P}}(1),
\]
and by the second part of Proposition~\ref{Proposition: b-spline properties}(d),
\[
K^{-1}\|\mbb{\widehat{\Sigma}}^{-1} \|_{\op}  \leq K^{-1}\| \mbb{\widehat{\Sigma}}^{-1} - \mbb{\Sigma}_P^{-1} \|_{\op} + K^{-1}\|\mbb{\Sigma}_P^{-1}\|_{\op} = O_{\mathcal{P}}(1),
\]
as required.
\end{proof}
\end{lemma}

Proposition~\ref{Proposition: general spline regression} below provides estimation and both in-sample and out-of-sample prediction bounds for spline regression.  It is based on \citet[][Theorem~4.1]{belloni2015some}, but here we control the errors in a uniform fashion over a family of distributions, and those authors did not require in-sample bounds.
\begin{proposition}
	\label{Proposition: general spline regression}
	Let $\mathcal{P}$ be a family of distributions of $(Y, Z)$ on $\mathbb{R} \times [0, 1]^d$ with regression function $f_P$ given by $f_P(z) := \E_P(Y \given Z=z)$, and let $(Y_1, Z_1), \dots, (Y_n, Z_n)$ be independent and identically distributed copies of $(Y, Z)$. Suppose that
	\begin{enumerate}[(i)]
	\item The $L_2(P)$-best approximant $f_P^\dagger$ of $f_P$ in $\mathcal{S}_{r, N}^d$ satisfies 
		\[
			\sup_{P \in \mathcal{P}} \| f_P	- f_P^\dagger \|_\infty = O(K^{-\zeta}),
		\]
		for some $\zeta \equiv \zeta(d, r) > 0$.
		\item Each $P \in \mathcal{P}$ is absolutely continuous with respect to Lebesgue measure on $[0,1]^d$, with corresponding density $p_P$ satisfying $C := \sup_{P \in \mathcal{P}} \sup_{z \in [0,1]^d} p_P(z) < \infty$ and $c := \inf_{P \in \mathcal{P}} \inf_{z \in [0, 1]^d} p_P(z) > 0$.
		\item There exists a positive sequence $(\sigma_n^2)_{n \in \mathbb{N}}$ such that $\Var_P(Y \given Z) \leq \sigma_n^2$ for all $P \in \mathcal{P}$.
	\end{enumerate} 
	Let $\mbb{\phi}$ denote the $d$-tensor B-spline basis of $\mathcal{S}^d_{r, N}$ and let $\mbb{\widehat{\beta}}$ denote the ordinary least squares estimate from regressing $Y_1,\ldots,Y_n$ onto $\mbb{\phi}(Z_1),\ldots,\mbb{\phi}(Z_n)$.  Assume that $K \log(K) / n \to 0$.  Then
	\[
		\frac{1}{n} \sum_{i=1}^n \bigl(f_P(Z_i)-\mbb{\widehat{\beta}}^\top \mbb{\phi}(Z_i)\bigr)^2 = O_{\mathcal{P}}\bigl(K^{-2\zeta} + \sigma_n^2 K/n \bigr).
	\]
	Letting $\mbb{\beta}_P \in \mathbb{R}^K$ be the unique solution to $f_P^\dagger(z) = \mbb{\beta}_P^\top \mbb{\phi}(z)$, we have under the same assumptions that
	\[		\| \widehat{\mbb{\beta}} - \mbb{\beta}_P \|_2^2 = O_{\mathcal{P}}\bigl(K^{-(2\zeta-2)}/n + \sigma_n^2 K^2/n\bigr).
	\]
	Further, if $(Y^*, Z^*)$ is a new observation of $(Y, Z)$ independent of the original sample, then
	\[
		\E_P \Bigl( \bigl\{ f_P(Z^*)-\mbb{\widehat{\beta}}^\top \mbb{\phi}(Z^*) \bigr\}^2  \bigm| \widehat{\mbb{\beta}} \Bigr) = O_{\mathcal{P}}\bigl(K^{-2\zeta} + \sigma_n^2 K/n\bigr).
	\]
	\begin{proof}
		Let $\mbb{\widehat{\Sigma}} := \frac{1}{n} \sum_{i=1}^n \mbb{\phi}(Z_i) \mbb{\phi}(Z_i)^\top$ and for $i \in [n]$, let $h_i:= f_P(Z_i) - f_P^\dagger(Z_i)$ and $\varepsilon_i := Y_i - f_P(Z_i)$. Then, recalling that $f_P^\dagger(z) = \mbb{\beta}_P^\top \mbb{\phi}(z)$, we have
		\begin{align}
			\nonumber
			&\frac{1}{n} \sum_{i=1}^n \bigl(f_P(Z_i) - \mbb{\widehat{\beta}}^\top \mbb{\phi}(Z_i)\bigr)^2 \leq 2 \| f_P - f_P^\dagger \|_\infty^2 +  2 (\mbb{\widehat{\beta}}-\mbb{\beta}_P)^\top \mbb{\widehat{\Sigma}} (\mbb{\widehat{\beta}}-\mbb{\beta}_P)\\ \nonumber
			&= 2 \| f_P - f_P^\dagger \|_\infty^2 +  2 \biggl( \frac{1}{n} \sum_{i=1}^n (h_i + \varepsilon_i) \mbb{\phi}(Z_i) \biggr)^\top \mbb{\widehat{\Sigma}}^{-1} \mbb{\widehat{\Sigma}} \mbb{\widehat{\Sigma}}^{-1}  \biggl( \frac{1}{n} \sum_{i=1}^n (h_i + \varepsilon_i) \mbb{\phi}(Z_i) \biggr) \\
			&\leq 2 \| f_P - f_P^\dagger \|_\infty^2 +  2 \| \mbb{\widehat{\Sigma}} \|_{\op} \bigl\| \mbb{\widehat{\Sigma}}^{-1} \bigr\|_{\op} \biggl\| \mbb{\widehat{\Sigma}}^{-1/2} \biggl( \frac{1}{n} \sum_{i=1}^n (h_i + \varepsilon_i) \mbb{\phi}(Z_i) \biggr) \biggr\|_2^2.			\label{Eq: spline mse bound}
		\end{align}
		Now $\| \mbb{\widehat{\Sigma}} \|_{\op} \bigl\| \mbb{\widehat{\Sigma}}^{-1} \bigr\|_{\op} = O_\mathcal{P}(1)$ by~\eqref{Eq: sigma operator norm} in Lemma~\ref{lemma: rudelson lln}.  Moreover, 
		\begin{align*}
		\biggl\| \mbb{\widehat{\Sigma}}^{-1/2} \biggl( \frac{1}{n} \sum_{i=1}^n (h_i + \varepsilon_i) \mbb{\phi}(Z_i) \biggr) \biggr\|_2^2 &\leq 2 \underbrace{\biggl\| \mbb{\widehat{\Sigma}}^{-1/2} \biggl( \frac{1}{n} \sum_{i=1}^n h_i \mbb{\phi}(Z_i) \biggr) \biggr\|_2^2}_{\RN{1}_n}\\
		&\hspace{1cm}+ 2 \underbrace{\biggl\| \mbb{\widehat{\Sigma}}^{-1/2} \biggl( \frac{1}{n} \sum_{i=1}^n \varepsilon_i \mbb{\phi}(Z_i) \biggr) \biggr\|_2^2}_{\RN{2}_n}.
		\end{align*}
		To deal with $\RN{1}_n$, let $\mbb{\Sigma}_P:= \E_P\bigl(\mbb{\phi}(Z)\mbb{\phi}(Z)^\top\bigr) \in \mathbb{R}^{K \times K}$, and note that
	\begin{equation}
		\label{eq: ls estimator mean}
		\begin{aligned}
		\E_P\bigl(f_P^\dagger(Z) \mbb{\phi}(Z)^\top \bigr) &= 
		\E_P\bigl(\mbb{\beta}_P^\top \mbb{\phi}(Z)\mbb{\phi}(Z)^\top \bigr) \\
		&= \E_P\bigl(f_P(Z)\mbb{\phi}(Z)^\top \bigr) \mbb{\Sigma}^{-1}_P \mathbb{E}_P\bigl(\mbb{\phi}(Z) \mbb{\phi}(Z)^\top\bigr) = \E_P\bigl(f_P(Z)\mbb{\phi}(Z)^\top\bigr).
		\end{aligned}
	\end{equation}
	It follows that
	\begin{align}
	\label{Eq: tighter projection bound}
	\sup_{P \in \mathcal{P}}\E_P\biggl(\biggl\| \frac{1}{n} \sum_{i=1}^n h_i \mbb{\phi}(Z_i) \biggr\|_2^2 \biggr) &= \frac{1}{n} \sup_{P \in \mathcal{P}} \E_P\bigl\{\tr\bigl((h_1)^2 \mbb{\phi}(Z_1) \mbb{\phi}(Z_1)^{\top}\bigr)\bigr\} \nonumber\\
	&\leq \frac{1}{n} \sup_{P \in \mathcal{P}} \|f_P-f_P^\dagger\|_{\infty}^2 \tr(\mbb{\Sigma}_P),
	\end{align}
so
\begin{equation}
\label{Eq: spline regression bias term}
|\RN{1}_n| \leq \bigl\|\mbb{\widehat{\Sigma}}^{-1}\bigr\|_{\mathrm{op}}\biggl\| \frac{1}{n} \sum_{i=1}^n h_i \mbb{\phi}(Z_i) \biggr\|_2^2 = O_\mathcal{P}\bigl(K^{-(2\zeta-1)}/n\bigr)
\end{equation}
by our assumption on $\sup_{P \in \mathcal{P}} \|f_P-f_P^\dagger\|_{\infty}$, Proposition~\ref{Proposition: b-spline properties}(d),~\eqref{Eq: sigma operator norm} in Lemma~\ref{lemma: rudelson lln} and Lemma~\ref{Lemma: unconditionalisation via Markov}. 
		To deal with $\RN{2}_n$, we note that $\varepsilon_1, \dots, \varepsilon_n$ are conditionally independent given $Z_1, \dots, Z_n$, so 
		\begin{equation}
			\label{Eq: spline regression variance term}
			\E_P(\RN{2}_n \given Z_1, \dots, Z_n) = \frac{1}{n^2} \tr\biggl( \mbb{\widehat{\Sigma}}^{-1} \sum_{i=1}^n \E(\varepsilon_i^2 \given Z_i) \mbb{\phi}(Z_i) \mbb{\phi}(Z_i)^\top \biggr) \leq \frac{\sigma_n^2}{n} \tr\bigl( \mbb{\widehat{\Sigma}}^{-1} \mbb{\widehat{\Sigma}} \bigr) \leq \frac{\sigma_n^2 K}{n}. 
		\end{equation}

		Putting things together, since $\|f_P-f_P^\dagger\|_\infty^2 = O(K^{-2\zeta})$ by assumption, we have 
		\[
		\frac{1}{n} \sum_{i=1}^n \bigl(f_P(Z_i) - \mbb{\widehat{\beta}}^\top \mbb{\phi}(Z_i)\bigr)^2 = O_\mathcal{P}\bigl(K^{-2\zeta} + \sigma_n^2 K/n\bigr),
		\]
		as desired.
		
		For the second claim, observe that
		\begin{align*}
			&\|\mbb{\widehat{\beta}} - \mbb{\beta}_P\|_2^2 = \biggl\| \mbb{\widehat{\Sigma}}^{-1} \biggl( \frac{1}{n} \sum_{i=1}^n (h_i + \varepsilon_i) \mbb{\phi}(Z_i) \biggr) \biggr\|_2^2\\
			&\leq \bigl\| \mbb{\widehat{\Sigma}}^{-1} \bigr\|_{\op} \biggl\| \mbb{\widehat{\Sigma}}^{-1/2} \biggl( \frac{1}{n} \sum_{i=1}^n (h_i + \varepsilon_i) \mbb{\phi}(Z_i) \biggr) \biggr\|_2^2 = O_\mathcal{P}\bigl(K^{-(2\zeta-2)}/n + \sigma_n^2 K^2/n\bigr),
		\end{align*}
		by~\eqref{Eq: sigma operator norm} in Lemma~\ref{lemma: rudelson lln} and our results above. 

		Finally, we have following the argument in \eqref{Eq: spline mse bound}, we have
		\begin{align*}
		\E_P \Bigl( \bigl\{ f_P(Z^*)-\mbb{\widehat{\beta}}^\top \mbb{\phi}(Z^*) \bigr\}^2  \bigm| \widehat{\mbb{\beta}} \Bigr) &\leq 2 \| f_P - f_P^\dagger \|_\infty^2 +  2 \bigl\| \mbb{\Sigma}_P \bigr\|_{\op} \bigl\|\mbb{\widehat{\beta}}-\mbb{\beta}_P\bigr\|^2\\
		&= O_{\mathcal{P}}\bigl(K^{-2\zeta} + \sigma_n^2K/n\bigr), 
		\end{align*}
		by Proposition~\ref{Proposition: b-spline properties}(d) and the second claim of the proposition.
	\end{proof}
\end{proposition}
Under standard smoothness assumptions, we can derive the following consequence of Proposition~\ref{Proposition: general spline regression}:
\begin{corollary}
	\label{Corollary: spline regression}
 Let $\mathcal{P}$ be a family of distributions of $(Y, Z)$ on $\mathbb{R} \times [0, 1]^d$, and let $f_P$ denote the regression function given by $f_P(z):=\E_P(Y \given Z=z)$ for $P \in \mathcal{P}$.  Suppose there exist $C,c,s > 0$ such that \begin{enumerate}[(i)]
		\item $f_P \in \mathcal{H}_s$ with $\|f_P\|_{\mathcal{H}_s} \leq C$ for all $P \in \mathcal{P}$.
  		\item Each $P \in \mathcal{P}$ is absolutely continuous with respect to Lebesgue measure on $[0,1]^d$, with corresponding density $p_P$ satisfying 
		\[
			\sup_{P \in \mathcal{P}} \sup_{z \in [0,1]^d} p_P(z) \leq C \quad \text{and} \quad \inf_{P \in \mathcal{P}} \inf_{z \in [0, 1]^d} p_P(z) \geq c. 
		\]
		\item $\Var_P(Y \given Z) \leq C$ for all $P \in \mathcal{P}$.
	\end{enumerate} 
	If $K\log(K)/n \rightarrow 0$, then the conclusions of Proposition~\ref{Proposition: general spline regression} hold for any $r \geq s$ with $\zeta = s/d$ and $\sigma_n^2 = C$.
	\begin{proof}
	Under Assumptions~\emph{(i)} and~\emph{(ii)} of the corollary, we have that Assumption~\emph{(i)} of Proposition~\ref{Proposition: general spline regression} holds with $\zeta = s/d$ when $r \geq s$ by Propositions~\ref{Proposition: spline approximation} and \ref{Proposition: spline least squares approximation}.  Assumptions~\emph{(ii)} and~\emph{(iii)} of Proposition~\ref{Proposition: general spline regression} also hold by hypothesis with $\sigma_n^2 = C$, so the conclusion follows. 
	\end{proof}
\end{corollary}

Now suppose that $\mbb{\phi}^X$ and $\mbb{\phi}^Z$ are the $d_X$- and $d_Z$-tensor B-spline bases of $\mathcal{S}_{r,N_X}^{d_X}$ and $\mathcal{S}_{r,N_Z}^{d_Z}$ respectively.  It will be convenient to have the following decomposition of functions in the span of $\mbb{\phi}^X \otimes \mbb{\phi}^Z$.

\begin{proposition} \label{Proposition: tensor product projection}
Let $\mbb{\phi} := \mbb{\phi}^X \otimes \mbb{\phi}^Z$, and let $K_X := (N_X + r)^{d_X}$, $K_Z:= (N_Z + r)^{d_Z}$ and $K_{XZ} := K_{XZ}$. Let $V$ denote the subspace of $R^{K_{XZ}}$ given by $V := \{\mbb{1} \otimes \mbb{v} : \mbb{v} \in \mathbb{R}^{K_Z} \}$ where $\mbb{1}$ denotes the vector of ones in $\mathbb{R}^{K_X}$.  Let $V^\perp$ denote the orthogonal complement of $V$ in $\mathbb{R}^{K_{XZ}}$, and let $\mbb{\Pi}:\mathbb{R}^{K_{XZ}} \rightarrow V^\perp$ denote the projection onto $V^\perp$.  Let $\mbb{\beta} = (\beta_1,\ldots,\beta_{K_{XZ}})^\top \in \mathbb{R}^{K_{XZ}}$, and define $f:[0,1]^{d_X+d_Z} \rightarrow \mathbb{R}$ by $f(x, z) := \mbb{\phi}(x, z)^\top \mbb{\beta}$.  Then, writing $\bar{\mbb{\beta}} = (\bar{\beta}_1,\ldots,\bar{\beta}_{K_Z})^\top \in \mathbb{R}^{K_Z}$, where $\bar{\beta}_k := K_X^{-1}\sum_{\ell=1}^{K_X} \beta_{(k-1)K_X+\ell}$, we have $(\mbb{I}-\mbb{\Pi})\mbb{\beta} = \mbb{1} \otimes \bar{\mbb{\beta}}$ and
\begin{equation}
\label{Eq:fdecomp}
f(x,z) = \mbb{\phi}(x, z)^\top \mbb{\Pi}\mbb{\beta} + \mbb{\phi}^Z(z)^\top \bar{\mbb{\beta}}.
\end{equation}
Moreover, $\|\mbb{\Pi}\mbb{\beta}\|_\infty \leq 2\|\mbb{\beta}\|_\infty$.
\begin{proof}
	We claim that $V^\perp = \bigl\{\mbb{u} = (\mbb{u}_1^\top,\ldots,\mbb{u}_{K_Z}^\top)^\top \in \mathbb{R}^{K_{XZ}}: \mbb{1}^\top \mbb{u}_k = 0 \ \forall k \in [K_Z]\bigr\}$.  To see this, note that if 
	$\mbb{u} = (\mbb{u}_1^\top,\ldots,\mbb{u}_{K_Z}^\top)^\top \in \mathbb{R}^{K_{XZ}}$ satisfies $\mbb{1}^\top \mbb{u}_k = 0$ for all $k \in [K_Z]$ and $\mbb{1} \otimes \mbb{v} \in V$ for some $\mbb{v} = (v_1,\ldots,v_{K_Z})^\top \in \mathbb{R}^{K_Z}$, then
	\[
	(\mbb{1} \otimes \mbb{v})^\top \mbb{u} = \sum_{k=1}^{K_Z} v_k (\mbb{1}^\top \mbb{u}_k) = 0,
	\]
	which establishes our claim.   We can therefore write
	\[
	\mbb{\beta} = \mbb{1} \otimes \bar{\mbb{\beta}} + \mbb{\beta} - (\mbb{1} \otimes \bar{\mbb{\beta}}),
	\]
	where $\mbb{1} \otimes \bar{\mbb{\beta}} \in V$ and $\mbb{\beta} - (\mbb{1} \otimes \bar{\mbb{\beta}}) \in V^\perp$, so $(\mbb{I}-\mbb{\Pi})\mbb{\beta} = \mbb{1} \otimes \bar{\mbb{\beta}}$ and $\mbb{\Pi}\mbb{\beta} = \mbb{\beta} - (\mbb{1} \otimes\bar{\mbb{\beta}})$.  Hence,
	\begin{align*}
	\mbb{\phi}(x, z)^\top (\mbb{I}-\mbb{\Pi})\mbb{\beta} &= \mathrm{vec}\bigl(\mbb{\phi}^X(x)\mbb{\phi}^Z(z)^\top\bigr)^\top \mathrm{vec}(\mbb{1}\bar{\mbb{\beta}}^\top) = \sum_{\ell=1}^{K_X} \sum_{k=1}^{K_Z} \mbb{\phi}^X_\ell(x) \mbb{\phi}^Z_k(z) \bar{\beta}_k\\
	&= \mbb{\phi}^Z(z)^\top \bar{\mbb{\beta}},
	\end{align*}
by Proposition~\ref{Proposition: b-spline properties}(a), from which~\eqref{Eq:fdecomp} follows.  Finally,  $\|\mbb{\Pi}\mbb{\beta}\|_\infty \leq \|\mbb{\beta}\|_\infty + \|\mbb{1} \otimes \bar{\mbb{\beta}}\|_\infty \leq 2\|\mbb{\beta}\|_\infty$.
\end{proof}
\end{proposition}
Our next two lemmas will be used in the proof of Proposition~\ref{Proposition: spline on spline regression}, which is the analogue of Corollary~\ref{Corollary: spline regression} for a key setting for us, namely where our response variable for spline regression consists of fitted values from an earlier spline regression.
\begin{lemma}\label{Lemma: Regression of basis functions}
	Let $P$ denote a distribution of $(X, Z)$ on $[0,1]^{d_X} \times [0, 1]^{d_Z}$ that is absolutely continuous with respect to Lebesgue measure, and let $p_{X|Z}$ denote the conditional density of $X$ given $Z$. Assume that $p_{X|Z}(x \given \cdot) \in \mathcal{H}^{d_Z}_s$ for every $x \in [0,1]^{d_X}$ and that there exists $C > 0$ such that 
	\[
		\sup_{x \in [0, 1]^{d_X}} \| p_{X|Z}(x \given \cdot) \|_{\mathcal{H}_s} \leq C.
	\]
	Let $\mbb{\phi} = (\phi_1,\ldots,\phi_K)^\top$ denote the $d_X$-tensor B-spline basis of $\mathcal{S}_{r,N}^{d_X}$, let $\mbb{\beta} = (\beta_1,\ldots, \beta_{K})^\top \in \mathbb{R}^K$ and define $g:[0,1]^{d_Z} \rightarrow \mathbb{R}$ by
	\begin{align*}
		g(z) := \mbb{\beta}^\top \E\bigl(\mbb{\phi}(X) \given Z = z\bigr) = \sum_{k=1}^{K} \beta_k \E\bigl(\phi_{k}(X) \given Z=z\bigr).
	\end{align*}
	Then $g \in \mathcal{H}_s^{d_Z}$ and $\| g \|_{\mathcal{H}_s} \leq C \| \mbb{\beta} \|_\infty$.
	\begin{proof}
		Repeated application of \citet[][Theorem 6.28]{klenke2020} allows us to interchange derivatives and integrals such that for any multi-index $\mbb{\alpha} = (\alpha_1,\ldots,\alpha_{d_Z})^\top \in \mathbb{N}_0^{d_Z}$ with $|\mbb{\alpha}| \leq \ceil{s} - 1 =: s_0$, we have
		\begin{align*}
			D^{\mbb{\alpha}} g(z) =  \sum_{k=1}^{K} \beta_k  \int_{[0,1]^{d_X}}   \phi_{k}(x) \cdot D^{\mbb{\alpha}} p_{X|Z}(x \given z) \, \mathrm{d}x. 
		\end{align*}
		Thus, by Hölder's inequality and the fact that the $\{\phi_k\}_{k=1}^K$ are non-negative and form a partition of unity by Proposition~\ref{Proposition: b-spline properties}(a), we have
		\[
			\| D^{\mbb{\alpha}} g \|_\infty \leq \| \mbb{\beta}\|_\infty \sup_{z \in [0, 1]^{d_Z}}  \int_{[0,1]^{d_X}}   \sum_{k=1}^{K} \phi_{k}(x) \bigl| D^{\mbb{\alpha}} p_{X|Z}(x \given z) \bigr| \, \mathrm{d}x \leq C \| \mbb{\beta}\|_\infty .
		\]
		By a similar argument, when $|\mbb{\alpha}| = s_0$, we have
		\begin{align*}
			\big|D^{\mbb{\alpha}} g(z) - D^{\mbb{\alpha}} g(z')\big| &\leq \| \mbb{\beta}\|_\infty \int_{[0,1]^{d_X}}   \sum_{k=1}^{K} \phi_{k}(x) \bigl| D^{\mbb{\alpha}} p_{X|Z}(x \given z) - D^{\mbb{\alpha}} p_{X|Z}(x \given z') \bigr| \, \mathrm{d}x \\
			&\leq C \|\mbb{\beta}\|_{\infty} \|z - z'\|_2^{s - s_0},
		\end{align*}
		for all $z,z' \in [0,1]^{d_Z}$, as required.
	\end{proof}
\end{lemma}

\begin{lemma}\label{Lemma: spline on spline approximation}
	 Let $\{\phi_k\}_{k=1}^K$ denote the uniform $d$-tensor B-spline basis of $\mathcal{S}^d_{r, N}$.  For $k \in [K]$, suppose that $h_k:[0,1]^d \rightarrow \mathbb{R}$ can be written as $h_k = s_k + r_k$ where $s_k \in \mathcal{S}^d_{r, N}$ and $r_k \in \mathcal{H}_s^d$, and let $M := \max_{k \in [K]} \|r_k\|_{\mathcal{H}_s}$. Define $m:[0,1]^d \rightarrow \mathbb{R}$ by
	\[
		m(z) := \sum_{k=1}^K g_k(z) \phi_k(z).
	\]
	Then there exist $C(d, r) > 0$ and $m^* \in \mathcal{S}_{2r-1, N}^d$ such that 
	\[
		\| m - m^* \|_\infty \leq \frac{M C(d, r)}{(2rK)^{\min(s,r)/d}}.
	\]
	\begin{proof}
		Let 
		\[
			\widetilde{\mathcal{S}} := \mathrm{span}\bigl( \{\phi_k(z)\phi_{\ell}(z) \}_{k, \ell \in [K]} \bigr) \subseteq \mathcal{S}^d_{2r-1, N}.
		\]
		For $k \in [K]$, let $r_k^*$ denote a supremum norm approximant to $r_k$ in $\mathcal{S}_{r,N}^d$ (see Proposition~\ref{Proposition: spline approximation}), so that $m^* := \sum_{k=1}^K (s_k + r_k^*) \phi_k \in \widetilde{\mathcal{S}}$.  Then by H\"older's inequality, Proposition~\ref{Proposition: b-spline properties}(a) and  Proposition~\ref{Proposition: spline approximation}, we have
		\[
		\| m - m^* \|_\infty = \biggl\| \sum_{k=1}^K (r_k-r_k^*) \phi_k  \biggr\|_\infty \leq \max_{k \in K} \|r_k- r_k^*\|_\infty \leq \frac{M C(d, r)}{(2rK)^{\min(s,r)/d}},
		\]
		as desired.
	\end{proof}
\end{lemma}
We are now in a position to state our main result on the performance of the spline on spline regression procedure.
\begin{proposition}
	\label{Proposition: spline on spline regression}
Let $r \in \mathbb{N}$, let $d = d_X + d_Z$ and let $\mbb{\phi}$ denote the $d$-tensor B-spline basis of $\mathcal{S}_{r, N}^d$.  Let $\mathcal{P}$ be a family of distributions of $(X, Z)$ on $[0, 1]^{d_X} \times [0, 1]^{d_Z}$.  Suppose that each $P \in \mathcal{P}$ is absolutely continuous with respect to Lebesgue measure on $[0,1]^d$.  Further, suppose that there exist $C,c > 0$ such that:
\begin{enumerate}[(i)]
	\item There exists $s \in (0,r]$ such that the conditional density $p_{X|Z, P}$ of $X$ given $Z$, satisfies $p_{X|Z, P}(x|\cdot) \in \mathcal{H}_s^{d_Z}$ for every $x \in [0, 1]^{d_X}$ and $\sup_{x \in [0, 1]^{d_X}} \|p_{X|Z, P}(x \given \cdot) \|_{\mathcal{H}_s} \leq C$. 
	\item For every $P \in \mathcal{P}$, the density $p_{Z, P}$ of $Z$ satisfies 	
	\[
		\sup_{P \in \mathcal{P}} \sup_{z \in [0,1]^d} p_{Z,P}(z) \leq C \quad \text{and} \quad \inf_{P \in \mathcal{P}} \inf_{z \in [0, 1]^d} p_{Z,P}(z) \geq c. 
	\]
\end{enumerate}
Let $(X_1, Z_1), \dots, (X_n, Z_n)$ be independent and identically distributed copies of $(X, Z)$.    For $n \in \mathbb{N}$, let $\mbb{\beta} \equiv \mbb{\beta}_n \in \mathbb{R}^{K_{XZ}}$ satisfy $\|\mbb{\Pi}\mbb{\beta}\|_\infty = O(1)$ where $\mbb{\Pi}$ is defined in Proposition~\ref{Proposition: tensor product projection}, and  define $f_n \in \mathcal{S}^d_{r, N}$ by $f_n(x, z) := \mbb{\beta}^\top \mbb{\phi}(x,z)$.  Further, define $g_{P,n}:[0,1]^{d_Z} \rightarrow \mathbb{R}$ by $g_{P,n}(z) := \mathbb{E}_P\bigl(f_n(X,Z) \given Z=z\bigr)$, let $\mbb{\psi}$ denote the $d_Z$-tensor B-spline basis of $\mathcal{S}_{2r-1,N}^{d_Z}$ and let $\widetilde{K}_Z := (2r-1+N)^{d_Z}$.  Let $Y_i := f_n(X_i, Z_i)$ for $i\in [n]$, and let $\mbb{\widehat{\theta}}$ denote the ordinary least squares estimate from regressing $Y_1,\ldots,Y_n$ onto $\mbb{\psi}(Z_1),\ldots,\mbb{\psi}(Z_n)$.  If $\widetilde{K}_Z \log (\widetilde{K}_Z) / n \to 0$, then
	\[
		\frac{1}{n} \sum_{i=1}^n \bigl(g_{P,n}(Z_i)-\mbb{\widehat{\theta}}^\top \mbb{\psi}(Z_i)\bigr)^2 = O_{\mathcal{P}}\bigl(\|\mbb{\Pi} \mbb{\beta}\|_\infty^2 \{\widetilde{K}_Z^{-2s/d_Z} + \widetilde{K}_Z/n \} \bigr).
	\]
	  Letting $\mbb{\theta}_P \in \mathbb{R}^{\widetilde{K}_Z}$ be the unique solution to $g_{P,n}^\dagger(z) = \mbb{\theta}_P^\top \mbb{\psi}(z)$, we have under the same assumptions that
	\[		\| \widehat{\mbb{\theta}} - \mbb{\theta}_P \|_2^2 = O_{\mathcal{P}}\bigl( \|\mbb{\Pi} \mbb{\beta}\|_\infty^2 \widetilde{K}_Z^2/n\bigr).
	\]
	Finally, if $(X^*, Z^*)$ is a new observation of $(X, Z)$ independent of the original sample, then
	\[
		\E_P \Bigl( \bigl\{ g_{P,n}(Z^*)-\mbb{\widehat{\theta}}^\top \mbb{\phi}(Z^*) \bigr\}^2  \bigm| \widehat{\mbb{\theta}} \Bigr) = O_{\mathcal{P}}\bigl(\|\mbb{\Pi} \mbb{\beta}\|_\infty^2 \{\widetilde{K}_Z^{-2s/d_Z} +  \widetilde{K}_Z/n \}\bigr).
	\]
\begin{proof}
We check the conditions of Proposition~\ref{Proposition: general spline regression} with $\mathcal{P}$ in that result taken to be the set of distributions of $(Y_1,Z_1)$.   To this end, let $\mbb{\phi}^Z$ and $\mbb{\phi}^X$ denote the $d_Z$- and $d_X$-tensor B-spline bases of $\mathcal{S}_{r, N}^{d_Z}$ and $\mathcal{S}_{r, N}^{d_X}$, respectively, so that $\mbb{\phi}(x, z)= \mbb{\phi}^X(x) \otimes \mbb{\phi}^Z(z)$. By Proposition~\ref{Proposition: tensor product projection}, we can write
\[
f_n(x, z) = \mbb{\phi}(x,z)^\top \mbb{\Pi}\mbb{\beta} + \mbb{\phi}^Z(z)^\top \bar{\mbb{\beta}}.
\]
Thus,
\[
g_{P, n}(z) = \sum_{k=1}^{K_Z} \phi_k^Z(z) \biggl[ \sum_{\ell=1}^{K_X} (\mbb{\Pi}\mbb{\beta})_{(k-1)K_X + \ell} \E(\phi^X_\ell(X) \given Z=z) + \bar{\beta}_k \biggr] =:  \sum_{k=1}^{K_Z} \phi_k^Z(z) h_k(z).
\]
By Lemma~\ref{Lemma: Regression of basis functions} and Assumption~\emph{(i)}, we have for every $k \in [K_Z]$ that the function $r_k:[0,1]^{d_Z} \rightarrow \mathbb{R}$ given by
\[
r_k(z) := \sum_{\ell=1}^{K_X} (\mbb{\Pi}\mbb{\beta})_{(k-1)K_X + \ell} \E(\phi^X_\ell(X) \given Z=z) 
\]
belongs to $\mathcal{H}_s^{d_Z}$ with $\|r_k\|_{\mathcal{H}_s} \leq C \|\mbb{\Pi} \mbb{\beta} \|_{\infty}$.  Since the constant function $z \mapsto \bar{\beta}_k$ belongs to $\mathcal{S}_{r,N}^{d_Z}$, we deduce from Proposition~\ref{Proposition: spline least squares approximation} and Lemma~\ref{Lemma: spline on spline approximation} that the $L_2(\mathcal{P})$-best approximant $g_{P,n}^\dagger$ to $g_{P,n}$ in $\mathcal{S}_{2r-1,N}^{d_Z}$ satisfies for each $P \in \mathcal{P}$ that
\begin{align*}
\| g_{P,n}	- g_{P,n}^\dagger \|_\infty &\leq M(C,c,d_Z,2r-1)\| g_{P,n}	- g_{P,n}^* \|_\infty \leq \frac{M(C,c,d_Z,2r-1)C \|\mbb{\Pi} \mbb{\beta} \|_{\infty}}{(2rK_Z)^{s/d_Z}} \\
&\leq \frac{2^{d_Z}M(C,c,d_Z,2r-1)C \|\mbb{\Pi} \mbb{\beta} \|_{\infty}}{(2r\widetilde{K}_Z)^{s/d_Z}}.
\end{align*}
Thus, Assumption~(i) of Proposition~\ref{Proposition: general spline regression} is satisfied with $\zeta = s/d_Z$.  Assumption~(ii) of Proposition~\ref{Proposition: general spline regression} is true by hypothesis, and Assumption~(iii) of Proposition~\ref{Proposition: general spline regression} holds with $\sigma_n = \| \mbb{\Pi} \mbb{\beta} \|_\infty$ since
\[
	\Var(Y_1 \given Z_1) = \Var(f_n(X, Z) \given Z) = \Var(\mbb{\phi}(X,Z)^\top \mbb{\Pi}\mbb{\beta} \given Z) \leq \|\mbb{\phi}(x,z)^\top \mbb{\Pi}\mbb{\beta}\|_\infty^2 \leq \| \mbb{\Pi} \mbb{\beta} \|_\infty^2
\]
by Proposition~\ref{Proposition: b-spline properties}(b).  The conclusions therefore follow from Proposition~\ref{Proposition: general spline regression}.
\end{proof}
\end{proposition}
Finally in this section, we present two results that control two different out-of-sample product errors in a sharper way than would be obtained via a naive application of the Cauchy--Schwarz inequality.  The first can be regarded as a restated and uniform version of Theorem~8 and Lemma~A5 in \citet{newey2018cross}.
\begin{proposition}\label{Proposition: product error}
	Let $\mathcal{P}$ be a family of distributions of $(X, Y, Z)$ on $\mathbb{R} \times \mathbb{R} \times [0, 1]^d$ with corresponding regression functions $f_P,g_P:[0,1]^d \rightarrow \mathbb{R}$ given by $f_P(z):=\E_P(Y \given Z=z)$ and $g_P(z):=\E_P(X \given Z=z)$ satisfying: 
	\begin{enumerate}[(i)]
	\item  There exist $\zeta_f(d, r) \equiv \zeta_f > 0$ and $\zeta_g(d,r) \equiv \zeta_g > 0$ such that
		\[
			\sup_{P \in \mathcal{P}} \|f_P-f_P^\dagger\|_\infty =  O(K^{-\zeta_f}), \quad \sup_{P \in \mathcal{P}}  \|g_P-g_P^\dagger\|_\infty = O(K^{-\zeta_g}),
		\]
		where $f_P^\dagger$ and $g_P^\dagger$ denote the $L_2(P)$-best approximants of $f_P$ and $g_P$ respectively in $\mathcal{S}^d_{r, N}$. 
  		\item Each $P \in \mathcal{P}$ is absolutely continuous with respect to Lebesgue measure on $[0,1]^d$, with corresponding density $p_P$ satisfying $C := \sup_{P \in \mathcal{P}} \sup_{z \in [0,1]^d} p_P(z) < \infty$ and $c := \inf_{P \in \mathcal{P}} \inf_{z \in [0, 1]^d} p_P(z) > 0$.
  		\item There exists a positive sequence $(\sigma_n^2)_{n \in \mathbb{N}}$ such that $\max\bigl\{\Var(Y \given Z), \Var(X \given Z) \bigr\} \leq \sigma_n^2 = O(1)$.
	\end{enumerate}
	Now suppose we are given three independent samples $(X^f_i, Y^f_i, Z^f_i)_{i=1}^n$, $(X^g_i, Y^g_i, Z^g_i)_{i=1}^n$ and $(X_i, Y_i, Z_i)_{i=1}^n$, each consisting of $n$ independent and identically distributed copies of $(X, Y, Z)$.  Let $\mbb{\phi}$ denote the $d$-tensor B-spline basis of $\mathcal{S}^d_{r, N}$.  Let $\mbb{\widehat{\beta}}_f$ and $\mbb{\widehat{\beta}}_g$ denote the ordinary least squares estimates from regressing $Y_1^f,\ldots,Y_n^f$ onto $\mbb{\phi}(Z_1^f),\ldots,\mbb{\phi}(Z_n^f)$ and $X_1^g,\ldots,X_n^g$ onto $\mbb{\phi}(Z_1^g),\ldots,\mbb{\phi}(Z_n^g)$ respectively.  Define fitted regression functions $\fhat$ and $\ghat$ by $\fhat(z) = \mbb{\widehat{\beta}}_f^\top \mbb{\phi}(z)$ and $\ghat(z) = \mbb{\widehat{\beta}}_g^\top \mbb{\phi}(z)$ respectively.  If $K \log(K) / n \to 0$, then
	\begin{align*}
		\frac{1}{n} \sum_{i=1}^n \bigl\{\fhat(Z_i) - f_P(Z_i)\bigr\}&\bigl\{\ghat(Z_i) - g_P(Z_i)\bigr\}\\
		&= O_{\mathcal{P}}\biggl(K^{-(\zeta_f+\zeta_g)} + \frac{K^{1/2}}{n} + \frac{K^{2-\max(\zeta_f,\zeta_g)} \log K}{n^2}\biggr).
	\end{align*}
\begin{proof}
	Suppose without loss of generality that $\zeta_f \geq \zeta_g$.  Define $\mbb{\beta}_{P, f}, \mbb{\beta}_{P, g} \in \mathbb{R}^K$ so that $f_P^\dagger(z) = \mbb{\beta}_{P, f}^\top \mbb{\phi}(z)$ and $g_P^\dagger(z) = \mbb{\beta}_{P, g}^\top \mbb{\phi}(z)$.  We start with the decomposition  
	\begin{align*}
		\frac{1}{n} \sum_{i=1}^n \bigl\{\fhat(Z_i) - f_P(Z_i)\bigr\}\bigl\{\ghat(Z_i) - g_P(Z_i)\bigr\} &= \underbrace{\frac{1}{n} \sum_{i=1}^n \{f_P^\dagger(Z_i) - f_P(Z_i)\}\{g_P^\dagger(Z_i) - g_P(Z_i)\}}_{\RN{1}_n}\\
		&+ \underbrace{\frac{1}{n} \sum_{i=1}^n \bigl\{\fhat(Z_i) - f_P^\dagger(Z_i)\bigr\}\bigl\{g_P^\dagger(Z_i) - g_P(Z_i)\bigr\}}_{\RN{2}_n^f}\\
		&+ \underbrace{\frac{1}{n} \sum_{i=1}^n \bigl\{f_P^\dagger(Z_i) - f_P(Z_i)\bigr\}\bigl\{\ghat(Z_i) - g_P^\dagger(Z_i)\bigr\}}_{\RN{2}_n^g}\\
		&+ \underbrace{\frac{1}{n} \sum_{i=1}^n \bigl\{\fhat(Z_i) - f_P^\dagger(Z_i)\bigr\}\bigl\{\ghat(Z_i) - g_P^\dagger(Z_i)\bigr\}}_{\RN{3}_n}.
	\end{align*}
	By Assumption~\emph{(i)},
	\[
		\sup_{P \in \mathcal{P}}|\RN{1}_n| \leq \sup_{P \in \mathcal{P}} \|f_P-f_P^\dagger\|_\infty \| g_P-g_P^\dagger \|_\infty = O\bigl(K^{-(\zeta_f+\zeta_g)}\bigr).
	\]
Using~\eqref{eq: ls estimator mean} in the proof of Proposition~\ref{Proposition: general spline regression} (with $g_P$ and $g_P^\dagger$ in place of $f_P$ and $f_P^\dagger$ there), we deduce that 
	\begin{align*}
		\E_P\bigl((\RN{2}_n^f)^2 \given \fhat\bigr) &= \frac{1}{n} \bigl(\mbb{\widehat{\beta}}_f - \mbb{\beta}_{P, f}\bigr)^\top  \E_P\bigl( \{g_P^\dagger(Z) - g_P(Z)\}^2 \mbb{\phi}(Z)\mbb{\phi}(Z)^\top\bigr) \bigl(\mbb{\widehat{\beta}}_f - \mbb{\beta}_{P, f}\bigr)\\
		&\leq \frac{1}{n} \|g_P^\dagger - g_P\|_\infty^2 \| \mbb{\Sigma}_P \|_{\op} \|\mbb{\widehat{\beta}}_f - \mbb{\beta}_{P, f}\|_2^2 = O_{\mathcal{P}}\bigl( K^{-(2\zeta_g-1)}n^{-2}\bigr),
	\end{align*}
	by our hypothesis on $\|g_P^\dagger - g_P\|_\infty$, Proposition~\ref{Proposition: b-spline properties}(d) and Proposition~\ref{Proposition: general spline regression}. Thus by Lemma~\ref{Lemma: unconditionalisation via Markov},
	\[
	\RN{2}_n^f = O_{\mathcal{P}}(K^{-(\zeta_g-1/2)} n^{-1}) = O_{\mathcal{P}}(K^{1/2} n^{-1}).
	\]
	Similarly,
	\[
	\RN{2}_n^g = O_{\mathcal{P}}(K^{1/2} n^{-1}).
	\]
	To deal with the $\RN{3}_n$ term, define $\mbb{\widehat{\Sigma}} := n^{-1} \sum_{i=1}^n \mbb{\phi}(Z_i)\mbb{\phi}(Z_i)^\top$, as well as $\mbb{\widehat{\Sigma}}_f := n^{-1} \sum_{i=1}^n \mbb{\phi}(Z_i^f)\mbb{\phi}(Z_i^f)^\top$ and $\mbb{\widehat{\Sigma}}_g := n^{-1} \sum_{i=1}^n \mbb{\phi}(Z_i^g)\mbb{\phi}(Z_i^g)^\top$.  For $i \in [n]$, let $\varepsilon_i^f:= Y_i^f - f_P(Z_i^f)$, $\varepsilon_i^g:=X_i^g - g_P(Z_i^g)$, $h_i^f :=  f_P(Z_i^f) - f_P^\dagger(Z_i^f)$ and $h_i^g :=  g_P(Z_i^g) - g_P^\dagger(Z_i^g)$.  We write
	\begin{align*}
		\RN{3}_n &= (\mbb{\widehat{\beta}}_f - \mbb{\beta}_{P, f})^\top \mbb{\widehat{\Sigma}} (\mbb{\widehat{\beta}}_g - \mbb{\beta}_{P, g}) \\
		&= \underbrace{\biggl(\frac{1}{n} \sum_{i=1}^n \varepsilon_i^f \mbb{\phi}(Z_i^f)^\top \biggr) \mbb{\widehat{\Sigma}}_f^{-1}  \mbb{\widehat{\Sigma}} (\mbb{\widehat{\beta}}_g - \mbb{\beta}_{P, g})}_{\RN{3}_n^{(1)}}\\
		&\hspace{3cm}+ \underbrace{\biggl(\frac{1}{n} \sum_{i=1}^n h_i^f \mbb{\phi}(Z_i^f)^\top \biggr) \mbb{\widehat{\Sigma}}_f^{-1} \mbb{\widehat{\Sigma}} \mbb{\widehat{\Sigma}}_g^{-1} \biggl(\frac{1}{n} \sum_{i=1}^n h_i^g \mbb{\phi}(Z_i^g) \biggr)}_{\RN{3}_n^{(2)}} \\
		&\hspace{3cm}+ \underbrace{\biggl(\frac{1}{n} \sum_{i=1}^n h_i^f \mbb{\phi}(Z_i^f)^\top \biggr) \mbb{\widehat{\Sigma}}_f^{-1} \mbb{\widehat{\Sigma}} \mbb{\widehat{\Sigma}}_g^{-1} \biggl(\frac{1}{n} \sum_{i=1}^n \varepsilon_i^g \mbb{\phi}(Z_i^g) \biggr)}_{\RN{3}_n^{(3)}}.
	\end{align*}
	To deal with the $\RN{3}_n^{(1)}$ term, we have using the fact that $\E_P(\varepsilon_1^f \given Z_1^f) = 0$ and $\var_P(Y_1^f \given Z_1^f) = \E\bigl((\varepsilon_1^f)^2 \given Z_1^f\bigr) \leq \sigma_n^2$ that
	\begin{align*}
		\E_P\bigl( (\RN{3}_n^{(1)})^2 \given \mbb{\widehat{\beta}}_g, (Z_i, Z_i^f)_{i=1}^n\bigr)  &\leq \frac{\sigma_n^2}{n}  (\mbb{\widehat{\beta}}_g - \mbb{\beta}_{P, g})^\top \mbb{\widehat{\Sigma}}  \mbb{\widehat{\Sigma}}_f^{-1}   \mbb{\widehat{\Sigma}}_f \mbb{\widehat{\Sigma}}_f^{-1}  \mbb{\widehat{\Sigma}} (\mbb{\widehat{\beta}}_g - \mbb{\beta}_{P, g})\\
		&\leq \frac{\sigma_n^2}{n} \|\mbb{\widehat{\Sigma}}\|_{\op}^2   \|\mbb{\widehat{\Sigma}}_f^{-1} \|_{\op} \|\mbb{\widehat{\beta}}_g - \mbb{\beta}_{P, g} \|_2^2 = O_{\mathcal{P}}(Kn^{-2}),
	\end{align*}
	by our hypothesis on $\sigma_n^2$,~\eqref{Eq: sigma operator norm} in Lemma~\ref{lemma: rudelson lln} and Proposition~\ref{Proposition: general spline regression}.   Hence, by another application of Lemma~\ref{Lemma: unconditionalisation via Markov},
	\[
	\RN{3}_n^{(1)} = O_{\mathcal{P}}(K^{1/2} n^{-1}).
	\]
	To deal with the $\RN{3}_n^{(2)}$ term, by the Cauchy--Schwarz inequality,
	\[
		|\RN{3}_n^{(2)} | \leq \bigl\| \mbb{\widehat{\Sigma}}_f^{-1/2} \mbb{\widehat{\Sigma}} \mbb{\widehat{\Sigma}}_g^{-1/2}  \bigr\|_{\op} \biggl\| \mbb{\widehat{
			\Sigma}}_f^{-1/2} \frac{1}{n} \sum_{i=1}^n h_i^f \mbb{\phi}(Z_i^f) \biggr\|_2 \biggl\|\mbb{\widehat{\Sigma}}_g^{-1/2} \frac{1}{n} \sum_{i=1}^n h_i^g \mbb{\phi}(Z_i^g) \biggr\|_2.
	\]
	By \eqref{Eq: sigma operator norm} in Lemma~\ref{lemma: rudelson lln}, 
	\[
		\bigl\| \mbb{\widehat{\Sigma}}_f^{-1/2} \mbb{\widehat{\Sigma}} \mbb{\widehat{\Sigma}}_g^{-1/2}  \bigr\|_{\op} = O_{\mathcal{P}}(1).
	\]
    The same argument as in \eqref{Eq: spline regression bias term} in the proof of Proposition~\ref{Proposition: general spline regression} now yields that 
    \[
    \RN{3}_n^{(2)} = O_{\mathcal{P}}(K^{-(\zeta_f+\zeta_g-1)}n^{-1}) = O_{\mathcal{P}}(K^{-(\zeta_f+\zeta_g)}).
    \]
	To deal with the $\RN{3}_n^{(3)}$ term we write
	\begin{align*}
	\RN{3}_n^{(3)} &= \underbrace{\biggl(\frac{1}{n} \sum_{i=1}^n h_i^f \mbb{\phi}(Z_i^f)^\top \biggr) \mbb{\Sigma}_P^{-1}  \mbb{\widehat{\Sigma}} \mbb{\widehat{\Sigma}}_g^{-1} \biggl(\frac{1}{n} \sum_{i=1}^n \varepsilon_i^g \mbb{\phi}(Z_i^g) \biggr)}_{\RN{3}_n^{(3, 1)}}\\
	&+ \underbrace{\biggl(\frac{1}{n} \sum_{i=1}^n h_i^f \mbb{\phi}(Z_i^f)^\top \biggr) (\mbb{\widehat{\Sigma}}_f^{-1} - \mbb{\Sigma}_P^{-1})  \mbb{\widehat{\Sigma}} \mbb{\Sigma}_P^{-1} \biggl(\frac{1}{n} \sum_{i=1}^n \varepsilon_i^g \mbb{\phi}(Z_i^g) \biggr)}_{\RN{3}_n^{(3, 2)}}\\
	&+ \underbrace{\biggl(\frac{1}{n} \sum_{i=1}^n h_i^f \mbb{\phi}(Z_i^f)^\top \biggr)(\mbb{\widehat{\Sigma}}_f^{-1} - \mbb{\Sigma}_P^{-1}) \mbb{\widehat{\Sigma}} (\mbb{\widehat{\Sigma}}_g^{-1} - \mbb{\Sigma}_P^{-1}) \biggl(\frac{1}{n} \sum_{i=1}^n \varepsilon_i^g \mbb{\phi}(Z_i^g) \biggr)}_{\RN{3}_n^{(3, 3)}}.
	\end{align*}
	For the first term, we have by an argument similar to the $\RN{3}_n^{(1)}$ term that 
	\begin{align*}
	\E_P\bigl((\RN{3}_n^{(3, 1)})^2 \given (Z_i, Z_i^f, Z_i^g)_{i=1}^n\bigr) &\leq \frac{\sigma_n^2}{n} \biggl\| \frac{1}{n} \sum_{i=1}^n h_i^f \mbb{\phi}(Z_i^f) \biggr\|_2^2 \bigl\| \mbb{\Sigma}_P^{-1} \mbb{\widehat{\Sigma}} \mbb{\widehat{\Sigma}}_g^{-1} \mbb{\widehat{\Sigma}}_g \mbb{\widehat{\Sigma}}_g^{-1} \mbb{\widehat{\Sigma}} \mbb{\Sigma}_P^{-1} \bigr\|_{\op}.
	\end{align*}
	By Lemma~\ref{Lemma: unconditionalisation via Markov}, our assumption on $\sigma_n$,~\eqref{Eq: tighter projection bound} in the proof of Proposition~\ref{Proposition: general spline regression}, Proposition~\ref{Proposition: b-spline properties}(d) and \eqref{Eq: sigma operator norm} in the proof of Lemma~\ref{lemma: rudelson lln}, we therefore have 
	\[
		\RN{3}_n^{(3, 1)} = O_{\mathcal{P}}(K^{1/2-\zeta_f}n^{-1}) = O_{\mathcal{P}}(K^{1/2}n^{-1}).
	\]
	Similarly,
\begin{align*}
		\E_P\bigl((\RN{3}_n^{(3, 2)})^2 \given &(Z_i, Z_i^f, Z_i^g)_{i=1}^n\bigr) \\
		&\leq \frac{\sigma_n^2}{n} \biggl\| \frac{1}{n} \sum_{i=1}^n h_i^f \mbb{\phi}(Z_i^f) \biggr\|_2^2 \bigl\| \mbb{\Sigma}_P^{-1} - \mbb{\widehat{\Sigma}}_f^{-1} \bigr\|_{\op}^2 \bigl\|  \mbb{\widehat{\Sigma}}  \mbb{\Sigma}_P^{-1} \mbb{\widehat{\Sigma}}_g
		\mbb{\Sigma}_P^{-1}
		\mbb{\widehat{\Sigma}} \bigr\|_{\op}.
	\end{align*}
	Hence, by the same arguments as for $\RN{3}_n^{(3, 1)}$, together with  the second result in Lemma~\ref{lemma: rudelson lln},
	\[
	\RN{3}_n^{(3, 2)} = O_{\mathcal{P}}\biggl(\frac{K^{-(\zeta_f-1)} \log^{1/2}(eK)}{n^{3/2}}\biggr) = O_{\mathcal{P}}(K^{1/2}n^{-1}).
	\]
	Finally, by the Cauchy--Schwarz inequality, we have 
	\begin{align*}
		&\bigl|\RN{3}_n^{(3, 3)} \bigr| \leq \|\mbb{\widehat{\Sigma}}\|_{\op} \|\mbb{\widehat{\Sigma}}_f^{-1} - \mbb{\Sigma}_P^{-1} \|_{\op} \bigl\| \mbb{\widehat{\Sigma}}_g^{-1} - \mbb{\Sigma}_P^{-1} \bigr\|_{\op} \biggl\| \frac{1}{n} \sum_{i=1}^n \varepsilon_i^g \mbb{\phi}(Z_i^g) \biggr\|_2 \biggl\| \frac{1}{n} \sum_{i=1}^n h_i^f \mbb{\phi}(Z_i^f) \biggr\|_2,
	\end{align*}
	so
	\[
	\RN{3}_n^{(3, 3)} = O_{\mathcal{P}}\biggl(\frac{K^{2-\zeta_f}\log (eK)}{n^2}\biggr) = O_{\mathcal{P}}\biggl(K^{-(\zeta_f+\zeta_g)} + \frac{K^{2 - \max(\zeta_f,\zeta_g)}\log K}{n^2}\biggr)
	\]
	from our previous bounds.	The result follows.
\end{proof}
\end{proposition}
Our second and final result controls a different type of product error and is loosely based on Theorem~8 and Corollary~9 in a working version of \citet{ichimura2015influence}.
\begin{proposition}\label{proposition: error times smooth function}
	Let $\mathcal{P}$ be a family of distributions of $(Y, Z)$ on $\mathbb{R} \times [0, 1]^d$ with corresponding regression function $f_P:[0,1]^d \rightarrow \mathbb{R}$ given by $f_P(z):=\E_P(Y \given Z=z)$.  Further, let $(g_P)_{P \in \mathcal{P}}$ be a family of functions from $[0, 1]^d$ to $\mathbb{R}$ with $\rho_P := \E_P\bigl(g_P(Z)^2\bigr) < \infty$ and $\inf_{P \in \mathcal{P}} \rho_P > 0$.  Suppose that 
	\begin{enumerate}[(i)]
	\item  There exist $\zeta_f(d, r) \equiv \zeta_f > 0$ and $\zeta_g(d,r) \equiv \zeta_g > 0$ such that
		\[
			\sup_{P \in \mathcal{P}} \|f_P-f_P^\dagger\|_\infty =  O(K^{-\zeta_f}), \quad \sup_{P \in \mathcal{P}}  \|g_P-g_P^\dagger\|_\infty = O(K^{-\zeta_g}),
		\]
		where $f_P^\dagger$ and $g_P^\dagger$ denote the $L_2(P)$-best approximants of $f_P$ and $g_P$ respectively in $\mathcal{S}^d_{r, N}$. 
  		\item Each $P \in \mathcal{P}$ is absolutely continuous with respect to Lebesgue measure on $[0,1]^d$, with corresponding density $p_P$ satisfying $C := \sup_{P \in \mathcal{P}} \sup_{z \in [0,1]^d} p_P(z) < \infty$ and $c := \inf_{P \in \mathcal{P}} \inf_{z \in [0, 1]^d} p_P(z) > 0$.
  		\item There exists a positive sequence $(\sigma_n^2)_{n \in \mathbb{N}}$ such that $\Var(Y \given Z) \leq \sigma_n^2 = O(1)$.
	\end{enumerate}
	Now suppose we are given two independent samples $(Y^f_i, Z^f_i)_{i=1}^n$ and $(Y_i, Z_i)_{i=1}^n$, each consisting of $n$ independent and identically distributed copies of $(Y, Z)$.  Let $\mbb{\phi}$ denote the $d$-tensor B-spline basis of $\mathcal{S}^d_{r, N}$.  Let $\mbb{\widehat{\beta}}$ denote the ordinary least squares estimate from regressing $Y_1^f,\ldots,Y_n^f$ onto $\mbb{\phi}(Z_1^f),\ldots,\mbb{\phi}(Z_n^f)$.  Define the fitted regression function $\fhat$ by $\fhat(z) = \mbb{\widehat{\beta}}^\top \mbb{\phi}(z)$.  If $K \log(K) / n \to 0$, then
	\begin{align*}
		\frac{1}{n} \sum_{i=1}^n g_P(Z_i) &\bigl\{\fhat(Z_i) - f_P(Z_i)\bigr\} \\
		&= O_{\mathcal{P}}\bigl(K^{-(\zeta_f+\zeta_g)}  + K^{-(\zeta_g-1/2)} n^{-1} + \rho_P^{1/2} n^{-1/2} \{1 + K^{-(\zeta_f-1/2)} \} \bigr).
	\end{align*}
	\begin{proof}
	Define $\mbb{\beta}_{P, f}, \mbb{\beta}_{P, g} \in \mathbb{R}^K$ so that $f_P^\dagger(z) = \mbb{\beta}_{P, f}^\top \mbb{\phi}(z)$ and $g_P^\dagger(z) = \mbb{\beta}_{P, g}^\top \mbb{\phi}(z)$.  By Proposition~\ref{Proposition: b-spline properties}(b),
	\[
	\|\mbb{\beta}_{P, g}\|_2 \leq K^{1/2} c_s(r)^{-d} \{ \| g_P^\dagger - g_P\|_\infty + \|g_P\|_2\} \leq K^{1/2} c_s(r)^{-d} \{ \| g_P^\dagger - g_P\|_\infty + \rho_P^{1/2} c^{-1/2} \}.
	\]
	Thus, by~\emph{(i)}, 
	\begin{equation}
		\label{Eq: beta_g bound}
		\|\mbb{\beta}_{P, g}\|_2 = O_{\mathcal{P}}(K^{-(\zeta_g-1/2)} + \rho_P^{1/2}K^{1/2}).
	\end{equation}
	We can write 
	\begin{align*}
		&\frac{1}{n} \sum_{i=1}^n g_P(Z_i) \bigl\{\fhat(Z_i) \!-\! f_P(Z_i)\bigr\} \\
		= &\underbrace{\frac{1}{n} \sum_{i=1}^n \{g_P(Z_i)\!-\!g_P^\dagger(Z_i) \}  \{f_P^\dagger(Z_i) \!-\! f_P(Z_i)\} }_{\RN{1}_n} + \underbrace{\frac{1}{n} \sum_{i=1}^n g_P^\dagger(Z_i) \{f_P^\dagger(Z_i) \!-\! f_P(Z_i)\} }_{\RN{2}_n}\\
		&\hspace{1.0cm}+ \underbrace{\frac{1}{n} \sum_{i=1}^n\{g_P(Z_i) - g_P^\dagger(Z_i)\} \{\fhat(Z_i) - f_P^\dagger(Z_i)\}}_{\RN{3}_n} + \underbrace{\frac{1}{n} \sum_{i=1}^n g_P^\dagger(Z_i) \{\fhat(Z_i) - f_P^\dagger(Z_i)\}}_{\RN{4}_n}.
	\end{align*}
	By assumption~\emph{(i)} again,
	\[
		\sup_{P \in \mathcal{P}}|\RN{1}_n| \leq \sup_{P \in \mathcal{P}} \|f_P-f_P^\dagger\|_\infty \| g_P-g_P^\dagger \|_\infty = O\bigl(K^{-(\zeta_f+\zeta_g)}\bigr).
	\]
	From~\eqref{eq: ls estimator mean} in the proof of Proposition~\ref{Proposition: general spline regression},
	\[
	\E_P\bigl(g_P^\dagger(Z)\{f_P(Z)- f_P^\dagger(Z)\}\bigr) = \E_P\bigl(\{f_P(Z)- f_P^\dagger(Z)\}\mbb{\phi}(Z)^\top \mbb{\beta}_{P,g}\bigr) = 0,
	\]
	and therefore 
	\begin{align*}
	\E_P(\RN{2}_n^2) &= \frac{1}{n} \mbb{\beta}_{P,g}^\top \E_P\bigl(\{f_P(Z)- f_P^\dagger(Z)\}^2 \mbb{\phi}(Z) \mbb{\phi}(Z)^\top \bigr) \mbb{\beta}_{P, g} \\
	&\leq \frac{1}{n} \|f_P-f_P^\dagger \|_\infty^2 \| \mbb{\beta}_{P,g} \|_2^2 \| \mbb{\Sigma}_P \|_{\op}.
	\end{align*}
	Thus
	\[
	    \RN{2}_n = O_{\mathcal{P}}\bigl(K^{-(\zeta_f-\zeta_g)}n^{-1/2} + \rho_P^{1/2}K^{-\zeta_f}n^{-1/2} \bigr) = O_{\mathcal{P}}\bigl(K^{-\zeta_f-\zeta_g} +  \rho_P^{1/2}n^{-1/2}\bigr).
	\]
	The same argument as for the $\RN{2}_n^f$ term in the proof of Proposition~\ref{Proposition: product error} yields that 
	\[
	\RN{3}_n = O_{\mathcal{P}}\bigl(K^{-(\zeta_g-1/2)} n^{-1}\bigr).
	\]
	To deal with $\RN{4}_n$, we write, with quantities defined as in the proof of Proposition~\ref{Proposition: product error},
	\begin{align*}
	\RN{4}_n &= \mbb{\beta}_{P, g}^\top \mbb{\widehat{\Sigma}} (\mbb{\widehat{\beta}}- \mbb{\beta}_{P ,f}) \\
	&= \underbrace{\mbb{\beta}_{P, g}^\top \mbb{\widehat{\Sigma}} \mbb{\widehat{\Sigma}}_f^{-1} \biggl(\frac{1}{n} \sum_{i=1}^n \varepsilon_i^f \mbb{\phi}(Z_i^f) \biggr)}_{\RN{4}_n^{(1)}} + \underbrace{\mbb{\beta}_{P, g}^\top \mbb{\widehat{\Sigma}} \mbb{\widehat{\Sigma}}_f^{-1} \biggl(\frac{1}{n} \sum_{i=1}^n h_i^f \mbb{\phi}(Z_i^f) \biggr)}_{\RN{4}_n^{(2)}}.
	\end{align*}
	Since $\E_P\bigl(\varepsilon_1^f \given Z_1^f\bigr) = 0$ and $\var_P\bigl(Y_1^f \given Z_1^f\bigr) = \E_P\bigl((\varepsilon_1^f)^2 \given Z_1^f\bigr) \leq \sigma_n^2$, we have
	\begin{align*}
	\E_P\bigl((\RN{4}_n^{(1)})^2 \given (Z_i, Z_i^f)_{i=1}^n\bigr) &\leq \frac{\sigma_n^2}{n} \mbb{\beta}_{P, g}^\top \mbb{\widehat{\Sigma}} \mbb{\widehat{\Sigma}}_f^{-1} \mbb{\widehat{\Sigma}}_f \mbb{\widehat{\Sigma}}_f^{-1} \mbb{\widehat{\Sigma}} \mbb{\beta}_{P, g} \\
	&\leq \frac{\sigma_n^2}{n} \| \mbb{\widehat{\Sigma}} \mbb{\widehat{\Sigma}}_f^{-1} \mbb{\widehat{\Sigma}}_f \mbb{\widehat{\Sigma}}_f^{-1} \mbb{\widehat{\Sigma}} \|_{\op} \|\mbb{\beta}_{P, g}\|_2^2,
	\end{align*}
	so
	\[
	\RN{4}_n^{(1)} = O_{\mathcal{P}}\bigl((K^{-\zeta_g} + \rho_P^{1/2})n^{-1/2}\bigr) = O_{\mathcal{P}}(\rho_P^{1/2}n^{-1/2}).
	\]
	Finally, by the Cauchy--Schwarz inequality, 
	\begin{align*}
	|\RN{4}_n^{(2)}| &\leq \| \mbb{\beta}_{P, g}\|_2 \|\mbb{\widehat{\Sigma}} \mbb{\widehat{\Sigma}}_f^{-1} \|_{\op} \biggl\|\frac{1}{n} \sum_{i=1}^n h_i^f \mbb{\phi}(Z_i^f)\biggr\|_2 \\
	&= O_{\mathcal{P}}(K^{-(\zeta_f+\zeta_g)+1/2}n^{-1/2} + \rho_P^{1/2} K^{-(\zeta_f-1/2)}n^{-1/2} ) \\
	&= O_{\mathcal{P}}\bigl(K^{-(\zeta_f+\zeta_g)} + \rho_P^{1/2} K^{-(\zeta_f-1/2)}n^{-1/2}\bigr).
	\end{align*}
	The result follows.
	\end{proof}
\end{proposition}

\section{Univariate linear model analysis} \label{Section: full linear analysis}
In this section we give a more detailed analysis of the setting considered in Section~\ref{Section: Linear models linear projection}. In contrast to the remainder of this paper, we let $\mathcal{D}_1$ contain $n_1$ observations and $\mathcal{D}_2$ contain $n_2$ observations, and we let $2n=n_1+n_2$ for this subsection only.  All limiting statements in this section are interpreted as $\min\{n_1, n_2\} \to 0$. This will facilitate a discussion of the effect of the splitting ratio on the power, and to compare the power of the proposed test more precisely with existing methods. 
To simplify our analysis, we set $\vhat \equiv 1$. We now formally write down the assumption required for the main result of this section (Proposition~\ref{Proposition: asymptotic power expression}).
\begin{assumption} \label{Assumption: univariate X} \normalfont
	Suppose that the family $\mathcal{P}$ of joint distributions $P$ of $(X,Y,Z)$ satisfies the linear model~(\ref{Eq: linear model with univariate X}). Let $\mbb{\eta}_P$ and $\mbb{\theta}_P$ denote the population least squares projections of $X$ on $Z$ and $Y$ on $Z$, respectively. Let $\widehat{\beta}$, $\widehat{\mbb{\eta}}$ and $\mbb{\widehat{\theta}}$ denote estimators of $\beta_P$, $\mbb{\eta}_P$ and $\mbb{\theta}_P$, respectively, where $\widehat{\beta}$ is trained on $\mathcal{D}_2$ and $\widehat{\mbb{\eta}}$ and $\mbb{\widehat{\theta}}$ are trained on $\mathcal{D}_1$.  Assume that $\sigma_{\beta_P}^2 := \lim_{n_2 \rightarrow \infty} \mathrm{Var}_P \bigl(\sqrt{n_2}(\widehat{\beta} - \beta_P)\bigr)$ exists in $(0,\infty)$, and suppose that $\widehat{\beta}$, $\mbb{\widehat{\theta}}$ and $\mbb{\widehat{\eta}}$ satisfy
	\begin{align} 
		\sup_{P \in \mathcal{P}}\sup_{t \in \mathbb{R}} \Big| \prob_{P}\big(\sqrt{n_2} \sigma_{\beta_P}^{-1}(\widehat{\beta} - \beta_P)  \leq t \big) - \Phi(t) \Big| &= o(1), \label{Eq: assumption 1 univariate-X linear model} \\
		\sqrt{n_1} \| \mbb{\widehat{\theta}} - \mbb{\theta}_P\|_2 \cdot  \bigg\| \frac{1}{n_1} \sum_{i \in \mathcal{I}_1} (X_i - \mbb{\eta}_P^\top Z_i) Z_i \bigg\|_2 &= o_{\mathcal{P}}(1), \label{Eq: assumption 2 univariate-X linear model}  \\
		\sqrt{n_1} \| \mbb{\widehat{\eta}} - \mbb{\eta}_P\|_2 \cdot \bigg\| \frac{1}{n_1} \sum_{i \in \mathcal{I}_1} (Y_i - \mbb{\theta}_P^\top Z_i) Z_i \bigg\|_2 &= o_{\mathcal{P}}(1), \label{Eq: assumption 3 univariate-X linear model} \\
		\sqrt{n_1} \| \mbb{\widehat{\eta}} - \mbb{\eta}_P\|_2 \cdot  \| \mbb{\widehat{\theta}} - \mbb{\theta}_P\|_2 \cdot  \bigg\| \frac{1}{n_1} \sum_{i \in \mathcal{I}_1} Z_i Z_i^\top \bigg\|_{\mathrm{op}} &= o_{\mathcal{P}}(1), \label{Eq: assumption 4 univariate-X linear model} \\
		\| \mbb{\widehat{\theta}} - \mbb{\theta}_P\|_2^2 \cdot \frac{1}{n_1} \sum_{i \in \mathcal{I}_1} (X_i - \mbb{\eta}_P^\top Z_i)^2 \|Z_i\|_2^2 &= o_\mathcal{P}(1), \label{Eq: assumption 6 univariate-X linear model}\\
	 \| \mbb{\widehat{\eta}} - \mbb{\eta}_P\|_2^2 \cdot \frac{1}{n_1} \sum_{i \in \mathcal{I}_1} (Y_i - \mbb{\theta}_P^\top Z_i)^2 \|Z_i\|_2^2 &= o_\mathcal{P}(1), \label{Eq: assumption 5 univariate-X linear model}\\
		\| \mbb{\widehat{\eta}} - \mbb{\eta}_P\|_2^2 \cdot \| \mbb{\widehat{\theta}} - \mbb{\theta}_P\|_2^2 \cdot \frac{1}{n_1} \sum_{i \in \mathcal{I}_1} \|Z_i\|_2^4 &= o_\mathcal{P}(1). \label{Eq: assumption 7 univariate-X linear model}
	\end{align}
\end{assumption}

In Section~\ref{Section: Linear models linear projection} we consider a simpler but less general assumption (Assumption~\ref{Assumption: OLS}) that suffices for the analysis when the estimators are OLS estimators. These more general assumptions allow for settings where alternate estimators are used or the dimension is allowed to increase with $n$.

As mentioned before, we set the estimated weight function $\vhat \equiv 1$ in this analysis, which yields $L_{i} = \widehat{\beta}\bigl(Y_i - \mbb{\widehat{\theta}}^\top Z_i\bigr)\bigl(X_i - \mbb{\widehat{\eta}}^\top Z_i\bigr) =: \widehat{\beta} R_{i}$ for $i=1,\ldots,n_1$. The resulting PCM statistic is
\begin{align*} 
	T = \sign(\widehat{\beta}) \frac{\frac{1}{\sqrt{n_1}} \sum_{i \in \mathcal{I}_1} R_{i} }{\sqrt{\frac{1}{n_1} \sum_{i \in \mathcal{I}_1} R_{i}^2 - \bigl( \frac{1}{n_1} \sum_{i \in \mathcal{I}_1} R_{i} \bigr)^2}}.
\end{align*}
To simplify our presentation, we write, $\xi_P := X - \mbb{\eta}_P^\top Z$, $\varepsilon_P := Y - \mbb{\theta}_P^\top Z$, $\sigma_{P,\xi}^2 := \Var_P(\xi_P)$ and $\sigma_{P, \varepsilon\xi}^2 := \var_P(\varepsilon_P \xi_P)$. The following result provides asymptotic size and power expressions for the PCM test in this context.
\begin{proposition} \label{Proposition: asymptotic power expression}
	Suppose that $\mathcal{P}$ is a family of distributions $P$ of $(X,Y,Z)$ for which the estimators $\widehat{\beta}$, $\mbb{\widehat{\eta}}$ and $\mbb{\widehat{\theta}}$ satisfy Assumption~\ref{Assumption: univariate X}. In addition, assume that there exist $c, C, \delta>0$ such that $\sigma_{P,\varepsilon\xi}^2 > c$ and $\E_P \bigl\{\bigl|\varepsilon_P \xi_P \bigr|^{2+\delta}\bigr\} \leq C$ for all $P \in \mathcal{P}$ and $n \in \mathbb{N}$. Then, by letting
	\begin{align*}
		\psi_{P,\alpha,n} := \Phi\biggl( \frac{\sqrt{n_2}\beta_P}{\sigma_{\beta_P}} \biggr) \Phi\biggl( z_\alpha + \frac{\sqrt{n_1} \beta_P \sigma_{P,\xi}^2}{\sigma_{P,\varepsilon\xi}} \biggr) + \Phi\biggl( -\frac{\sqrt{n_2}\beta_P}{\sigma_{\beta_P}} \biggr) \Phi\biggl( z_\alpha - \frac{\sqrt{n_1} \beta_P \sigma_{P,\xi}^2}{\sigma_{P,\varepsilon\xi}} \biggr),
	\end{align*}	
	the power of the PCM test satisfies
	\begin{align*}
		\sup_{P \in \mathcal{P}} \big| \prob_{P}(T > z_{1-\alpha}) - \psi_{P,\alpha,n} \big| \rightarrow 0, \ \text{as $\min\{n_1,n_2\} \rightarrow\infty$.}
	\end{align*}
	Furthermore, when $\alpha < 1/2$, we have $\psi_{P,\alpha,n} \geq \alpha$ and, when $\sigma_{P,\xi}^2/\sigma_{P,\varepsilon\xi} >0$, equality holds if and only if $\beta_P = 0$. 
\end{proposition}
Proposition~\ref{Proposition: asymptotic power expression} confirms that under Assumption~\ref{Assumption: univariate X} and the given moment conditions, our proposed test is asymptotically valid uniformly over the null hypothesis $\mathcal{P}_{0} := \{P \in \mathcal{P}: \beta_P =0\}$. In terms of splitting ratio, a consequence of Proposition~\ref{Proposition: asymptotic power expression}, as stated formally in Corollary~\ref{Cor:PainfulAnalysis} is that in this linear model setting one cannot hope to achieve high power against a local alternative where $\target_P \asymp n^{-1}$ unless $n_1 \asymp n_2$.  
While limited to the linear model, this result nevertheless instills confidence in our choice of balanced splitting ratio, and also reveals that the choice of splitting ratio that maximises the asymptotic power depends on the underlying (unknown) parameters. For this reason, we consider $n_1=n_2$ by default for simplicity.

For the specific class of linear alternatives considered in Proposition~\ref{Proposition: asymptotic power expression}, the asymptotic power of the GCM test~\citep{shah2020hardness} without sample splitting is
\begin{align*}
	\Phi \Biggr( z_{\alpha/2}  + \frac{\sqrt{n_1 + n_2}\beta_P \sigma_{P,\xi}^2}{\sigma_{P,\varepsilon\xi}} \Biggr) + \Phi \Biggr(z_{\alpha/2}  - \frac{\sqrt{n_1 + n_2} \beta_P \sigma_{P,\xi}^2}{\sigma_{P,\varepsilon\xi}} \Biggr).
\end{align*}
Comparing this expression with $\psi_{P,\alpha,n}$, one can see that the GCM test is typically more powerful than the proposed test, but only by a constant factor when $n_1 \asymp n_2$. However, as mentioned earlier, the proposed test can have power against broader alternatives than the GCM test depending on the choice of projection. In comparison with the tests of \cite{williamson2020unified} and \cite{dai2021}, the proposed test achieves higher power. In particular, their tests become powerless whenever $\sqrt{n}\target_P \rightarrow 0$, which is true for both parametric and nonparametric settings. Moreover, as pointed out by \cite{williamson2020unified} and further demonstrated in Section~\ref{Section: misc results} of the supplementary material, their tests via sample splitting may not control the Type I error when $(X,Y,Z)$ are mutually independent. In contrast, our approach does not suffer from this issue and can be powerful even when $\sqrt{n}\target_P \rightarrow 0$, as demonstrated by Proposition~\ref{Proposition: asymptotic power expression}. In the next subsection, we provide the proof of Proposition~\ref{Proposition: asymptotic power expression}.

\subsection{Proof of Proposition~\ref{Proposition: asymptotic power expression}}
Throughout this proof we suppress $P$ subscripts from the notation as in the proofs in Section~\ref{Section: Proofs}. As $L_{i} = \widehat{\beta}\bigl(Y_i - \mbb{\widehat{\theta}}^\top Z_i\bigr)\bigl(X_i - \mbb{\widehat{\eta}}^\top Z_i\bigr) =: \widehat{\beta} R_{i}$ for $i \in [n_1]$, recall that our test statistic is
\begin{align} \label{Eq: linear test statistic}
	T = \sign(\widehat{\beta}) \frac{\frac{1}{\sqrt{n_1}} \sum_{i=1}^{n_1} R_{i} }{\sqrt{\frac{1}{n_1} \sum_{i=1}^{n_1} R_{i}^2 - \bigl( \frac{1}{n_1} \sum_{i=1}^{n_1} R_{i} \bigr)^2}}.
\end{align}
Let
\begin{align*}
	T_{R}:= \frac{\frac{1}{\sqrt{n_1}} \sum_{i=1}^{n_1} (R_{i} - \beta \sigma_{\xi}^2)}{\sigma_{\varepsilon\xi}},
\end{align*}
and for now suppose that the following approximations hold:
\begin{equation} \label{Eq: normality of T_R}
	\sup_{P \in \mathcal{P}}\sup_{t \in \mathbb{R}} \big| \prob\big(T_{R} \leq t \big) - \Phi(t) \big| = o(1) 
\end{equation}
and
\begin{equation}
	\label{Eq: convergence of ratio}
	\Bigg\{\frac{\sigma_{\varepsilon\xi}^2}{\frac{1}{n_1} \sum_{i=1}^{n_1} R_{i}^2 -  \big( \frac{1}{n_1} \sum_{i=1}^{n_1} R_{i}\big)^2}  \Bigg\}^{1/2} = 1 +  o_\mathcal{P}(1).
\end{equation}
Then, by~(\ref{Eq: convergence of ratio}), we have
\begin{align*}
	T &= \biggl\{\sign(\widehat{\beta}) T_{R} + \sign(\widehat{\beta}) \frac{\sqrt{n_1} \beta \sigma_{\xi}^2}{\sigma_{\varepsilon\xi}}\biggr\}(1 + V_n),
\end{align*}
where $V_n$ is a remainder term satisfying $V_n = o_{\mathcal{P}}(1)$.
Define for brevity $s_{\beta} := \sqrt{n_1} \beta \sigma_{\xi}^2/ \sigma_{\varepsilon\xi}$. Now, since $(R_{i})_{i=1}^{n_1}$ and $\sign(\widehat{\beta})$ are formed on independent data and are thus independent, we have
\begin{align*}
	&\sup_{P \in \mathcal{P}} | \pr(T > z_{1-\alpha}) - \psi_{\alpha,n} | \\
	&= \sup_{P \in \mathcal{P}} \Bigl| \pr \Bigl(\sign(\widehat{\beta}) T_{R}  > z_{1-\alpha} - \sign(\widehat{\beta}) s_{\beta} - \sign(\widehat{\beta}) V_n (T_{R}  + s_{\beta})  \Bigr) - \psi_{\alpha,n} \Bigr|\\
	&\leq \sup_{P \in \mathcal{P}} \biggl| \pr(\widehat{\beta} > 0) \pr \Bigl(T_{R}  > z_{1-\alpha} -  s_{\beta} - V_n (T_{R}  + s_{\beta}) \Bigr) - \Phi\biggl( \frac{\sqrt{n_2} \beta}{\sigma_{\beta}} \biggr) \Phi\bigl( z_\alpha +  s_{\beta} \bigr) \biggr| \\
	&\hspace{0.5cm}+ \sup_{P \in \mathcal{P}} \biggl| \pr(\widehat{\beta} < 0) \pr \Bigl(-T_{R} > z_{1-\alpha} +  s_{\beta} + V_n (T_{R}  + s_{\beta}) \Bigr) - \Phi\biggl( \frac{-\sqrt{n_2} \beta}{\sigma_{\beta}} \biggr) \Phi\bigl( z_\alpha - s_{\beta} \bigr) \biggr| \\
	&\hspace{11.5cm} + \sup_{P \in \mathcal{P}} \pr(\widehat{\beta}=0).
\end{align*}
The last term here is $o(1)$ by~\eqref{Eq: assumption 1 univariate-X linear model}.  The first two terms are dealt with similarly, so we only show how to argue that the first term is $o(1)$. To this end, we have 
\begin{align*}
	\sup_{P \in \mathcal{P}} \biggl| \pr(\widehat{\beta} > 0) \pr \Bigl(T_{R}  > z_{1-\alpha} -  s_{\beta} &- V_n (T_{R}  + s_{\beta}) \Bigr) - \Phi\biggl( \frac{\sqrt{n_2} \beta}{\sigma_{\beta}} \biggr) \Phi\bigl( z_\alpha + s_{\beta} \bigr) \biggr| \\
	\leq &\sup_{P \in \mathcal{P}} \biggl| \pr(\widehat{\beta} > 0) - \Phi\biggl( \frac{\sqrt{n_2} \beta}{\sigma_{\beta}} \biggr)\biggr| \\
	+ &  \sup_{P \in \mathcal{P}} \Bigl|  \pr \Bigl(T_{R}  > z_{1-\alpha} -  s_{\beta} - V_n (T_{R}  + s_{\beta}) \Bigr)  - \Phi\bigl( z_\alpha +s_{\beta} \bigr) \Bigr|.
\end{align*}
The first term above is $o(1)$ by \eqref{Eq: assumption 1 univariate-X linear model}. To deal with the second term, for an arbitrary $\epsilon \in (0,1)$, write
\begin{align*}
	\sup_{P \in \mathcal{P}} &\Bigl|  \pr \Bigl(T_{R}  > z_{1-\alpha} -  s_{\beta} - V_n (T_R  + s_{\beta}) \Bigr)  - \Phi\bigl( z_\alpha +s_{\beta} \bigr) \Bigr| \\
	& \leq  \sup_{P \in \mathcal{P}} \Bigl|  \pr \Bigl(T_{R}  > z_{1-\alpha} -  s_{\beta} - V_n (T_R  + s_{\beta}), \, |V_n T_{R}| < \epsilon, \, |V_n| < \epsilon \Bigr)  - \Phi\bigl( z_\alpha +s_{\beta} \bigr) \Bigr|  \\
	&\hspace{6cm}+ \sup_{P \in \mathcal{P}} \pr_P (|V_n| \geq \epsilon) + \sup_{P \in \mathcal{P}} \pr_P (|V_n T_{R}| \geq \epsilon).
\end{align*}
Since $V_n = o_{\mathcal{P}}(1)$ by \eqref{Eq: convergence of ratio} and $V_n T_{R} = o_{\mathcal{P}}(1)$ by Lemma~\ref{Lemma: product of o and O} and \eqref{Eq: normality of T_R}, the last two terms are $o(1)$.  Moreover,
\begin{align*}
	 \sup_{P \in \mathcal{P}} \Bigl|  \pr \Bigl(T_{R}  > &z_{1-\alpha} -  s_{\beta} - V_n (T_{R}  + s_{\beta}), \, |V_n T_{R}| < \epsilon, \, |V_n| < \epsilon \Bigr)  - \Phi\bigl( z_\alpha +s_{\beta} \bigr) \Bigr| \\
	 &\leq \sup_{P \in \mathcal{P}} \max \Bigl\{ \pr \Bigl(T_{R}  > z_{1-\alpha} -  s_{\beta} - \epsilon(1  + |s_{\beta}|) \Bigr)  - \Phi\bigl( z_\alpha +s_{\beta} \bigr), \\
	 &\hspace{4cm} \Phi\bigl( z_\alpha +s_{\beta} \bigr) -  \pr \Bigl(T_{R}  > z_{1-\alpha} -  s_{\beta} + \epsilon(1  +  |s_{\beta}|) \Bigr)   \Bigr\} \\
	 &\leq \underbrace{\sup_{P \in \mathcal{P}}  \Bigl| \pr \Bigl(T_{R}  > z_{1-\alpha} -  s_{\beta} - \epsilon(1  + |s_{\beta}|) \Bigr)  - \Phi\bigl( z_\alpha +s_{\beta} \bigr) \Bigr|}_{\RN{1}_n} + \\
	 &\hspace{3cm} \underbrace{\sup_{P \in \mathcal{P}}  \Bigl| \Phi\bigl( z_\alpha +s_{\beta} \bigr) -  \pr \Bigl(T_{R}  > z_{1-\alpha} -  s_{\beta} + \epsilon(1  +  |s_{\beta}|) \Bigr) \Bigr|}_{\RN{2}_n}.
\end{align*}
We only show that $\RN{1}_n$ is $o(1)$ as $\RN{2}_n$ can be handled similarly. Now letting $W \sim N(0,1)$, we have 
\begin{align*}
	\RN{1}_n &= \sup_{P \in \mathcal{P}} \Bigl|  \pr \Bigl(T_{R}  > z_{1-\alpha} -  s_{\beta} - \epsilon(1  +  |s_{\beta}|) \Bigr)  - \pr\bigl( W > z_{1-\alpha} - s_{\beta} \bigr) \Bigr| \\
	&\leq \sup_{P \in \mathcal{P}}  \Bigl|  \pr \Bigl(T_{R}  > z_{1-\alpha} -  s_{\beta} - \epsilon(1  +  |s_{\beta}|) \Bigr)  - \pr\bigl( W > z_{1-\alpha} -  s_{\beta} - \epsilon(1  +  |s_{\beta}|) \bigr) \Bigr| \\
	&\hspace{5cm}+ \sup_{P \in \mathcal{P}} \Bigl|  \pr \Bigl(W  \in \bigl(z_{1-\alpha} -  s_{\beta} - \epsilon(1  +  |s_{\beta}|), z_{1-\alpha} -  s_{\beta}\bigr] \Bigr). 
\end{align*}
By the asymptotic normality of $T_{R}$ in \eqref{Eq: normality of T_R}, the first term is $o(1)$. On the other hand, for the second term
\begin{align*}
	\sup_{P \in \mathcal{P}} \Bigl|  &\pr \Bigl(W  \in \bigl(z_{1-\alpha} -  s_{\beta} - \epsilon(1  +  |s_{\beta}|), z_{1-\alpha} -  s_{\beta}\bigr]  \\
	&\leq \sup_{P \in \mathcal{P}} \min \biggl\{ \frac{1}{\sqrt{2\pi}} (1 + |s_{\beta}|)\epsilon, \, \Phi\bigl( z_{1-\alpha} - s_{\beta} \bigr), \,  \Phi\bigl( z_{\alpha} +  s_{\beta} + \epsilon(1  +  |s_{\beta}|)\bigr)  \biggr\}.
\end{align*}
We now analyse the upper bound depending on the sign of $\beta$. 

\medskip

\noindent \textbf{Case (i)} Suppose that $\beta > 0$ and $\Phi(z_{1-\alpha} - s_{\beta}) \leq \epsilon/\sqrt{2\pi}$. Then 
\begin{align*}
	  \min \biggl\{ \frac{1}{\sqrt{2\pi}} (1 + |s_{\beta}|)\epsilon, \, \Phi\bigl( z_{1-\alpha} - s_{\beta} \bigr), \,  \Phi\bigl( z_{\alpha} +  s_{\beta} + \epsilon(1  +  |s_{\beta}|) \bigr)  \biggr\} 
	 \leq \frac{\epsilon}{\sqrt{2\pi}}.
\end{align*}
On the other hand, if $\Phi(z_{1-\alpha} - s_{\beta}) > \epsilon/\sqrt{2\pi}$, then $0 < s_{\beta} < z_{1-\alpha} - \Phi^{-1}\bigl(\frac{\epsilon}{\sqrt{2\pi}}\bigr)$. Thus
\begin{align*}
	\min \biggl\{ \frac{1}{\sqrt{2\pi}} (1 + |s_{\beta}|)\epsilon, \, \Phi\bigl( z_{1-\alpha} - s_{\beta} \bigr)&, \,  \Phi\bigl( z_{\alpha} +  s_{\beta} + \epsilon(1  +  |s_{\beta}|) \bigr)  \biggr\}  \\
	\leq  & \frac{\epsilon}{\sqrt{2\pi}}\biggl\{1 + z_{1-\alpha} - \Phi^{-1}\biggl(\frac{\epsilon}{\sqrt{2\pi}}\biggr)\biggr\}.
\end{align*}

\medskip

\noindent \textbf{Case (ii)} Next suppose that $\beta < 0$ and $\Phi\bigl( z_{\alpha} +  s_{\beta} + \epsilon(1  +  |s_{\beta}|) \bigr) \leq \epsilon/\sqrt{2\pi}$. Then 
\begin{align*}
	\min \biggl\{ \frac{1}{\sqrt{2\pi}} (1 + |s_{\beta}|)\epsilon, \, \Phi\bigl( z_{1-\alpha} - s_{\beta} \bigr), \,  \Phi\bigl( z_{\alpha} +  s_{\beta} + \epsilon(1  +  |s_{\beta}|) \bigr) \biggr\}  \leq  \frac{\epsilon}{\sqrt{2\pi}}.
\end{align*}
On the other hand, if $\Phi\bigl( z_{\alpha} +  s_{\beta} + \epsilon(1  +  |s_{\beta}|) \bigr) > \epsilon/\sqrt{2\pi}$, then
\begin{align*}
	\frac{1}{1-\epsilon} \biggl\{\Phi^{-1}\biggl(\frac{\epsilon}{\sqrt{2\pi}}\biggr) - z_{\alpha} - \epsilon \biggr\} < s_{\beta} < 0.
\end{align*}
Thus
\begin{align*}
	\min \biggl\{ \frac{1}{\sqrt{2\pi}} (1 + |s_{\beta}|)\epsilon, \, \Phi\bigl( z_{1-\alpha} - s_{\beta} \bigr), \,  &\Phi\bigl( z_{\alpha} +  s_{\beta} + \epsilon(1  +  |s_{\beta}|) \bigr)  \biggr\}  \\
	\leq  & \frac{\epsilon}{\sqrt{2\pi}}\biggl[1 + \frac{1}{1-\epsilon} \biggl\{z_{\alpha} + \epsilon -\Phi^{-1}\biggl(\frac{\epsilon}{\sqrt{2\pi}}\biggr) \biggr\} \biggr].
\end{align*}

\medskip

\noindent \textbf{Case (iii)} When $\beta =0$, we have 
\begin{align*}
	\min \biggl\{ \frac{1}{\sqrt{2\pi}} (1 + |s_{\beta}|)\epsilon, \, \Phi\bigl( z_{1-\alpha} - s_{\beta} \bigr), \,  \Phi\bigl( z_{\alpha} +  s_{\beta} + \epsilon(1  +  |s_{\beta}|) \bigr)  \biggr\}  \leq  \frac{\epsilon}{\sqrt{2\pi}}.
\end{align*}
Combining the previous results we have for every $\beta \in \mathbb{R}$ that
\begin{align*}
	 \sup_{P \in \mathcal{P}} &\min \biggl\{ \frac{1}{\sqrt{2\pi}} (1 + |s_{\beta}|)\epsilon, \, \Phi\bigl( z_{1-\alpha} - s_{\beta} \bigr), \,  \Phi\bigl( z_{\alpha} +  s_{\beta} + \epsilon(1  +  |s_{\beta}|) \bigr)  \biggr\}  \\
	 &\leq \frac{\epsilon}{\sqrt{2\pi}} \max \biggl\{ 1, \,  1 + z_{1-\alpha} - \Phi^{-1} \biggl(\frac{\epsilon}{\sqrt{2\pi}}\biggr),\,  1 + \frac{1}{1-\epsilon} \biggl[z_{\alpha} + \epsilon -\Phi^{-1}\biggl(\frac{\epsilon}{\sqrt{2\pi}}\biggr) \biggr] \biggr\}.
\end{align*}
We further note that the bound $\mathbb{P}(W \geq x) \leq (1/2) \cdot e^{-x^2/2}$ for $x \geq 0$ gives for $\epsilon < (1/2) \cdot \sqrt{2\pi}$ that
\begin{align*}
	-\epsilon \Phi^{-1}\biggl(\frac{\epsilon}{\sqrt{2\pi}}\biggr) \leq \epsilon \sqrt{2 \log \biggl(\frac{1}{2(1 - \epsilon/\sqrt{2\pi})}}\biggr) \rightarrow 0,
\end{align*}
as $\epsilon \rightarrow 0$. We deduce that $\RN{1}_n = o(1)$, so the first claim of the proposition will follow once we establish~\eqref{Eq: normality of T_R} and~\eqref{Eq: convergence of ratio}. 

For the claim~\eqref{Eq: normality of T_R}, consider the decomposition 
\begin{align*}
	T_{R} &=  \underbrace{\frac{\sigma_{\varepsilon\xi}^{-1}}{\sqrt{n_1}} \sum_{i=1}^{n_1}\bigl\{(Y_i - \mbb{\theta}^\top Z_i )(X_i - {\mbb{\eta}}^\top Z_i) - \beta \sigma_{\xi}^2\bigr\}}_{T^{(1)}} \\
	&- \underbrace{\frac{\sigma_{\varepsilon\xi}^{-1}}{\sqrt{n_1}} \sum_{i=1}^{n_1} (Y_i - \mbb{\theta}^\top Z_i) ({\mbb{\widehat{\eta}} - \mbb{\eta}})^\top Z_i}_{T^{(2)}} - \underbrace{\frac{\sigma_{\varepsilon\xi}^{-1}}{\sqrt{n_1}} \sum_{i=1}^{n_1} (\mbb{\widehat{\theta}} - \mbb{\theta})^\top Z_i  (X_i - \mbb{\eta}^\top Z_i)}_{T^{(3)}}\\
	&+ \underbrace{\frac{\sigma_{\varepsilon\xi}^{-1}}{\sqrt{n_1}}  \sum_{i=1}^{n_1} (\mbb{\widehat{\theta}} - \mbb{\theta})^\top Z_i  Z_i^\top ( {\mbb{\widehat{\eta}} - \mbb{\eta}})}_{T^{(4)}} .
\end{align*}
By the assumption that $\E \bigl\{\bigl|(Y - \mbb{\theta}^\top Z)(X - \mbb{\eta}^\top Z)\bigr|^{2+\delta}\bigr\} \leq C$, \citet[][Lemma~18]{shah2020hardness} yields that
\begin{align*}
	\sup_{P \in \mathcal{P}}\sup_{t \in \mathbb{R}} \big| \prob\big(T^{(1)} \leq t \big) - \Phi(t) \big| \rightarrow 0. 
\end{align*}
Moreover,
\[
|T^{(2)}| \leq c \sqrt{n_1} \| \mbb{\widehat{\eta}} - \mbb{\eta} \|_2 \biggl\| \frac{1}{n_1} \sum_{i=1}^{n_1} (Y_i - \mbb{\theta}^\top Z_i)Z_i \biggr\|_2 = o_\mathcal{P}(1),
\]
by Cauchy--Schwarz, the assumption that $\sigma_{P,\varepsilon\xi} > c$ and \eqref{Eq: assumption 3 univariate-X linear model}.  We can argue similarly that $T^{(3)} = o_{\mathcal{P}}(1)$ using \eqref{Eq: assumption 2 univariate-X linear model}. Finally,
\[
|T^{(4)}| \leq c \sqrt{n_1} \| \mbb{\widehat{\eta}} - \mbb{\eta} \|_2 \| \mbb{\widehat{\theta}} - \mbb{\theta} \|_2 \biggl\| \frac{1}{n_1} \sum_{i=1}^{n_1} Z_i Z_i^{\top}  \biggr\|_{\mathrm{op}} = o_\mathcal{P}(1)
\]
by similar arguments as above and~\eqref{Eq: assumption 4 univariate-X linear model}. Combining the above with the uniform version of Slutsky's theorem, we have the desired claim~\eqref{Eq: normality of T_R}. \\
To prove \eqref{Eq: convergence of ratio}, we let $\widetilde{R}_{n,i} := R_{n,i} - \beta \sigma_{\xi}^2$ for $i \in [n_1]$ and note that
\begin{align*}
	\frac{1}{n_1} \sum_{i=1}^{n_1} R_{i}^2 - \biggl( \frac{1}{n_1} \sum_{i=1}^{n_1} R_{i}\biggr)^2 
	= \frac{1}{n_1} \sum_{i=1}^{n_1} \widetilde{R}_{n,i}^2 - \biggl( \frac{1}{n_1} \sum_{i=1}^{n_1} \widetilde{R}_{n,i} \biggr)^2 = \frac{1}{n_1} \sum_{i=1}^{n_1} \widetilde{R}_{n,i}^2  + o_\mathcal{P}(1),
\end{align*}
where the second equality follows from the proof of~\eqref{Eq: normality of T_R} above. To ease the notation further, for $i \in [n_1]$, we write  
\begin{align*}
	\widetilde{R}_{n,i} &= \underbrace{(Y_i - \mbb{\theta}^\top Z_i)(X_i - \mbb{\eta}^\top Z_i ) - \beta \sigma_{\xi}^2}_{\RN{1}_{i}} - \underbrace{ (\mbb{\widehat{\theta}} - \mbb{\theta})^\top Z_i(X_i - \mbb{\eta}^\top Z_i)}_{\RN{2}_{i}} \\
	&\hspace{3cm}- \underbrace{(Y_i - \mbb{\theta}^\top Z_i) ( \mbb{\widehat{\eta}} - {\mbb{\eta}})^\top Z_i }_{\RN{3}_{i}} + \underbrace{(\mbb{\widehat{\theta}} - \mbb{\theta})^\top Z_i Z_i^\top  ( \mbb{\widehat{\eta}} - {\mbb{\eta}})}_{\RN{4}_{i}} .
\end{align*}
Then
\begin{align*}
	\frac{1}{n_1} \sum_{i=1}^{n_1} \widetilde{R}_{n,i}^2 &= \frac{1}{n_1} \sum_{i=1}^{n_1} \RN{1}_{i}^2 +  \frac{1}{n_1} \sum_{i=1}^{n_1} \RN{2}_{i}^2 +  \frac{1}{n_1} \sum_{i=1}^{n_1} \RN{3}_{i}^2 +  \frac{1}{n_1} \sum_{i=1}^{n_1} \RN{4}_{i}^2 \\
	&\hspace{2cm}- \frac{2}{n_1} \sum_{i=1}^{n_1} \RN{1}_{i} \RN{2}_{i} - \frac{2}{n_1} \sum_{i=1}^{n_1} \RN{1}_{i} \RN{3}_{i} +  \frac{2}{n_1} \sum_{i=1}^{n_1} \RN{1}_{i} \RN{4}_{i} \\
	&\hspace{2cm}+  \frac{2}{n_1} \sum_{i=1}^{n_1} \RN{2}_{i} \RN{3}_{i} -  \frac{2}{n_1} \sum_{i=1}^{n_1} \RN{2}_{i} \RN{4}_{i} -  \frac{2}{n_1} \sum_{i=1}^{n_1} \RN{3}_{i} \RN{4}_{i}.
\end{align*}
By the assumption that $\E \bigl\{\bigl|(Y - \mbb{\theta}^\top Z)(X - \mbb{\eta}^\top Z)\bigr|^{2+\delta}\bigr\} \leq C$, \citet[][Lemma~19]{shah2020hardness} yields that $\sigma_{\varepsilon\xi}^{-2} n_1^{-1}\sum_{i=1}^{n_1} \RN{1}_{i}^2  = 1 + o_\mathcal{P}(1)$.  Moreover, by Cauchy--Schwarz,
\begin{align*}
	\frac{1}{n_1} \sum_{i=1}^{n_1} \RN{2}_{i}^2 \leq  \|\mbb{\widehat{\theta}} - \mbb{\theta}\|_2^2 \cdot \frac{1}{n_1} \sum_{i=1}^{n_1} (X_i - \mbb{\eta}^\top Z_i)^2 \|Z_i\|_2^2,
\end{align*}
so \eqref{Eq: assumption 6 univariate-X linear model} together with $\sigma_{\varepsilon\xi}^2 > c$ implies that $\frac{\sigma_{\varepsilon\xi}^{-2}}{n_1} \sum_{i=1}^{n_1} \RN{2}_{i}^2 = o_\mathcal{P}(1)$.  Similarly, 
\[
\frac{\sigma_{\varepsilon\xi}^{-2}}{n_1} \sum_{i=1}^{n_1} \RN{3}_{i}^2 = o_\mathcal{P}(1)
\]
by~\eqref{Eq: assumption 5 univariate-X linear model}. By two applications of Cauchy--Schwarz, we have 
\[
\frac{1}{n_1} \sum_{i=1}^{n_1} \RN{4}_{i}^2 \leq  \| \mbb{\widehat{\theta}} - \mbb{\theta}\|_2^2 \| \mbb{\widehat{\eta}} - \mbb{\eta}\|_2^2
\cdot \frac{1}{n_1} \sum_{i=1}^{n_1} \| Z_i \|_2^4,  
\]
so~\eqref{Eq: assumption 7 univariate-X linear model} combined with the lower bound on $\sigma_{\varepsilon\xi}^2$ yields that $\frac{\sigma_{\varepsilon\xi}^{-2}}{n_1} \sum_{i=1}^{n_1} \RN{4}_{i}^2 = o_\mathcal{P}(1)$. Turning to the cross-product terms, by Cauchy--Schwarz and the previous analysis,
\begin{align*}
	\biggl| \frac{\sigma_{\varepsilon\xi}^{-2}}{n_1} \sum_{i=1}^{n_1} \RN{1}_{i} \RN{2}_{i} \biggr| \leq \sigma_{\varepsilon\xi}^{-2} \biggl(\frac{1}{n_1} \sum_{i=1}^{n_1} \RN{1}_{i}^2 \biggr)^{1/2} \biggl( \frac{1}{n_1} \sum_{i=1}^{n_1} \RN{2}_{i}^2 \biggr)^{1/2} = o_\mathcal{P}(1).
\end{align*}
The other terms can be similarly analysed and shown to be $o_\mathcal{P}(1)$. We have thus established by the uniform version of Slutsky's theorem that 
\begin{align*}
	\frac{1}{\sigma_{\varepsilon\xi}^2}\biggl\{\frac{1}{n_1} \sum_{i=1}^{n_1} R_{i}^2 - \biggl( \frac{1}{n_1} \sum_{i=1}^{n_1} R_{i}\biggr)^2 \biggr\} = 1 + o_\mathcal{P}(1).
\end{align*}
Finally,~\eqref{Eq: convergence of ratio} follows by the above result combined with Lemma~\ref{Lemma: uniform convergence in probability under continuous transformation}. This completes the proof of the first claim in Proposition~\ref{Proposition: asymptotic power expression}. \\
To prove the second claim, let us assume that $\beta \geq 0$ (the case $\beta < 0$ can be handled very similarly), and denote
\begin{align*}
	\psi_{\alpha,n} &= \underbrace{\Phi\biggl( \frac{\sqrt{n_2}\beta}{\sigma_{\beta}} \biggr)}_{V(\beta)} \cdot \underbrace{\Phi\biggl( z_\alpha + \frac{\sqrt{n_1} \beta \sigma_{\xi}^2}{\sigma_{\varepsilon\xi}} \biggr)}_{W_1(\beta)} + \underbrace{\Phi\biggl( -\frac{\sqrt{n_2}\beta}{\sigma_{\beta}} \biggr)}_{1-V(\beta)} \cdot \underbrace{\Phi\biggl( z_\alpha - \frac{\sqrt{n_1} \beta \sigma_{\xi}^2}{\sigma_{\varepsilon\xi}} \biggr)}_{W_2(\beta)} \\[.5em]
	&= W_2(\beta) + V(\beta)\{W_1(\beta) - W_2(\beta)\}.
\end{align*}
Then $ V(\beta) \geq 1/2$ and $W_1(\beta) - W_2(\beta) \geq 0$, so $\psi_{\alpha,n} \geq W_2(\beta) + \{W_1(\beta) - W_2(\beta)\}/2 = W_1(\beta)/2 + W_2(\beta)/2$. \\
Next observe that the function $\delta \mapsto \Phi(z_\alpha + \delta)/2 +  \Phi(z_\alpha - \delta)/2$ is continuous on $\mathbb{R}$, and when $\alpha < 1/2$, it is decreasing when $\delta < 0$ and increasing when $\delta > 0$. It follows that 
\begin{align*}
	\psi_{\alpha,n} \geq \frac{1}{2} \Phi\biggl( z_\alpha + \frac{\sqrt{n_1} \beta \sigma_{\xi}^2}{\sigma_{\varepsilon\xi}} \biggr) + \frac{1}{2}\Phi\biggl( z_\alpha - \frac{\sqrt{n_1} \beta \sigma_{\xi}^2}{\sigma_{\varepsilon\xi}} \biggr) \geq \alpha,
\end{align*}
and when $\sigma_{\xi}^2/\sigma_{\varepsilon\xi} > 0$, we have equality in both inequalities if and only if $\beta = 0$.

\section{Additional simulations: GAMs with binary responses}
\label{Section: additive models binary}
Here we consider settings similar to those considered in Section~\ref{Section: additive models}, but with $Y$ binary. Our null settings use
\[
\prob(Y = 1 \given X, Z) = \mathrm{expit}(\sin(2\pi Z_1)),
\]
and we consider three alternative settings mirroring those in Section~\ref{Section: additive models}:
\begin{enumerate}
	\item $\prob(Y = 1 \given X, Z) = \mathrm{expit}(\sin(2\pi Z_1)  + 0.25X^2)$,
	\item $\prob(Y = 1 \given X, Z) = \mathrm{expit}(\sin(2\pi Z_1)  + 0.5X^2)$,
	\item $\prob(Y = 1 \given X, Z) = \mathrm{expit}(\sin(2\pi Z_1)  + 0.5X^2 Z_2)$.
\end{enumerate}

For all regressions with binary responses, we fit a binomial generalised additive model with logistic link, and we use additive models for all other regressions; we use the implementations in the R package \texttt{mgcv} \citep{wood2017}.  We do not fit a new regression model for $\widetilde{v}$ but instead utilize that $Y$ is binary and set $\widetilde{v}(x,z) := \ghat(x, z)(1-\ghat(x, z))$. The additive models are tuned as in Section~\ref{Section: additive models} while the binomial additive models use half as many basis functions i.e.\ $\lfloor (N-1)/(2d) \rfloor$ (where $N$ and $d$ are the number of observations and predictors on which the model is trained, respectively) to avoid issues with convergence of the generalised additive model fits.
The results can be seen in Figure~\ref{fig:gam binary comparison} and are broadly in line with those in Section~\ref{Section: additive models} with the PCM performing favourably though being powerless in Setting 3 with pure interactions (as to be expected), and \texttt{wgsc} and most notably \texttt{gam} not maintaining Type~I error control.
\begin{figure}
    \centering
    \includegraphics[scale=0.44]{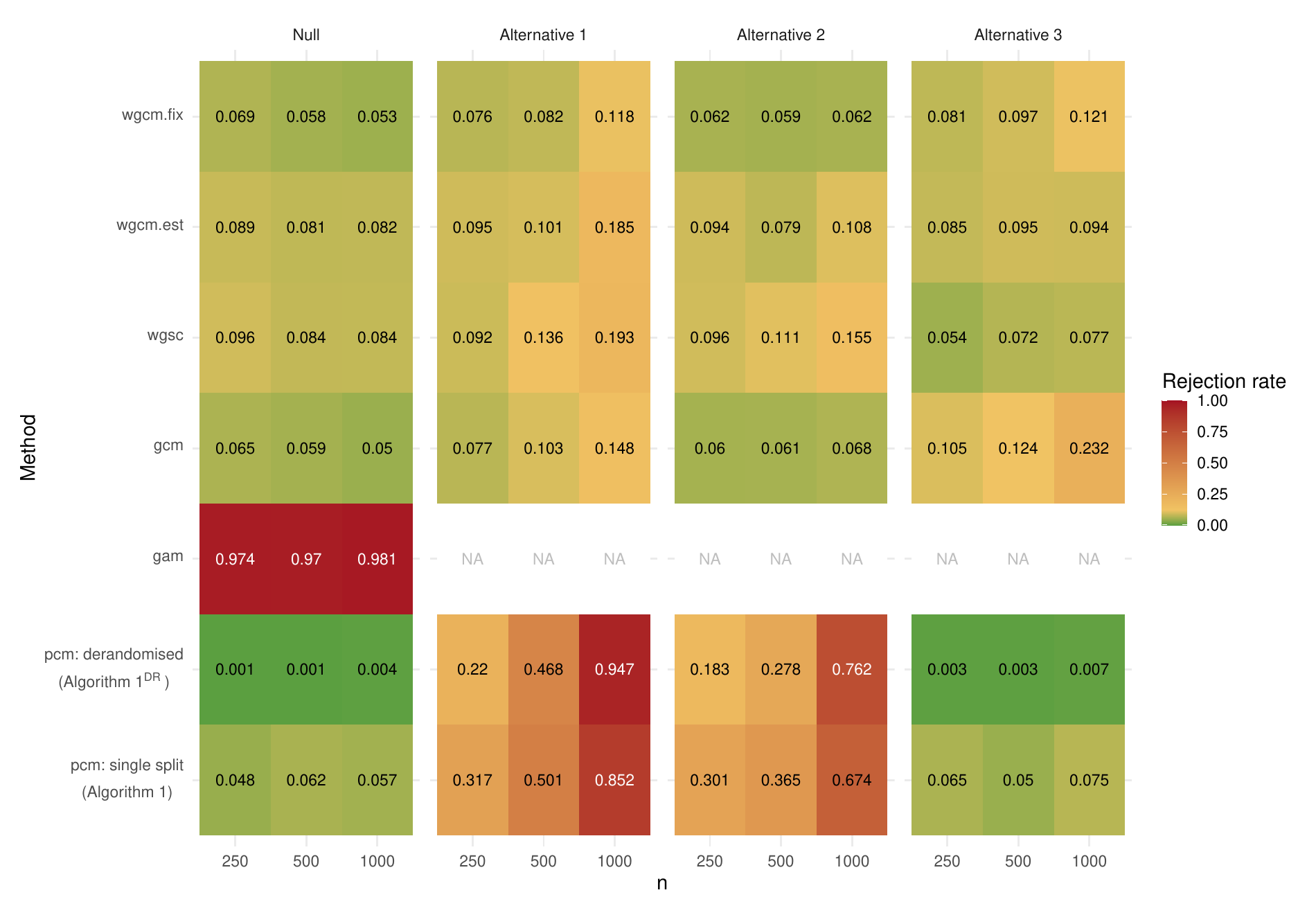}
    \caption{Rejection rates in the various settings considered in Section~\ref{Section: additive models binary} for nominal 5\%-level tests.}
    \label{fig:gam binary comparison}
\end{figure}

\end{document}